\documentclass[reqno]{amsart}

\usepackage{xcolor}
\usepackage{amssymb,amsmath,amsthm,amsfonts}
\usepackage{latexsym}
\usepackage{bbm}
\usepackage{mathrsfs}
\usepackage{hyperref}
\usepackage{makecell}
\usepackage{multirow}
\usepackage{diagbox}
\usepackage{braket}
\usepackage{enumitem}

\makeatletter\renewcommand*{\eqref}[1]{%
	\hyperref[{#1}]{\textup{\tagform@{\!\!\ref*{#1}}}}%
}\makeatother 

\makeatletter

\@addtoreset{equation}{section}
\makeatother
\theoremstyle{plain}
\newtheorem{theorem}{Theorem}[section]
\newtheorem{lemma}[theorem]{Lemma}
\newtheorem{proposition}[theorem]{Proposition}
\newtheorem{corollary}[theorem]{Corollary}
\theoremstyle{definition}

\newtheorem{remark}[theorem]{Remark}

\newcommand{\les}{\lesssim}

\def\ac{\mathop{\mathrm{ac}}\nolimits}

\def\BMO{{\mathop{\mathrm{BMO}}}}
\def\R{{\mathbb{R}}}

\def\<{{\langle}}
\def\>{{\rangle}}

\makeatletter
\renewcommand{\l@section}{\@tocline{1}{0pt}{1em}{2.5pc}{}}
\renewcommand{\l@subsection}{\@tocline{2}{0pt}{2em}{2pc}{}}
\makeatother

\title[Endpoint Mapping Properties]
{Endpoint Mapping Properties of Wave Operators for
Two-Dimensional Schr\"odinger Operators}
\author{Han Cheng, Changxing Miao, Xiaohua Yao}

\address {Han Cheng, Institute of Applied Physics and Computational Mathematics, Beijing, 100088, People's Republic of China}
\email{chmathh@163.com}

\address {Changxing Miao, School of Mathematics and Physics,
University of Science and Technology Beijing,
Beijing 100083, China}
\email{miao\_changxing@ustb.edu.cn}

\address {Xiaohua Yao, School of Mathematics and Statistics, Key Laboratory of Nonlinear Analysis and Applications (Ministry of Education), Central China Normal University, Wuhan, 430079, P.R. China}
\email{yaoxiaohua@ccnu.edu.cn}
\date{\today}

\begin{document}
	\begin{abstract}
We establish sharp endpoint mapping properties for the wave operators 
$W_\pm(H,-\Delta)$ of two-dimensional Schr\"odinger operators 
$H=-\Delta+V$ with real-valued decaying potentials $V$. Together with 
the known non-endpoint $L^p$ theory, our results give a complete 
classification of the $L^p$ mapping properties of the two-dimensional 
wave operators, and reveal an unexpected reversal of the usual threshold 
paradigm at the endpoints $p=1$ and $p=\infty$.

When zero is a regular point of $H$, the wave operators fail to be 
bounded on $L^1(\R^2)$ and on $L^\infty(\R^2)$, but they satisfy the 
natural substitute estimates of Calder\'on--Zygmund type:
$$
L^1(\R^2)\longrightarrow L^{1,\infty}(\R^2),
\qquad
\mathcal{H}^1(\R^2)\longrightarrow L^1(\R^2),
\qquad
L^\infty(\R^2)\longrightarrow \mathrm{BMO}(\R^2).
$$
When zero is instead a threshold singularity of the first kind---an 
s-wave resonance with no other threshold obstruction, the wave 
operators are bounded on both endpoint spaces $L^1(\R^2)$ and  $L^\infty(\R^2)$. Thus, in dimension two, 
an s-wave resonance improves the endpoint 
behavior of the wave operators, in sharp contrast with dimensions 
$n\ge3$, where the only regular case is the favorable one. The reversal is 
explained by a simple principle that unifies the known endpoint results 
in all dimensions: the favorable case for endpoint boundedness is the 
one in which $H$ and $H_0$ share the same threshold type, and the 
two-dimensional free operator $-\Delta$ itself has a first-kind 
threshold singularity, with the constant function as its s-wave 
resonance.

We also determine the endpoint behavior in the remaining zero-energy 
spectral configurations of $H$. A p-wave resonance obstructs both the 
$L^1$- and the $L^\infty$-boundedness of the wave operators, while in 
the zero-eigenvalue case we obtain necessary and sufficient conditions 
for endpoint boundedness, expressed in terms of the presence of s- and 
p-wave resonances and of explicit second-order harmonic moment 
cancellations satisfied by the zero-energy eigenfunctions.

The proofs combine sharp asymptotics for oscillatory integrals with 
logarithmic symbols at low energy with new physical-space kernel 
estimates at high energy, the latter yielding in particular the 
boundedness of the high-energy component at both endpoints.
\end{abstract}
	\maketitle
	
	\tableofcontents  
\section{Introduction}
\subsection{Background and motivations}
Let $H_0=-\Delta$ and $H=H_0+V$ be the free and perturbed Schr\"odinger operators, respectively, where $V$ is a real-valued potential. In the framework of time-dependent scattering theory, if $|V(x)|\les \langle x\rangle^{-\delta}$ for some $\delta>1$ (an Agmon potential), then $(0,\infty)$ is contained in the absolutely continuous spectrum of $H$, and the wave operators, defined as the strong limits in $L^2(\R^n)$,
\[
W_\pm=W_\pm(H,H_0)
:=
\operatorname*{s-lim}_{t\to\pm\infty}
e^{itH}e^{-itH_0},
\]
exist and are asymptotically complete (see, e.g., Agmon \cite{Agmon} or Reed and Simon \cite{Reed-Simon-III}). Moreover, they satisfy the identities
\[
W_\pm W_\pm^{*}=P_{\ac}(H),
\qquad
W_\pm^{*}W_\pm=I,
\]
where $P_{\ac}(H)$ denotes the projection onto the absolutely continuous spectral subspace of $H$. In this work, we consider the $L^p$ theory of wave operators.

The $L^p$ mapping properties of wave operators are of interest from several complementary perspectives. A first motivation comes from the intertwining property. For a bounded Borel function $m$, one has
\begin{equation}\label{eq:intro-intertwining}
	m(H)P_{\ac}(H)
	=
	W_\pm m(H_0)W_\pm^*.
\end{equation}
Consequently, estimates for the free Fourier multiplier $m(-\Delta)$ may be transferred to the perturbed operator:
\[
\bigl\|m(H)P_{\ac}(H)\bigr\|_{L^p\to L^q}
\leq
\|W_\pm\|_{L^q\to L^q}
\|m(-\Delta)\|_{L^p\to L^q}
\|W_\pm^*\|_{L^p\to L^p}.
\]
This transference principle applies to spectral multipliers, Sobolev estimates, and evolution operators arising in the Schr\"odinger, wave, and Klein--Gordon equations. In particular, the $L^\infty$ bound of the wave operators is precisely what is needed to transfer free $L^1$--$L^\infty$ dispersive estimates directly through \eqref{eq:intro-intertwining}.

A second, structural perspective is provided by the explicit factorization of wave operators. The wave operators admit the representation
\begin{equation}\label{eq:intro-distorted-Fourier}
	W_\pm=\mathcal{F}_\mp^{*}\,\mathcal{F},\qquad\mathcal{F}_\mp=\mathcal{F} W_\pm^*
\end{equation}
where $\mathcal{F}$ denotes the ordinary Fourier transform and $\mathcal{F}_\pm$ the distorted Fourier transform associated with $H$ (see \cite{Agmon,Kuroda}). Identity \eqref{eq:intro-distorted-Fourier} reveals that the $L^p$ theory of wave operators provides a natural paradigm for investigating the mapping properties of distorted Fourier transforms. This perspective has proved particularly fruitful in the analysis of nonlinear dispersive equations with external potentials, where the distorted Fourier transform serves as the natural substitute for the flat Fourier transform; see, for instance, Chen--Pusateri \cite{Chen-Pusat}, Donninger--Krieger \cite{Donninger-Krie}, Germain--Pusateri--Rousset \cite{Ger-Pu-Rou} and Pusateri--Soffer \cite{Pusa-Soffer} on modified scattering and long-time asymptotics for nonlinear equations with potentials, and the references therein.

The study of the $L^p$ boundedness of wave operators was initiated by Yajima in his seminal works \cite{Yajima-JMSJ-95, Yajima-1995-even} for spatial dimensions $n\ge3$. Over the past three decades, this topic has been extensively studied by many authors, and the results depend sensitively on the spatial dimension and on the presence or absence of resonances or eigenfunctions at zero energy; we review these developments in detail in Subsection~\ref{known results} below.

Despite this extensive body of work, the following endpoint problems have remained widely open, to the best of our knowledge, even under strong decay and smoothness assumptions on $V$.
\begin{itemize}
	\item[($a$)] In $\R^2$, determining the $L^p$-boundedness or unboundedness of wave operators at the endpoints $p=1$ and $p=\infty$ under all possible spectral conditions at zero.
	\item[($b$)] In $\R^4$, determining the $L^1$-unboundedness of wave operators when $H$ has a resonance.
	\item[($c$)] In dimensions $n\ge 3$, characterizing necessary and sufficient conditions for the $L^\infty$-boundedness of wave operators when $H$ has only zero-energy eigenfunctions.
\end{itemize}
Among these, the two-dimensional problem $(a)$ constitutes the largest remaining piece 
of the $L^p$ theory of wave operators and has remained open for a long time. Its exceptional difficulty comes from the 
criticality of dimension two: it is the only dimension in which the spatial dimension 
coincides with the order of the operator. As a consequence, the free Green function is 
logarithmic rather than polynomially decaying as in higher dimensions, and the free 
operator already carries a threshold singularity. Two-dimensional Schr\"odinger operators therefore exhibit the richest 
threshold structure among all dimensions, and problem $(a)$ has resisted all prior 
approaches 
\cite{Erdogan-Goldberg-Green-JFA-2018,Goldberg-Green-2016,Yajima-2016-3d,Yajima-2022}: 
the obstruction is structural, and it persists in both the low-energy and the 
high-energy regimes (see Subsection~\ref{sec:outline}). The purpose of this paper is 
to settle problem $(a)$.

\subsection{Main results}
The endpoint behavior of the wave operators depends on the spectral condition of $H$ at zero energy. To state our results, we first recall the precise notions of threshold obstructions in $\R^2$ (see, e.g., \cite{Erdogan-Goldberg-Green-JFA-2018, JN01}).

A zero-energy eigenfunction of $H$ is an $L^2$-solution of the equation
\begin{equation}\label{eq:Hil equation}
	H\psi=-\Delta \psi + V\psi = 0
\end{equation}
in the sense of distributions. A \emph{threshold resonance} is a distributional solution of \eqref{eq:Hil equation} that is bounded but not square-integrable. Such resonances are classified into two types. An \emph{s-wave resonance} is a solution $\psi\in L^\infty(\mathbb{R}^2)$ of \eqref{eq:Hil equation} that admits a decomposition $\psi=c+\Lambda$, where $c$ is a nonzero constant and $\Lambda\in L^p(\mathbb{R}^2)$ for some $p\in(2,\infty)$. A \emph{p-wave resonance} is a solution $\psi$ of \eqref{eq:Hil equation} with $\psi \in L^p(\mathbb{R}^2) \setminus L^2(\mathbb{R}^2)$ for all $2 < p \le \infty$.

We refer collectively to threshold resonances and zero-energy eigenfunctions as \emph{threshold obstructions}. According to the spectral behavior of $H$ at zero energy, we distinguish the following four configurations:
\begin{itemize}
	\item If $H$ has neither a threshold resonance nor a zero-energy eigenvalue, then zero is called a \emph{regular point} of $H$.
	
	\item If $H$ admits an s-wave resonance but neither a p-wave resonance nor a zero-energy eigenfunction, then zero is called a \emph{first-kind threshold singularity}.
	
	\item If $H$ admits a p-wave resonance but no zero-energy eigenfunction, then zero is called a \emph{second-kind threshold singularity}.
	
	\item If $H$ admits a nontrivial zero-energy eigenfunction, then zero is called a \emph{zero-energy eigenvalue} of $H$.
\end{itemize}

Throughout the paper, $L^{1,\infty}(\R^2)$ denotes the weak $L^1$ space, $\mathcal{H}^1(\R^2)$ denotes the real Hardy space, and $\BMO(\R^2)$ denotes its dual, the space of functions of bounded mean oscillation. We write $\langle x\rangle=(1+|x|^2)^{1/2}$, and $a+$, respectively $a-$, denotes $a+\varepsilon$, respectively $a-\varepsilon$, for some arbitrarily small fixed $\varepsilon>0$.

The main objective of this paper is to determine the endpoint behavior of the two-dimensional wave operators according to the spectral structure at zero. Our first main result reveals a sharp and somewhat unexpected distinction between the regular and first-kind threshold cases.

\begin{theorem}\label{thm-main result-1}
	Let $H=-\Delta+V$ be a Schr\"odinger operator on $\R^2$, and suppose that $|V(x)|\lesssim \langle x\rangle^{-6-}$. Then the following statements hold.
	
	\medskip
	
	\noindent
	\emph{(i)}
	If zero is a regular point of $H$, then the wave operators $W_\pm(H,-\Delta)$ are bounded on neither $L^1(\R^2)$ nor $L^\infty(\R^2)$. They nevertheless extend to bounded operators with the mapping properties
	\[
	L^1(\R^2)\to L^{1,\infty}(\R^2),\quad
	\mathcal{H}^1(\R^2)\to L^1(\R^2),\quad\text{and}\quad
	L^\infty(\R^2)\to \BMO(\R^2).
	\]
	
	\medskip
	
	\noindent
	\emph{(ii)}
	If zero is a first-kind threshold singularity of $H$, then $W_\pm(H,-\Delta)$ are bounded on both $L^1(\R^2)$ and $L^\infty(\R^2)$.
	
	By real interpolation with the $L^2$ bound, together with duality, it follows that $W_\pm(H,-\Delta)$ are bounded on $L^p(\R^2)$ for all $1<p<\infty$ in the regular case, and for all $1\le p\le\infty$ in the first-kind threshold case.
\end{theorem}

Along the way, we also establish the endpoint boundedness of the high-energy component of the two-dimensional wave operators (Proposition~\ref{pro:bound for high}), which was previously known only in the non-endpoint range $1<p<\infty$; see Yajima \cite{Yajima-99-CMP}.

Our second theorem addresses the remaining two spectral conditions at zero energy in two dimensions.

\begin{theorem}\label{thm-main result-2}
	Let $H=-\Delta+V$
	be a Schr\"odinger operator on $\R^2$, and suppose that
	$|V(x)|\lesssim \langle x\rangle^{-8-}$.
	Then the following statements hold.
	
	\medskip
	
	\noindent
	\emph{(i)}
	If zero is a second-kind threshold singularity of $H$, then $W_\pm(H,-\Delta)$ are unbounded on both $L^1(\R^2)$ and $L^\infty(\R^2)$.
	
	\medskip
	
	\noindent
	\emph{(ii)}
	Suppose that zero is a zero-energy eigenvalue of $H$. Then $W_\pm(H,-\Delta)$ are bounded on $L^1(\R^2)$ if and only if $H$ has an s-wave resonance but no p-wave resonance.  $W_\pm(H,-\Delta)$ are bounded on $L^\infty(\R^2)$ if and only if $H$ has an s-wave resonance but no p-wave resonance and every zero-energy eigenfunction $\psi$ satisfies
	\begin{equation}\label{eq:condi_for_eigen-0}
		\int_{\R^2} x_1x_2\,V(x)\psi(x)\,dx=0
		\quad\text{and}\quad
		\int_{\R^2} \bigl(x_1^2-x_2^2\bigr)V(x)\psi(x)\,dx=0.
	\end{equation}
    Here, \eqref{eq:condi_for_eigen-0} is equivalent to
	\begin{equation*}
		\int_{\R^2} P(x)\,V(x)\psi(x)\,dx = 0
		\qquad\text{for all }P\in\mathcal{A}_2,
	\end{equation*}
	where $\mathcal{A}_2$ denotes the space of homogeneous harmonic polynomials of degree $2$ on $\R^2$.
\end{theorem}

Theorem~\ref{thm-main result-2} shows that the $L^1$ endpoint detects the presence of a p-wave resonance: such a resonance produces a non-integrable leading contribution to the wave operators. In the zero-eigenvalue case, however, the mere absence of a p-wave resonance is not sufficient: an s-wave resonance must also be present in order to provide the cancellation required for $L^1$ boundedness.
The $L^\infty$ endpoint is even more sensitive, as it further depends on the second-order harmonic moments of the zero-energy eigenfunctions.

Taken together, Theorems~\ref{thm-main result-1} and \ref{thm-main result-2} show that the endpoint mapping properties are governed by a subtle interaction between the type of threshold obstruction and the moment cancellations encoded in the generalized eigenfunctions. In conjunction with the known non-endpoint $L^p$ results reviewed in Subsection~\ref{known results} below, they yield a complete classification of the $L^p$ mapping properties of the two-dimensional wave operators: under the decay assumptions above, $W_{\pm}$ fail to be bounded on $L^p(\mathbb{R}^2)$ if and only if at least one of the following conditions holds:
\begin{itemize}
	\item $H$ has no s-wave resonance and $p\in\{1,\infty\}$;
	\smallskip
	\item $H$ has a p-wave resonance and $p\in\{1\}\cup(2,\infty]$;
	\smallskip
	\item $H$ has a zero-energy eigenfunction for which \eqref{eq:condi_for_eigen-0} fails, and $p=\infty$.
\end{itemize}
The classification is summarized in Table~\ref{tab:Lp-classification}.

\begin{table}[ht]
	\centering
	\small
	\renewcommand{\arraystretch}{1.4}
	\setlength{\tabcolsep}{5pt}
	\setlength{\arrayrulewidth}{0.6pt}
	\begin{tabular}{
			|>{\centering\arraybackslash}p{0.16\textwidth}
			|>{\centering\arraybackslash}p{0.28\textwidth}
			|c|c|c|c|}
		\hline
		\textbf{Zero type}
		& \textbf{Subcase}
		& $\boldsymbol{p=1}$
		& $\boldsymbol{1<p\le 2}$
		& $\boldsymbol{2<p<\infty}$
		& $\boldsymbol{p=\infty}$
		\\ \hline
		Regular
		& ---
		& U & B & B & U
		\\ \hline
		First kind
		& ---
		& B & B & B & B
		\\ \hline
		Second kind
		& ---
		& U & B & U & U
		\\ \hline
		\multirow{4}{*}{Eigenvalue}
		& no s-wave and no p-wave resonance
		& U & B & B & U
		\\ \cline{2-6}
		& presence of a p-wave resonance
		& U & B & U & U
		\\ \cline{2-6}
		& s-wave, no p-wave, and \eqref{eq:condi_for_eigen-0} holds
		& B & B & B & B
		\\ \cline{2-6}
		& s-wave, no p-wave, and \eqref{eq:condi_for_eigen-0} fails
		& B & B & B & U
		\\ \hline
	\end{tabular}
	\caption{The $L^p$-mapping properties of the two-dimensional Schr\"odinger wave operators. Here $\mathrm{B}$ and $\mathrm{U}$ stand for boundedness and unboundedness, respectively.}
	\label{tab:Lp-classification}
\end{table}

We make some comments on Theorems~\ref{thm-main result-1} and~\ref{thm-main result-2}.

\begin{remark}[{\it Reversal of the threshold paradigm in two dimensions}]
A striking feature of Theorems~\ref{thm-main result-1} and~\ref{thm-main result-2} 
is that the usual threshold paradigm is reversed in dimension two. In dimensions 
$n\ge3$, the regular case is the favorable one: threshold resonances or zero-energy eigenvalues typically 
obstruct the endpoint boundedness of the wave operators. In dimension two, however, the opposite occurs. If zero is a 
regular point, then $W_\pm$ are unbounded on both $L^1(\R^2)$ and $L^\infty(\R^2)$; 
if zero is a first-kind threshold singularity, then $W_\pm$ are bounded on both 
endpoint spaces. Thus, in dimension two, the presence of an s-wave resonance improves 
the endpoint behavior of the wave operators, instead of destroying it.

Upon closer inspection, however, this reversal is entirely natural. The underlying principle is the same in 
every dimension: the wave operators compare the perturbed evolution with the free 
one, and their endpoint boundedness favors the case in which 
$H$ and $H_0$ share the same threshold type.  In dimension two, the free operator $H_0=-\Delta$ already has a first-kind 
threshold singularity at zero, with the constant function as its s-wave resonance. 
Consequently, when $H$ has the same threshold type, the low-energy behavior of the 
two evolutions matches at the threshold, and the term responsible for the endpoint 
failure, namely the operator $\mathcal W_{\mathrm{sing}}^{\mathrm L}$ of 
Subsection~\ref{sec:outline}, disappears from the expansion. When zero is a regular 
point of $H$, this matching is lost, and the surviving singular term destroys the 
endpoint bounds. 

The same principle also explains the known results in all other 
dimensions. For $n\ge3$, zero is a regular point of $H_0$, so the principle favors 
the regular case, which is exactly what the higher-dimensional theory shows. For 
$n=1$, the free operator again has a threshold resonance, and the endpoint bounds 
hold only in the resonant case, under Weder's condition \eqref{eq:Jost-endpoint}, 
which requires the zero-energy Jost solution of $H$ to coincide with the free one at 
$-\infty$. In this sense, dimension two reverses the conclusions of the threshold 
paradigm but not the paradigm itself: it is the same principle, seen in the only 
dimension where the free operator itself has a first-kind threshold singularity.

This improvement is remarkable also because it is specific to the wave operators in dimension two. For problems involving $H$ alone, there is no free counterpart to match, 
and no such phenomenon occurs. For instance, the time decay 
of the Schr\"odinger group, Erdo\u{g}an--Green \cite{Erd-Gre-13-cmp} and Murata 
\cite{Murata-82} proved that, when zero is a regular point,
\[
\bigl\|e^{itH}P_{\mathrm{ac}}(H)\bigr\|_
{\mathbb B(L^2_\sigma(\R^2),L^2_{-\sigma}(\R^2))}
\les
\frac{1}{|t|\log^2|t|},
\qquad
t\ge2,\quad \sigma>3,
\]
while in the presence of a first-kind threshold singularity, the weighted decay 
deteriorates to $|t|^{-1}$. Moreover, the dispersive estimate
\begin{equation*}
	\bigl\|e^{itH}P_{\mathrm{ac}}(H)\bigr\|_
	{\mathbb B(L^1(\R^2),L^\infty(\R^2))}
	\les |t|^{-1}
\end{equation*}
holds in both cases \cite{Erd-Gre-13,Schlag-2D-dis}, with the same decay rate as the 
free propagator $e^{itH_0}$. Hence, for the Schr\"odinger group, a first-kind 
threshold singularity is never helpful and may even be harmful, in sharp contrast 
with its effect on the wave operators.	

We believe that the principle behind this phenomenon---namely, that the 
endpoint boundedness of the wave operators favors the case in which $H$ and 
the comparison operator share the same threshold type---is not specific to 
potential perturbations of the Laplacian, but should govern the endpoint 
behavior of wave operators for other classes of perturbations as well.
\end{remark}

\begin{remark}[{\it Examples of the first-kind threshold singularity}]
	In two dimensions, there are many potentials for which $H$ possesses an s-wave resonance but neither p-wave resonances nor zero-energy eigenfunctions.  For instance, let $\Lambda$ be a positive Schwartz function and define
	\[
	V_\varepsilon = \frac{\varepsilon \Delta \Lambda}{1 + \varepsilon \Lambda}, \qquad \varepsilon > 0.
	\]
	For each $\varepsilon > 0$, the operator $H_\varepsilon = -\Delta + V_\varepsilon$ possesses an s-wave resonance $c+\varepsilon \Lambda$. We claim that $H_\varepsilon$ has no p-wave resonances or zero-energy eigenfunctions for sufficiently small $\varepsilon$. Indeed, by \cite[Lemma~6.4]{JN01}, if $\psi$ denotes either a p-wave resonance or a zero-energy eigenfunction, then with $v = \sqrt{|V_\varepsilon|}$ and $U = \operatorname{sgn}(V_\varepsilon)$, the function $f = U v \psi$ belongs to $L^2(\mathbb{R}^2)$ and satisfies
	\begin{equation}\label{eq:the identity trans}
		U f - \frac{1}{2\pi} v \int_{\mathbb{R}^2} \log(|x-y|)\,(v f)(y) \, dy = 0
	\end{equation}
	together with
	\[
	\psi = \frac{1}{2\pi} \int_{\mathbb{R}^2} \log(|x-y|)\,(v f)(y) \, dy.
	\]
	The integral operator with kernel $v(x)\log(|x-y|)v(y)$ is Hilbert--Schmidt, and a direct computation shows that its Hilbert--Schmidt norm is $\mathcal{O}(\varepsilon)$ as $\varepsilon \to 0$. Hence, for sufficiently small $\varepsilon$, equation \eqref{eq:the identity trans} forces $f\equiv 0$, and consequently $\psi\equiv 0$.
\end{remark}

\subsection{Known results}\label{known results}

We briefly review the known results on the $L^p$-boundedness of Schr\"odinger wave operators.
A \emph{zero-energy resonance} is a distributional solution of \eqref{eq:Hil equation} 
belonging to $L^2_{-1/2-}(\mathbb{R}^n)\setminus L^2(\mathbb{R}^n)$ in dimension $n=1, 3$, 
and to $L^2_{-0-}(\mathbb{R}^4)\setminus L^2(\mathbb{R}^4)$ in dimension $n=4$.

$\bullet$ \textbf{Dimension $n=1$.}
Weder \cite{Weder-99-CMP} proved, under suitable decay assumptions on $V$, that the wave operators $W_\pm$ are bounded on $L^p(\mathbb R)$ for every $1<p<\infty$, both in the regular case and in the presence of a zero-energy resonance. He further showed that, if zero is resonant and the zero-energy Jost solution satisfies
\begin{equation}\label{eq:Jost-endpoint}
	\lim_{x\to-\infty} f_1(x,0)=1,
\end{equation}
then $W_\pm$ are bounded on $L^1(\mathbb R)$ and $L^\infty(\mathbb R)$. Galtbayar and Yajima \cite{Gal-Yaj-00} and D'Ancona and Fanelli \cite{AF06} subsequently established $L^p$-boundedness for $1<p<\infty$ under alternative, and in part weaker, assumptions on $V$.

Weder's low-energy decomposition exhibits Hilbert-transform terms in the remaining cases, which suggests the failure of strong endpoint boundedness. The actual unboundedness on $L^1(\mathbb R)$ and $L^\infty(\mathbb R)$ in the regular case, as well as in the resonant case where \eqref{eq:Jost-endpoint} fails, was established only recently by Huang and Yao \cite{Huang-Yao-2026}. Together with Weder's positive endpoint result, this completes the classification of the $L^p$-boundedness of the one-dimensional wave operators under the assumptions considered there. We recall that, for potentials satisfying $|V(x)|\les\langle x\rangle^{-2-}$, zero cannot be an $L^2$-eigenvalue in dimension one; see, for example, \cite{Deift-T}.


$\bullet$ \textbf{Dimension $n=2$.}
When zero is a regular point of $H$, Yajima \cite{Yajima-99-CMP} and Jensen--Yajima \cite{Jensen-Yajima-02-CMP} proved that $W_\pm$ are bounded on $L^p(\mathbb R^2)$ for every $1<p<\infty$. Erdo\u{g}an, Goldberg, and Green \cite{Erdogan-Goldberg-Green-JFA-2018} showed that the same conclusion remains valid if zero is a first-kind threshold singularity, or a zero-energy eigenvalue with no threshold resonance present.

Yajima \cite{Yajima-2022} subsequently obtained a sharp classification in the non-endpoint range: $W_\pm$ are bounded on $L^p(\mathbb R^2)$ for all $1<p<\infty$ if and only if $H$ has no p-wave resonance. If a p-wave resonance is present, then $W_\pm$ are bounded for $1<p\le 2$ and unbounded for $2<p<\infty$. Prior to the present work, the endpoint cases $p=1$ and $p=\infty$ remained unresolved for general decaying potentials.


$\bullet$ \textbf{Dimension $n=3$.}
If zero is a regular point, Yajima \cite{Yajima-JMSJ-95} established the boundedness of $W_\pm$ on $L^p(\mathbb R^3)$ for the full range $1\le p\le\infty$. Beceanu and Schlag \cite{Beceanu-Schlag-20} later obtained the same conclusion under substantially weaker assumptions on the potential.

If a zero-energy resonance is present, Yajima \cite{Yajima-2016,Yajima-2016-3d} proved that $W_\pm$ are bounded on $L^p(\mathbb R^3)$ for $1<p<3$, but are unbounded on $L^1(\mathbb R^3)$ and on $L^p(\mathbb R^3)$ for $3\le p\le\infty$.

Suppose instead that zero is an eigenvalue but no zero-energy resonance is present. Then $W_\pm$ are bounded on $L^p(\mathbb R^3)$ for $1\le p<3$. Moreover, they are bounded for $3\le p<\infty$ if and only if every zero-energy eigenfunction $\psi$ satisfies
\begin{equation}\label{vanishing-condi}
	\int_{\mathbb R^3} x^\alpha V(x)\psi(x)\,dx=0,
	\qquad |\alpha|\le 1.
\end{equation}
If this condition fails, $W_\pm$ are unbounded on $L^p(\mathbb R^3)$ for $3\le p<\infty$. Yajima \cite{Yajima-2016-3d} and Erdo\u{g}an-Green-Lamaster \cite{Erd-Gre-arxiv} further proved that $W_\pm$ are bounded on $L^\infty(\mathbb R^3)$ provided that condition \eqref{vanishing-condi} holds for all $|\alpha|\le 2$ and for every zero-energy eigenfunction $\psi$.


$\bullet$ \textbf{Dimension $n=4$.}
When zero is a regular point, Yajima \cite{Yajima-1995-even} proved that $W_\pm$ are bounded on $L^p(\mathbb R^4)$ for all $1\le p\le\infty$. If a zero-energy resonance is present, Yajima \cite{Yajima-2022-4} showed that $W_\pm$ are bounded on $L^p(\mathbb R^4)$ for $1<p\le 2$ and unbounded for $2<p\le\infty$. 

Suppose that zero is an eigenvalue but no resonance is present. The results of Jensen--Yajima \cite{Jensen-Yajima-08-PLMS} and Goldberg--Green \cite{Goldberg-Green-2017} imply that $W_\pm$ are bounded on $L^p(\mathbb R^4)$ for $1\le p<4$ in general. If every zero-energy eigenfunction $\psi$ satisfies the analogue of condition \eqref{vanishing-condi} in $\mathbb R^4$, then $W_\pm$ are bounded on $L^p(\mathbb R^4)$ for every $1\le p<\infty$. Conversely, Yajima \cite{Yajima-2022-4} proved that, if this first-moment condition fails, then $W_\pm$ are unbounded on $L^p(\mathbb R^4)$ for $4<p<\infty$.
Erdo\u{g}an-Green-Lamaster \cite{Erd-Gre-arxiv} obtained the sufficient condition for $L^\infty$-boundedness.

$\bullet$ \textbf{Dimensions $n\ge 5$.}
In these dimensions, non-$L^2$ zero-energy resonances do not occur. If zero is a regular point, the wave operators are bounded on $L^p(\mathbb R^n)$ throughout the full range $1\le p\le\infty$; see \cite{Finco-Yajima-2006-even,Yajima-JMSJ-95,Yajima-1995-even,Yajima-JMSJ-2006}.

If zero is an eigenvalue, Finco--Yajima \cite{Finco-Yajima-2006-even,Yajima-JMSJ-2006} and Goldberg--Green \cite{Goldberg-Green-2016} proved that $W_\pm$ are bounded on $L^p(\mathbb R^n)$ for $1\le p<n/2$ without additional cancellation assumptions. If every zero-energy eigenfunction $\psi$ satisfies
\[
\int_{\mathbb R^n} V(x)\psi(x)\,dx=0,
\]
then $W_\pm$ are bounded for $1\le p<n$. If, in addition,
\[
\int_{\mathbb R^n} x_j V(x)\psi(x)\,dx=0,
\qquad j=1,\ldots,n,
\]
then they are bounded for every $1\le p<\infty$. In \cite{Erd-Gre-arxiv}, the authors obtain $L^\infty$-boundedness under the condition 
$\int_{\mathbb{R}^n} x^{\alpha} V(x)\psi(x)\,dx=0$ for all $|\alpha|\le 2$.  Yajima \cite{Yajima-2016} further showed that these moment conditions are also necessary in the corresponding upper ranges: the zeroth-moment condition is necessary for $n/2\le p<n$, while the zeroth- and first-moment conditions are necessary for $n\le p<\infty$.

$\bullet$ \textbf{Higher order Schr\"odinger operators.}
Recent progress has also been made on the $L^p$-boundedness of wave operators for higher-order Schr\"odinger operators $(-\Delta)^m+V$, $m\geq2$. Goldberg and Green \cite{GG21} treated the three-dimensional fourth-order case, while Erdo\u{g}an and Green \cite{EG22,EG23} extended the analysis to higher dimensions. Further results for the low-dimensional fourth-order case were obtained by Mizutani, Wan and Yao \cite{MWY23, MWY23-1, MWY23-2} and Galtbayar and Yajima \cite{GY24, GY26}. More recently, the effects of threshold singularities for general higher-order Schr\"odinger operators were studied in \cite{Erd-Gre-arxiv, CSWY25}.

\subsection{Ideas of the proof}\label{sec:outline}
We outline the main ideas underlying the proofs of Theorems~\ref{thm-main result-1} and~\ref{thm-main result-2}. In the course of the proofs we introduce three new ingredients: a physical-space treatment of the two-dimensional high-energy component, based on a hyperbolic parametrization of the phase; sharp asymptotics for oscillatory integrals and Fourier multipliers with logarithmic symbols; and, in the eigenvalue case, a spherical-harmonic decomposition of the low-energy singular component, which identifies the second-order harmonic moments of $V\psi$ as the precise obstruction to the endpoint bounds. We highlight the role of each ingredient below.

By the symmetry $\overline{W_+f}=W_-\overline{f}$, it suffices to study $W_-$. Our proof begins with the stationary representation of the wave operators (see, e.g., Kuroda \cite{Kuroda}):
\begin{equation}\label{eq-stationary formula}
	\begin{aligned}
		W_-(H,-\Delta)f
		&=
		f-\frac{1}{\pi i}
		\int_0^\infty
		\lambda R^+(\lambda^2)V
		\bigl(R_0^+(\lambda^2)-R_0^-(\lambda^2)\bigr)f
		\,d\lambda  \\
		&=:f-\mathcal Wf ,
	\end{aligned}
\end{equation}
where the limiting resolvent operators $R^\pm(\lambda^2)$ and $R_0^\pm(\lambda^2)$ are defined by
\[
R^\pm(\lambda^2)
=
\lim_{\varepsilon\to0^+}
\bigl(H-(\lambda^2\pm i\varepsilon)\bigr)^{-1},
\qquad
R_0^\pm(\lambda^2)
=
\lim_{\varepsilon\to0^+}
\bigl(-\Delta-(\lambda^2\pm i\varepsilon)\bigr)^{-1}.
\]

Let $\chi$ be a smooth cut-off function satisfying
\begin{equation}\label{cut-off-fun}
	\chi(\lambda)=1
	\quad\text{for }|\lambda|\leq1/2,
	\quad
	\chi(\lambda)=0
	\quad\text{for }|\lambda|\geq1, \quad \widetilde\chi(\lambda)=1-\chi(\lambda).
\end{equation}
We now decompose $\mathcal W$ into its low- and high-energy components:
\begin{equation*}
	\begin{aligned}
		\mathcal W^{\mathrm L}f
		&=
		\frac{1}{\pi i}
		\int_0^\infty
		\lambda R^+(\lambda^2)V
		\bigl(R_0^+(\lambda^2)-R_0^-(\lambda^2)\bigr)f\,
		\chi(\lambda/\lambda_0)\,d\lambda,\\
		\mathcal W^{\mathrm H}f
		&=
		\frac{1}{\pi i}
		\int_0^\infty
		\lambda R^+(\lambda^2)V
		\bigl(R_0^+(\lambda^2)-R_0^-(\lambda^2)\bigr)f\,
		\widetilde\chi(\lambda/\lambda_0)\,d\lambda,
	\end{aligned}
\end{equation*}
where $\lambda_0>0$ is a fixed suitable parameter.

From the stationary representation \eqref{eq-stationary formula} we have the decomposition
\[
W_-=I-\mathcal W^{\mathrm H}-\mathcal W^{\mathrm L}.
\]
Here the identity operator $I$ is trivially bounded on every $L^p$. Consequently, the endpoint behavior of $W_-$ on $L^1$ and $L^\infty$ is determined by that of the high-energy component $\mathcal W^{\mathrm H}$ and the low-energy component $\mathcal W^{\mathrm L}$. Theorems~\ref{thm-main result-1} and~\ref{thm-main result-2} are established once the following five propositions are proved.

The positive results read as follows.

\begin{proposition}\label{pro:bound for high}
	Suppose the potential satisfies $|V(x)|\les \langle x\rangle^{-\beta}$ with $\beta>\frac72$. Then, for any fixed $\lambda_0>0$, $\mathcal W^{\mathrm H}$ is bounded on $L^p(\R^2)$ for all $1\le p\le\infty$.
\end{proposition}

\begin{proposition}\label{pro:bound for low energy}
	Suppose that $|V(x)|\les\langle x\rangle^{-6-}$. Then the following statements hold.
	
	\smallskip
	\noindent
	\emph{(i)} If zero is a regular point of $H$, then $\mathcal W^{\mathrm L}$ is bounded from $\mathcal H^1(\R^2)$ to $L^1(\R^2)$, from $L^1(\R^2)$ to $L^{1,\infty}(\R^2)$, and from $L^\infty(\R^2)$ to $\BMO(\R^2)$.
	
	\smallskip
	\noindent
	\emph{(ii)} If zero is a first-kind threshold singularity of $H$, then $\mathcal W^{\mathrm L}$ is bounded on $L^p(\R^2)$ for $1\le p\le\infty$.
\end{proposition}

\begin{proposition}\label{pro:bound for low energy-1}
	Assume that $|V(x)|\les\langle x\rangle^{-8-}$ and that zero is either a threshold singularity of the second kind or a zero-energy eigenvalue of $H$. Then the following statements hold.
	
	\smallskip
	\noindent
	\emph{(i)} $\mathcal W^{\mathrm L}$ is bounded on $L^p(\R^2)$ for all $1<p\le2$, and also for all $2<p<\infty$ provided that $H$ has no p-wave resonance.
	
	\smallskip
	\noindent
	\emph{(ii)} If $H$ has an s-wave resonance but no p-wave resonance, then $\mathcal W^{\mathrm L}$ is also bounded on $L^1(\R^2)$.
	
	\smallskip
	\noindent
	\emph{(iii)} If $H$ has an s-wave resonance but no p-wave resonance, and every zero-energy eigenfunction satisfies \eqref{eq:condi_for_eigen-0}, then $\mathcal W^{\mathrm L}$ is further bounded on $L^\infty(\R^2)$.
\end{proposition}

\begin{remark}
	Propositions~\ref{pro:bound for high} and~\ref{pro:bound for low energy-1} together provide an alternative proof of the positive result in \cite{Yajima-2022}, under a weaker decay assumption on $V$.
\end{remark}

As these statements indicate, the obstruction to the endpoint boundedness comes entirely from the low-energy component. The negative results for $\mathcal W^{\mathrm L}$ read as follows.

\begin{proposition}\label{pro:unbound}
	Suppose that $|V(x)|\les\langle x\rangle^{-6-}$ and that zero is a regular point of $H$. Then $\mathcal W^{\mathrm L}$ is unbounded on $L^1(\R^2)$ and $L^\infty(\R^2)$.
\end{proposition}

\begin{proposition}\label{pro:unbound-1}
	Assume that $|V(x)|\les\langle x\rangle^{-8-}$.
	
	\smallskip
	\noindent
	\emph{(i)} If $H$ does not have an s-wave resonance or possesses a p-wave resonance, then $\mathcal W^{\mathrm L}$ is unbounded on $L^1(\R^2)$.
	
	\smallskip
	\noindent
	\emph{(ii)} If $H$ does not have an s-wave resonance, or if some zero-energy eigenfunction of $H$ fails to satisfy \eqref{eq:condi_for_eigen-0}, then $\mathcal W^{\mathrm L}$ is unbounded on $L^\infty(\R^2)$.
\end{proposition}

\subsubsection*{High energy.}
The high-energy component illustrates most clearly why the two-dimensional endpoint problem had resisted the existing approaches. For any fixed cutoff $\lambda_0>0$, the works \cite{Gal-Yaj-00,Weder-99-CMP,Yajima-JMSJ-95,Yajima-1995-even} established that $\mathcal W^{\mathrm H}$ is bounded on $L^p(\R^n)$ for the full range $1\le p\le\infty$ in all dimensions $n\neq2$. In dimension two, by contrast, the only previously known result is due to Yajima \cite{Yajima-99-CMP}, who obtained $L^p$-boundedness for $1<p<\infty$ by a method in frequency space, while the endpoint cases $p=1,\infty$ remained open. We close this gap by giving a physical-space treatment of $\mathcal W^{\mathrm H}$, which reaches both endpoints (see Proposition~\ref{pro:bound for high}).

To prove the boundedness of $\mathcal W^{\mathrm H}$, we first use the resolvent identity
\[
R^+(\lambda^2)
=
R_0^+(\lambda^2)
-
R_0^+(\lambda^2)VR^+(\lambda^2)
+
R_0^+(\lambda^2)VR_0^+(\lambda^2)VR^+(\lambda^2),
\]
together with the stationary representation formula to decompose $\mathcal W^{\mathrm H}$ into three terms:
\begin{align*}
	\mathcal W^{\mathrm H}_0
	&:=
	\frac{1}{\pi i}
	\int_0^\infty
	\lambda R_0^+(\lambda^2)V
	\bigl(R_0^+(\lambda^2)-R_0^-(\lambda^2)\bigr)
	\widetilde\chi(\lambda/\lambda_0)\,d\lambda,\\
	\mathcal W^{\mathrm H}_1
	&:=
	\frac{1}{\pi i}
	\int_0^\infty
	\lambda R_0^+(\lambda^2)VR_0^+(\lambda^2)V
	\bigl(R_0^+(\lambda^2)-R_0^-(\lambda^2)\bigr)
	\widetilde\chi(\lambda/\lambda_0)\,d\lambda,\\
	\mathcal W^{\mathrm H}_2
	&:=
	\frac{1}{\pi i}
	\int_0^\infty
	\lambda R_0^+(\lambda^2)VR_0^+(\lambda^2)VR^+(\lambda^2)V
	\bigl(R_0^+(\lambda^2)-R_0^-(\lambda^2)\bigr)
	\widetilde\chi(\lambda/\lambda_0)\,d\lambda.
\end{align*}
Thanks to the extra decay in $\lambda$ from
\[
R_0^\pm(\lambda^2)(x,y)
=
\mathcal O\bigl(\lambda^{-1/2}|x-y|^{-1/2}\bigr),
\]
the last two terms $\mathcal W^{\mathrm H}_1$ and $\mathcal W^{\mathrm H}_2$ can be controlled directly by oscillatory integral estimates. The main difficulty arises in the study of $\mathcal W^{\mathrm H}_0$: the standard stationary phase method cannot be applied directly, since a naive estimate of the oscillatory integral
\[
\int_0^\infty
e^{i\lambda(|x-z|-|z-y|)}
\widetilde\chi(\lambda/\lambda_0)\,d\lambda
\]
produces a singularity $\frac{1}{|x-z|-|z-y|}$. This singularity behaves like a radial Hilbert transform and is far too rough to yield $L^1$ or $L^\infty$ bounds; it is this obstruction that keeps the endpoint cases out of reach of the frequency-space approach.

Two new devices overcome this difficulty. The first is a Gaussian modulation of the potential. Recall that, although not explicitly stated in \cite{Yajima-99-CMP}, the $L^p$-boundedness of $\mathcal W^{\mathrm H}_0$ for all $1\le p\le\infty$ under the extra condition $\int_{\R^2}V\,dx=0$ essentially follows from Yajima's proof therein. Motivated by this fact, we write $V=(V-V_0)+V_0$ with
\[
V_0(x)
=
\frac1\pi
\left(\int_{\R^2}V\,dx\right)e^{-|x|^2},
\]
so that $V-V_0$ has zero mean and its contribution is covered by the zero-mean case of Yajima's argument. The boundedness of $\mathcal W^{\mathrm H}_0$ is thereby reduced to that of its $V_0$-part, whose kernel is a sum of terms of the form
\begin{equation*}
	\frac{ic^+c^\pm}{16\pi}
	\int_{\R^2}\int_0^\infty
	\frac{
		e^{i\lambda(|x-z|-|z-y|)}V_0(z)
	}{
		|x-z|^{1/2}|z-y|^{1/2}
	}
	\widetilde\chi(\lambda/\lambda_0)
	\,d\lambda\,dz.
\end{equation*}
The crucial point is that $V_0$ is smooth; it is precisely this regularity that enables the second device.

The second device is a hyperbolic parametrization of the phase: on the dangerous region, the integration variable $z$ is parametrized by the level sets $\{|x-z|-|z-y|=\rho\}$, which are hyperbolas with foci $x$ and $y$, and the resulting one-dimensional integral in $\rho$ is treated by a frequency-dependent decomposition combined with integration by parts in the spatial variables. This produces enhanced decay in $\lambda$ and shows that the resulting kernel, which the naive stationary-phase estimate pronounces inadmissible, is in fact a Schur-admissible kernel; see Section~\ref{sec:High-bound} for the details.

We emphasize that, apart from Yajima's zero-mean result used in the first step, the argument is carried out entirely in physical space: to our knowledge, it yields the first endpoint estimates for the two-dimensional high-energy component, and the hyperbolic parametrization of the phase level sets appears to be new in this context. The picture of the high-energy component is thereby completed across all dimensions.

\subsubsection*{Low energy.}
To consider the low energy part, let
\[
M^\pm(\lambda):=U+vR_0^\pm(\lambda^2)v,
\]
and
\[
v(x):=|V(x)|^{1/2},
\qquad
U(x):=\operatorname{sgn}V(x)=
\begin{cases}
	\ \ 1, & \text{if } V(x)\ge0,\\
	-1, & \text{if } V(x)<0.
\end{cases}
\]
For the low-energy part $\mathcal W^{\mathrm L}$, using the resolvent identity
\[
R^+(\lambda^2)
=
R_0^+(\lambda^2)v
\bigl(M^+(\lambda)\bigr)^{-1}v,
\]
we obtain the representation
\begin{equation}\label{eq:station rep for low}
	\mathcal W^{\mathrm L}
	=
	\frac{1}{\pi i}
	\int_0^\infty
	R_0^+(\lambda^2)v
	\bigl(M^+(\lambda)\bigr)^{-1}v
	\bigl(R_0^+(\lambda^2)-R_0^-(\lambda^2)\bigr)
	\lambda\chi(\lambda/\lambda_0)\,d\lambda.
\end{equation}
Let
\[
\mathscr E(\lambda)
:=
R_0^+(\lambda^2)v
\bigl(M^+(\lambda)\bigr)^{-1}v
\bigl(R_0^+(\lambda^2)-R_0^-(\lambda^2)\bigr).
\]
The core difficulty in studying $\mathcal W^{\mathrm L}$ lies in the refined analysis of the asymptotic behavior of the kernel $\mathscr E(\lambda)(x,y)$ as $\lambda\to0^+$.

\medskip
\noindent
\textbf{Positive results.} As an illustration, we consider the case where zero is a regular point or a first-kind threshold singularity. We first derive the asymptotic expansion of $\bigl(M^+(\lambda)\bigr)^{-1}$ (see Lemma~\ref{thm-M inverse-even}); substituting it into the stationary representation formula \eqref{eq:station rep for low} reduces the problem to analyzing the boundedness of each term separately. Apart from the term $h_+(\lambda)^{-1}P$ in \eqref{eq-rep-W_s}, all remaining terms are bounded on both $L^1(\R^2)$ and $L^\infty(\R^2)$. The term $h_+(\lambda)^{-1}P$, on the other hand, gives rise to the singular contribution
\[
\mathcal W_{\mathrm{sing}}^{\mathrm L}
=
\frac{1}{\pi i}
\int_0^\infty
\lambda h_+(\lambda)^{-1}\chi(\lambda/\lambda_0)
R_0^+(\lambda^2)vPv
\bigl(R_0^+(\lambda^2)-R_0^-(\lambda^2)\bigr)
\,d\lambda.
\]
In the regular case, this operator is only weakly bounded (see the negative results in Section \ref{sec:unbound}); when zero is a first-kind threshold singularity, it is absent from the expansion altogether, and the strong endpoint bounds are recovered. See Section~\ref{sec:Low-bound} for the full details.

\medskip
\noindent
\textbf{Negative results.} To show the unboundedness of $W_\pm$ when zero is a regular point, we first note that for $1\le p\le\infty$, $W_\pm$ are bounded on $L^p(\R^2)$ if and only if $\mathcal W^{\mathrm L}$ is bounded on $L^p(\R^2)$. Thus, we reduce the unboundedness of $W_\pm$ to that of $\mathcal W_{\mathrm{sing}}^{\mathrm L}$. We shall show the $L^1$-unboundedness of $\mathcal W_{\mathrm{sing}}^{\mathrm L}$ by proving that
\[
\mathcal W_{\mathrm{sing}}^{\mathrm L}
\bigl(\mathcal F^{-1}\bigl(\chi(2|\cdot|/\lambda_0)\bigr)\bigr)
\notin L^1(\R^2)
\]
using a frequency method, and the $L^\infty$-unboundedness of $\mathcal W_{\mathrm{sing}}^{\mathrm L}$ by proving that for the characteristic functions $\Psi_\mu$ of the balls $B(0,\mu)$,
\[
\bigl|\mathcal W_{\mathrm{sing}}^{\mathrm L}\Psi_\mu(x)\bigr|
=
-\frac{1}{2}\log\log\mu+\mathcal O(1)
\quad
\text{for } |x|\in[\mu-\tfrac12,\mu+\tfrac12]
\text{ and } \mu\gg1.
\]
In particular, the endpoint failure exhibits an explicitly identified double-logarithmic profile, which shows that the weak-type bounds of Proposition~\ref{pro:bound for low energy}\,\textup{(i)} are sharp. See Section~\ref{sec:unbound} for the full details.

In the presence of a p-wave resonance, the singular component associated with these resonances is the
operator
\[
\mathcal{W}_{\mathrm{p}, \mathrm{sing}}^{\mathrm{L}}=\frac{1}{\pi i}\int_0^{\infty}
\lambda^{-1} g_1^{+}(\lambda)^{-1} \chi(\lambda/\lambda_0)\,
R_0^{+}(\lambda^2)v(S_2-S_3)D_2(S_2-S_3)v
\bigl(R_0^{+}(\lambda^2)-R_0^{-}(\lambda^2)\bigr)\, d\lambda,
\]
where $S_2$, $S_3$ are the finite-rank projections associated with the threshold spectral subspaces.
Expanding the two resolvent kernels reduces it, modulo $L^1$-bounded cross terms, to a separable
kernel $\mathcal{I}_{\mathrm{p},\mathrm{sing}}(x,y)=k(\langle x\rangle,\langle y\rangle)T(x,y)$: an
order-one Bessel--Hankel radial factor tensored with a first-order moment pairing. A finite-rank
decoupling of the moments freezes $T$ to a nonzero constant on a narrow angular sector, while the
same logarithmic asymptotics as above, now of order one, yield the lower bound
$|k(\nu,\mu)|\gtrsim(\mu\,\delta\log\delta)^{-1}$ along the shoulder $\nu=\mu+\delta$,
$\delta\in[\delta_1,\sqrt{\mu}\,]$. Integrating over the associated cone then gives column norms
$\int_{\mathbb{R}^2}|\mathcal{I}_{\mathrm{p},\mathrm{sing}}(x,y)|\,dx\gtrsim\log\log\mu$ for
$\langle y\rangle\sim\mu$, and the endpoint Schur identity turns this double-logarithmic growth into
the $L^1$-unboundedness of Proposition~\ref{pro:unbound-1}\,\textup{(i)}.

In the eigenvalue case, the blow-up mechanism has a distinctly algebraic flavor. Expanding the singular kernel into spherical harmonics---equivalently, into homogeneous harmonic polynomials---shows that at the critical order of the expansion only the degree-two channel survives: the singular component of $\mathcal W^{\mathrm L}$ is governed by the moments
\[
\int_{\R^2}P(x)V(x)\psi(x)\,dx,
\]
with $P$ ranging over the homogeneous harmonic polynomials of degree two, i.e. over $\operatorname{span}\{x_1x_2,\ x_1^2-x_2^2\}$. When these moments vanish for every zero-energy eigenfunction $\psi$, the singular component disappears and the $L^\infty$ bound is recovered; when one of them fails to vanish, the surviving degree-two harmonic produces an explicit logarithmic blow-up on suitable test data, which yields the unboundedness.

\medskip
\noindent
\textbf{Logarithmic symbols.} A further ingredient, which pervades both the positive and the negative arguments above, is a family of sharp asymptotic estimates for oscillatory integrals and Fourier multipliers whose symbols involve logarithmic factors such as $\log\lambda$ and $(\log\lambda)^{-1}$. Symbols of this type are generated universally by the two-dimensional resolvent expansions, and the classical multiplier theorems do not apply to them at the endpoints. The required estimates are developed in Sections~\ref{sec:Pre},~\ref{sec:low-energy-kernel} and~\ref{sec:unbound}, and may be of independent interest.

\subsection{Notations}
We list the common notations and conventions used throughout the paper as follows:
\begin{itemize}
	\item
	The notation $A\lesssim B$ means that $A\leq CB$ for some constant
	$C>0$ independent of the relevant variables. The notation
	$A\lesssim_l B$ means that $A\leq C_l B$, where the constant
	$C_l>0$ may depend on $l$.
	
	\item The notation $f(x)=\mathcal{O}(g(x))$ means that
	\[
	|f(x)|\leq C|g(x)|
	\]
	for some constant $C>0$ independent of the relevant variables.
	
	\item The notation $a+$, respectively $a-$, denotes $a+\varepsilon$, respectively $a-\varepsilon$, for some arbitrarily small fixed $\varepsilon>0$.
	
	\item We set
	\[
	\mathbb{N}_0=\{0,1,2,\ldots\},
	\qquad
	\mathbb{Z}=\{0,\pm1,\pm2,\ldots\}.
	\]
	
	\item For $1\leq p\leq\infty$, we write
	$L^p=L^p(\mathbb{R}^n;\mathbb{C}).$
	For $\sigma\in\mathbb{R}$, the weighted space $L^2_\sigma(\mathbb{R}^n)$ is defined by
	\[
	L^2_\sigma(\mathbb{R}^n)
	=
	\left\{
	f:\langle x\rangle^\sigma f(x)\in L^2(\mathbb{R}^n)
	\right\}.
	\]
	The notation $L^{p,q}(\mathbb{R}^n)$ denotes the Lorentz space for
	$1\leq p,q\leq\infty$. In particular,
	$L^{1,\infty}(\mathbb{R}^n)$ denotes the weak $L^1$ space.
	Moreover, $\mathcal{H}^1(\mathbb{R}^n)$ denotes the real Hardy
	space, while $\mathrm{BMO}(\mathbb{R}^n)$ denotes its dual, the
	space of functions of bounded mean oscillation; see, for example,
	\cite{graf,graf2}.
	
	\item The space $C_c^\infty(\mathbb{R}^n)$ denotes the space of
	smooth compactly supported functions on $\mathbb{R}^n$, while
	$\mathcal{S}(\mathbb{R}^n)$ denotes the Schwartz space.
	
	\item For Banach spaces $X$ and $Y$, $\mathbb{B}(X,Y)$ denotes the
	space of bounded linear operators from $X$ to $Y$, and
	$\mathbb{B}(X)=\mathbb{B}(X,X).$
	
	\item Both $\mathcal{F}f$ and $\widehat{f}$ denote the Fourier
	transform of $f$, defined by
	\[
	\mathcal{F}f(\xi)
	=
	\widehat{f}(\xi)
	=
	\int_{\mathbb{R}^n}
	e^{-ix\cdot\xi}f(x)\,dx.
	\]
	Similarly, both $\mathcal{F}^{-1}g$ and $\check{g}$ denote the
	inverse Fourier transform of $g$, defined by
	\[
	\mathcal{F}^{-1}g(x)
	=
	\check{g}(x)
	=
	\frac{1}{(2\pi)^n}
	\int_{\mathbb{R}^n}
	e^{ix\cdot\xi}g(\xi)\,d\xi.
	\]
\end{itemize}

The rest of the paper is organized as follows. Section~\ref{sec:Pre} provides preliminary material, including properties of the free resolvent and several Fourier multiplier and oscillatory integral estimates that will be used throughout the paper. Section~\ref{sec:High-bound} is devoted to proving the boundedness of the high-energy component $\mathcal{W}^{\mathrm{H}}$. Section~\ref{sec:low-energy-kernel} develops the properties of the kernels of the free resolvents for small $\lambda$, as a preparation for the two sections that follow. Sections~\ref{sec:Low-bound} and~\ref{low-bound II} establish the boundedness of $\mathcal{W}^{\mathrm{L}}$, and Section~\ref{sec:unbound} establishes the unboundedness of the wave operators. Appendix~\ref{sec:proof of exp lemma} contains the proof of the resolvent expansion lemma used in Section~\ref{low-bound II}. Finally, Appendix~\ref{sec:pro for lemm in sec2} collects the proofs of the auxiliary lemmas stated in Section~\ref{sec:Pre}.

\section{Preliminaries}\label{sec:Pre}
In this section, we fix the notation for symbol classes, recall the 
asymptotic behavior of the two-dimensional free resolvent kernel, and 
collect the oscillatory integral and Fourier multiplier estimates that 
will be used throughout the paper.

\subsection{Free resolvent kernels, Hankel and Bessel functions}\label{Sec:some Technical lemmas}
We denote by
\[
f^{(k)}(z)=\frac{d^k}{dz^k}f(z)
\]
the $k$-th derivative of a function $f$. We write 
$f(\lambda)=\mathcal{O}_m(g(\lambda))$ for 
$\lambda\in\Omega\subset\mathbb{R}$ if
\begin{equation*}
	\left|f^{(k)}(\lambda)\right|\lesssim_k |g(\lambda)|\,|\lambda|^{-k},
	\qquad \lambda\in\Omega,\quad k=0,1,\dots,m.
\end{equation*}
The notation extends naturally to $m=\infty$; for brevity, we also write 
$\mathcal{O}(g(\lambda))=\mathcal{O}_0(g(\lambda))$.

Recall that the integral kernels of the free resolvent 
$R_0^{\pm}(\lambda^2)$ are given by
\begin{equation}\label{eq:reso and Hankel-2}
	R_0^{\pm}(\lambda^2)(x,y) = \pm\frac{i}{4} H_0^{\pm}(\lambda|x-y|),  
	\quad H_0^{\pm}(z)=J_0(z)\pm iY_0(z).
\end{equation}
In particular,
\begin{equation*}
	R_0^{+}(\lambda^2)(x,y)-R_0^{-}(\lambda^2)(x,y) 
	= \frac{i}{2} J_0(\lambda|x-y|).
\end{equation*}
Thus, the analysis of the free resolvent kernels reduces to that of the 
Hankel and Bessel functions.

We now record some properties of the Hankel and Bessel functions; these 
facts are standard and can be found, for instance, in \cite{Abra, GR02, 
	Schlag-2D-dis, Yajima-99-CMP}. For $0<z\les 1$, the Hankel functions 
$H_0^{\pm}(z)$ admit the asymptotic expansion
\begin{equation}\label{eq:asy-Hanel-0}
	H_0^{\pm}(z) = 1 \pm \frac{2i}{\pi}\gamma \pm \frac{2i}{\pi}\log\frac{z}{2}
	\mp \frac{i}{2\pi}\bigl(z^2\log z + \beta_{\pm} z^2\bigr) 
	+ \mathcal{O}_{\infty}(z^3),
\end{equation}
where $\beta_{\pm}=\big(\gamma-1-\ln 2\big)\mp\frac{i\pi}{2}\in \mathbb{C}\setminus\R$. For $z \gtrsim 1$, the 
Hankel functions can be expressed as
\begin{equation}\label{eq:asy-Hanel-1}
	H_0^{\pm}(z) =e^{\pm iz}\Theta^{\pm}(z)
	=\sqrt{\frac{2}{\pi}}\, \frac{e^{\pm i\bigl(z-\pi/4\bigr)}}{z^{1/2}} 
	+ e^{\pm iz}\mathcal{O}_{\infty}(z^{-3/2}),
\end{equation}
with
\begin{equation}\label{eq:est for Theta-0}
	\Theta^{\pm}(z)=\mathcal{O}_{\infty}(z^{-1/2}).
\end{equation}
Combined with \eqref{eq:asy-Hanel-0}, this shows that the representation 
$H_0^{\pm}(z)=e^{\pm iz}\Theta^{\pm}(z)$, with $\Theta^{\pm}$ satisfying 
\eqref{eq:est for Theta-0}, in fact holds for all $z\in(0,\infty)$. We 
shall also need the following variant of \eqref{eq:asy-Hanel-0}, valid on 
the whole half-line:
\begin{equation*}
	H_0^{\pm}(z) = 1 \pm \frac{2i}{\pi}\gamma \pm \frac{2i}{\pi}
	\log\frac{z}{2} + \zeta_0^{\pm}(z), \qquad z\in (0,\infty),
\end{equation*}
where, for every $k\in \mathbb{N}_0$, the remainders $\zeta_0^{\pm}(z)$ 
satisfy
\begin{equation}\label{eq:est for r-0}
	\frac{d^k}{dz^k} \zeta_0^{\pm}(z)=
	\begin{cases}
		\mathcal{O}\bigl(|z|^{3/2 - k}\bigr), & |z| \lesssim 1,\\[6pt]
		\mathcal{O}\bigl(|z|^{\max\{3/2 - k,\ -1/2\}}\bigr), & |z| \gtrsim 1.
	\end{cases}
\end{equation}

\subsection{Fourier transform and oscillatory integral estimates}
In this subsection, we collect several auxiliary estimates for 
oscillatory integrals and Fourier multipliers with logarithmic symbols, 
which will be used in the proof of our main results; their proofs are 
deferred to Appendix~\ref{sec:pro for lemm in sec2}.

\begin{lemma}\label{lem-oscillatory estimate}
	Let $0<\lambda_0<1$, and suppose that 
	$m(\lambda)= \mathcal{O}_{[b]+2}\big(|\log\lambda|^{-a} \lambda^b\big)$ 
	for $\lambda\in (0,\lambda_0)$ and some fixed $a\ge 0$ and $b>-1$. Then
	\begin{equation}\label{eq.oscillatory-est}
		\left|\int_0^\infty e^{i\lambda x} m(\lambda) \chi(\lambda/\lambda_0) 
		\, d\lambda\right| \lesssim_{m,\lambda_0} 
		\frac{1}{\langle x\rangle^{b+1}\log^{a}(2+|x|)}, \quad x \in \mathbb{R},
	\end{equation}
	where the cut-off function $\chi$ is given in \eqref{cut-off-fun}.
\end{lemma}

\begin{lemma}\label{lem-oscillatory estimate-h}
	Let $\lambda_0>0$ and $b\in \R$ be fixed, and assume that 
	$m(\lambda)= \mathcal{O}_{[b]+2}\big(\lambda^b\big)$ for 
	$\lambda\in (\lambda_0/2, \infty)$. Then, for every $k\in\mathbb{N}_0$ 
	and all $x\in\mathbb{R}$,
	\begin{equation}\label{eq.oscillatory-est-h}
		\left|\int_0^\infty e^{i\lambda x} m(\lambda) \tilde{\chi}(\lambda/\lambda_0) 
		\, d\lambda \right| \lesssim_{m,\lambda_0,k}
		\begin{cases}
			\dfrac{1}{\langle x\rangle^{k+1}|x|^{b+1}}, & \text{if } b \neq -1,\\[8pt]
			\dfrac{1}{\langle x\rangle^{k+1}\bigl(1+|\log |x||\bigr)}, & \text{if } b = -1,
		\end{cases}
	\end{equation}
	where the cut-off function $\tilde{\chi}$ is given in \eqref{cut-off-fun}.
\end{lemma}

\begin{lemma}\label{eq-fourier-mulitiplier-1} 
	Let $0<\lambda_0<1$ and $a>0$ be fixed, and assume that 
	$m(\lambda)=\mathcal{O}_{[\tfrac{n+1}{2}]+1}\big(|\log\lambda|^{-a}\big)$ 
	for $\lambda\in (0,\lambda_0)$. Then the Fourier multiplier 
	$m(\sqrt{-\Delta})\chi\big(\sqrt{-\Delta}/\lambda_0\big)$ is bounded on 
	$\mathcal{H}^1(\mathbb{R}^n)$, $\mathrm{BMO}(\mathbb{R}^n)$, and 
	$L^p(\mathbb{R}^n)$ for all $1\le p\le \infty$.
\end{lemma}

\begin{lemma}\label{lem-Fou-multiplier}
	Let $a>0$ and $0<\lambda_0<1$ be fixed. Then
	\[
	\left|\mathcal{F}^{-1}\!\left((\log|\cdot|)^{-a}
	\chi\!\left(\tfrac{|\cdot|}{\lambda_0}\right)\right)(x)
	\right|\les_{a,n,\lambda_0}\frac{1}{|x|^n\log^{a+1}|x|},\quad |x|\gg 1.
	\]
	Accordingly, the Fourier multiplier operator 
	$\log^{-a}(\sqrt{-\Delta})\, \chi\big(\sqrt{-\Delta}/\lambda_0\big)$ is 
	bounded on $\mathcal{H}^1(\mathbb{R}^n)$, $\mathrm{BMO}(\mathbb{R}^n)$, 
	and $L^p(\mathbb{R}^n)$ for every $1\le p\le \infty$.
\end{lemma}

For an operator $T$, we write $T(x,y)$ for its integral kernel. Such a 
kernel is called \emph{admissible} if
\begin{equation*}
	\sup_{x\in\mathbb{R}^n}\int_{\mathbb{R}^n}|T(x,y)|\,dy
	+\sup_{y\in\mathbb{R}^n}\int_{\mathbb{R}^n}|T(x,y)|\,dx<\infty.
\end{equation*}
The classical Schur lemma then reads as follows.

\begin{lemma}[Schur's lemma]\label{Schur Lemma}
	If $T$ has an admissible kernel, then it is bounded on 
	$L^p(\mathbb{R}^n)$ for every $1\le p\le\infty$, with
	\[
	\|T\|_{L^p\to L^p}\le \sup_{x\in\mathbb{R}^n}
	\int_{\mathbb{R}^n}|T(x,y)|\,dy+\sup_{y\in\mathbb{R}^n}
	\int_{\mathbb{R}^n}|T(x,y)|\,dx.
	\]
\end{lemma}

We will also need the following integral estimate.

\begin{lemma}[{\cite[Lemma 3.8]{GV06}}]\label{lemma-GV}
	Let $n\ge 1$, $k<n$ and $k+\ell>n$. Then
	\begin{equation*}
		\int_{\mathbb{R}^n}|x-y|^{-k}\langle y\rangle^{-\ell}\,dy 
		\les_{n,k,\ell}  \langle x\rangle^{-\min\{k,\,k+\ell-n\}}.
	\end{equation*}
\end{lemma}


\section{High-energy estimates}\label{sec:High-bound}

In this section we prove Proposition~\ref{pro:bound for high}, namely, that the high-energy component $\mathcal{W}^{\mathrm{H}}$ is bounded on $L^p(\mathbb{R}^2)$ for the full range $1\le p\le \infty$.

By the stationary representation formula \eqref{eq-stationary formula},
\[
\mathcal{W}^{\mathrm{H}} = \frac{1}{\pi i}\int_0^{+\infty}\lambda R^{+}(\lambda^{2})V\big(R_0^+(\lambda^{2})-R_0^-(\lambda^{2})\big) \tilde{\chi}(\lambda/\lambda_0) \, d\lambda.
\]
Inserting the resolvent identity
\[
R^{+}(\lambda^{2})=R_0^{+}(\lambda^{2})-R_0^{+}(\lambda^{2})VR^{+}(\lambda^{2})+R_0^{+}(\lambda^{2})VR_0^{+}(\lambda^{2})VR^{+}(\lambda^{2})
\]
into this formula yields the decomposition $\mathcal{W}^{\mathrm{H}} = \mathcal{W}^{\mathrm{H}}_0 - \mathcal{W}^{\mathrm{H}}_1 + \mathcal{W}^{\mathrm{H}}_2$, where
\begin{align*}
	\mathcal{W}^{\mathrm{H}}_0 &:= \frac{1}{\pi i}\int_0^{+\infty}\lambda R_0^{+}(\lambda^{2})V\big(R_0^+(\lambda^{2})-R_0^-(\lambda^{2})\big) \tilde{\chi}(\lambda/\lambda_0) \, d\lambda, \\
	\mathcal{W}^{\mathrm{H}}_1 &:= \frac{1}{\pi i}\int_0^{+\infty}\lambda R_0^{+}(\lambda^{2})VR_0^{+}(\lambda^{2})V\big(R_0^+(\lambda^{2})-R_0^-(\lambda^{2})\big) \tilde{\chi}(\lambda/\lambda_0) \, d\lambda, \\
	\mathcal{W}^{\mathrm{H}}_2 &:= \frac{1}{\pi i}\int_0^{+\infty}\lambda R_0^{+}(\lambda^{2})VR_0^{+}(\lambda^{2})VR^{+}(\lambda^{2})V\big(R_0^+(\lambda^{2})-R_0^-(\lambda^{2})\big) \tilde{\chi}(\lambda/\lambda_0) \, d\lambda.
\end{align*}

It was shown by Yajima \cite[Proposition~3.1]{Yajima-99-CMP} that $\mathcal{W}^{\mathrm{H}}_2$ is bounded on $L^p(\mathbb{R}^2)$ for all $1 \leq p \leq \infty$ whenever $|V(x)| \lesssim \langle x \rangle^{-\beta}$ with $\beta > \frac{7}{2}$. Proposition~\ref{pro:bound for high} therefore follows from the corresponding bounds for $\mathcal{W}^{\mathrm{H}}_0$ and $\mathcal{W}^{\mathrm{H}}_1$, which are the content of Propositions~\ref{lem:boundedness of W-1} and~\ref{lem:boundedness of W-2} below.

\begin{proposition}\label{lem:boundedness of W-1}
	Suppose that $|V(x)| \lesssim \langle x \rangle^{-\beta}$ with $\beta > \frac{7}{2}$. Then the operator $\mathcal{W}^{\mathrm{H}}_0$ is bounded on $L^p(\mathbb{R}^2)$ for all $1\le p\le \infty$.
\end{proposition}

\begin{proof}
	The starting point is the following fact, which is not explicitly stated in \cite{Yajima-99-CMP} but is essentially contained in the proof of \cite[Proposition~2.1]{Yajima-99-CMP}: under the additional assumption
	\[
	\widehat{V}(0)=\int_{\mathbb{R}^2} V(x)\,dx=0,
	\]
	the operator $\mathcal{W}^{\mathrm{H}}_0$ is $L^p$-bounded for all $1 \le p \le \infty$. To recall the argument, \cite[Lemma~2.1]{Yajima-99-CMP} gives the representation
	\[
	\mathcal{W}^{\mathrm{H}}_0f(x) = \frac{i}{4\pi} \int_{\mathbb{S}^1} d\omega \int_{0}^{\infty} K(t + 2x\cdot\omega, \omega) \big(\tilde{\chi}(\sqrt{-\Delta}/\lambda_0)f\big)(x + t\omega)\, dt,
	\]
	where $\mathbb{S}^1$ denotes the unit circle and the kernel $K(t,\omega)$ is given by
	\[
	K(t,\omega) = \int_{0}^{\infty} \widehat{V}(r\omega) e^{i t r / 2}\, dr.
	\]
	As in the proof of \cite[Proposition~2.1]{Yajima-99-CMP}, integration by parts in $r$ gives the decomposition
	\begin{equation}\label{eq-deco-for-K}
		K(t,\omega) = \chi(t)K(t,\omega) + \tilde{\chi}(t)\frac{2i\widehat{V}(0)}{t} + \tilde{\chi}(t)\frac{2i}{t}\int_{0}^{\infty} e^{i t r / 2} \frac{\partial}{\partial r}\widehat{V}(r\omega)\,dr,
	\end{equation}
	with $\chi$ and $\tilde{\chi}$ as defined in \eqref{cut-off-fun}. Since $\langle x \rangle^s V(x) \in L^2(\mathbb{R}^2)$ for some $s>1$ by the assumption $|V(x)| \lesssim \langle x \rangle^{-\beta}$ with $\beta>\frac72$, the first and third summands on the right-hand side of \eqref{eq-deco-for-K} belong to $L^1(\mathbb{R} \times \mathbb{S}^1)$, while the middle summand vanishes precisely when $\widehat{V}(0)=0$. Hence $K\in L^1(\mathbb{R}\times\mathbb{S}^1)$ in the mean-zero case, and since translations and the reflections $x \mapsto x_\omega := x - 2(x\cdot\omega)\omega$ preserve the $L^p$-norm, Minkowski's inequality yields the $L^p$-boundedness of $\mathcal{W}^{\mathrm{H}}_0$ for every $1 \leq p \leq \infty$; we refer to \cite[Proposition~2.1]{Yajima-99-CMP} for the details.
	
	To reduce the general case to the mean-zero case, we introduce the Gaussian $V_0(x) = a e^{-|x|^2}$ with the amplitude $a = \frac{1}{\pi}\widehat{V}(0)$, chosen so that
	\[
	\widehat{(V - V_0)}(0)=\int_{\mathbb{R}^2} \big(V(x) - V_0(x)\big)\,dx = 0.
	\]
	By the case just treated, the operator
	\[
	\frac{1}{\pi i} \int_0^{+\infty} \lambda R_0^{+}(\lambda^2) (V-V_0) \big( R_0^+(\lambda^2) - R_0^-(\lambda^2) \big) \tilde{\chi}(\lambda/\lambda_0) \, d\lambda
	\]
	is $L^p$-bounded for all $1 \leq p \leq \infty$, and the $L^p$-boundedness of $\mathcal{W}^{\mathrm{H}}_0$ is reduced to that of
	\[
	\mathcal{W}^{\mathrm{H},V_0}_{0} := \frac{1}{\pi i} \int_{0}^{+\infty} \lambda R_0^+(\lambda^2) V_0 \big(R_0^+(\lambda^2)-R_0^-(\lambda^2)\big) \tilde{\chi}(\lambda/\lambda_0)\, d\lambda.
	\]
	The advantage of replacing $V$ by the Gaussian $V_0$ is that the resulting kernel can be analyzed directly: the smoothness of $V_0$ allows us to integrate by parts in the spatial variables, and we will show that the kernel of $\mathcal{W}^{\mathrm{H},V_0}_{0}$ is admissible.
	
	Recall (cf. \cite{Schlag-2D-dis, Yajima-99-CMP}) that in two dimensions the integral kernels of the free resolvents $R_0^{\pm}(\lambda^2)$ are given by
	\begin{equation}
		R_0^{\pm}(\lambda^2)(x,y) = \pm\frac{i}{4} H_0^{\pm}(\lambda|x-y|),
	\end{equation}
	where $H_0^{\pm}=J_0\pm iY_0$ are the Hankel functions of order zero, with $J_0$ and $Y_0$ the Bessel functions of the first and second kind of order zero, respectively. By \cite[Chapter~9]{Abra},
	\begin{equation}\label{eq:asy-Hanel-2}
		H_0^\pm(z) = e^{\pm iz}\Theta^{\pm}(z) =  \sqrt{\frac{2}{\pi}}\, \frac{e^{\pm i(z-\pi/4)}}{z^{1/2}} + e^{\pm iz}\Theta_r^\pm(z), \qquad z\in (0,\infty),
	\end{equation}
	where the functions $\Theta^{\pm}$ and $\Theta_r^\pm$ satisfy
	\[
	\frac{d^k}{dz^k}\Theta^{\pm}(z)=\mathcal{O}\bigl(z^{-\frac12 - k}\bigr),\qquad
	\frac{d^k}{dz^k}\Theta_r^{\pm}(z)= \mathcal{O}\bigl(z^{-\frac32 - k}\bigr),\qquad k\in \mathbb{N}_0.
	\]
	Inserting the expansion \eqref{eq:asy-Hanel-2} into the definition of $\mathcal{W}^{\mathrm{H},V_0}_0$, one finds that its kernel is a finite sum of terms of the following four types:
	\begin{align*}
		\mathcal{T}_1^\pm(x,y) &:= \frac{ i c^+c^\pm}{ 16\pi }\int_{\mathbb{R}^2}\int_0^{+\infty} e^{i \lambda \rho_{\pm}} \frac{  V_0(z) }{|x-z|^{1/2} |z-y|^{1/2}} \, \tilde{\chi}(\lambda/\lambda_0)  \, d\lambda \, dz, \\[4pt]
		\mathcal{T}_2^\pm(x,y) &:= \frac{ ic^+}{ 16\pi }\int_{\mathbb{R}^2}\int_0^{+\infty}
		e^{i\lambda  \rho_{\pm}} \frac{ V_0(z)  \Theta_r^{\pm}(\lambda|z-y|)}{|x-z|^{1/2}} \, \lambda^{1/2} \,\tilde{\chi}(\lambda/\lambda_0)  \, d\lambda \, dz, \\[4pt]
		\mathcal{T}_3^\pm(x,y) &:= \frac{ i c^\pm}{ 16 \pi } \int_{\mathbb{R}^2}\int_0^{+\infty}
		e^{i \lambda \rho_{\pm}}   \frac{\Theta_r^{+}(\lambda|x-z|) V_0(z)}{  |z-y|^{1/2} } \, \lambda^{1/2} \, \tilde{\chi}(\lambda/\lambda_0)  \, d\lambda \, dz, \\[4pt]
		\mathcal{T}_4^\pm(x,y) &:= \frac{ i }{ 16 \pi }\int_{\mathbb{R}^2}\int_0^{+\infty}
		e^{i \lambda \rho_{\pm}}  V_0(z) \Theta_r^{+}(\lambda|x-z|) \Theta_r^{\pm}(\lambda|z-y|) \, \lambda \, \tilde{\chi}(\lambda/\lambda_0)  \, d\lambda \, dz,
	\end{align*}
	where $\rho_{\pm} := |x-z| \pm |z-y|$ and $c^\pm:= \sqrt{\frac{2}{\pi}} \, e^{\mp \pi i/4}$. The proposition is therefore a consequence of Lemmas~\ref{lem:boundedness of W^V_0-1} and~\ref{lem:boundedness of W^V_0-2} below.
\end{proof}

\begin{lemma}\label{lem:boundedness of W^V_0-1}
	The operators associated with the kernels $\mathcal{T}_2^\pm(x,y)$, $\mathcal{T}_3^\pm(x,y)$ and $\mathcal{T}_4^\pm(x,y)$ are bounded on $L^p(\mathbb{R}^2)$ for all $1\le p\le \infty$.
\end{lemma}

\begin{proof}
	The three kernels are treated by the same argument; we give the details for $\mathcal{T}_2^\pm(x,y)$ only. For $\mathcal{T}_3^\pm$ the roles of $|x-z|$ and $|z-y|$ are interchanged, while for $\mathcal{T}_4^\pm$ the amplitude has even better decay in $\lambda$; the computation below applies verbatim in both cases.
	
	The amplitude of the $\lambda$-integral in $\mathcal{T}_2^\pm$ satisfies
	\begin{equation*}
		\Big|\partial_{\lambda}^k \big(\lambda^{1/2} \tilde{\chi}(\lambda/\lambda_0)\Theta_r^{\pm}(\lambda|z-y|) \big) \Big| \lesssim \lambda^{-1-k} |z-y|^{-3/2}, \qquad k\in \mathbb{N}_0,
	\end{equation*}
	so Lemma~\ref{lem-oscillatory estimate-h} yields the pointwise bound
	\[
	\big| \mathcal{T}_2^\pm(x,y) \big| \lesssim \int_{\mathbb{R}^2}
	\frac{|V_0(z)| \bigl(1+\bigl|\log|\rho_{\pm}|\bigr|\bigr)}{|x-z|^{1/2} |z-y|^{3/2}} \, \langle\rho_{\pm}\rangle^{-2} \, dz,
	\qquad \rho_{\pm} = |x-z| \pm |z-y|.
	\]
	Integrating in $x$ in polar coordinates centered at $z$ and applying Fubini's theorem, we obtain
	\begin{align*}
		\sup_{y}\int_{\mathbb{R}^2}\big| \mathcal{T}_2^\pm(x,y) \big| \, dx
		&\lesssim \sup_{y}\int_{\mathbb{R}^2}\frac{|V_0(z)|}{|z-y|^{3/2}} \int_0^{\infty}
		\frac{t^{1/2}\bigl(1+\bigl|\log|t\pm|z-y||\bigr|\bigr)}{\langle t \pm |z-y|\rangle^{2}} \, dt \, dz \\
		&\lesssim \sup_{y}\int_{\mathbb{R}^2}
		\left( \frac{|V_0(z)|}{|z-y|^{3/2}} + \frac{|V_0(z)|}{|z-y|} \right) dz \lesssim 1,
	\end{align*}
	where the elementary inequality $t^{1/2} \lesssim \bigl|t \pm |z-y|\bigr|^{1/2} + |z-y|^{1/2}$ was used to control the $t$-integral by a constant multiple of $1+|z-y|^{1/2}$. A symmetric argument gives
	\[
	\sup_{x}\int_{\mathbb{R}^2}\big| \mathcal{T}_2^\pm(x,y) \big| \, dy \lesssim 1,
	\]
	and the $L^p$-boundedness of the operators associated with $\mathcal{T}_2^\pm(x,y)$ for all $1 \le p \le \infty$ follows from Schur's lemma.
\end{proof}

\begin{lemma}\label{lem:boundedness of W^V_0-2}
	The operators associated with $\mathcal{T}_1^\pm(x,y)$ are bounded on $L^p(\mathbb{R}^2)$ for all $1\le p\le \infty$.
\end{lemma}

\begin{proof}
	By Schur's lemma, it suffices to show that the kernels $\mathcal{T}_1^\pm(x,y)$ are admissible.
	
	\smallskip
	\noindent\emph{Step 1: the kernel $\mathcal{T}_1^+$.} Since $\rho_+\ge |x-y|$ and $u\mapsto u^{-1}\langle u\rangle^{-2}$ is decreasing on $(0,\infty)$, Lemma~\ref{lem-oscillatory estimate-h} yields
	\[
	\bigl|\mathcal{T}_1^+(x,y)\bigr| \lesssim \int_{\mathbb{R}^2} \frac{|V_0(z)|}{|x-z|^{1/2}|z-y|^{1/2}} \, |\rho_{+}|^{-1}\langle \rho_{+} \rangle^{-2} \, dz \lesssim |x-y|^{-1}\langle x-y \rangle^{-2},
	\]
	where we also used the uniform bound $\int_{\mathbb{R}^2}\frac{|V_0(z)|}{|x-z|^{1/2}|z-y|^{1/2}}\,dz\lesssim 1$, which follows from the rapid decay of $V_0$ and the local integrability of $|\,\cdot\,|^{-1/2}$. Since $|x-y|^{-1}\langle x-y\rangle^{-2}$ is admissible, so is $\mathcal{T}_1^+$.
	
	\smallskip
	\noindent\emph{Step 2: the kernel $\mathcal{T}_1^-$.} Write $r:=|x-y|$ and $\rho:=\rho_-=|x-z|-|z-y|$ for brevity. Inserting the partition of unity
	\[
	1=\tilde\chi(\rho)+\chi(\rho)\bigl(1-\chi(2\rho/r)\bigr)+\chi(\rho)\chi(2\rho/r)
	\]
	into the $z$-integral defining $\mathcal{T}_1^-$, we estimate the three resulting pieces separately; on the support of the three cutoffs one has, respectively, $|\rho|\ge 1/2$, $|\rho|\le 1$ with $|\rho|\ge r/4$, and $|\rho|\le\min\{1,\,r/2\}$.
	
	\medskip
	\noindent$\bullet$ \emph{The region $\{|\rho| \ge 1/2\}$.} On this region, integration by parts in $\lambda$ shows that the corresponding piece of $\mathcal{T}_1^{-}(x,y)$ is bounded by
	\[
	\int_{\mathbb{R}^2} \frac{|V_0(z)|}{|x-z|^{1/2}|z-y|^{1/2}} \, \langle \rho \rangle^{-2} \, dz
	\]
	(the factor $|\rho|^{-1}$ produced by the oscillatory estimate is absorbed since $|\rho|\ge 1/2$). Passing to polar coordinates $y=z+t\omega$ in the $y$-integral, we find
	\begin{align*}
		\sup_{x} \int_{\mathbb{R}^2} \int_{\mathbb{R}^2} \frac{|V_0(z)|}{|x-z|^{1/2}|z-y|^{1/2}} \, \langle \rho \rangle^{-2} \, dz\, dy
		&\lesssim \sup_{x} \int_{\mathbb{R}^2}\frac{|V_0(z)|}{|x-z|^{1/2}} \int_{0}^{\infty} t^{1/2}\langle t-|x-z| \rangle^{-2} \, dt\, dz
		\\
		&\lesssim \sup_{x} \int_{\mathbb{R}^2}\left(|V_0(z)|+\frac{|V_0(z)|}{|x-z|^{1/2}}\right)dz
		\lesssim 1,
	\end{align*}
	where we used $\int_0^\infty t^{1/2}\langle t-s\rangle^{-2}\,dt\lesssim 1+s^{1/2}$ for $s\ge 0$. The same computation with the roles of $x$ and $y$ interchanged gives
	\[
	\sup_{y} \int_{\mathbb{R}^2} \int_{\mathbb{R}^2}
	\frac{|V_0(z)|}{|x-z|^{1/2}|z-y|^{1/2}}\,
	\langle \rho \rangle^{-2} \, dz\, dx
	\lesssim 1,
	\]
	so this piece of $\mathcal{T}_1^{-}(x,y)$ is admissible.
	
	\medskip
	\noindent$\bullet$ \emph{The region $\{|\rho| \le 1,\ |\rho|\gtrsim r\}$.} We are reduced to the kernel
	\begin{equation}\label{eq:kerne-1}
		\int_{\{|\rho| \le 1,\ |\rho| \gtrsim r\}} \int_0^{+\infty} e^{i \lambda \rho} \frac{  V_0(z) }{|x-z|^{1/2} |z-y|^{1/2}} \, \tilde{\chi}(\lambda/\lambda_0)  \, d\lambda \, dz.
	\end{equation}
	On this region $|\rho|^{-1}\langle\rho\rangle^{-2}\lesssim r^{-1}\langle r\rangle^{-2}$, so Lemma~\ref{lem-oscillatory estimate-h} gives
	\[
	\bigl|\eqref{eq:kerne-1}\bigr| \lesssim \int_{\{|\rho| \le 1,\ |\rho| \gtrsim r\}} \frac{|V_0(z)|}{|x-z|^{1/2}|z-y|^{1/2}} \, |\rho|^{-1} \langle \rho \rangle^{-2} \, dz \lesssim r^{-1}\langle r \rangle^{-2},
	\]
	which is again an admissible kernel.
	
	\medskip
	\noindent$\bullet$ \emph{The region $\{|\rho| \le 1,\ |\rho|\le r/2\}$.} It remains to estimate the kernel
	\begin{equation}\label{eq:restri to small pho-1}
		\mathcal{L}(x,y):=\int_{\mathbb{R}^2} \int_0^{+\infty} \frac{e^{i \lambda \rho} V_0(z)}{|x-z|^{1/2} |z-y|^{1/2}} \, \chi(\rho)\chi\!\left(\frac{2\rho}{r}\right) \tilde{\chi}\!\left(\frac{\lambda}{\lambda_0}\right) d\lambda \, dz.
	\end{equation}
	Note that on the support of the cutoffs,
	\begin{equation}\label{eq:support of region 3}
		|\rho|\le \min\{1,\,r/2\},\qquad\text{hence }\ r^2-\rho^2\sim r^2.
	\end{equation}
	By the change of variables $z = u + \frac{x+y}{2}$, the integral \eqref{eq:restri to small pho-1} becomes
	\begin{equation}\label{eq:restri to small pho-1-change}
		\int_{\mathbb{R}^2} \int_0^{+\infty} \frac{e^{i\lambda \rho} \, V_0\!\left(u+\frac{x+y}{2}\right)}{\big|\frac{x-y}{2}-u\big|^{1/2} \big|u+\frac{x-y}{2}\big|^{1/2}} \, \chi(\rho)\chi\!\left(\frac{2\rho}{r}\right) \tilde{\chi}\!\left(\frac{\lambda}{\lambda_0}\right) d\lambda \, du,
	\end{equation}
	where $\rho=\big|\frac{x-y}{2}-u\big|-\big|u+\frac{x-y}{2}\big|$. The bounds produced below involve $x$ and $y$ only through the distances $|x-y|$, $|x-z|$, $|z-y|$ and through integrals of $|V_0|$ and $|\nabla V_0|$, all of which are invariant under rigid motions of the triple $(x,y,z)$. After a translation and a rotation we may therefore assume that
	\[
	x=\Big(-\frac{r}{2},\,0\Big),\qquad y=\Big(\frac{r}{2},\,0\Big),\qquad r=|x-y|,
	\]
	throughout the remainder of the proof.
	
	For $0 < |\rho| < r$, the equation $\big|\frac{x-y}{2}-u\big| - \big|u+\frac{x-y}{2}\big| = \rho$ defines one branch of a hyperbola with foci $x$ and $y$. We parametrize $u$ on the region $\{u\in\mathbb{R}^2 : |\rho| \le r\}$ by
	\begin{equation}\label{eq:tranform-to-hyper}
		u = \left( \frac{\rho}{2}\sec\theta,\; \frac{\sqrt{r^2-\rho^2}}{2}\tan\theta \right), \quad \theta \in \left(-\frac{\pi}{2},\frac{\pi}{2}\right),
	\end{equation}
	in which
	\[
	|x-z| = \frac{r\sec\theta + \rho}{2}, \qquad |z-y| = \frac{r\sec\theta - \rho}{2}, \qquad |x-z|\,|z-y|=\frac{r^2\sec^2\theta-\rho^2}{4},
	\]
	and the Jacobian determinant is
	\[
	\det(J) = \frac{\sec\theta\,(r^2\sec^2\theta - \rho^2)}{4\sqrt{r^2-\rho^2}} = \frac{|x-z|\,|z-y|\,\sec\theta}{\sqrt{r^2-\rho^2}}.
	\]
	Thus, up to a constant factor, the integral \eqref{eq:restri to small pho-1-change} becomes
	\begin{align}\label{eq:restri to small pho-1-1}
		\mathcal{L}=\int_{-\frac{\pi}{2}}^{\frac{\pi}{2}} \int_0^\infty \int_{\mathbb{R}} e^{i\lambda\rho} \, V_0\!\left(u(\rho,\theta)+\frac{x+y}{2}\right) \tau\sec\theta \, \chi(\rho)\chi\!\left(\frac{2\rho}{r}\right) d\rho \; \tilde{\chi}\!\left(\frac{\lambda}{\lambda_0}\right) d\lambda \, d\theta,
	\end{align}
	where $u(\rho,\theta)$ is given by \eqref{eq:tranform-to-hyper} and
	\[
	\tau=\tau(\rho,\theta):=\frac{\sqrt{r^2\sec^2\theta-\rho^2}}{\sqrt{r^2-\rho^2}}
	\qquad\Bigl(\,=\frac{2\sqrt{|x-z|\,|z-y|}}{\sqrt{r^2-\rho^2}}\,\Bigr).
	\]
	Note that $\tau\ge 1$ and, by \eqref{eq:support of region 3}, $\tau\sim\sec\theta$ on the support of the integrand. We now insert the decomposition $1 = \chi\big(\frac{\lambda}{\sec\theta}\big) + \tilde{\chi}\big(\frac{\lambda}{\sec\theta}\big)$ into \eqref{eq:restri to small pho-1-1} and write $\mathcal{L}=\mathcal{L}_1+\mathcal{L}_2$, where $\mathcal{L}_1$ and $\mathcal{L}_2$ carry the cutoffs $\chi(\frac{\lambda}{\sec\theta})$ and $\tilde{\chi}(\frac{\lambda}{\sec\theta})$, respectively. The scale $\sec\theta$ is chosen here, rather than the comparable quantity $\tau$, because it does not depend on $\rho$; this will allow us to integrate by parts in $\rho$ when treating $\mathcal{L}_2$.
	
	For $\mathcal{L}_1$, we perform the $\lambda$-integration first. In view of the derivative bound
	\[
	\sup_{\sigma\ge 1} \sigma^{-1/2} \left| \frac{d^k}{d\lambda^k} \chi\!\left(\frac{\lambda}{\sigma}\right) \right| \lesssim \lambda^{-1/2 - k}, \qquad k \in \mathbb{N}_0,
	\]
	Lemma~\ref{lem-oscillatory estimate-h}, applied with $\sigma=\sec\theta$, yields
	\begin{equation}\label{eq:restri to small pho-1-2-1}
		\left|\mathcal{L}_1\right|\lesssim \int_{-\frac{\pi}{2}}^{\frac{\pi}{2}} \int_{\mathbb{R}} \frac{|V_0| \,\tau^{3/2} \sec\theta}{|\rho|^{1/2}} \, \chi(\rho)\chi\!\left(\frac{2\rho}{r}\right) d\rho \, d\theta,
	\end{equation}
	where we used $\tau(\sec\theta)^{3/2}\sim \tau^{3/2}\sec\theta$, and $|V_0|$ is evaluated at $u(\rho,\theta)+\frac{x+y}{2}$.
	
	For $\mathcal{L}_2$, we integrate by parts in $\rho$ (the boundary terms vanish, the integrand being compactly supported in $\rho$) and are led to
	\[
	\mathcal{L}_2 = i \int_{-\frac{\pi}{2}}^{\frac{\pi}{2}} \int_{\mathbb{R}} \int_0^\infty e^{i\lambda\rho} \, \tilde{\chi}\!\left(\frac{\lambda}{\sec\theta}\right) \tilde{\chi}\!\left(\frac{\lambda}{\lambda_0}\right) \frac{d\lambda}{\lambda} \;
	\partial_{\rho}\!\Big(\tau \, V_0 \, \chi(\rho)\chi\!\left(\frac{2\rho}{r}\right)\Big) d\rho \; \sec\theta \, d\theta.
	\]
	In view of
	\[
	\sup_{\sigma\ge 1} \sigma^{1/2} \left| \partial_\lambda^k \tilde{\chi}\!\left(\frac{\lambda}{\sigma}\right) \right| \lesssim \lambda^{1/2 - k}, \qquad k \in \mathbb{N}_0,
	\]
	the amplitude $\lambda^{-1}\tilde{\chi}(\lambda/\sec\theta)\,\tilde{\chi}(\lambda/\lambda_0)$ obeys the symbol estimates of order $-1/2$ with constant $(\sec\theta)^{-1/2}$, so Lemma~\ref{lem-oscillatory estimate-h} gives
	\[
	\left|\mathcal{L}_2\right|\lesssim \int_{-\frac{\pi}{2}}^{\frac{\pi}{2}} \int_{\mathbb{R}} \frac{\sec\theta}{\sqrt{|\rho|\,\tau}} \,
	\left| \partial_{\rho}\!\Big(\tau \, V_0 \, \chi(\rho)\chi\!\left(\frac{2\rho}{r}\right)\Big) \right| d\rho \, d\theta,
	\]
	where we used $(\sec\theta)^{1/2}\sim\sec\theta/\sqrt{\tau}$. By the Leibniz rule,
	\[
	\partial_{\rho}\!\Big(\tau \, V_0 \, \chi(\rho)\chi\!\left(\frac{2\rho}{r}\right)\Big)
	= (\partial_{\rho}\tau) \, V_0 \, \chi(\rho)\chi\!\left(\frac{2\rho}{r}\right)
	+ \tau \, (\partial_{\rho}V_0) \, \chi(\rho)\chi\!\left(\frac{2\rho}{r}\right)
	+ \tau \, V_0 \, \partial_{\rho}\!\Big(\chi(\rho)\chi\!\left(\frac{2\rho}{r}\right)\Big).
	\]
	On the support of $\chi(\rho)\chi(2\rho/r)$, direct computations using \eqref{eq:support of region 3} show that
	\[
	|\partial_\rho V_0| = \left| \nabla V_0 \cdot \left( \frac{\sec\theta}{2},\; -\frac{\rho\tan\theta}{\sqrt{r^2-\rho^2}} \right) \right| \lesssim |\nabla V_0| \, \tau,
	\]
	\[
	|\partial_\rho \tau| \le \frac{|\rho|\sqrt{r^2\sec^2\theta-\rho^2}}{(r^2-\rho^2)^{3/2}} + \frac{|\rho|}{\sqrt{r^2-\rho^2}\,\sqrt{r^2\sec^2\theta-\rho^2}} \lesssim \frac{\tau}{r},
	\]
	and
	\[
	\left| \partial_{\rho}\!\Big(\chi(\rho)\chi\!\left(\frac{2\rho}{r}\right)\Big) \right| \lesssim 1 + \frac{1}{r}.
	\]
	Therefore
	\[
	\frac{\sec\theta}{\sqrt{|\rho|\,\tau}} \left| \partial_{\rho}\!\Big(\tau \, V_0 \, \chi(\rho)\chi\!\left(\frac{2\rho}{r}\right)\Big) \right|
	\lesssim \frac{|\nabla V_0|\,\tau^{3/2}\sec\theta}{\sqrt{|\rho|}}
	+ \frac{|V_0|\sqrt{\tau}\,\sec\theta}{r\sqrt{|\rho|}}
	+ \frac{|V_0|\sqrt{\tau}\,\sec\theta}{\sqrt{|\rho|}},
	\]
	and combining this with \eqref{eq:restri to small pho-1-2-1} we arrive at
	\begin{equation}\label{eq:restri to small pho-1-4}
		\left|\mathcal{L}(x,y)\right|\lesssim \int_{-\frac{\pi}{2}}^{\frac{\pi}{2}} \int_{|\rho|\le \min\{1,\, r/2\}} \left( \frac{(|\nabla V_0|+|V_0|)\,\tau^{3/2}\sec\theta}{\sqrt{|\rho|}}
		+ \frac{|V_0|\sqrt{\tau}\,\sec\theta}{r\sqrt{|\rho|}}
		+ \frac{|V_0|\sqrt{\tau}\,\sec\theta}{\sqrt{|\rho|}} \right) d\rho \, d\theta.
	\end{equation}
	The right-hand side is the kernel $\mathfrak{L}(x,y)$ of Lemma~\ref{lem:admi of main term in high} below; that lemma yields the admissibility of $\mathcal{L}$, and hence of $\mathcal{T}_1^{-}$. Together with Step 1, this completes the proof.
\end{proof}

\begin{lemma}\label{lem:admi of main term in high}
	Let $r=|x-y|$, and let $\rho=|x-z|-|z-y|$, $\tau=\frac{2(|x-z|\,|z-y|)^{1/2}}{\sqrt{r^2-\rho^2}}$ and $u(\rho,\theta)$ be as in the proof of Lemma~\ref{lem:boundedness of W^V_0-2}. Then the kernel
	\begin{equation*}
		\mathfrak{L}(x,y)=\int_{-\frac{\pi }{2}}^{\frac{\pi }{2}}\int_{|\rho|\le \min\{1, r/2 \}} \left(\frac{\big(|V_0|+|\nabla V_0|\big) \tau^{3/2} \sec\theta}{\sqrt{|\rho|}}
		+\frac{|V_0|\sqrt{\tau}\,\sec\theta}{ r \sqrt{|\rho|}}
		+\frac{|V_0|\sqrt{\tau} \,\sec\theta }{\sqrt{|\rho|}}\right) d\rho \, d\theta,
	\end{equation*}
	where $V_0$ and $\nabla V_0$ are evaluated at $z=u(\rho,\theta)+\frac{x+y}{2}$, is admissible.
\end{lemma}

\begin{proof}
	Reversing the change of variables \eqref{eq:tranform-to-hyper} shows that
	\[
	\mathfrak{L}(x,y)\lesssim \mathfrak{L}_1(x,y)+\mathfrak{L}_2(x,y)+\mathfrak{L}_3(x,y),
	\]
	where
	\begin{align*}
		\mathfrak{L}_1(x,y)&=  \int_{|\rho|\le \min\{1, r/2 \}}\frac{\tau^{1/2} \big(|V_0(z)|+|\nabla V_0(z)|\big)}{|x-z|^{1/2}|z-y|^{1/2}|\rho|^{1/2}}\,dz,
		\\
		\mathfrak{L}_2(x,y)&= \int_{|\rho|\le \min\{1, r/2 \}}\frac{|V_0(z)|}{r\,|x-z|^{1/2}|z-y|^{1/2}|\rho|^{1/2}\tau^{1/2}}\,dz,
		\\
		\mathfrak{L}_3(x,y)&=\int_{|\rho|\le \min\{1, r/2 \}}\frac{|V_0(z)|}{|x-z|^{1/2}|z-y|^{1/2}|\rho|^{1/2}\tau^{1/2}}\,dz.
	\end{align*}
	Each $\mathfrak{L}_j$ depends on $(x,y)$ only through $r=|x-y|$, $|\rho|=\big||x-z|-|z-y|\big|$, $|x-z|$ and $|z-y|$, and is therefore invariant under the exchange $x\leftrightarrow y$. Hence it suffices to show that $\sup_{y}\|\mathfrak{L}_j(\cdot,y)\|_{L^1_x}\lesssim 1$ for $j=1,2,3$.
	
	\medskip
	\noindent$\bullet$ \emph{Estimate for $\mathfrak{L}_1$.} Using
	\begin{equation}\label{eq:est for tau}
		\tau= \frac{2\,|x-z|^{1/2}|z-y|^{1/2}}{ \sqrt{r^2-\rho^2}} \lesssim \frac{|x-z|^{1/2}|z-y|^{1/2}}{r}
		\qquad (|\rho|\le r/2),
	\end{equation}
	we obtain
	\begin{equation}\label{eq:est for C-1}
		\left|\mathfrak{L}_1(x,y)\right| \lesssim \frac{1}{r^{1/2}}\int_{|\rho|\le 1}\frac{\big(|V_0|+|\nabla V_0|\big)(z)}{|x-z|^{1/4}|z-y|^{1/4}|\rho|^{1/2}}\,dz.
	\end{equation}
	We estimate the $L^1_x$-norm of the right-hand side separately on the regions $\{|z-y|\le 10\}$ and $\{|z-y|>10\}$.
	
	On the region $\{|z-y|\le 10\}$ one has $|x-y|\le 2|z-y|+|\rho|\le 30$. H\"older's inequality with exponents $\frac32$ and $3$, followed by the polar coordinates $x=z+t\omega$ with $s:=|z-y|$ and $t:=|x-z|$, gives
	\begin{equation}\label{eq:in the |z-y|le 10}
		\begin{aligned}
			&\sup_{y}\int_{\mathbb{R}^2}\frac{1}{|x-y|^{1/2}}\int_{\substack{|\rho|\le 1 \\ |z-y|\le 10}}\frac{\big(|V_0|+|\nabla V_0|\big)(z)}{|x-z|^{1/4}|z-y|^{1/4}|\rho|^{1/2}}\,dz\,dx \\
			&\quad \lesssim \sup_{y}\int_{|z-y| \le 10} \left(\int_{|\rho|\le 1} \frac{dx}{|x-z|^{3/8}|\rho|^{3/4}}\right)^{2/3} \frac{\big(|V_0|+|\nabla V_0|\big)(z)}{|z-y|^{1/4}} \, dz \\
			&\quad \lesssim \sup_{y}\int_{|z-y|\le 10} \left(\int_{|t-s|\le 1} \frac{t^{5/8}}{|t-s|^{3/4}}\,dt\right)^{2/3} \frac{\big(|V_0|+|\nabla V_0|\big)(z)}{|z-y|^{1/4}} \, dz \\
			&\quad \lesssim \sup_{y}\int_{\mathbb{R}^2} \frac{\big(|V_0|+|\nabla V_0|\big)(z)}{|z-y|^{1/4}} \, dz \lesssim 1,
		\end{aligned}
	\end{equation}
	where in the last step we used $\int_{|t-s|\le1} t^{5/8}|t-s|^{-3/4}\,dt\lesssim \langle s\rangle^{5/8}\lesssim 1$ for $s\le 10$; the other H\"older factor $\big\||x-y|^{-1/2}\big\|_{L^3(|x-y|\le 30)}$ is finite.
	
	On the region $\{|z-y|>10\}$, the condition $|\rho|\le 1$ forces $|x-z|\in[|z-y|-1,\,|z-y|+1]$, so in particular $|x-z|\sim|z-y|$ and
	\begin{equation}\label{eq:restri to small pho-1-4-2}
		\begin{aligned}
			&\sup_{y}\int_{\mathbb{R}^2}\frac{1}{|x-y|^{1/2}}\int_{\substack{|\rho|\le 1 \\ |z-y|> 10}}\frac{\big(|V_0|+|\nabla V_0|\big)(z)}{|x-z|^{1/4}|z-y|^{1/4}|\rho|^{1/2}}\,dz\,dx \\
			&\quad \lesssim \sup_{y}\int_{|z-y|>10} \left(\int_{|\rho|\le 1}\frac{dx}{|\rho|^{1/2}|x-y|^{1/2}}\right) \frac{\big(|V_0|+|\nabla V_0|\big)(z)}{|z-y|^{1/2}} \, dz.
		\end{aligned}
	\end{equation}
	To bound the inner integral we decompose its domain into the three subregions
	\[
	|x-y| < 2,\qquad  |x-y| \geq \tfrac{s}{2},
	\qquad 2 \leq |x-y| < \tfrac{s}{2},\qquad s=|z-y|.
	\]
	For fixed $|z-y|\ge 10$ and $|x-y|<2$, the angle between $x-z$ and $y-z$ is $\mathcal{O}(|z-y|^{-1})$; H\"older's inequality then gives
	\begin{align*}
		\int_{\substack{|\rho|\leq 1 \\ |x-y|<2}} \frac{dx}{|\rho|^{1/2}|x-y|^{1/2}}
		&\lesssim \left(\int_{\substack{|\rho|\leq 1 \\ |x-y|<2}} \frac{dx}{|\rho|^{3/4}}\right)^{2/3} \left(\int_{\substack{|x-y|\leq 2}} \frac{dx}{|x-y|^{3/2}}\right)^{1/3}
		\\
		&\lesssim \left(\int_{|\theta|\lesssim s^{-1}} \int_{|t-s|\leq 1} \frac{t}{|t-s|^{3/4}} \, dt\, d\theta\right)^{2/3}
		\lesssim 1 \lesssim |z-y|^{1/2},
	\end{align*}
	where we wrote $x-z$ in polar coordinates $(t,\theta)$ centered at $z$, with $s=|z-y|$ and $t=|x-z|$. Next,
	\begin{align*}
		\int_{\substack{|\rho|\leq 1 \\ |x-y|\geq s/2}} \frac{dx}{|\rho|^{1/2}|x-y|^{1/2}}
		\lesssim |z-y|^{-1/2}\int_{|t-s|\leq 1} \frac{t}{|t-s|^{1/2}}\,dt
		\lesssim |z-y|^{1/2},
	\end{align*}
	and, using $(r,\rho)$ as coordinates for $x$ in the remaining subregion,
	\begin{align*}
		\int_{\substack{|\rho|\leq 1 \\ 2 \leq |x-y| < s/2}} \frac{dx}{|\rho|^{1/2}|x-y|^{1/2}}
		\lesssim \int_{r \lesssim s} \int_{|\rho|\leq 1} r^{-1/2}|\rho|^{-1/2}\,d\rho\,dr
		\lesssim |z-y|^{1/2},
	\end{align*}
	where $dx = \left|\frac{\partial x}{\partial(r,\rho)}\right| dr\,d\rho$ and the Jacobian satisfies $\left|\frac{\partial x}{\partial(r,\rho)}\right|\lesssim 1$ on $\{2 \leq |x-y| < |z-y|/2\}$ with $|z-y|\ge 10$: indeed, on this region the circles $\{|x-y|=r\}$ and $\{|x-z|=|z-y|+\rho\}$ intersect at an angle bounded away from $0$ and $\pi$, since $2\le r\le s/2$ and $\big||x-z|-s\big|\le 1$. Combining the three subregions, we conclude that
	\begin{equation*}
		\int_{|\rho|\leq 1} \frac{dx}{|\rho|^{1/2}|x-y|^{1/2}} \lesssim |z-y|^{1/2}
		\qquad (|z-y|\ge 10).
	\end{equation*}
	Inserting this estimate into \eqref{eq:restri to small pho-1-4-2}, we obtain
	\[
	\sup_{y}\int_{|z-y|>10}\frac{1}{|x-y|^{1/2}}\int_{|\rho|\le 1}\frac{\big(|V_0|+|\nabla V_0|\big)(z)}{|x-z|^{1/4}|z-y|^{1/4}|\rho|^{1/2}}\,dz\,dx
	\lesssim \sup_{y} \int_{\mathbb{R}^2} \big(|V_0(z)| + |\nabla V_0(z)|\big) \, dz \lesssim 1,
	\]
	which, together with \eqref{eq:in the |z-y|le 10}, yields
	\begin{equation}\label{eq:admi of C1}
		\sup_{y} \bigl\| \mathfrak{L}_1(\cdot,y) \bigr\|_{L^1_x} \lesssim 1.
	\end{equation}
	
	\medskip
	\noindent$\bullet$ \emph{Estimate for $\mathfrak{L}_2$.} We proceed in the same spirit as for $\mathfrak{L}_1$. By \eqref{eq:est for tau},
	\[
	\left|\mathfrak{L}_2(x,y)\right| \lesssim \frac{1}{r^{1/2}}\int_{|\rho|\le 1}\frac{|V_0(z)|}{|x-z|^{3/4}|z-y|^{3/4}|\rho|^{1/2}}\,dz.
	\]
	On the region $\{|z-y|\le10\}$, H\"older's inequality as above gives
	\begin{align*}
		&\sup_{y}\int_{\mathbb{R}^2}\frac{1}{r^{1/2}}\int_{\substack{|\rho|\le 1 \\ |z-y|\le 10}}\frac{|V_0(z)|}{|x-z|^{3/4}|z-y|^{3/4}|\rho|^{1/2}}\,dz\,dx \\
		&\quad \lesssim \sup_{y}\int_{\mathbb{R}^2} \left(\int_{|\rho|\le 1} \frac{dx}{|x-z|^{9/8}|\rho|^{3/4}}\right)^{2/3} \frac{|V_0(z)|}{|z-y|^{3/4}} \, dz
		\lesssim 1,
	\end{align*}
	since $\int_{|t-s|\le1}t^{-1/8}|t-s|^{-3/4}\,dt\lesssim 1$ uniformly in $s\le 10$; on the region $\{|z-y|>10\}$, the inner integral bound obtained above yields
	\begin{align*}
		&\sup_{y}\int_{\mathbb{R}^2}\frac{1}{r^{1/2}}\int_{\substack{|\rho|\le 1 \\ |z-y|> 10}}\frac{|V_0(z)|}{|x-z|^{3/4}|z-y|^{3/4}|\rho|^{1/2}}\,dz\,dx \\
		&\quad \lesssim \sup_{y}\int_{|z-y|>10} \left(\int_{|\rho|\le 1}\frac{dx}{|\rho|^{1/2}|x-y|^{1/2}}\right) \frac{|V_0(z)|}{|z-y|^{3/2}} \, dz\\
		&\quad \lesssim \sup_{y}\int_{\mathbb{R}^2} \frac{|V_0(z)|}{|z-y|} \, dz \lesssim 1.
	\end{align*}
	These two estimates show that
	\begin{equation}\label{eq:admi of C2}
		\sup_{y}\left\|\mathfrak{L}_2(\cdot,y)\right\|_{L^1_x} \lesssim 1.
	\end{equation}
	
	\medskip
	\noindent$\bullet$ \emph{Estimate for $\mathfrak{L}_3$.} Since $\tau \geq 1$,
	\[
	\left|\mathfrak{L}_3(x,y)\right| \lesssim \int_{|\rho|\le 1}\frac{|V_0(z)|}{|x-z|^{1/2}|z-y|^{1/2}|\rho|^{1/2}}\,dz.
	\]
	Exchanging the order of integration and passing to polar coordinates $x=z+t\omega$ with $s=|z-y|$ and $t=|x-z|$,
	\[
	\int_{|\rho|\le 1}\frac{dx}{|x-z|^{1/2}|\rho|^{1/2}}
	= 2\pi\int_{|t-s|\le 1}\frac{t^{1/2}}{|t-s|^{1/2}}\,dt
	\lesssim \langle s\rangle^{1/2},
	\]
	hence
	\[
	\sup_{y}\left\|\mathfrak{L}_3(\cdot,y)\right\|_{L^1_x}
	\lesssim \sup_y\int_{\mathbb{R}^2} |V_0(z)|\,\frac{\langle s\rangle^{1/2}}{s^{1/2}}\,dz
	\lesssim \sup_y\int_{\mathbb{R}^2} |V_0(z)|\,\bigl(1+|z-y|^{-1/2}\bigr)\,dz
	\lesssim 1.
	\]
	Combining this with \eqref{eq:admi of C1} and \eqref{eq:admi of C2}, we conclude that the kernel $\mathfrak{L}(x,y)$ is admissible. The proof is complete.
\end{proof}

\medskip
It remains to treat $\mathcal{W}^{\mathrm{H}}_1$, for which direct oscillatory integral estimates suffice.

\begin{proposition}\label{lem:boundedness of W-2}
	Suppose that $|V(x)| \lesssim \langle x \rangle^{-\beta}$ with $\beta > \frac{7}{2}$. Then the operator $\mathcal{W}^{\mathrm{H}}_1$ is bounded on $L^p(\mathbb{R}^2)$ for all $1\le p\le \infty$.
\end{proposition}

\begin{proof}
	By Schur's lemma, it suffices to show that the kernel of $\mathcal{W}^{\mathrm{H}}_1$ is admissible.
	
	In view of \eqref{eq:asy-Hanel-2}, up to a constant factor the kernel of $\mathcal{W}^{\mathrm{H}}_1$ is the sum of two terms of the form
	\[
	\int_{\mathbb{R}^2}\int_{\mathbb{R}^2} \int_0^{+\infty} e^{i\lambda \rho_{\pm}} \, V(z)V(w) \, \mathcal{G}^{\pm}(\lambda, |x-z|,|z-w|,|w-y|) \, \lambda \, \tilde{\chi}\big(\tfrac{\lambda}{\lambda_0}\big) \, d\lambda \, dz \, dw,
	\]
	corresponding to the two choices of sign, where $\rho_{\pm} = |x-z| + |z-w| \pm |w-y|$ and
	\[
	\mathcal{G}^{\pm}(\lambda, |x-z|,|z-w|,|w-y|)=\Theta^{+}(\lambda|x-z|)\,\Theta^{+}(\lambda|z-w|)\,\Theta^{\pm}(\lambda|w-y|).
	\]
	By the derivative bounds on $\Theta^{\pm}$,
	\[
	\left|\partial_{\lambda}^k \mathcal{G}^\pm (\lambda, |x-z|,|z-w|,|w-y|)\right| \lesssim \lambda^{-3/2-k} \big(|x-z|\,|z-w|\,|w-y|\big)^{-1/2}, \quad k \in \mathbb{N}_0,
	\]
	so Lemma~\ref{lem-oscillatory estimate-h} yields
	\[
	\big|\mathcal{W}^{\mathrm{H}}_1(x,y)\big|\lesssim
	\int_{\mathbb{R}^2}\int_{\mathbb{R}^2}
	\frac{|V(z)V(w)|\, |\rho_{\pm}|^{-1/2}\langle \rho_{\pm} \rangle^{-2}}
	{\big(|x-z|\,|z-w|\,|w-y|\big)^{1/2}}
	\,dz\,dw.
	\]
	Integrating in $x$ in polar coordinates centered at $z$ and applying Fubini's theorem,
	\begin{align*}
		\sup_{y} \left\|\mathcal{W}^{\mathrm{H}}_1(\cdot,y)\right\|_{L^1_x}
		&\lesssim \sup_{y}\int_{\mathbb{R}^4} \left(\int_{0}^\infty \frac{t^{1/2}}{|t+s_\pm|^{1/2}\langle t+s_\pm\rangle^{2}} \, dt\right) \frac{|V(z)V(w)|}{(|z-w|\,|w-y|)^{1/2}} \, dz \, dw \\
		&\lesssim \sup_{y} \int_{\mathbb{R}^4} \frac{|V(z)V(w)|\,\big(1 + |z-w|^{1/2}+|w-y|^{1/2}\big)}{(|z-w|\,|w-y|)^{1/2}} \, dz \, dw
		\lesssim 1,
	\end{align*}
	where $s_\pm := |z-w| \pm |w-y|$; here we used $\int_0^\infty t^{1/2}|t+s|^{-1/2}\langle t+s\rangle^{-2}\,dt\lesssim 1+|s|^{1/2}$ for $s\in\mathbb{R}$, followed by $|s_\pm|^{1/2}\le |z-w|^{1/2}+|w-y|^{1/2}$, and the last step uses the uniform bound $\sup_{a\in\mathbb{R}^2}\int_{\mathbb{R}^2}|V(b)|\,|b-a|^{-1/2}\,db\lesssim 1$. The estimate $\sup_{x} \|\mathcal{W}^{\mathrm{H}}_1(x,\cdot)\|_{L^1_y}\lesssim 1$ follows in the same way, integrating first in $y$ in polar coordinates centered at $w$. Hence the kernel of $\mathcal{W}^{\mathrm{H}}_1$ is admissible, which completes the proof.
\end{proof}


\section{Low-energy properties of the resolvent kernel}\label{sec:low-energy-kernel}

In this section we establish several lemmas describing the low-energy
behavior of the free resolvent kernels. They will be used in the analysis of
the spectral measure and in the proof of the boundedness of
$\mathcal{W}^{\mathrm{L}}$, and they also account for the decay condition
$|V(x)|\lesssim\langle x\rangle^{-6-}$ imposed in
Theorem~\ref{thm-main result-1}.

The projections appearing in the expansion of $M^+(\lambda)^{-1}$ satisfy
certain vanishing moment conditions. It is therefore natural to work with
operators of the form $\mathcal{Q}_m v$, where $\mathcal{Q}_m$ is an
$L^2$-projection satisfying
\begin{equation}\label{eq:def of Qm}
\mathcal{Q}_m\bigl(v(x)\,x^{\alpha}\bigr)=0\qquad\text{for all }|\alpha|<m.
\end{equation}
Recall that the Hankel asymptotics recorded in
Subsection~\ref{Sec:some Technical lemmas} imply that
$vR_0^\pm(\lambda^2)(\cdot,y)$ belongs to $L^2(\mathbb{R}^2)$ whenever
$|v(x)|\lesssim\langle x\rangle^{-3-}$. Consequently,
$\mathcal{Q}_m vR_0^\pm(\lambda^2)$ has a well-defined distributional kernel,
denoted by
\[
\bigl[\mathcal{Q}_m vR_0^\pm(\lambda^2)\bigr](x,y).
\]

Throughout this section we write $\hat y:=y/|y|$ and use the abbreviations
\begin{equation}\label{eq:abre}
a=\langle y\rangle,\qquad b=|x-y|-a,\qquad
r_y(x)=x\cdot\hat y,\qquad
\tilde{b}=b+r_y(x).
\end{equation}
We now state the three main results of this section. The first is elementary;
the proofs of the other two occupy the remainder of this section and rely on
a unified expansion scheme---the master Lemma~\ref{lem:Taylor-scheme-0} and
Corollaries~\ref{lem:Taylor-scheme} and~\ref{lem:log-expansion}---which
reduces them to the verification of symbol bounds.

\medskip

\begin{lemma}\label{lem:est-vR-0}
Let $|v(x)|\lesssim\langle x\rangle^{-3-}$. Then we have
\begin{equation}\label{eq:expr-vH0}
v(x)R_0^\pm(\lambda^2)(x,y)=e^{\pm i\lambda a}\,\omega_0^\pm(\lambda,x,y),
\end{equation}
where the amplitudes $\omega_0^\pm$ satisfy, for $\lambda\in(0,1/2)$ and $k=0,1,2$,
\begin{align}
\label{eq:omega-est-1}
\bigl\|\partial_\lambda^k\omega_0^\pm(\lambda,\cdot,y)\bigr\|_{L^2(\mathbb{R}^2)}&\lesssim\lambda^{-\frac12-k}a^{-\frac12},\\[8pt]
\label{eq:omega-est-2}
\bigl\|\partial_\lambda^k\omega_0^\pm(\lambda,\cdot,y)\bigr\|_{L^2(\mathbb{R}^2)}&\lesssim(-\log\lambda)\,\lambda^{-k}.
\end{align}
\end{lemma}

\begin{lemma}\label{lem:est-S0-J0}
Assume $|v(x)|\lesssim\langle x\rangle^{-3-}$. Then for $m=1,2$ we have
\begin{equation}\label{eq:S0vR+R-}
\mathcal{Q}_m v\bigl(R_0^+(\lambda^2)(\cdot,y)-R_0^-(\lambda^2)(\cdot,y)\bigr)
=e^{i\lambda a}\,\mathcal{Q}_m\eta_m^+(\lambda,\cdot,y)+e^{-i\lambda a}\,\mathcal{Q}_m\eta_m^-(\lambda,\cdot,y),
\end{equation}
where the amplitude functions $\eta_m^\pm(\lambda,x,y)$ satisfy, for $\lambda\in(0,1/2)$ and $k=0,1,2$,
\begin{align}
\bigl\|\partial_\lambda^k\eta_m^\pm(\lambda,\cdot,y)\bigr\|_{L^2(\mathbb{R}^2)}&\lesssim\lambda^{m-\frac12-k}a^{-\frac12},\label{eq:est-eta-w-1}\\[4pt]
\bigl\|\partial_\lambda^k\eta_m^\pm(\lambda,\cdot,y)\bigr\|_{L^2(\mathbb{R}^2)}&\lesssim\lambda^{m-k}.\label{eq:est-eta-w-2}
\end{align}
\end{lemma}

\begin{lemma}\label{lem:est-S0vR0-2}
Assume $|v(x)|\lesssim\langle x\rangle^{-3-}$. Then for $m=1,2$ we have the decomposition
\begin{equation}\label{eq:S0vR-G-0}
\mathcal{Q}_m v\big(R_0^\pm(\lambda^2)(\cdot,y)\big)=e^{\pm i\lambda a}\,\mathcal{Q}_m\omega_m^\pm(\lambda,\cdot,y)-\mathcal{Q}_m\psi_m(\cdot,y)\,\chi(\lambda a),
\end{equation}
where $\omega_m^\pm(\lambda,x,y)$ satisfy, for $\lambda\in(0,1/2)$ and $k=0,1,2$,
\begin{equation}\label{eq:omega-est-3}
\bigl\|\partial_\lambda^k\omega_m^\pm(\lambda,\cdot,y)\bigr\|_{L^2(\mathbb{R}^2)}\lesssim\lambda^{m-\frac12-k}a^{-\frac12},
\end{equation}
and
\begin{equation}\label{eq:psi-est-0}
\|\psi_m(\cdot,y)\|_{L^2(\mathbb{R}^2)}\lesssim a^{-m}.
\end{equation}
\end{lemma}

\begin{proof}[Proof of Lemma~\ref{lem:est-vR-0}]
Factoring out the oscillation at frequency $a$, we define
$\omega_0^\pm(\lambda,x,y):=e^{\mp i\lambda a}v(x)R_0^\pm(\lambda^2)(x,y)$,
so that \eqref{eq:expr-vH0} holds by definition. By \eqref{eq:asy-Hanel-2}
and \eqref{eq:reso and Hankel-2},
\[
\omega_0^\pm(\lambda,x,y)=\pm\tfrac{i}{4}\,e^{\pm i\lambda b}v(x)\,\Theta^\pm(\lambda|x-y|),\qquad \Theta^\pm(z)=\mathcal{O}_\infty(z^{-1/2}).
\]
Each $\lambda$-derivative falls either on the phase, producing a factor
$b=\mathcal{O}(\langle x\rangle)$, or on the symbol $\Theta^\pm$, producing a
factor $\lambda^{-1}$. Hence Lemma~\ref{lemma-GV} and
$|v(x)|\lesssim\langle x\rangle^{-3-}$ give
\[
\bigl\|\partial_\lambda^k\omega_0^\pm(\lambda,\cdot,y)\bigr\|_{L^2}
\lesssim\lambda^{-\frac12-k}\left\|\frac{\langle x\rangle^2|v(x)|}{|x-y|^{1/2}}\right\|_{L^2(\mathbb{R}^2_x)}
\lesssim\lambda^{-\frac12-k}a^{-\frac12},
\]
which is \eqref{eq:omega-est-1}. When $\lambda a>1$ one has
$\lambda^{-\frac12}a^{-\frac12}\lesssim1\lesssim(-\log\lambda)$, so
\eqref{eq:omega-est-2} follows from \eqref{eq:omega-est-1}.

For $\lambda a\le 1$ we use instead the low-energy expansion of the Hankel
function recorded in Subsection~\ref{Sec:some Technical lemmas}:
\[
\omega_0^\pm(\lambda,x,y)=\pm\tfrac{i}{4}e^{\mp i\lambda a}v(x)\left(1\pm\frac{2i}{\pi}\gamma\pm\frac{2i}{\pi}\log\frac{\lambda|x-y|}{2}+\zeta_0^\pm(\lambda|x-y|)\right),
\quad \zeta_0^\pm(z)=\mathcal{O}_{\infty}(z^{3/2}),
\]
hence
\[
\bigl|\partial_\lambda^k\omega_0^\pm(\lambda,x,y)\bigr|\lesssim\lambda^{-k}|v(x)|\left(1-\log\lambda+\bigl|\log|x-y|\bigr|+\lambda^{3/2}|x-y|^{3/2}\right).
\]
Taking the $L^2$-norm in $x$ and invoking Lemma~\ref{lemma-GV}, we obtain
\[
\bigl\|\partial_\lambda^k\omega_0^\pm(\lambda,\cdot,y)\bigr\|_{L^2}\lesssim\lambda^{-k}\left(-\log\lambda+\log a+\lambda^{3/2}a^{3/2}\right)\lesssim(-\log\lambda)\,\lambda^{-k},
\]
since $\log a\lesssim |\log\lambda|$ and $\lambda^{3/2}a^{3/2}\lesssim1$ when
$\lambda a\le 1$ and $\lambda<1/2$. This proves \eqref{eq:omega-est-2}.
\end{proof}

\medskip

The proofs of Lemmas~\ref{lem:est-S0-J0} and~\ref{lem:est-S0vR0-2} rest on a
single Taylor-expansion principle, applied to
$f\in\{J_0^\pm,\,H_{0,g}^{\pm},\,h_0,\,h_1\}$: expand at the base point
$\lambda a$, recenter the dipole term at $r_y(x)$ via $b=\tilde b-r_y(x)$,
and peel the phase $e^{\pm i\lambda a}$ off the remainder. We isolate this
principle in the following master lemma.

\medskip

To implement the scheme, denote by
\begin{equation}\label{eq:Taylor-remainder}
\mathcal{R}_{m,f}(\lambda,a,b)
=\frac{(\lambda b)^{m}}{(m-1)!}\int_{0}^{1}(1-t)^{m-1}
f^{(m)}\bigl(\lambda a+t\lambda b\bigr)\,dt
\end{equation}
the $m$th-order integral Taylor remainder of $f$ expanded at $\lambda a$, and
set
\begin{equation}\label{eq:seconde- remainders}
\mathcal{A}_{f}^{\pm}(\lambda,a,b)
=e^{\mp i\lambda a}\Bigl[\lambda\tilde b\,f'(\lambda a)
+\mathcal{R}_{2,f}(\lambda,a,b)\Bigr].
\end{equation}
Recall from \eqref{eq:abre} the geometric identity
\begin{equation}\label{eq:geometric-b}
b=\tilde b-r_y(x),
\qquad\text{where}\quad
|b|\lesssim\langle x\rangle,\quad
|\tilde b\,|\lesssim\frac{\langle x\rangle^{2}}{a};
\end{equation}
the bound for $\tilde b$ follows by distinguishing the cases
$|x|\lesssim|y|$ and $|x|\gg|y|$. Applying Taylor's theorem to
$t\mapsto f\bigl(\lambda(a+tb)\bigr)$ at $t=0$ yields
\begin{equation}\label{eq:exp-usual}
f(\lambda|x-y|)=f(\lambda a)-\lambda\,r_y(x)\,f'(\lambda a)
+e^{\pm i\lambda a}\mathcal{A}_{f}^{\pm}(\lambda,a,b).
\end{equation}

\begin{lemma}\label{lem:Taylor-scheme-0}
Suppose that $f$ is smooth on $(0,\infty)$ and that
\begin{equation}\label{eq:assm of f}
e^{\mp iz}f^{(m)}(z)=\Pi_{m}^{\pm}(z)=\mathcal{O}_{\infty}(z^{\sigma_{m}}),
\qquad m=0,1,2.
\end{equation}
Then, for $k=0,1,2$,
\begin{equation}\label{eq:est-Af-1}
\bigl|\partial_{\lambda}^{k}\mathcal{A}_{f}^{\pm}(\lambda,a,b)\bigr|
\lesssim \lambda^{2-k}\langle x\rangle^{2}
\Bigl[(\lambda(a+b))^{\sigma_{0}}+(\lambda a)^{\sigma_{0}}
+(\lambda a)^{\sigma_{1}}\Bigr],
\qquad \lambda|b|\gtrsim 1,
\end{equation}
and, whenever $\sigma_{2}\in[-2,0]$ and $\lambda|b|\lesssim1$,
\begin{equation}\label{eq:est-Af-2}
\begin{aligned}
\bigl|\partial_{\lambda}^{k}\mathcal{A}_{f}^{\pm}(\lambda,a,b)\bigr|
\lesssim\lambda^{2-k}\langle x\rangle^{2}
\begin{cases}
(\lambda a)^{\sigma_{1}-1}+\bigl(\lambda a\bigr)^{\sigma_{2}},
&\sigma_2\in (-2, 0], \\[4pt]
(\lambda a)^{\sigma_{1}-1}
+(\lambda a)^{-\frac32}(\lambda a+\lambda b)^{-\frac12},
&\sigma_2=-2.
\end{cases}
\end{aligned}
\end{equation}
\end{lemma}

\begin{proof}
For $\lambda|b|\gtrsim 1$, we combine \eqref{eq:exp-usual} with
\eqref{eq:assm of f} to write
\[
\mathcal{A}_{f}^{\pm}(\lambda,a,b)
=e^{\pm i\lambda b}\Pi_{0}^{\pm}(\lambda a+\lambda b)
-\Pi_{0}^{\pm}(\lambda a)
+\lambda\,r_y(x)\,\Pi_{1}^{\pm}(\lambda a).
\]
Differentiation in $\lambda$ falls either on the symbols $\Pi_{m}^{\pm}$,
costing a factor $\lambda^{-1}$, or on the phase $e^{\pm i\lambda b}$,
costing $|b|$. Since $\lambda|b|\gtrsim1$ and $|b|\lesssim\langle x\rangle$,
the bound \eqref{eq:est-Af-1} follows directly from \eqref{eq:assm of f}.

For $\lambda|b|\lesssim 1$ the phase $e^{\pm i\lambda b}$ is bounded, and it
is more efficient to keep the second-order expansion intact: from
\eqref{eq:Taylor-remainder}, \eqref{eq:seconde- remainders} and
\eqref{eq:assm of f},
\[
\mathcal{A}_{f}^{\pm}(\lambda,a,b)
=\lambda\tilde b\,\Pi_{1}^{\pm}(\lambda a)
+(\lambda b)^{2}\int_{0}^{1}(1-t)\,e^{\pm it\lambda b}\,
\Pi_{2}^{\pm}\bigl(\lambda(a+tb)\bigr)\,dt.
\]
Since $|\tilde b|\lesssim\langle x\rangle^{2}/a$, the symbol bounds
\eqref{eq:assm of f} give
$\bigl|\partial_{\lambda}^{k}\bigl[\lambda\tilde b\,\Pi_{1}^{\pm}(\lambda a)\bigr]\bigr|
\lesssim\lambda^{2-k}\langle x\rangle^{2}(\lambda a)^{\sigma_{1}-1}$,
which is the first term on the right-hand side of \eqref{eq:est-Af-2}.
For the integral term, each $\lambda$-derivative costs at most
$\lambda^{-1}$, and \eqref{eq:geometric-b} gives
$a+tb\ge\max\bigl\{(1-t)a,\,t(a+b)\bigr\}$; inserting
$\Pi_2^\pm=\mathcal{O}_\infty(z^{\sigma_2})$ (for $\sigma_2=-2$, split the
power $-2$ as $-\frac32-\frac12$ and use the two lower bounds separately)
yields
\[
\biggl|\int_{0}^{1}(1-t)\,e^{\pm it\lambda b}\,
\Pi_{2}^{\pm}\bigl(\lambda(a+tb)\bigr)\,dt\biggr|
\lesssim
\begin{cases}
\bigl(\lambda a\bigr)^{\sigma_{2}}, &\sigma_2\in(-2,0],\\[4pt]
(\lambda a)^{-3/2}(\lambda a+\lambda b)^{-1/2}, &\sigma_2=-2,
\end{cases}
\]
since the resulting integrals $\int_{0}^{1}(1-t)^{\sigma_2+1}\,dt$
($\sigma_2>-2$) and $\int_{0}^{1}(1-t)^{-1/2}t^{-1/2}\,dt$ are finite.
Multiplying by $(\lambda b)^2\lesssim\lambda^2\langle x\rangle^2$ gives the
second term of \eqref{eq:est-Af-2}.
\end{proof}

Together with Lemma~\ref{lemma-GV}, this immediately yields the following
corollary.

\begin{corollary}\label{lem:Taylor-scheme}
Let $f$ satisfy the assumptions of Lemma~\ref{lem:Taylor-scheme-0}, and let
$|v(x)|\lesssim\langle x\rangle^{-3-}$.  Then, for $k=0,1,2$,
\begin{equation}\label{eq:est-AF}
\Bigl\|v(x)\,\partial_{\lambda}^{k}\mathcal{A}_{f}^{\pm}(\lambda,a,b)\Bigr\|_{L_{x}^{2}}
\lesssim
\begin{cases}
\lambda^{\frac32-k}a^{-\frac12}
&\text{if }\lambda|b|\gtrsim1\text{ and }
(\sigma_{0},\sigma_{1})=(-\tfrac12,-\tfrac12),\\[4pt]
\lambda^{\frac32-k}a^{-\frac12}
&\text{if }\lambda|b|\lesssim 1\text{ and }
(\sigma_{1},\sigma_{2})=(\tfrac12,-\tfrac12),\\[4pt]
\end{cases}
\end{equation}
and
\begin{equation}\label{eq:est-AF-1}
\Bigl\|v(x)\,\partial_{\lambda}^{k}\mathcal{A}_{f}^{\pm}(\lambda,a,b)\Bigr\|_{L_{x}^{2}}
\lesssim
\begin{cases}
\lambda^{2-k}
&\text{if }\lambda|b|\gtrsim1\text{ and }
(\sigma_{0},\sigma_{1})=(0,0),\\[4pt]
\lambda^{2-k}
&\text{if }\lambda|b|\lesssim1\text{ and }
(\sigma_{1},\sigma_{2})=(1,0).
\end{cases}
\end{equation}
\end{corollary}
\begin{proof}
Inserting the specified exponents into
\eqref{eq:est-Af-1}--\eqref{eq:est-Af-2} and applying the resulting pointwise
bounds on the corresponding regions of~$x$, the claim follows from
$\bigl\|v(x)\langle x\rangle^{2}(a+b)^{-1/2}\bigr\|_{L^2(\mathbb{R}^2_x)}\lesssim a^{-1/2}$,
which is a consequence of Lemma~\ref{lemma-GV} since $a+b=|x-y|$.
\end{proof}

\medskip

\begin{corollary}\label{lem:log-expansion}
Let $h_0(z)=\log z\,\chi(z)$ and $h_1(z)=\log z\,\tilde{\chi}(z)$, and let
$|v(x)|\lesssim\langle x\rangle^{-3-}$.
Then, for $j=0,1$, we have
\begin{equation*}
h_j(\lambda|x-y|)-h_j(\lambda a)+\lambda\,r_y(x)\,h_j'(\lambda a)
=e^{\pm i\lambda a}\,\mathcal{A}_{h_j}^{\pm}(\lambda,a,b),
\end{equation*}
where the remainders satisfy, for all $\lambda\in(0,1/2)$ and $k=0,1,2$,
\begin{equation}\label{eq:est-h0-remainder}
\bigl\|v(x)\partial_\lambda^k\mathcal{A}_{h_0}^{\pm}(\lambda,a,b)\bigr\|_{L_x^2}
\lesssim\lambda^{\frac32-k}a^{-\frac12},
\qquad\text{whenever }\lambda a\gtrsim 1,
\end{equation}
and
\begin{equation}\label{eq:est-h1-remainder}
\bigl\|v(x)\partial_\lambda^k\mathcal{A}_{h_1}^{\pm}(\lambda,a,b)\bigr\|_{L_x^2}
\lesssim\lambda^{\frac32-k}a^{-\frac12},
\qquad\text{whenever }\lambda a\lesssim 1.
\end{equation}
\end{corollary}

\begin{proof}
Since $h_0$ is supported near the origin and $h_1$ away from it, one checks
directly that
$e^{\pm iz}h_0(z)=\mathcal{O}_{\infty}(z^{-1/2})$,
$e^{\pm iz}h_0^{(\ell)}(z)=\mathcal{O}_{\infty}(z^{-\ell})$ for $\ell=1,2$,
and $h_1^{(\ell)}(z)=\mathcal{O}_{\infty}(z)$ for $\ell=0,1$,
$h_1^{(2)}(z)=\mathcal{O}_{\infty}(1)$.
Applying Lemma~\ref{lem:Taylor-scheme-0} with
$(\sigma_0,\sigma_1,\sigma_2)=(-\frac12,-1,-2)$ for $h_0$, and using
Lemma~\ref{lemma-GV} for the terms involving
$(a+b)^{-1/2}=|x-y|^{-1/2}$, every resulting term is
$\lesssim\lambda^{\frac32-k}a^{-\frac12}$ when $\lambda a\gtrsim 1$, which
proves \eqref{eq:est-h0-remainder}. For $h_1$, which is supported away from
zero, a direct computation gives
$\bigl|\partial_{\lambda}^k\mathcal{A}_{h_1}^{\pm}(\lambda,a,b)\bigr|
\lesssim\langle x\rangle^{2}\lambda^{2-k}$ when $\lambda a\lesssim 1$;
since $\lambda^{2-k}\lesssim\lambda^{\frac32-k}a^{-\frac12}$ in this regime,
multiplication by $v(x)$ gives \eqref{eq:est-h1-remainder}.
\end{proof}

\medskip

We now prove Lemmas~\ref{lem:est-S0-J0} and~\ref{lem:est-S0vR0-2}. In each
case \eqref{eq:exp-usual} is applied to
$f\in\{J_0^\pm,\,H_{0,g}^{\pm},\,h_0,\,h_1\}$ and $\mathcal{Q}_m$ is then
applied to $v$ times the expansion; by \eqref{eq:def of Qm} the zeroth-order
term drops out for $m=1,2$ and the dipole term for $m=2$.

\medskip

\begin{proof}[Proof of Lemma~\ref{lem:est-S0-J0}]
Recall that the difference of the free resolvents is expressed by the Bessel
kernel
$\bigl(R_0^+(\lambda^2)-R_0^-(\lambda^2)\bigr)(x,y)=\tfrac{i}{2}J_0(\lambda|x-y|)$.
We decompose $J_0$ by means of the Hankel asymptotics
\eqref{eq:asy-Hanel-0}--\eqref{eq:asy-Hanel-1} (see also
\cite[Chapter~9]{Abra}):
\begin{equation*}
J_0(z)=\zeta_1(z)+\frac{e^{iz}\Theta^+(z)}{2}\tilde{\chi}(z)+\frac{e^{-iz}\Theta^-(z)}{2}\tilde{\chi}(z)
=:J_0^+(z)+J_0^-(z),
\end{equation*}
where $\Theta^\pm(z)=\mathcal{O}_\infty(z^{-1/2})$, the remainder
$\zeta_1\in\mathcal{S}(\mathbb{R})$ satisfies
\begin{equation}\label{eq:est-for-zeta0}
\zeta_1(z)=1-\frac{z^2}{4}+\mathcal{O}_\infty(z^4)\qquad\text{as }z\to0,
\end{equation}
and $J_0^\pm(z):=\frac12\zeta_1(z)+\frac12e^{\pm iz}\Theta^\pm(z)\tilde\chi(z)$.
From \eqref{eq:est-for-zeta0} and the bounds for $\Theta^\pm$, the
derivatives of $J_0^\pm$ admit the unified representation
\begin{equation}\label{eq:rep-J0-pm}
e^{\mp iz}(J_0^\pm)^{(m)}(z)=\Pi_m^\pm(z),\quad m=0,1,2,
\end{equation}
where
\begin{equation}\label{eq:est-for-Pi-0}
\Pi_m^\pm(z)=\mathcal{O}_\infty(z^\sigma)\quad\text{for }m\in\{0,2\}\text{ and every }\sigma\in[-\tfrac12,0],
\end{equation}
and
\begin{equation}\label{eq:est-for-Pi-1}
\Pi_1^\pm(z)=\mathcal{O}_\infty(z^\sigma)\quad\text{for every }\sigma\in[-\tfrac12,1].
\end{equation}
Thus $J_0^+$ satisfies the hypotheses of
Corollary~\ref{lem:Taylor-scheme}, and $J_0^-$ does so by complex
conjugation. Applying \eqref{eq:exp-usual} to $f=J_0^+$ and its complex
conjugate to $f=J_0^-$, and adding the two identities, we obtain
\begin{equation}\label{eq:exp-for-J0-low}
J_0(\lambda|x-y|)-J_0(\lambda a)+\lambda\,r_y(x)\,J_0'(\lambda a)
=e^{i\lambda a}\,\mathcal{A}_{J_0^+}^{+}(\lambda,a,b)
+e^{-i\lambda a}\,\mathcal{A}_{J_0^-}^{-}(\lambda,a,b).
\end{equation}
Applying $\mathcal{Q}_m$ to $v$ times \eqref{eq:exp-for-J0-low} and noting
that $J_0'(z)=e^{iz}\Pi_1^+(z)+e^{-iz}\Pi_1^-(z)$ by \eqref{eq:rep-J0-pm},
we arrive at \eqref{eq:S0vR+R-} with
\begin{align}\label{eq:def-eta1}
\eta_1^\pm(\lambda,x,y)&=\tfrac{i}{2}v(x)\biggl(\mathcal{A}_{J_0^\pm}^\pm(\lambda,a,b)-\lambda\,r_y(x)\,\Pi_1^\pm(\lambda a)\biggr),\\[4pt]
\eta_2^\pm(\lambda,x,y)&=\tfrac{i}{2}v(x)\mathcal{A}_{J_0^\pm}^\pm(\lambda,a,b).\notag
\end{align}

It remains to verify \eqref{eq:est-eta-w-1}--\eqref{eq:est-eta-w-2}. In view
of \eqref{eq:est-for-Pi-0}--\eqref{eq:est-for-Pi-1},
Corollary~\ref{lem:Taylor-scheme} gives
\[
\bigl\|\partial_\lambda^k\bigl[v(x)\mathcal{A}_{J_0^\pm}^\pm(\lambda,a,b)\bigr]\bigr\|_{L^2}
\lesssim\min\bigl\{\lambda^{\frac32-k}a^{-\frac12},\ \lambda^{2-k}\bigr\},
\]
which is admissible in \eqref{eq:est-eta-w-1}--\eqref{eq:est-eta-w-2} for
both $m=1$ and $m=2$, since $\lambda<1/2$. For the dipole term in
\eqref{eq:def-eta1}, \eqref{eq:est-for-Pi-1} with $\sigma=0$ and
$\sigma=-\frac12$ yields
$\bigl|\partial_\lambda^k\bigl[\lambda\Pi_1^\pm(\lambda a)\bigr]\bigr|
\lesssim\min\bigl\{\lambda^{1-k},\,\lambda^{\frac12-k}a^{-\frac12}\bigr\}$;
since $|r_y(x)|\le|x|$, multiplication by
$|v(x)|\lesssim\langle x\rangle^{-3-}$ shows that this term also satisfies
\eqref{eq:est-eta-w-1}--\eqref{eq:est-eta-w-2}. This completes the proof.
\end{proof}


\begin{proof}[Proof of Lemma~\ref{lem:est-S0vR0-2}]
We only treat the $+$ sign; the $-$ case follows by complex conjugation.
We begin with the decomposition
\begin{equation}\label{eq:exp-for-R_000}
\begin{aligned}
 R_0^+(\lambda^2)(x,y)&=\frac{i}{4}H_{0,g}^{+}(\lambda|x-y|)-\frac{1}{2\pi}h_0(\lambda|x-y|) \\[2pt]
 &=\frac{i}{4}H_{0,g}^{+}(\lambda|x-y|)-\frac{1}{2\pi}\Bigl[\chi(\lambda a)\log(\lambda|x-y|)
-\chi(\lambda a)h_1(\lambda|x-y|)
+\tilde{\chi}(\lambda a)h_0(\lambda|x-y|)\Bigr],
\end{aligned}
\end{equation}
where $h_0,h_1$ are as in Corollary~\ref{lem:log-expansion} and
$H_{0,g}^{+}(z):=H_0^{+}(z)-\frac{2i}{\pi}\log z\,\chi(z)$.

Expanding each of $H_{0,g}^+$, $h_0$ and $h_1$ by \eqref{eq:exp-usual} and
applying $\mathcal{Q}_m$ to $v$ times the resulting identity, the
zeroth-order term drops out for $m=1,2$ and the dipole term drops out for
$m=2$: for $f\in\{H_{0,g}^{+},\,h_0,\,h_1\}$,
\begin{equation}\label{eq:exp-for-f-spec}
\mathcal{Q}_m\bigl[v\,f(\lambda|\cdot-y|)\bigr]
=\begin{cases}
-\lambda f'(\lambda a)\,\mathcal{Q}_1(v\,r_y)
+e^{i\lambda a}\,\mathcal{Q}_1\bigl(v\,\mathcal{A}_{f}^{+}\bigr), &m=1, \\[4pt]
e^{i\lambda a}\,\mathcal{Q}_2\bigl(v\,\mathcal{A}_{f}^{+}\bigr), &m=2,
\end{cases}
\end{equation}
where $\mathcal{Q}_m$ acts in the $x$-variable and all amplitudes are
evaluated at $(\lambda,a,b)$, with $b=b(x,y)=|x-y|-a$ as in \eqref{eq:abre}.

For the remaining logarithmic term,
\begin{equation}\label{eq:log-mean-zero}
\frac{1}{2\pi}\mathcal{Q}_m v\bigl[\log(\lambda|\cdot-y|)\bigr]
=\frac{1}{2\pi}\mathcal{Q}_m v\bigl[\log|\cdot-y|\bigr],
\end{equation}
since $\mathcal{Q}_m v=0$ kills the additive constant $\log\lambda$. On the
region $|x-y|\ge a/2$ we expand
\[
\log|x-y|=\log a+\frac{b}{a}+\mathcal{O}\biggl(\frac{\langle x\rangle^2}{a^2}\biggr)
=\log a-\frac{r_y(x)}{a}+\frac{\tilde b}{a}
+\mathcal{O}\biggl(\frac{\langle x\rangle^2}{a^2}\biggr);
\]
on the complementary region $|x-y|\le a/2$ one has $\langle x\rangle\gtrsim a$,
and an elementary estimate gives
$\bigl\|v(x)\log|x-y|\,\mathbf{1}_{\{|x-y|\le a/2\}}\bigr\|_{L^2_x}\lesssim a^{-2}$,
so this contribution is absorbed into $\psi_m$ directly. The constant
$\log a$ drops out for $m=1,2$ and the dipole $r_y/a$ drops out for $m=2$;
we therefore set
\begin{align}\label{eq:def-psi}
\psi_1(x,y):=\frac{1}{2\pi}v(x)\biggl[\frac{b}{a}
+\mathcal{O}\biggl(\frac{\langle x\rangle^2}{a^2}\biggr)\biggr],\quad
\psi_2(x,y):=\frac{1}{2\pi}v(x)\biggl[\frac{\tilde b}{a}
+\mathcal{O}\biggl(\frac{\langle x\rangle^2}{a^2}\biggr)\biggr],\notag
\end{align}
with $b,\tilde b$ from \eqref{eq:abre} viewed as functions of $x$, so that
$\frac{1}{2\pi}\mathcal{Q}_m v\bigl[\log|\cdot-y|\bigr]
=\mathcal{Q}_m\psi_m(\cdot,y)$ up to the absorbed contribution above.

Combining \eqref{eq:exp-for-R_000}, \eqref{eq:exp-for-f-spec} and
\eqref{eq:log-mean-zero}, we arrive at \eqref{eq:S0vR-G-0} with
\begin{equation}\label{eq:def-omega-2}
\omega_2^+(\lambda,x,y)=v(x)\Bigl[\frac{i}{4}\mathcal{A}_{H_{0,g}^{+}}^{+}
+\frac{1}{2\pi}\chi(\lambda a)\mathcal{A}_{h_1}^{+}
-\frac{1}{2\pi}\tilde{\chi}(\lambda a)\mathcal{A}_{h_0}^{+}\Bigr],
\end{equation}
all amplitudes being evaluated at $(\lambda,a,b)$, and
\begin{equation}\label{eq:def-omega-1}
\begin{aligned}
\omega_1^+(\lambda,x,y)=\omega_2^+(\lambda,x,y)
-\lambda\,r_y(x)\,e^{-i\lambda a}v(x)\,G_0(\lambda a),
\end{aligned}
\end{equation}
where $G_0(z):=\tfrac{i}{4}(H_{0,g}^{+})'(z)
-\tfrac{1}{2\pi}\tilde{\chi}(z)h_0'(z)
+\tfrac{1}{2\pi}\chi(z)h_1'(z).$

It remains to verify \eqref{eq:omega-est-3} and \eqref{eq:psi-est-0}.
From the Hankel asymptotics \eqref{eq:asy-Hanel-0}--\eqref{eq:asy-Hanel-1} we
read off
$e^{-iz}(H_{0,g}^{+})^{(m)}(z)=\mathcal{O}_\infty(z^{-1/2})$ for $m=0,2$ and
$e^{-iz}(H_{0,g}^{+})'(z)=\mathcal{O}_\infty(z^{\sigma})$ for
$\sigma\in\{-\frac12,\frac12\}$. The bounds in \eqref{eq:est-AF}, applied
with $(\sigma_0,\sigma_1)=(-\frac12,-\frac12)$ on the region
$\lambda|b|\gtrsim1$ and with $(\sigma_1,\sigma_2)=(\frac12,-\frac12)$ on
$\lambda|b|\lesssim1$, then give, for all $\lambda\in(0,1/2)$ and $k=0,1,2$,
\begin{equation}\label{eq:est-for-H0-remainder}
\Bigl\|v(x)\partial_\lambda^k\mathcal{A}_{H_{0,g}^{+}}^{+}(\lambda,a,b)\Bigr\|_{L_x^2}
\lesssim\lambda^{\frac32-k}a^{-\frac12}.
\end{equation}
In \eqref{eq:def-omega-2} each amplitude is paired with the cut-off of the
regime in which its estimate holds, and differentiating the cut-offs is
harmless since $(\lambda a)^j\chi^{(j)}(\lambda a)$ and
$(\lambda a)^j\tilde{\chi}^{(j)}(\lambda a)$ are bounded for every $j\ge0$;
therefore \eqref{eq:est-for-H0-remainder}, \eqref{eq:est-h0-remainder} and
\eqref{eq:est-h1-remainder} yield \eqref{eq:omega-est-3} for $\omega_2^+$.
For the dipole term in \eqref{eq:def-omega-1},
$G(z):=e^{-iz}G_0(z)=\mathcal{O}_{\infty}(z^{-1/2})$ (the main contribution
is $e^{-iz}(H_{0,g}^+)'(z)=\mathcal{O}_\infty(z^{-1/2})$, while
$\tilde{\chi}h_0'$ and $\chi h_1'$ are smooth and compactly supported in
$(0,\infty)$), so the dipole contribution $-\lambda\,r_y(x)\,v(x)\,G(\lambda a)$
to $\omega_1^+$ satisfies
$\bigl|\partial_{\lambda}^k\bigl[\lambda\,r_y(x)\,G(\lambda a)\bigr]\bigr|
\lesssim\lambda^{\frac12-k}\langle x\rangle a^{-\frac12}$,
and \eqref{eq:omega-est-3} for $\omega_1^+$ follows upon multiplication by
$|v(x)|\lesssim\langle x\rangle^{-3-}$. Finally, \eqref{eq:psi-est-0} is
immediate from $|b|\lesssim\langle x\rangle$,
$|\tilde b|\lesssim\langle x\rangle^2/a$ and
$|v(x)|\lesssim\langle x\rangle^{-3-}$, which give
$|\psi_m(x,y)|\lesssim\langle x\rangle^{-1-}a^{-m}$.
\end{proof}


\begin{remark}\label{rem:rela-of-omegas}
The function $G_0$ in \eqref{eq:def-omega-1} (and its $-$-analogue)
simplifies algebraically: since
$H_{0,g}^{\pm}=H_0^{\pm}\mp\frac{2i}{\pi}h_0$ by definition and
$h_0'+h_1'=\bigl(\log z\,(\chi+\tilde\chi)\bigr)'=\frac{1}{z}$,
\[
\frac{i}{4}(H_{0,g}^{\pm})'\mp\frac{1}{2\pi}\,\tilde{\chi}h_0'
\pm\frac{1}{2\pi}\,\chi h_1'
=\frac{i}{4}(H_0^{\pm})'\pm\frac{\chi(z)}{2\pi z}.
\]
Multiplied by $-\lambda r_y(x)$ and with $\mathcal{Q}_1$ applied, the last
term on the right-hand side becomes
$\mp\frac{\chi(\lambda a)}{2\pi a}\mathcal{Q}_1[v\,r_y](\cdot,y)$,
which has the same form as the
$\mathcal{Q}_1\psi_1(\cdot,y)\chi(\lambda a)$ term in
\eqref{eq:S0vR-G-0} and is absorbed into it; the bound \eqref{eq:psi-est-0}
is unaffected since $|v(x)r_y(x)|/a\lesssim\langle x\rangle^{-2-}a^{-1}$.
Consequently,
\begin{align*}
\mathcal{Q}_1 v\bigl(R_0^\pm(\lambda^2)(\cdot,y)\bigr)
=e^{\pm i\lambda a}\,\mathcal{Q}_1\omega_2^\pm(\lambda,\cdot,y)
-\mathcal{Q}_1\psi_1(\cdot,y)\,\chi(\lambda a)
\mp \tfrac{i}{4}\lambda (H_0^\pm)'(\lambda a)\,\mathcal{Q}_1[v\,r_y](\cdot,y).
\end{align*}
\end{remark}

\begin{remark}\label{remk:on the diff}
The same argument, with the base point $\lambda|y|$ in place of $\lambda a$
(using $|y|\sim\langle y\rangle$ for $|y|\ge1$), shows that for
$|v(x)|\lesssim\langle x\rangle^{-3-}$ and $|y|\ge 1$ one has
\[
vR_0^+(\lambda^2)(\cdot,y)-\tfrac{i}{4}vH_0^+(\lambda|y|)
=e^{i\lambda|y|}\tilde{\omega}_1^+(\lambda,\cdot,y)
+\chi(\lambda a)\,\tilde{\psi}_1(\cdot,y),
\]
where $\tilde{\omega}_1^+$ and $\tilde{\psi}_1$
satisfy \eqref{eq:omega-est-3} and \eqref{eq:psi-est-0}, respectively,
with $m=1$.
\end{remark}

\section{Low-energy estimates I: the regular point and first-kind resonance}\label{sec:Low-bound}

This section is devoted to establishing the boundedness of the low-energy component $\mathcal{W}^{\mathrm{L}}$ when zero is a regular point or a first-kind threshold singularity of $H$, as stated in Proposition~\ref{pro:bound for low energy}. Unless otherwise specified, we assume throughout this section that $|V(x)|\lesssim\langle x\rangle^{-6-}$, in accordance with the hypothesis of Proposition~\ref{pro:bound for low energy}.

Recall the stationary representation of the low-energy part $\mathcal{W}^{\mathrm{L}}$:
\begin{equation}\label{eq:stationary formula-Low-2}
	\mathcal{W}^{\mathrm{L}} = \frac{1}{\pi i}\int_0^{+\infty} \lambda\, R^+(\lambda^2)v \big(M^+(\lambda)\big)^{-1} v \big(R_0^+(\lambda^2) - R_0^-(\lambda^2)\big) \chi(\lambda/\lambda_0)\,d\lambda.
\end{equation}
In view of \eqref{eq:stationary formula-Low-2}, the analysis of $\mathcal{W}^{\mathrm{L}}$ requires a precise description of the operator $R^+(\lambda^2)v \big(M^+(\lambda)\big)^{-1} v \big(R_0^+(\lambda^2) - R_0^-(\lambda^2)\big)$ for $\lambda$ near zero. As a first step, we study the asymptotic behavior of $\big(M^{\pm}(\lambda)\big)^{-1}$ near zero energy.

\subsection{Asymptotic behavior of $(M^{\pm}(\lambda))^{-1}$}

The asymptotic expansion of $(M^{\pm}(\lambda))^{-1}$ in two dimensions was first established in \cite[Theorem 6.2]{JN01}; alternative proofs and refinements can be found in \cite{Schlag-2D-dis, Jensen-Yajima-02-CMP, Erd-Gre-13, Erdogan-Goldberg-Green-JFA-2018}. We remark that the invertibility of $M^{\pm}(\lambda)$ for $\lambda>0$ follows from the absence of embedded eigenvalues; the expansion below provides an alternative, direct proof of this invertibility for small $\lambda$.

To state the expansion in a concise form, we introduce the following notation. For $\Omega\subset\mathbb{R}$, we write $f(\lambda)=\tilde{\mathcal{O}}_k(g(\lambda))$ for a $\mathbb{B}(L^2)$-valued function $f$ if
\begin{equation*}
	\left\|\frac{d^\ell f(\lambda)}{d\lambda^\ell}\right\|_{\mathbb{B}(L^2)}\lesssim_\ell |g(\lambda)|\,|\lambda|^{-\ell},\qquad \lambda\in\Omega,\quad \ell=0,1,\dots,k.
\end{equation*}
We occasionally abuse notation and write $\tilde{\mathcal{O}}_k(\lambda^b)$ for any operator-valued function in this class.

The zero-energy expansion of $(M^{\pm}(\lambda))^{-1}$ involves the integral operators arising from the free resolvent. Define
\begin{align*}
G_0f(x)&:=-\frac{1}{2\pi}\int_{\mathbb{R}^2}f(y)\log|x-y|\,dy,\\
G_1f(x)&:=\frac{1}{8\pi}\int_{\mathbb{R}^2}f(y)|x-y|^2\,dy,\\
G_2f(x)&:=\frac{1}{8\pi}\int_{\mathbb{R}^2}f(y)|x-y|^2\log|x-y|\,dy,
\end{align*}
and set
\begin{equation}\label{eq:def of T}
	T:=U+vG_0v.
\end{equation}

To describe the singularities of $(M^{\pm}(\lambda))^{-1}$ at zero energy, we need a hierarchy of projections on $L^2$ that detect the presence of resonances and eigenfunctions. For the definitions below, the condition $|V(x)|\lesssim\langle x\rangle^{-5-}$ suffices. We introduce the following mutually orthogonal projections on $L^2$:
\begin{align}
	&Pf:=\frac{\langle f,v\rangle}{\|V\|_{L^1}}\,v(x),\qquad Q:=I-P,\notag\\
	\label{eq-def-for-QS}
	&S_1L^2:=\ker(QTQ)\cap QL^2,\\
	\label{eq-def-for-S2}
	&S_2L^2:=\ker(S_1TPTS_1)\cap S_1L^2\quad\text{if }S_1\neq0,\\
	\label{eq-def-for-S3}
	&S_3L^2:=\ker(S_2vG_1vS_2)\cap S_2L^2\quad\text{if }S_2\neq0,
\end{align}
with the convention that $S_{j+1}=0$ whenever $S_j=0$. By construction,
\begin{equation}\label{eq:S2T=0}
	S_2T=TS_2=0.
\end{equation}
Indeed, for $\varphi\in S_1L^2$ one has $QTQ\varphi=0$, hence $T\varphi=PT\varphi=\frac{\langle\varphi,Tv\rangle}{\|V\|_{L^1}}\,v$; it follows that $S_1TPTS_1\varphi=\frac{\langle\varphi,S_1Tv\rangle}{\|V\|_{L^1}}\,S_1Tv$, so the defining condition of $S_2$ forces $\langle\varphi,Tv\rangle=0$ and therefore $T\varphi=0$. Since $T=T^{*}$, this gives \eqref{eq:S2T=0}.
Moreover, for any $\varphi\in S_2L^2\subset QL^2$ one has $\int_{\mathbb{R}^2}v\varphi\,dx=0$, and a direct computation using $|x-y|^2=|x|^2+|y|^2-2x\cdot y$ gives
\[
\langle vG_1v\varphi,\varphi\rangle=-\frac{1}{4\pi}\sum_{i=1}^2\left|\int_{\mathbb{R}^2}x_i v(x)\varphi(x)\,dx\right|^2.
\]
Consequently, the definition of $S_3$ can be rephrased as
\begin{equation}\label{eq:property of S-3}
	S_3L^2=\left\{\varphi\in S_2L^2\;\Big|\;\int_{\mathbb{R}^2}x_i v(x)\varphi(x)\,dx=0\text{ for }i=1,2\right\},
\end{equation}
and in particular
\begin{equation*}
	\dim\bigl((S_2-S_3)L^2\bigr)\le2.
\end{equation*}

The following result from \cite[Lemma 6.4]{JN01} or \cite[Section 5]{Erd-Gre-13} identifies the threshold subspaces of $H$ with the ranges of the projections $S_1-S_2$, $S_2-S_3$ and $S_3$.

\begin{lemma}\label{lem:one to one lem}
	Let $\mathscr{N}_s$, $\mathscr{N}_p$, and $\mathscr{N}_e$ denote the spaces of s-wave resonances, p-wave resonances, and zero-energy eigenfunctions of $H$, respectively. Then the mapping $\psi\mapsto Uv\psi$ induces isomorphisms from $\mathscr{N}_s$ onto $(S_1-S_2)L^2$, from $\mathscr{N}_p$ onto $(S_2-S_3)L^2$, and from $\mathscr{N}_e$ onto $S_3L^2$.
\end{lemma}

\begin{remark}\label{remk:one-to-one}
	Lemma~\ref{lem:one to one lem} implies that zero is a regular point of $H$ if and only if $S_1=0$, and that zero is a first-kind threshold singularity precisely when $S_1\neq0$ while $S_2=S_3=0$.

	We also note that $S_3(vx^\alpha)=0$ for all $|\alpha|\le 1$: the case $|\alpha|=0$ follows from $S_3L^2\subset QL^2$, and the case $|\alpha|=1$ from \eqref{eq:property of S-3}. Since $V\psi=v\,(Uv\psi)$, Lemma~\ref{lem:one to one lem} therefore yields
	\[
	\int_{\mathbb{R}^2} x^{\alpha} V(x)\psi(x)\,dx=0,\qquad |\alpha|\le 1,
	\]
	for every zero-energy eigenfunction $\psi$ of $H$.
\end{remark}

With these projections at hand, we now introduce the remaining auxiliary operators that appear in the expansion. Set
\begin{align}
	&D_0:=(QTQ+S_1)^{-1}, \quad D_1:=(S_1TPTS_1)^{-1},\notag\\
	\label{eq:def of K}
	&N:=P-PTQD_0Q-QD_0QTP+QD_0QTPTQD_0Q.
\end{align}
Here $QTQ+S_1$ is invertible on $QL^2$ (see \cite{JN01}), and $S_1TPTS_1$ is invertible on $S_1L^2$ whenever $S_2=0$. We also define
\begin{equation}\label{eq:def-of-hpm}
	h_{\pm}(\lambda):=-\frac{\|V\|_{L^1}}{2\pi}\log\lambda+z_0^{\pm},\qquad z_0^{\pm}\in\mathbb{C}\setminus\mathbb{R}.
\end{equation}

The following expansion, which is essentially \cite[Lemma 2.3]{Erdogan-Goldberg-Green-JFA-2018}, gives the precise asymptotic behavior of $(M^{\pm}(\lambda))^{-1}$ near zero energy.

\begin{lemma}\label{thm-M inverse-even}
	Suppose $|V(x)|\lesssim\langle x\rangle^{-5-}$. Then there exists $\lambda_0=\lambda_0(V)\in(0,1)$ such that for all $\lambda\in(0,\lambda_0)$ the operators $(M^{\pm}(\lambda))^{-1}\in\mathbb{B}(L^2)$ admit the following expansions.

	\smallskip
	\noindent\emph{(i)} If zero is a regular point of $H$, then
	\begin{equation}\label{eq-M expansion-even}
		(M^{\pm}(\lambda))^{-1}=h_{\pm}(\lambda)^{-1}N+QD_0Q+\tilde{\mathcal{O}}_2(\lambda^{3/2}).
	\end{equation}

	\smallskip
	\noindent\emph{(ii)} If zero is a first-kind threshold singularity of $H$, then
	\begin{equation}\label{eq-M expansion-even-1-1}
		(M^{\pm}(\lambda))^{-1}=-h_{\pm}(\lambda)S_1D_1S_1-NS_1D_1S_1-S_1D_1S_1N+QD_0Q+\tilde{\mathcal{O}}_2(\lambda^{3/2}).
	\end{equation}
\end{lemma}

\begin{remark}
	When zero is a first-kind threshold singularity of $H$, \cite[Lemma 2.3]{Erdogan-Goldberg-Green-JFA-2018} in fact gives the expansion
	\begin{equation}\label{eq-M expansion-even-1}
		\begin{aligned}
			(M^{\pm}(\lambda))^{-1}&=-h_{\pm}(\lambda)S_1D_1S_1-NS_1D_1S_1-S_1D_1S_1N\\
			&\quad-h_{\pm}(\lambda)^{-1}NS_1D_1S_1N+h_{\pm}(\lambda)^{-1}N+QD_0Q+\tilde{\mathcal{O}}_2(\lambda^{3/2})
		\end{aligned}
	\end{equation}
	in place of \eqref{eq-M expansion-even-1-1}. Using the identities
	\[
	S_1D_0=S_1,\quad S_1Q=QS_1=S_1,\quad S_1P=PS_1=0,\quad PTS_1D_1S_1TP=P,
	\]
	a direct computation gives $NS_1D_1S_1N=N$, so that \eqref{eq-M expansion-even-1} reduces to \eqref{eq-M expansion-even-1-1}.
\end{remark}

We now insert the expansions of Lemma~\ref{thm-M inverse-even} into the stationary representation \eqref{eq:stationary formula-Low-2}, using also the resolvent identity
\[
	R^+(\lambda^2)v=R_0^+(\lambda^2)v\big(M^+(\lambda)\big)^{-1}U.
\]
The boundedness of $\mathcal{W}^{\mathrm{L}}$ then follows by analyzing the contribution of each resulting term separately.

To organize the analysis, observe that the terms arising from the expansions \eqref{eq-M expansion-even} and \eqref{eq-M expansion-even-1-1} fall into four classes according to the projections involved:
\begin{itemize}
	\item terms whose rightmost projection is either $Q$ or $S_1$;
	\item terms whose leftmost projection is either $Q$ or $S_1$;
	\item terms with $P$ as both the leftmost and the rightmost projection;
	\item the error terms.
\end{itemize}
Recall that $Qv=0$ and $S_1v=0$, and let $\mathcal{Q}_1$ be as in \eqref{eq:def of Qm}. The proof of Proposition~\ref{pro:bound for low energy} is thereby reduced to establishing the boundedness of operators of the following canonical forms:
\begin{align}\label{eq-rep-W0}
	\mathcal{N}_0&=
	\frac{1}{\pi i}\int_0^{+\infty}\lambda f_0(\lambda)\chi(\lambda/\lambda_0)\,R_0^{+}(\lambda^2)v\mathcal{D}\mathcal{Q}_1 v\big(R_0^+(\lambda^{2})-R_0^-(\lambda^{2})\big)\,d\lambda, \\
	\label{eq-rep-W1}
	\mathcal{N}_1&=\frac{1}{\pi i}\int_0^{+\infty}\lambda f_1(\lambda)\chi(\lambda/\lambda_0)\,R_0^{+}(\lambda^2)v\mathcal{Q}_1\mathcal{D}\mathcal{P} v\big(R_0^+(\lambda^{2})-R_0^-(\lambda^{2})\big)\,d\lambda, \\
	\label{eq-rep-W2}
	\mathcal{N}_2&=\frac{1}{\pi i}\int_0^{+\infty}\lambda\,\chi(\lambda/\lambda_0)\,R_0^{+}(\lambda^2)v\Gamma_2(\lambda)v\big(R_0^+(\lambda^{2})-R_0^-(\lambda^{2})\big)\,d\lambda,
\end{align}
and
\begin{equation}\label{eq-rep-W_s}
	\mathcal{W}_\mathrm{sing}^{\mathrm{L}}=\frac{1}{\pi i}\int_0^{+\infty}\lambda h_{+}(\lambda)^{-1}\chi(\lambda/\lambda_0)\,R_0^{+}(\lambda^2)vPv\big(R_0^+(\lambda^{2})-R_0^-(\lambda^{2})\big)\,d\lambda.
\end{equation}
The first two classes above give rise to $\mathcal{N}_0$ and $\mathcal{N}_1$, the third to $\mathcal{W}_{\mathrm{sing}}^{\mathrm{L}}$, and the error terms are collected in $\mathcal{N}_2$. In the preceding integrals $\mathcal{D},\Gamma_2(\lambda)\in\mathbb{B}(L^2)$, $\mathcal{P}$ is a finite-rank projection on $L^2$, and $f_0,f_1$ are scalar functions whose precise hypotheses are stated in Lemmas~\ref{Lemm: bound of N_0} and~\ref{Lemm: bound of N_1} below.

\subsection{Boundedness of $\mathcal{N}_0$, $\mathcal{N}_1$ and $\mathcal{N}_2$}\label{sec:bound-of-N0-N2}

In this subsection we establish the $L^p$-boundedness of the operators
$\mathcal{N}_0$, $\mathcal{N}_1$ and $\mathcal{N}_2$ defined in \eqref{eq-rep-W0}--\eqref{eq-rep-W2}. The argument combines
the low-energy resolvent expansions of
Section~\ref{sec:low-energy-kernel} with the oscillatory integral
estimates of Section~\ref{sec:Pre}: in each case the kernel is written as an
oscillatory integral in $\lambda$ whose amplitude is a pairing of two
low-energy amplitudes, and the required pointwise bounds then follow from
Lemma~\ref{lem-oscillatory estimate}. 

\begin{lemma}\label{Lemm: bound of N_0}
	Let $|V(x)|\lesssim \langle x\rangle^{-6-}$. If $f_0 \in C^{2}(\mathbb{R})$ satisfies
	\begin{equation}\label{eq:assum for f_0}
		\big|\partial_{\lambda}^{k} f_0(\lambda)\big| \lesssim \lambda^{-k-}, \qquad k=0,1,2, \quad 0<\lambda<\lambda_0,
	\end{equation}
	then the operator $\mathcal{N}_0$ given in \eqref{eq-rep-W0} is bounded on $L^p(\mathbb{R}^2)$ for all $1\le p\le\infty$.
\end{lemma}

\begin{proof}
By Lemma~\ref{Schur Lemma}, it suffices to show that the
kernel of $\mathcal{N}_0$ is admissible.

From \eqref{eq-rep-W0}, the kernel of $\mathcal{N}_0$ admits
the representation
\begin{equation*}
\mathcal{N}_0(x,y)=\frac{1}{\pi i}\int_0^{+\infty} \lambda f_0(\lambda) \chi(\lambda/\lambda_0)
\bigl\langle \mathcal{D}\mathcal{Q}_1v\,\big(R_0^{+}-R_0^{-}\big)(\lambda^2)(\cdot,y),\; vR_0^{-}(\lambda^2)(\cdot, x)\bigr\rangle \, d\lambda ,
\end{equation*}
where the kernel of $\mathcal{Q}_1 v\bigl(R_0^+(\lambda^2)-R_0^-(\lambda^2)\bigr)$ is identified with the function
$\mathcal{Q}_1v\bigl(R_0^+(\lambda^2)(\cdot,y)-R_0^-(\lambda^2)(\cdot,y)\bigr)(z)$.
Since $\mathcal{Q}_1v=0$ and $|v(x)|\lesssim\langle x\rangle^{-3-}$, we may apply the expansion \eqref{eq:S0vR+R-} to the first factor of the pairing and \eqref{eq:expr-vH0} to the second; this yields
	\begin{equation}\label{represn-kernel-of-W2}
		\mathcal{N}_0(x,y)=\sum_{\pm}\frac{\pm 1}{\pi i} \int_0^{+\infty} e^{i\lambda(\langle x\rangle\pm \langle y\rangle)} \bigl\langle \mathcal{D}\mathcal{Q}_1\, \eta_1^{\pm}(\lambda,\cdot,y),\; \omega_0^{-}(\lambda,\cdot,x) \bigr\rangle\, \lambda f_0(\lambda) \chi(\lambda/\lambda_0) \, d\lambda, 
	\end{equation}
	with $\eta_1^{\pm}$ and $\omega_0^{-}$ as in
	Lemmas~\ref{lem:est-S0-J0} and~\ref{lem:est-vR-0}, respectively.

	By the Cauchy--Schwarz inequality and the estimates \eqref{eq:omega-est-1} and \eqref{eq:est-eta-w-1} with
	$m=1$, we have, for $\ell=0,1,2$ and
	$\lambda\in(0,\lambda_0)$,
	\begin{equation*}
		\begin{aligned}
			\Big|\partial_\lambda^{\ell} \bigl\langle \mathcal{D}\mathcal{Q}_1\, \eta_1^{\pm}(\lambda,\cdot,y),\; \omega_0^{-}(\lambda,\cdot,x) \bigr\rangle\Big|
			&\lesssim \sum_{\ell_1+\ell_2=\ell}
			\Bigl\| \partial_\lambda^{\ell_1} \eta_1^{\pm}(\lambda,\cdot,y) \Bigr\|_{L^2}
			\Bigl\| \partial_\lambda^{\ell_2} \omega_0^{-}(\lambda,\cdot,x) \Bigr\|_{L^2} \\
			&\lesssim \lambda^{-\ell} \langle x\rangle^{-\frac12} \langle y\rangle^{-\frac12}.
		\end{aligned}
	\end{equation*}
	In view of \eqref{eq:assum for f_0}, the full amplitude
	$\langle x\rangle^{\frac12} \langle y\rangle^{\frac12}\lambda f_0(\lambda)\chi(\lambda/\lambda_0)\mathcal{K}_0^{\pm}(\lambda,x,y)$
	therefore satisfies the hypotheses of Lemma~\ref{lem-oscillatory estimate}
	with $b=1-$, and \eqref{represn-kernel-of-W2} yields the pointwise bound
	\begin{equation}\label{eq:ptwise-N0}
		\bigl|\mathcal{N}_0(x,y)\bigr| \lesssim \sum_{\pm} \frac{\langle x\rangle^{-\frac12}\langle y\rangle^{-\frac12}}{\langle \langle x\rangle\pm \langle y\rangle \rangle^{2-}}
		\lesssim \frac{\langle x\rangle^{-\frac12}\langle y\rangle^{-\frac12}}{\langle \langle x\rangle - \langle y\rangle\rangle^{2-}},
	\end{equation}
	where in the second inequality we used $\langle \langle x\rangle-\langle y\rangle\rangle\le\langle\langle x\rangle+\langle y\rangle\rangle$.

	It remains to estimate the Schur norms of \eqref{eq:ptwise-N0}. Passing to polar
	coordinates, using $r\langle r\rangle^{-\frac12}\lesssim r^{\frac12}$ and applying Lemma~\ref{lemma-GV}, we obtain
	\begin{align*}
		\sup_{x\in\mathbb{R}^2}\int_{\mathbb{R}^2} |\mathcal{N}_0(x,y)|\,dy
		\lesssim \sup_{x}\, \langle x\rangle^{-\frac12}\int_{0}^{\infty} \frac{r^{\frac12}}{\langle r-\langle x\rangle\rangle^{2-}} \,dr
		\lesssim 1.
	\end{align*}
	Since the bound \eqref{eq:ptwise-N0} is symmetric in $x$ and $y$, the same estimate holds with the roles of $x$ and $y$ interchanged:
	\[
	\sup_{y\in\mathbb{R}^2}\int_{\mathbb{R}^2} |\mathcal{N}_0(x,y)|\,dx \lesssim 1.
	\]
	Thus $\mathcal{N}_0(x,y)$ is admissible, and the proof is complete.
\end{proof}


\begin{lemma}\label{Lemm: bound of N_1}
	Let $|V(x)|\lesssim \langle x\rangle^{-6-}$ and suppose that $f_1 \in C^{2}(\mathbb{R})$ satisfies
	\begin{equation}\label{eq:assum for f_1}
		\big|\partial_{\lambda}^{k} f_1(\lambda)\big| \lesssim \lambda^{-k-}, \qquad k=0,1,2, \quad 0<\lambda<\lambda_0.
	\end{equation}
	Then the operator $\mathcal{N}_1$ given in \eqref{eq-rep-W1} is bounded on $L^p(\mathbb{R}^2)$ for all $1\le p< \infty$.
	If, in addition, $f_1(\lambda)=1$ or $f_1(\lambda)=h_{+}(\lambda)^{-1}$, then
	$\mathcal{N}_1$ is bounded on $L^p(\mathbb{R}^2)$ for the full range
	$1\le p\le\infty$.
\end{lemma}

\begin{proof}
	We first prove the $L^p$-boundedness for $1\le p<\infty$ by writing the
	kernel as an oscillatory integral; the endpoint $p=\infty$ for the two
	special choices of $f_1$ is treated separately at the end.

\medskip

Using $\mathcal{Q}_1=\mathcal{Q}_1^{*}$, the kernel of $\mathcal{N}_1$ can be written as
\begin{equation*}
\mathcal{N}_1(x,y)=\frac{1}{\pi i}\int_0^{+\infty}\lambda f_1(\lambda)\,\chi(\lambda/\lambda_0)
\bigl\langle \mathcal{D}\mathcal{P} v \bigl(R_0^+-R_0^-\bigr)(\lambda^{2})(\cdot,y),\; \mathcal{Q}_1vR_0^{-}(\lambda^2)(\cdot,x) \bigr\rangle\, d\lambda .
\end{equation*}
Since $\mathcal{Q}_1v=0$ and $|v(x)|\lesssim\langle x\rangle^{-3-}$, the
decomposition \eqref{eq:S0vR-G-0} (with $m=1$) applies to the second factor
of the pairing. Substituting it and then moving $\mathcal{Q}_1$ back onto the
first factor (the amplitudes $\omega_1^{-}(\lambda,\cdot,x)$ and $\psi_1(\cdot,x)$ belong to $\mathrm{Ran}\,\mathcal{Q}_1$), we obtain
\begin{align*}
\mathcal{N}_1(x,y)
&=\frac{1}{\pi i}\int_0^{+\infty}\bigl(\mathcal{E}_{1,0}(\lambda)-\mathcal{E}_{1,1}(\lambda)\bigr)(x,y)\,\lambda f_1(\lambda)\,\chi(\lambda/\lambda_0)\,d\lambda\\
&=: \mathcal{N}_{1,0}(x,y) + \mathcal{N}_{1,1}(x,y),
\end{align*}
with
\begin{align*}
\mathcal{E}_{1,0}(\lambda)(x,y)&:=e^{i\lambda\langle x\rangle}\Bigl\langle \mathcal{Q}_1\mathcal{D}\mathcal{P} v \bigl(R_0^+-R_0^-\bigr)(\lambda^{2})(\cdot,y),\; \omega_1^{-}(\lambda,\cdot,x) \Bigr\rangle,\\
\mathcal{E}_{1,1}(\lambda)(x,y)&:=\chi(\lambda\langle x\rangle)\Bigl\langle \mathcal{Q}_1\mathcal{D}\mathcal{P} v \bigl(R_0^+-R_0^-\bigr)(\lambda^{2})(\cdot,y),\; \psi_1(\cdot,x) \Bigr\rangle,
\end{align*}
where $\mathcal{N}_{1,0}$ and $\mathcal{N}_{1,1}$ denote the contributions of $\mathcal{E}_{1,0}$ and $-\mathcal{E}_{1,1}$, respectively.

\medskip
\noindent\emph{Bound for $\mathcal{N}_{1,0}$.}
Expanding $v\bigl(R_0^+-R_0^-\bigr)(\lambda^{2})(\cdot,y)$ by means of \eqref{eq:expr-vH0}, the kernel $\mathcal{N}_{1,0}(x,y)$ can be written as
\[
\mathcal{N}_{1,0}(x,y)=\sum_{\pm}\frac{\pm 1}{\pi i}\int_0^{+\infty} e^{i\lambda(\langle x\rangle \pm \langle y\rangle)}\,\bigl\langle \mathcal{Q}_1\mathcal{D}\mathcal{P}\, \omega_0^{\pm}(\lambda,\cdot,y),\; \omega_1^{-}(\lambda,\cdot,x) \bigr\rangle
\lambda f_1(\lambda)\,\chi(\lambda/\lambda_0) \, d\lambda.
\]
By the Cauchy--Schwarz inequality and the estimates \eqref{eq:omega-est-1}
and \eqref{eq:omega-est-3} (with $m=1$),
\[
\Big|\partial_\lambda^{\ell} \bigl\langle \mathcal{Q}_1\mathcal{D}\mathcal{P}\, \omega_0^{\pm}(\lambda,\cdot,y),\; \omega_1^{-}(\lambda,\cdot,x) \bigr\rangle\Big| \lesssim \lambda^{-\ell} \langle x\rangle^{-\frac12}\langle y\rangle^{-\frac12},
\qquad \ell=0,1,2,\quad \lambda\in(0,\lambda_0).
\]
Therefore Lemma~\ref{lem-oscillatory estimate} with $b=1-$ gives the
pointwise bound
\[
|\mathcal{N}_{1,0}(x,y)| \lesssim \langle x\rangle^{-\frac12} \langle \langle x\rangle - \langle y\rangle\rangle^{-2+} \langle y\rangle^{-\frac12},
\]
which is exactly the bound \eqref{eq:ptwise-N0} obtained for $\mathcal{N}_0$. The Schur estimate above then shows that $\mathcal{N}_{1,0}(x,y)$ is admissible; hence
$\mathcal{N}_{1,0}$ is bounded on $L^p$ for all $1\le p\le\infty$.

\medskip
\noindent\emph{Bound for $\mathcal{N}_{1,1}$ with general $f_1$.}
Using the expansion \eqref{eq:expr-vH0} once more, we have
\begin{equation}\label{eq:ker of N11}
\mathcal{N}_{1,1}(x,y)=\sum_{\pm}\frac{\pm i}{\pi }\int_0^{+\infty} e^{\pm i\lambda\langle y\rangle}\,\chi(\lambda\langle x\rangle)\,\bigl\langle \mathcal{Q}_1\mathcal{D}\mathcal{P}\,  \omega_0^{\pm}(\lambda,\cdot,y),\; \psi_1(\cdot,x) \bigr\rangle
\, \lambda f_1(\lambda)\,\chi(\lambda/\lambda_0) \, d\lambda,
\end{equation}
with
where $\psi_1$ is the function from Lemma~\ref{lem:est-S0vR0-2}. It
follows from \eqref{eq:omega-est-1} and \eqref{eq:psi-est-0} (with $m=1$)
that
	\begin{equation}\label{eq:est of K11}
		\Big|\partial_\lambda^{\ell} \bigl\langle \mathcal{Q}_1\mathcal{D}\mathcal{P}\,  \omega_0^{\pm}(\lambda,\cdot,y),\; \psi_1(\cdot,x) \bigr\rangle\Big| \lesssim \lambda^{-\frac12-\ell} \langle x\rangle^{-1}  \langle y\rangle^{-\frac12},
		\qquad \ell=0,1,2,\quad \lambda\in(0,\lambda_0).
	\end{equation}
	We derive two pointwise bounds from \eqref{eq:ker of N11}. First, since
	$\lambda\langle x\rangle\lesssim 1$ on the support of
	$\chi(\lambda\langle x\rangle)$ and $f_1$ satisfies \eqref{eq:assum for f_1}, a direct computation gives
	\[
	|\mathcal{N}_{1,1}(x,y)| \lesssim \langle x\rangle^{-1}\langle y\rangle^{-\frac12} \int_0^{\langle x\rangle^{-1}} \lambda^{\frac12-} \, d\lambda \lesssim \langle x\rangle^{-\frac52+}\langle y\rangle^{-\frac12}.
	\]
	Consequently
	\[
	\|\mathcal{N}_{1,1}f\|_{L^1}\lesssim\Bigl(\int_{\mathbb{R}^2}\langle x\rangle^{-\frac52+}\,dx\Bigr)\|f\|_{L^1}\lesssim\|f\|_{L^1},
	\]
	which is the $L^1$-boundedness of $\mathcal{N}_{1,1}$. Second, since
	$\bigl|\partial_\lambda^\ell\chi(\lambda\langle x\rangle)\bigr|\lesssim\lambda^{-\ell}$ on the support of $\chi(\lambda\langle x\rangle)$, the
	amplitude in \eqref{eq:ker of N11} satisfies the hypotheses of
	Lemma~\ref{lem-oscillatory estimate} applied to the oscillation
	$e^{\pm i\lambda\langle y\rangle}$ with $b=\frac12-$; this gives
	\[
	|\mathcal{N}_{1,1}(x,y)| \lesssim \langle x\rangle^{-1} \langle y\rangle^{-2+}.
	\]
	For $2<p<\infty$ we have $\langle x\rangle^{-1}\in L^p(\mathbb{R}^2)$ and,
	choosing the loss sufficiently small, $\langle y\rangle^{-2+}\in
	L^{p'}(\mathbb{R}^2)$; H\"older's inequality then shows that
	$\mathcal{N}_{1,1}$ is bounded on $L^p$ for $2<p<\infty$. The
	Riesz--Thorin interpolation theorem now yields the $L^p$-boundedness of
	$\mathcal{N}_{1,1}$ for all $1\le p<\infty$, and together with
	Step~2 this proves the first assertion.

	\medskip
\noindent\emph{The endpoint $p=\infty$ for $f_1(\lambda)=1$ or $f_1(\lambda)=h_{+}(\lambda)^{-1}$.}
	Using the relation $\chi(\lambda\langle x\rangle)=1-\tilde{\chi}(\lambda\langle x\rangle)$, we decompose
	\begin{align*}
		\mathcal{N}_{1,1}(x,y)&=-\frac{1}{\pi i}\int_0^{+\infty}\mathcal{G}(\lambda)(x,y)\, \lambda f_1(\lambda)\,\chi(\lambda/\lambda_0) \, d\lambda \\
		&\qquad+\frac{1}{\pi i}\int_0^{+\infty} \mathcal{G}(\lambda)(x,y)\, \tilde{\chi}(\lambda\langle x\rangle)\,\lambda f_1(\lambda)\,\chi(\lambda/\lambda_0) \, d\lambda \\
		&=: \mathcal{N}_{1,1,m}(x,y) + \mathcal{N}_{1,1,r}(x,y),
	\end{align*}
	with the abbreviation
	\[
	\mathcal{G}(\lambda)(x,y):=\bigl\langle \mathcal{Q}_1\mathcal{D}\mathcal{P} v \bigl(R_0^+-R_0^-\bigr)(\lambda^{2})(\cdot,y),\; \psi_1(\cdot,x) \bigr\rangle.
	\]

	\medskip
	\noindent\emph{Bound for $\mathcal{N}_{1,1,m}$.}
	Since $\mathcal{P}$ has finite rank, we may write
	$\mathcal{P} = \sum_{i=1}^j \varphi_i \langle \cdot,\varphi_i\rangle$ for
	some $\varphi_i\in L^2$. Inserting this into the definition of
	$\mathcal{N}_{1,1,m}$ and using the identity
	\begin{equation}\label{eq:Four_trans_on_S}
		\frac{\lambda}{\pi i}\bigl(R_0^+(\lambda^{2})-R_0^-(\lambda^{2})\bigr)f(z)
		= \bigl(\hat f\, d\sigma_\lambda\bigr)^{\vee}(z),
	\end{equation}
	where $d\sigma_\lambda$ denotes the induced measure on the circle
	$\lambda \mathbb{S}^1$, the $\lambda$-integral can be rewritten in terms of the
	functional calculus of $-\Delta$: up to a universal constant,
	\[
	\mathcal{N}_{1,1,m}f(x)
	= \sum_{i=1}^j \Phi_i(x)
	\int_{\mathbb{R}^2} v(z)\,\overline{\varphi_i}(z)\,\Bigl( f_1(\sqrt{-\Delta})\, \chi\bigl(\sqrt{-\Delta}/\lambda_0\bigr) f\Bigr)(z) \, dz,
	\]
	where
	\[
	\Phi_i(x)=\bigl\langle \mathcal{Q}_1\mathcal{D}\varphi_i,\; \psi_1(\cdot,x) \bigr\rangle
	=\int_{\mathbb{R}^2} (\mathcal{Q}_1\mathcal{D}\varphi_i)(w)\,\overline{\psi_1(w,x)}\,dw .
	\]

	We claim that for the two choices of $f_1$ under consideration, the multiplier
	$f_1(\sqrt{-\Delta})\,\chi(\sqrt{-\Delta}/\lambda_0)$ is bounded on
	$L^\infty(\mathbb{R}^2)$.
	When $f_1(\lambda)=1$, the corresponding multiplier is
	$\chi(\sqrt{-\Delta}/\lambda_0)$ is obviously bounded on $L^\infty$. When
	$f_1(\lambda)=h_{+}(\lambda)^{-1}$, writing $b_0:=-\frac{\|V\|_{L^1}}{2\pi}$ so that $h_{+}(\lambda)=b_0\log\lambda+z_0^+$, we decompose
	\[
	h_{+}(\lambda)^{-1} = \bigl(b_0\log\lambda\bigr)^{-1} + \Bigl( h_{+}(\lambda)^{-1}-\bigl(b_0\log\lambda\bigr)^{-1} \Bigr),
	\]
	where the second term decays like $(\log\lambda)^{-2}$ as $\lambda\to0^+$;
	Lemmas~\ref{eq-fourier-mulitiplier-1} and~\ref{lem-Fou-multiplier} then give
	the $L^\infty$-boundedness of each piece. Moreover, by \eqref{eq:psi-est-0} (with $m=1$),
	\[
	|\Phi_i(x)|\le \|\mathcal{Q}_1\mathcal{D}\varphi_i\|_{L^2}\,\|\psi_1(\cdot,x)\|_{L^2}
	\lesssim \|\varphi_i\|_{L^2}\,\langle x\rangle^{-1}\lesssim 1,
	\]
	and $v\overline{\varphi_i}\in L^1$ since $|v(x)|\lesssim\langle
	x\rangle^{-3-}$ and $\varphi_i\in L^2$. H\"older's inequality and the
	$L^\infty$-boundedness of the multiplier therefore give
	\begin{align*}
		\bigl\|\mathcal{N}_{1,1,m}f\bigr\|_{L^\infty}
		&\lesssim \sum_{i=1}^j \| \Phi_i\|_{L^\infty} \, \| v\varphi_i \|_{L^1} \, \| f \|_{L^\infty}
		\lesssim \| f \|_{L^\infty},
	\end{align*}
	which establishes the $L^\infty$-boundedness of $\mathcal{N}_{1,1,m}$.

	\medskip
	\noindent\emph{Bound for $\mathcal{N}_{1,1,r}$.}
	Expanding $v\bigl(R_0^+-R_0^-\bigr)(\lambda^{2})(\cdot,y)$ by \eqref{eq:expr-vH0} as before, we
	write
	\begin{equation}\label{eq:ker of N11r}
		\mathcal{N}_{1,1,r}(x,y) = \sum_{\pm}\frac{\mp i}{\pi i}\int_0^{+\infty} e^{\pm i\lambda \langle y\rangle} \lambda f_1(\lambda)\,\chi(\lambda/\lambda_0) \tilde{\chi}(\lambda\langle x\rangle)
		 \bigl\langle \mathcal{Q}_1\mathcal{D}\mathcal{P}\,  \omega_0^{\pm}(\lambda,\cdot,y),\; \psi_1(\cdot,x) \bigr\rangle \, d\lambda.
	\end{equation}
	On the support of $\tilde{\chi}(\lambda\langle x\rangle)$ we have
	$\lambda\langle x\rangle\gtrsim1$, hence $\langle x\rangle^{-1}\lesssim\lambda$, and \eqref{eq:est of K11} yields
	\begin{equation*}
		\Bigl|\partial_\lambda^\ell \bigl\langle \mathcal{Q}_1\mathcal{D}\mathcal{P}\,  \omega_0^{\pm}(\lambda,\cdot,y),\; \psi_1(\cdot,x) \bigr\rangle \Bigr|
		\lesssim \lambda^{\frac12-\ell}\langle y\rangle^{-\frac12},
		\qquad \ell=0,1,2,\quad \lambda\in (0,\lambda_0),
	\end{equation*}
	uniformly in $x$. 
	Taking absolute values in \eqref{eq:ker of N11r} already gives
	$|\mathcal{N}_{1,1,r}(x,y)|\lesssim\langle y\rangle^{-\frac12}\lesssim 1$.
	Moreover, integrating by parts twice in $\lambda$ (the boundary terms
	vanish), we obtain
	\[
	|\mathcal{N}_{1,1,r}(x,y)| \lesssim \langle y\rangle^{-\frac52}\int_{0}^{2\lambda_0} \lambda^{-\frac12-}\,d\lambda \lesssim \langle y\rangle^{-\frac52}.
	\]
	Consequently
	$
	\sup_{x\in\mathbb{R}^2} \int_{\mathbb{R}^2}|\mathcal{N}_{1,1,r}(x,y)|\,dy
	\lesssim 1,
	$
	which implies the $L^\infty$-boundedness of
	$\mathcal{N}_{1,1,r}$.

	Combining the bounds for $\mathcal{N}_{1,1,m}$ and
	$\mathcal{N}_{1,1,r}$, we conclude that $\mathcal{N}_{1,1}$ is bounded on
	$L^\infty(\mathbb{R}^2)$ when $f_1(\lambda)=1$ or $f_1(\lambda)=h_{+}(\lambda)^{-1}$; together with
	the boundedness of $\mathcal{N}_{1,0}$ given above, this proves the second assertion.
\end{proof}


\begin{lemma}\label{Lemm: bound of N_2}
	Let $|V(x)|\lesssim \langle x\rangle^{-6-}$. If $\Gamma_2(\lambda)=\tilde{\mathcal{O}}_2(\lambda^{1/2+})$, then the operator $\mathcal{N}_2$ given in \eqref{eq-rep-W2} is bounded on $L^p(\mathbb{R}^2)$ for all $1\le p\le\infty$.
\end{lemma}

\begin{proof}
	The argument closely parallels that of Lemma~\ref{Lemm: bound of N_0}.
	Expanding both resolvent factors in \eqref{eq-rep-W2} by \eqref{eq:expr-vH0}, the kernel of
	$\mathcal{N}_2$ can be written as
	\begin{align*}
		\mathcal{N}_2(x,y) &=\sum_{\pm}\frac{\pm 1}{\pi i}\int_0^{+\infty} \lambda\,\chi(\lambda/\lambda_0)\, e^{i\lambda(\langle x\rangle \pm \langle y\rangle)}
		\bigl\langle \Gamma_2(\lambda)\,\omega_0^{\pm}(\lambda,\cdot,y),\; \omega_0^{-}(\lambda,\cdot,x) \bigr\rangle\, d\lambda \\
		&=:\sum_{\pm}\frac{\pm 1}{\pi i}\int_0^{+\infty} \lambda\,\chi(\lambda/\lambda_0)\, e^{i\lambda(\langle x\rangle \pm \langle y\rangle)} \mathcal{K}_2^{\pm}(\lambda,x,y)\, d\lambda.
	\end{align*}
	Since $\Gamma_2(\lambda)=\tilde{\mathcal{O}}_2(\lambda^{1/2+})$, the
	estimate \eqref{eq:omega-est-1} together with the Cauchy--Schwarz
	inequality gives
	\[
	\bigl|\partial_\lambda^{\ell} \mathcal{K}_2^{\pm}(\lambda,x,y)\bigr| \lesssim \lambda^{-\frac12 - \ell +} \langle x\rangle^{-\frac12} \langle y\rangle^{-\frac12},
	\qquad \ell=0,1,2,\quad \lambda\in(0,\lambda_0).
	\]
	The full amplitude
	$\langle x\rangle^{\frac12} \langle y\rangle^{\frac12} \lambda\chi(\lambda/\lambda_0)\mathcal{K}_2^{\pm}(\lambda,x,y)$ therefore
	satisfies the hypotheses of Lemma~\ref{lem-oscillatory estimate} with
	$b=\frac12+$, and we obtain the pointwise bound
	\[
	|\mathcal{N}_2(x,y)|\lesssim\frac{\langle x\rangle^{-\frac12}\langle y\rangle^{-\frac12}}{\langle \langle x\rangle-\langle y\rangle\rangle^{\frac32+}}.
	\]
	This kernel is admissible: indeed, with $a=\langle x\rangle$, polar coordinates
	and Lemma~\ref{lemma-GV} give
	\[
	\int_{\mathbb{R}^2}|\mathcal{N}_2(x,y)|\,dy
	\lesssim a^{-\frac12}\int_0^\infty \frac{r\,\langle r\rangle^{-\frac12}}{\langle r-a\rangle^{\frac32+}}\,dr
	\lesssim a^{-\frac12}\int_0^\infty \frac{r^{\frac12}}{\langle r-a\rangle^{\frac32+}}\,dr
	\lesssim 1,
	\]
	and the other Schur norm is estimated in the same way, since the bound is
	symmetric in $x$ and $y$. Hence $\mathcal{N}_2$ is bounded on $L^p(\mathbb{R}^2)$ for all
	$1\le p\le\infty$.
\end{proof}
	
\subsection{Boundedness of $\mathcal{W}_{\mathrm{sing}}^{\mathrm{L}}$}

In this subsection we establish the mapping properties of the singular operator $\mathcal{W}_{\mathrm{sing}}^{\mathrm{L}}$, which arises in the low-energy decomposition when zero is a regular point of $H$. The operator $\mathcal{W}_{\mathrm{sing}}^{\mathrm{L}}$ contains a singular integral component that cannot be treated by the admissible-kernel methods of the preceding subsection; the key step is to separate this component and to identify it, through a Parseval-type computation, as the Hilbert transform along the radial variable.

\begin{lemma}\label{lem-W0s-bound}
	Let $|V(x)|\lesssim \langle x\rangle^{-6-}$. Then the operator $\mathcal{W}_{\mathrm{sing}}^{\mathrm{L}}$ defined in \eqref{eq-rep-W_s} is bounded from $L^1(\mathbb{R}^2)$ to $L^{1,\infty}(\mathbb{R}^2)$, from $\mathcal{H}^1(\mathbb{R}^2)$ to $L^1(\mathbb{R}^2)$, and from $L^\infty(\mathbb{R}^2)$ to $\mathrm{BMO}(\mathbb{R}^2)$.
\end{lemma}

\begin{proof}
	The proof is divided into two parts. We first establish the $\mathcal{H}^1\to L^1$ and $L^\infty\to\mathrm{BMO}$ bounds by identifying the singular integral structure of $\mathcal{W}_{\mathrm{sing}}^{\mathrm{L}}$; the weak-type $(1,1)$ estimate is then obtained by a refined decomposition of the output variable near and away from the origin. 

	\medskip
	\noindent\textbf{Part I: $\mathcal{H}^1\to L^1$ and $L^\infty\to\mathrm{BMO}$ boundedness.}

	Recall the representation \eqref{eq-rep-W_s}:
	\begin{align*}
		\mathcal{W}_{\mathrm{sing}}^{\mathrm{L}}
		&= \frac{1}{\pi i}\int_0^{+\infty} R_0^{+}(\lambda^2)\, v P v \bigl(R_0^+(\lambda^{2})-R_0^-(\lambda^{2})\bigr)\, \lambda h_{+}(\lambda)^{-1} \chi(\lambda/\lambda_0) \, d\lambda \\
		&= \mathcal{I}_{\mathrm{sing}}\circ \mathcal{M}_{h},
	\end{align*}
	where $\mathcal{M}_{h}:=h_{+}(\sqrt{-\Delta})^{-1}\chi(\sqrt{-\Delta}/\lambda_0)$ and
	\begin{equation}\label{eq:def of I-sing}
		\mathcal{I}_{\mathrm{sing}}:=\frac{1}{\pi i}\int_0^{+\infty} R_0^{+}(\lambda^2)\, v P v\,\lambda \bigl(R_0^+(\lambda^{2})-R_0^-(\lambda^{2})\bigr) \, d\lambda ,
	\end{equation}
	the integral being understood in the sense of tempered distributions (see the computation below). The factorization in the second line follows from the functional calculus and \eqref{eq:Four_trans_on_S}: since $\mathcal{M}_{h}$ acts on the Fourier side of the circle $\lambda\mathbb{S}^{1}$ as multiplication by the scalar $h_{+}(\lambda)^{-1}\chi(\lambda/\lambda_0)$, one has
	\[
		\lambda h_{+}(\lambda)^{-1} \chi(\lambda/\lambda_0)\,\bigl(R_0^+(\lambda^{2})-R_0^-(\lambda^{2})\bigr)
		=\lambda\bigl(R_0^+(\lambda^{2})-R_0^-(\lambda^{2})\bigr)\,\mathcal{M}_{h}.
	\]
	Since $\mathcal{M}_{h}$ is bounded on $\mathcal{H}^1$ and on $L^p(\mathbb{R}^2)$ for all $1\le p\le\infty$ by Lemma~\ref{lem-Fou-multiplier}, it suffices to establish the corresponding mapping properties of $\mathcal{I}_{\mathrm{sing}}$.

	Recall that $Pf=\|V\|_{L^1}^{-1}\langle f,v\rangle\, v$, so that $vPv$ has the integral kernel $\|V\|_{L^1}^{-1}|V|(w)|V|(z)$. For any Schwartz function $f$ we compute the Fourier transform of $\mathcal{I}_{\mathrm{sing}}f$ in $x$. On the one hand, the free resolvent acts on the Fourier side as
	\[
		\widehat{R_0^+(\lambda^2)g}\,(\xi)=\frac{\widehat{g}(\xi)}{|\xi|^2-\lambda^2-i0};
	\]
	on the other hand, \eqref{eq:Four_trans_on_S} and Parseval's formula (with the bilinear pairing $\langle g,h\rangle=\int gh$) give
	\[
		F_f(\lambda):=\frac{1}{\pi i}\bigl\langle |V|,\, \lambda\bigl(R_0^+-R_0^-\bigr)(\lambda^{2})f\bigr\rangle
		= \frac{1}{(2\pi)^2}\int_{\lambda \mathbb{S}^1}\overline{\widehat{|V|}(\eta)}\,\hat f(\eta)\,d\sigma_{\lambda}(\eta).
	\]
	Combining these two identities and applying the coarea formula in $\eta$, we arrive at
	\begin{align*}
		\mathcal{F}(\mathcal{I}_{\mathrm{sing}} f)(\xi)
		&= \frac{\widehat{|V|}(\xi)}{\|V\|_{L^1}}\lim_{\varepsilon\to 0^+}\int_{0}^{+\infty} \frac{F_f(\lambda)\,d\lambda}{|\xi|^2-\lambda^2-i\varepsilon}  \\
		&= \frac{\widehat{|V|}(\xi)}{(2\pi)^2\,\|V\|_{L^1}}
		\lim_{\varepsilon\to 0^+}\int_{\mathbb{R}^2}
		\frac{\overline{\widehat{|V|}(\eta)}\,\hat f(\eta)}{|\xi|^2-|\eta|^2-i\varepsilon}\,d\eta .
	\end{align*}
	Using the elementary identity
	\[
		\frac{1}{s-i\varepsilon}=i\int_0^\infty e^{-it(s-i\varepsilon)}\,dt,\qquad s\in\mathbb{R},\ \varepsilon>0,
	\]
	with $s=|\xi|^2-|\eta|^2$, we obtain
	\[
		\mathcal{F}(\mathcal{I}_{\mathrm{sing}} f)(\xi) = \frac{i\,\widehat{|V|}(\xi)}{(2\pi)^2\,\|V\|_{L^1}}
		\lim_{\varepsilon\to 0^+}\int_0^\infty \int_{\mathbb{R}^2} e^{-it(|\xi|^2-|\eta|^2-i\varepsilon)}
		\,\overline{\widehat{|V|}(\eta)}\,\hat f(\eta)\,d\eta\,dt .
	\]
	Define, in the sense of tempered distributions,
	\[
		\mathbf{T}(x,y):=\frac{i}{(2\pi)^4}\lim_{\varepsilon\to 0^+}\int_0^\infty \int_{\mathbb{R}^2}\int_{\mathbb{R}^2}
		e^{ix\cdot\xi}\,e^{-it(|\xi|^2-|\eta|^2-i\varepsilon)}\,e^{iy\cdot\eta}\,d\xi\,d\eta\,dt .
	\]
	Unwinding the Fourier transforms then yields the kernel representation
	\begin{equation}\label{eq:iden-1}
		\mathcal{I}_{\mathrm{sing}} f(x) = \frac{1}{\|V\|_{L^1}}
		\int_{\mathbb{R}^2}\int_{\mathbb{R}^2} \int_{\mathbb{R}^2}\mathbf{T}(x-w,z-y)\,|V|(z)|V|(w)f(y)\,dz\,dw\,dy .
	\end{equation}
	Since $(e^{-it|\xi|^2})^{\vee}(x)=(4\pi i t)^{-1}e^{i|x|^2/(4t)}$, the $\xi$- and $\eta$-integrals give, with $A:=|x|^2-|y|^2$,
	\[
		\mathbf{T}(x,y)= \frac{i}{(4\pi)^2}\lim_{\varepsilon\to 0^+}\int_0^\infty
		e^{i\frac{A}{4t}-\varepsilon t}\,\frac{dt}{t^{2}}
		= \frac{4i}{(4\pi)^2}\int_0^\infty e^{isA}\,ds,
	\]
	where the second equality results from the change of variables $s=1/(4t)$ (with $t^{-2}dt=4\,ds$) and is understood in the sense of distributions in $A$. The Sokhotski--Plemelj formula (i.e., the distributional Fourier transform of the Heaviside function)
	\[
		\int_0^\infty e^{isA}\,ds = \pi\delta(A) + i\,\operatorname{p.v.}\frac{1}{A}
	\]
	therefore yields
	\begin{equation}\label{rep-for-0-s-3}
    \begin{aligned}
     \mathbf{T}(x,y)&= \frac{4i}{(4\pi)^2}\Bigl(\pi\,\delta(|x|^2-|y|^2)
		+ i\,\operatorname{p.v.}\frac{1}{|x|^2-|y|^2}\Bigr)  \\
		&\quad= \frac{i}{4\pi}\,\delta(|x|^2-|y|^2)
		-\frac{1}{4\pi^{2}}\,\operatorname{p.v.}\frac{1}{|x|^2-|y|^2}.   
    \end{aligned}	
	\end{equation}
	We now reduce $\mathcal{I}_{\mathrm{sing}}$ to the operator $\mathbf{T}$ with kernel $\mathbf{T}(x,y)$. Since $\mathbf{T}$ is even in its second argument, \eqref{eq:iden-1} can be rewritten as
	\[
		\mathcal{I}_{\mathrm{sing}}f(x)=\frac{1}{\|V\|_{L^1}}\int_{\mathbb{R}^2}|V|(w)\,\mathbf{T}\Phi(x-w)\,dw,
	\]
	with $\Phi(u):=\int_{\mathbb{R}^2}|V|(y)\,f(u+y)\,dy$. The translation invariance of $L^1$ and $\mathcal{H}^1$ then gives
	\[
    \|\Phi\|_{\mathcal{H}^1}\le\int_{\mathbb{R}^2}|V|(y)\,\|f(\cdot+y)\|_{\mathcal{H}^1}\,dy=\|V\|_{L^1}\|f\|_{\mathcal{H}^1}
    \]
	and
	\[
		\|\mathcal{I}_{\mathrm{sing}}f\|_{L^1}\le\frac{1}{\|V\|_{L^1}}\int_{\mathbb{R}^2}|V|(w)\,\|\mathbf{T}\Phi(\cdot-w)\|_{L^1}\,dw=\|\mathbf{T}\Phi\|_{L^1}.
	\]
	Hence the $\mathcal{H}^1\to L^1$ boundedness of $\mathcal{I}_{\mathrm{sing}}$ follows from that of $\mathbf{T}$.

	\smallskip
	\noindent\emph{Boundedness of the delta term.} By the coarea formula, for fixed $x\ne0$,
	\[
		\int_{\mathbb{R}^2} \delta(|x|^2-|y|^2)\,f(y)\,dy = \frac{1}{2|x|}\int_{|y|=|x|} f(y)\,dS(y).
	\]
	Polar coordinates show that the $L^1$- and $L^\infty$-norms of this expression are bounded by $\pi\|f\|_{L^1}$ and $\pi\|f\|_{L^\infty}$, respectively; by the Riesz--Thorin interpolation theorem the delta term is therefore bounded on every $L^p$, $1\le p\le \infty$, and in particular from $\mathcal{H}^1$ to $L^1$ and from $L^\infty$ to $\mathrm{BMO}$.

	\smallskip
	\noindent\emph{Boundedness of the singular term $\mathbf{T}_{\mathbb{H}}$.} It remains to treat the operator $\mathbf{T}_{\mathbb{H}}$ with kernel
	\[
		\mathbf{T}_{\mathbb{H}}(x,y)=\operatorname{p.v.}\frac{1}{|x|^2-|y|^2}.
	\]
	We prove that $\mathbf{T}_{\mathbb{H}}:\mathcal{H}^1\to L^1$ by the atomic decomposition. Recall that $a$ is an $\mathcal{H}^1(\mathbb{R}^2)$-atom associated with a ball $B(x_0,r)\subset\mathbb{R}^2$ if it satisfies
	\begin{enumerate}
		\item[(1)] the support condition: $\operatorname{supp} a \subset B(x_0,r)$;
		\item[(2)] the size condition: $\|a\|_{L^\infty} \le |B(x_0,r)|^{-1}$;
		\item[(3)] the vanishing moment condition: $\int_{\mathbb{R}^2} a(x) \, dx = 0$.
	\end{enumerate}
	It suffices to show that there exists a constant $A$ such that
	\begin{equation}\label{eq:atom-bound}
		\|\mathbf{T}_{\mathbb{H}} a\|_{L^1} \le A
	\end{equation}
	for every $\mathcal{H}^1(\mathbb{R}^2)$-atom $a$.

	Passing to polar coordinates $y=\sqrt\rho\,\theta$ with $dy=\frac12\,d\rho\,d\theta$ shows that $\mathbf{T}_{\mathbb{H}}a$ is radial, and writing $t=|x|^2$ with $dx=\frac12\,dt\,d\omega$ then gives
	\[
		\|\mathbf{T}_{\mathbb{H}}a\|_{L^1(\mathbb{R}^2)}
		= \frac{\pi}{2}\Bigl\| \operatorname{p.v.}\int_{\mathbb{R}^+}\frac{\int_{\mathbb{S}^1}a(\sqrt\rho\,\theta)\,d\theta}{t-\rho}\,d\rho\Bigr\|_{L^1(\mathbb{R}^+,\,dt)} .
	\]
	Since the function $\rho\mapsto\int_{\mathbb{S}^1}a(\sqrt\rho\,\theta)\,d\theta$ is supported in $\mathbb{R}^+$, the principal value on the right-hand side coincides, for $t>0$, with $\pi$ times the Hilbert transform of this function (extended by zero to $\mathbb{R}$). The $\mathcal{H}^1(\mathbb{R})\to L^1(\mathbb{R})$ boundedness of the Hilbert transform therefore yields
	\[
		\|\mathbf{T}_{\mathbb{H}}a\|_{L^1(\mathbb{R}^2)}
		\lesssim \Bigl\|\int_{\mathbb{S}^1}a(\sqrt\rho\,\theta)\,d\theta\Bigr\|_{\mathcal{H}^1(\mathbb{R})}.
	\]

	Set $\mathcal{A}a(\rho)=\frac{1}{18}\int_{\mathbb{S}^1}a(\sqrt\rho\,\theta)\,d\theta$. To obtain \eqref{eq:atom-bound}, it remains to verify that $\mathcal{A}a$ is an $\mathcal{H}^1(\mathbb{R})$-atom. The vanishing moment condition follows from that of $a$:
	\[
		\int_{\mathbb{R}}\mathcal{A}a(\rho)\,d\rho=\frac{1}{18}\int_{\mathbb{S}^1}\int_0^{+\infty} a(\sqrt\rho\,\theta)\,d\rho\,d\theta=\frac{1}{9}\int_{\mathbb{R}^2}a(y)\,dy=0 .
	\]
	If $a$ is supported in $B(x_0,r)$ with $|x_0| < 2r$, then $\operatorname{supp}\mathcal{A}a \subset I_0:=[0, 9r^2]$ and
	\[
		|\mathcal{A}a(\rho)|\le\frac{2\pi}{18}\,\|a\|_{L^\infty}\le\frac{2\pi}{18}\cdot\frac{1}{\pi r^{2}}=\frac{1}{9r^{2}}=|I_0|^{-1},
	\]
	so $\mathcal{A}a$ is an $\mathcal{H}^1(\mathbb{R})$-atom associated with the interval $I_0$. If instead $|x_0| \ge 2r$, then
	\[
		\operatorname{supp}\mathcal{A}a \subset I_0:=\bigl[(|x_0|-r)^2, (|x_0|+r)^2\bigr],\qquad |I_0|=4|x_0|r,
	\]
	and, since the circle of radius $\sqrt\rho$ meets the ball $B(x_0,r)$ in an arc of angular measure at most $2\arcsin(r/|x_0|)$,
	\[
		|\mathcal{A}a(\rho)|
		\le \frac{1}{9}\arcsin\!\Big(\frac{r}{|x_0|}\Big)\,\|a\|_{L^\infty}
		\le \frac{\pi r}{18\,|x_0|}\cdot\frac{1}{\pi r^{2}}
		= \frac{1}{18\,|x_0|\, r}
		\le |I_0|^{-1},
	\]
	for almost every $\rho$. Thus $\mathcal{A}a$ is again an $\mathcal{H}^1(\mathbb{R})$-atom associated with $I_0$. In both cases $\|\mathcal{A}a\|_{\mathcal{H}^1(\mathbb{R})}$ is bounded uniformly in $a$, which establishes \eqref{eq:atom-bound}.

	Consequently $\mathbf{T}_{\mathbb{H}}$, and hence $\mathbf{T}$, $\mathcal{I}_{\mathrm{sing}}$ and $\mathcal{W}_{\mathrm{sing}}^{\mathrm{L}}$, are bounded from $\mathcal{H}^1$ to $L^1$. For the endpoint at $p=\infty$ we argue by duality: the adjoint of $\mathcal{I}_{\mathrm{sing}}$ has the same form as $\mathcal{I}_{\mathrm{sing}}$, with the kernel $\mathbf{T}(x-w,z-y)$ in \eqref{eq:iden-1} replaced by $\overline{\mathbf{T}(y-w,z-x)}$, which differs from \eqref{rep-for-0-s-3} only by a sign in front of the principal-value part and a conjugation of the delta part. The argument above therefore applies to the adjoint as well; since $\mathrm{BMO}$ is the dual of $\mathcal{H}^1$, we conclude that $\mathcal{I}_{\mathrm{sing}}$, and hence $\mathcal{W}_{\mathrm{sing}}^{\mathrm{L}}$, are bounded from $L^\infty$ to $\mathrm{BMO}$.

	\medskip
	\noindent\textbf{Part II: Weak-type $(1,1)$ boundedness.}

	We now prove that $\mathcal{W}_{\mathrm{sing}}^{\mathrm{L}}$ is bounded from $L^1(\mathbb{R}^2)$ to $L^{1,\infty}(\mathbb{R}^2)$. The representation \eqref{eq:iden-1} cannot be used directly, because the convolution in the $w$-variable is not controlled in the weak space $L^{1,\infty}$; we therefore modify the argument according to the size of $|x|$. Write the kernel of $\mathcal{W}_{\mathrm{sing}}^{\mathrm{L}}$ as
	\begin{align*}
		\mathcal{W}_{\mathrm{sing}}^{\mathrm{L}}(x,y)= \frac{1}{\pi i\,\|V\|_{L^1}}
		&\int_0^{+\infty}\int_{\mathbb{R}^2}\int_{\mathbb{R}^2}
		R_0^+(\lambda^2)(x,w)\,|V|(w)|V|(z)\\
		&\qquad\times\bigl(R_0^+(\lambda^{2})-R_0^-(\lambda^{2})\bigr)(z,y)\,
		\frac{\lambda\,\chi(\lambda/\lambda_0)}{h_+(\lambda)}\,dz\,dw\,d\lambda ,
	\end{align*}
	and split the operator according to the spatial cutoff $\chi(|x|)$ (supported in $|x|\le1$) and its complement $\tilde\chi(|x|)=1-\chi(|x|)$ (supported in $|x|\ge\frac12$).

	\smallskip
	\emph{The region $|x|\le 1$.} Using the expansions \eqref{eq:asy-Hanel-0} and \eqref{eq:asy-Hanel-1}, we have for $|x|\le 1$,
	\begin{align*}
		\int_{\mathbb{R}^2}R_0^+(\lambda^2)(x,w)\,|V|(w)\,dw -\frac{1}{2\pi}\int_{\mathbb{R}^2}\log|x-w|\,|V|(w)\,dw + c(\lambda)\,\|V\|_{L^1} + g(\lambda,x),
	\end{align*}
	where, with a constant $c_0\in\mathbb{C}$,
	\[
		c(\lambda):=c_0-\frac{1}{2\pi}\log\lambda,
	\]
	and where $g$ satisfies $\sup_{|x|\le1}|\partial_\lambda^\ell g(\lambda,x)|\lesssim\lambda^{1-\ell}$ for $\ell=0,1,2$ and $\lambda\in (0,1)$ by \eqref{eq:est for r-0}. Note that the logarithmic factor $c(\lambda)$, rather than a genuine constant, appears here; matching it with the leading term of $h_{+}(\lambda)$ is essential for the multiplier argument below. The contribution of the part with $\chi(|x|)$ accordingly splits into three terms,
	\begin{align*}
		-\frac{\chi(|x|)}{2\pi\|V\|_{L^1}} &\int_{\mathbb{R}^2} \log|x-w| \, |V|(w) \, dw \int_{\mathbb{R}^2} |V|(y)\, \mathcal{M}_{h} f(y) \, dy, \\
		\frac{\chi(|x|)}{\|V\|_{L^1}} &\int_{\mathbb{R}^2} |V|(y)\, \mathcal{M}_{c,h} f(y) \, dy,
	\end{align*}
	and
	\[
		\frac{\chi(|x|)}{\pi i\|V\|_{L^1}}
		\int_0^{+\infty} g(\lambda,x)
		\int_{\mathbb{R}^2} |V|(z) \bigl(R_0^+(\lambda^{2})-R_0^-(\lambda^{2})\bigr) (z,y) \, \frac{\lambda  \chi(\lambda/\lambda_0)}{h_{+}(\lambda)} \,  dz \,d\lambda ,
	\]
	where $\mathcal{M}_{c,h}$ denotes the Fourier multiplier with symbol $c(\lambda)h_{+}(\lambda)^{-1}\chi(\lambda/\lambda_0)$, and where we have used the functional-calculus identity
	\[
		\frac{1}{\pi i}\int_0^{+\infty}\varphi(\lambda)\,\lambda\bigl(R_0^+(\lambda^{2})-R_0^-(\lambda^{2})\bigr)\,d\lambda=\varphi(\sqrt{-\Delta}),
	\]
	applied to $\varphi(\lambda)=h_{+}(\lambda)^{-1}\chi(\lambda/\lambda_0)$ and to $\varphi(\lambda)=c(\lambda)\,h_{+}(\lambda)^{-1}\chi(\lambda/\lambda_0)$; this identity follows from \eqref{eq:Four_trans_on_S} exactly as in Part I.

	Since $\mathcal{M}_{h}$ is bounded on $L^1$ and
	\[
		\sup_{|x|\le1}\int_{\mathbb{R}^2}\bigl|\log|x-w|\bigr|\,|V|(w)\,dw\lesssim1,
	\]
	the operator corresponding to the first term is bounded on $L^1$. For the second term, writing $b_0=-\frac{\|V\|_{L^1}}{2\pi}$ so that $h_{+}(\lambda)=b_0\log\lambda+z_0^+$ (cf. \eqref{eq:def-of-hpm}), we decompose
	\[
		\frac{c(\lambda)}{h_{+}(\lambda)}=\frac{c_0-\frac{1}{2\pi}\log\lambda}{b_0\log\lambda+z_0^+}
		=-\frac{1}{2\pi b_0}+\Bigl(c_0+\frac{z_0^+}{2\pi b_0}\Bigr)h_{+}(\lambda)^{-1}.
	\]
	Hence $\mathcal{M}_{c,h}$ is the sum of a constant multiple of $\mathcal{M}_{h}$ and a constant multiple of the cutoff multiplier $\chi(\sqrt{-\Delta}/\lambda_0)$, both of which are bounded on $L^1$ by Lemma~\ref{lem-Fou-multiplier}. The operator corresponding to the second term is therefore bounded on $L^1$ as well. The third term is handled exactly as in the proof of Lemma~\ref{Lemm: bound of N_2}: the improved symbol estimates for $g$ stated above provide more $\lambda$-decay than is required there, so the associated operator is also $L^1$-bounded.

	\smallskip
	\emph{The region $|x|\ge 1/2$.} For the part with $\tilde\chi(|x|)$, we replace the full resolvent $R_0^+(\lambda^2)(x,w)$ by its asymptotic profile $\frac{i}{4}H_0^+(\lambda|x|)$ and split accordingly into a main term $\mathcal{I}_0$ and a remainder $\mathcal{I}_1$. Since $\int_{\mathbb{R}^2}|V|(w)\,dw=\|V\|_{L^1}$, we set
	\[
		\mathcal{I}_0 =  \frac{\tilde{\chi}(|x|)}{4\pi}
		\int_0^{+\infty} \frac{\lambda\,\chi(\lambda/\lambda_0)}{h_{+}(\lambda)}\,H_0^{+}(\lambda|x|)
		\int_{\mathbb{R}^2} |V|(z)\, \bigl(R_0^+-R_0^-\bigr)(\lambda^{2}) (z,y) \, dz \, d\lambda,
	\]
	and
	\begin{align*}
		\mathcal{I}_1 =  \frac{\tilde{\chi}(|x|)}{\pi i\, \|V\|_{L^1}}
		\int_0^{+\infty}
		\int_{\mathbb{R}^2} \int_{\mathbb{R}^2} &\left(R_0^{+}(\lambda^2)(x,w) - \frac{i}{4}H_0^{+}(\lambda|x|)\right) |V|(w) |V|(z) \\
		&\quad \times \bigl(R_0^+-R_0^-\bigr)(\lambda^{2}) (z,y) \, \frac{\lambda  \chi(\lambda/\lambda_0)}{h_{+}(\lambda)} \, dz \, dw \, d\lambda.
	\end{align*}
	In view of Remark~\ref{remk:on the diff}, the kernel of $\mathcal{I}_1$ satisfies the same estimates as the kernels treated in Lemma~\ref{Lemm: bound of N_1}; hence $\mathcal{I}_1$ is bounded on $L^1$, and in particular from $L^1$ to $L^{1,\infty}$.

	For $\mathcal{I}_0$ we repeat the Parseval computation of Part I: since $\frac{i}{4}H_0^{+}(\lambda|x|)=R_0^+(\lambda^2)(x,0)$, the factor $\mathcal{F}(|V|)(\xi)$ there is replaced by the constant $\|V\|_{L^1}$ (the Fourier transform of the point mass $\|V\|_{L^1}\delta_0$), while the spectral factor $h_{+}(\lambda)^{-1}\chi(\lambda/\lambda_0)$ contributes the multiplier $\mathcal{M}_{h}$ acting on $f$. With $\mathbf{T}$ given by \eqref{rep-for-0-s-3} we obtain
	\begin{equation}\label{eq:inte kern by reduction}
		\mathcal{I}_0 f(x)=\tilde\chi(|x|)\int_{\mathbb{R}^2} \int_{\mathbb{R}^2} \mathbf{T}(x, z-y)\, |V|(z)\, \mathcal{M}_hf(y) dz\,dy  
		=\tilde\chi(|x|)\int_{\mathbb{R}^2} \mathbf{T}(x, u)\, \Phi(u)\,du,
	\end{equation}
	where $\Phi(u):=\int_{\mathbb{R}^2}|V|(u+y)\,\mathcal{M}_{h}f(y)\,dy$ satisfies
	\[
		\|\Phi\|_{L^1}\le\|V\|_{L^1}\,\|\mathcal{M}_{h}f\|_{L^1}\lesssim\|f\|_{L^1}.
	\]
	Since multiplication by $\tilde\chi$ is bounded on $L^{1,\infty}$, it suffices to prove the weak-type $(1,1)$ boundedness of $\mathbf{T}$.

	The delta part of $\mathbf{T}$ is bounded on $L^1$, hence of weak type $(1,1)$. For the Hilbert part, note first that
	\[
		\mathbf{T}_{\mathbb{H}}f(x)=\frac12\,\operatorname{p.v.}\int_{\mathbb{R}^+}\frac{\int_{\mathbb{S}^1}f(\sqrt\rho\,\theta)\,d\theta}{|x|^{2}-\rho}\,d\rho
	\]
	is a radial function. Passing to polar coordinates in the $x$-variable as in Part I gives
	\begin{align*}
		\|\mathbf{T}_{\mathbb{H}} f\|_{L^{1,\infty}(\mathbb{R}^2)}
		&= \frac{\pi}{2} \left\| \operatorname{p.v.}\int_{\mathbb{R}^+} \frac{ \int_{\mathbb{S}^1} f(\sqrt{\rho}\,\theta)\,d\theta}{t-\rho} \, d\rho \right\|_{L^{1,\infty}(\mathbb{R}^+)} \\[4pt]
		&\lesssim \left\| \int_{\mathbb{S}^1} f(\sqrt{\rho}\,\theta)\,d\theta \right\|_{L^1(\mathbb{R})} \\[4pt]
		&\le 2\,\|f\|_{L^1(\mathbb{R}^2)},
	\end{align*}
	where the second line follows from the weak-type $(1,1)$ boundedness of the Hilbert transform (as the integrand is supported on $\mathbb{R}^+$, the principal value coincides for $t>0$ with the Hilbert transform on $\mathbb{R}$), and the third from $d\rho\,d\theta=2\,dy$. Thus $\mathcal{I}_0$ is bounded from $L^1(\mathbb{R}^2)$ to $L^{1,\infty}(\mathbb{R}^2)$. Combining this with the $L^1$-boundedness of $\mathcal{I}_1$ and of the part with $|x|\le1$ completes the proof of the weak-type $(1,1)$ bound, and hence of Lemma~\ref{lem-W0s-bound}.
\end{proof}

\begin{remark}\label{eq:rem-to-Linfty}
	By Lemma~\ref{lem-Fou-multiplier}, the multipliers $\mathcal{M}_{h}$ and $\mathcal{M}_{c,h}$ are bounded on $L^\infty$. Therefore, the three-term decomposition used for $|x|\le1$ in Part II applies equally well with $\chi(|\cdot|/r)$ in place of $\chi(|\cdot|)$. Consequently, for each fixed $r>0$ the truncated operator $\chi(|\cdot|/r)\,\mathcal{W}_{\mathrm{sing}}^{\mathrm{L}}$ is bounded on $L^\infty$, with constant depending on $r$.
\end{remark}

\subsection{The proof of Proposition~\ref{pro:bound for low energy}}

We now assemble the preceding estimates to complete the proof of Proposition~\ref{pro:bound for low energy}.

\medskip
\noindent\emph{Case 1: zero is a regular point of $H$.}
Recall from \eqref{eq-M expansion-even} that
\[
	M^+(\lambda)^{-1} = h_{+}(\lambda)^{-1}N + QD_0Q + \tilde{\mathcal{O}}_2(\lambda^{3/2}),
\]
with $N=P-PTQD_0Q-QD_0QTP+QD_0QTPTQD_0Q$ as in \eqref{eq:def of K}.
Inserting this expansion into the stationary representation \eqref{eq:stationary formula-Low-2} via the resolvent identity $R^+(\lambda^2)v=R_0^+(\lambda^2)v\big(M^+(\lambda)\big)^{-1}U$, the part $h_{+}(\lambda)^{-1}P$ produces precisely the singular component $\mathcal{W}_{\mathrm{sing}}^{\mathrm{L}}$ of \eqref{eq-rep-W_s}, whose mapping properties are given by Lemma~\ref{lem-W0s-bound}:
\[
	\mathcal{W}_{\mathrm{sing}}^{\mathrm{L}}:\; L^1\to L^{1,\infty},\qquad \mathcal{H}^1\to L^1,\qquad L^\infty\to\mathrm{BMO}.
\]
Every remaining summand of $h_{+}(\lambda)^{-1}N$, the term $QD_0Q$, and the error term has $Q$ as its leftmost or rightmost projection; the cancellation $Qv=0$ therefore places each of their contributions in one of the canonical forms $\mathcal{N}_0$, $\mathcal{N}_1$ or $\mathcal{N}_2$ (see the classification preceding \eqref{eq-rep-W0}--\eqref{eq-rep-W_s}), with $f_0,f_1\in\{1,\,h_{+}^{-1}\}$ and $\Gamma_2(\lambda)=\tilde{\mathcal{O}}_2(\lambda^{3/2})$. Since $h_{+}^{-1}$ satisfies \eqref{eq:assum for f_0} and \eqref{eq:assum for f_1}, and since both $f_1\equiv1$ and $f_1=h_{+}^{-1}$ are covered by the endpoint assertion of Lemma~\ref{Lemm: bound of N_1}, Lemmas~\ref{Lemm: bound of N_0}--\ref{Lemm: bound of N_2} show that all these contributions are bounded on $L^p(\mathbb{R}^2)$ for the full range $1\le p\le\infty$; in particular they are bounded from $L^1$ to $L^{1,\infty}$, from $\mathcal{H}^1$ to $L^1$ and from $L^\infty$ to $\mathrm{BMO}$. This establishes statement (i) of Proposition~\ref{pro:bound for low energy}.

\medskip
\noindent\emph{Case 2: zero is a first-kind threshold singularity of $H$.}
From \eqref{eq-M expansion-even-1-1} we have
\[
	M^+(\lambda)^{-1}= -h_{+}(\lambda)S_1D_1S_1 - NS_1D_1S_1 - S_1D_1S_1N+ QD_0Q + \tilde{\mathcal{O}}_2(\lambda^{3/2}).
\]
In contrast to Case 1, no singular component of the form $\mathcal{W}_{\mathrm{sing}}^{\mathrm{L}}$ occurs: every term in this expansion has $Q$ or $S_1$ as its rightmost or leftmost projection. More precisely, according to the classification preceding \eqref{eq-rep-W0}--\eqref{eq-rep-W_s}:
\begin{itemize}
	\item the contributions of $-h_{+}(\lambda)S_1D_1S_1$ and of $-NS_1D_1S_1$ (rightmost projection $S_1$) are of the form $\mathcal{N}_0$, with $f_0=-h_{+}$ and $f_0=1$, respectively;
	\item the contribution of $-S_1D_1S_1N$ (leftmost projection $S_1$) is of the form $\mathcal{N}_1$ with $f_1=1$;
	\item the term $QD_0Q$ gives $\mathcal{N}_0$ with $f_0=1$, and the error term gives $\mathcal{N}_2$ with $\Gamma_2(\lambda)=\tilde{\mathcal{O}}_2(\lambda^{3/2})$.
\end{itemize}
Since $h_{+}(\lambda)=b_0\log\lambda+z_0^+$ satisfies
\[
	\bigl|\partial_\lambda^{k}h_{+}(\lambda)\bigr|\lesssim\lambda^{-k},\qquad k=0,1,2,\quad 0<\lambda<\lambda_0,
\]
both $f_0=-h_{+}$ and $f_0=1$ verify the hypothesis \eqref{eq:assum for f_0} of Lemma~\ref{Lemm: bound of N_0}, whose conclusion covers the full range $1\le p\le\infty$; likewise $f_1=1$ is covered by the endpoint assertion of Lemma~\ref{Lemm: bound of N_1}. Lemmas~\ref{Lemm: bound of N_0}--\ref{Lemm: bound of N_2} therefore show that every term in the expansion yields an operator bounded on $L^p(\mathbb{R}^2)$ for all $1\le p\le\infty$. This establishes statement (ii) of Proposition~\ref{pro:bound for low energy} and completes the proof.

\section{Low-energy estimates II: second-kind resonance and zero eigenvalue}\label{low-bound II}

In this section we establish the $L^p$-boundedness of the low-energy component $\mathcal{W}^{\mathrm{L}}$ under the assumption that zero is either a second-kind threshold singularity or an eigenvalue of $H$.

We recall that Yajima~\cite{Yajima-2022} proved, under the decay condition $\langle x\rangle^{8+}V\in L^1$, that $\mathcal{W}^{\mathrm{L}}$ is bounded on $L^p(\mathbb{R}^2)$ for $1<p\le 2$ in the presence of a p-wave resonance, and for $1<p<\infty$ in its absence. Below we present a new proof of this result under a weaker decay assumption on the potential.

Our analysis still rests on the stationary representation
\begin{equation*}
	\mathcal{W}^{\mathrm{L}} = \frac{1}{\pi i}\int_0^{+\infty} \lambda R^+(\lambda^2)v \bigl(M^+(\lambda)\bigr)^{-1} v \bigl(R_0^+(\lambda^2) - R_0^-(\lambda^2)\bigr) \chi(\lambda/\lambda_0) \, d\lambda.
\end{equation*}
To proceed, we first analyze the asymptotic behavior of $(M^{\pm}(\lambda))^{-1}$ as $\lambda\to 0$ when zero is a second-kind threshold singularity or an eigenvalue.

\subsection{Asymptotic behavior of $(M^{\pm}(\lambda))^{-1}$}

The asymptotics of $(M^{\pm}(\lambda))^{-1}$ near zero energy have been studied in \cite{Erd-Gre-13, Erdogan-Goldberg-Green-JFA-2018, JN01, Yajima-2022} when zero is either a second-kind threshold singularity or an eigenvalue of $H$. For our purposes, we require the following refined expansion; its proof is deferred to Appendix~\ref{sec:proof of exp lemma}.

To state the expansion, we define
\begin{equation}\label{eq:def-of-g_1}
g_1^{\pm}(\lambda):=\log\lambda+\beta_{\pm},\qquad \beta_{\pm}\in\mathbb{C}\setminus\mathbb{R},
\end{equation}
where $\beta_{\pm}$ are determined by \eqref{eq:asy-Hanel-0}. By Lemma~\ref{lem:one to one lem}, the projection $S_2$ is non-zero when zero is a second-kind threshold singularity or an eigenvalue. We then set
\begin{align*}
	D_2&:=\bigl((S_2-S_3)vG_1v(S_2-S_3)\bigr)^{-1}
	=\bigl(S_2vG_2vS_2\bigr)^{-1}\quad\text{if }S_2\neq S_3,\\[2pt]
	D_3&:=(S_3vG_2vS_3)^{-1}\quad\text{if }S_3\neq0.
\end{align*}
Finally, let
\begin{equation}\label{eq:def of K_1}
	\begin{aligned}
		N_1&:=(S_2-S_3)D_2(S_2-S_3)-(S_2-S_3)D_2vG_2vS_3D_3S_3\\
		&\qquad-S_3D_3S_3vG_2vD_2(S_2-S_3)\\
		&\qquad+S_3D_3S_3vG_2v(S_2-S_3)D_2(S_2-S_3)vG_2vS_3D_3S_3.
	\end{aligned}
\end{equation}

\begin{lemma}\label{thm-M inverse-even-1}
	Suppose $|V(x)|\lesssim\langle x\rangle^{-8-}$ and that zero is either a second-kind threshold singularity or an eigenvalue of $H$. Then there exists $\lambda_0=\lambda_0(V)\in(0,1)$ such that for $\lambda\in(0,\lambda_0)$ the operators $(M^{\pm}(\lambda))^{-1}\in\mathbb{B}(L^2)$ admit the expansion
	\begin{equation}\label{eq-M expansion-even-2}
		\begin{aligned}
			\bigl(M^{\pm}(\lambda)\bigr)^{-1}
			&=(I-S_2)\mathcal{L}_{00}^{\pm}(\lambda)(I-S_2)
			+(I-S_2)\mathcal{L}_{01}^{\pm}(\lambda)S_2\\
			&\quad+S_2\mathcal{L}_{10}^{\pm}(\lambda)(I-S_2)
			+S_2\mathcal{L}_{11}^{\pm}(\lambda)S_2,
		\end{aligned}
	\end{equation}
	where
	\begin{equation}\label{eq:mathcalL00}
		\mathcal{L}_{00}^{\pm}(\lambda)=
		\begin{cases}
			h_{\pm}(\lambda)^{-1}N+QD_0Q+\tilde{\mathcal{O}}_2(\lambda^{3/2}), 
			&\text{if no s-wave},\\[4mm]
			-h_{\pm}(\lambda)(S_1-S_2)D_1(S_1-S_2)
			-N(S_1-S_2)D_1(S_1-S_2)\\[1mm]
			\quad-(S_1-S_2)D_1(S_1-S_2)N+QD_0Q
			+\tilde{\mathcal{O}}_2(\lambda^{3/2}), 
			& \text{if  has s-wave},
		\end{cases}
	\end{equation}
	and
	\begin{equation}\label{eq:mathcalL11}
		\mathcal{L}_{11}^{\pm}(\lambda)=
		\begin{cases}
			\dfrac{S_3D_3S_3}{\lambda^2}+\tilde{\mathcal{O}}_2(\lambda^{-1+}), 
			&\text{if no p-wave},\\[6mm]
			\dfrac{S_3D_3S_3}{\lambda^2}
			+\dfrac{N_1}{\lambda^2g_1^{\pm}(\lambda)}
			+\tilde{\mathcal{O}}_2\!\left(\lambda^{-2}g_1^{\pm}(\lambda)^{-2}\right), 
			& \text{if has p-wave},
		\end{cases}
	\end{equation}
	and
	\begin{equation}\label{eq:mathcalL01 10}
		\mathcal{L}_{01}^{\pm}(\lambda)=\tilde{\mathcal{O}}_2(\log^2\lambda),\qquad 
		\mathcal{L}_{10}^{\pm}(\lambda)=\tilde{\mathcal{O}}_2(\log^2\lambda).
	\end{equation}
	Here we set $S_3D_3S_3=0$ when $S_3=0$. Furthermore, when $H$ has s-wave but no p-wave resonances,
	\begin{equation}\label{eq:special mathcalL10}
		\mathcal{L}_{10}^{\pm}(\lambda)=\tilde{\mathcal{O}}_2(\log^2\lambda)\,Q+\Xi P+\tilde{\mathcal{O}}_2(\lambda^2),
	\end{equation}
	where $\Xi$ is a bounded operator on $L^2(\mathbb{R}^2)$.
\end{lemma}

In view of the expansion \eqref{eq-M expansion-even-2}, it suffices to analyze the contributions arising from each term $\mathcal{L}_{ij}^\pm(\lambda)$ for $i,j\in\{0,1\}$. We define the corresponding contributions by
\begin{align*}
	\mathcal{W}_{00}^{\mathrm{L}}
	&=\frac{1}{\pi i}\int_0^{+\infty} \lambda R^+(\lambda^2)v(I-S_2)
	\mathcal{L}_{00}^+(\lambda)(I-S_2)v
	\bigl(R_0^+(\lambda^2)-R_0^-(\lambda^2)\bigr)
	\chi(\lambda/\lambda_0)\,d\lambda,\\[2mm]
	\mathcal{W}_{01}^{\mathrm{L}}
	&=\frac{1}{\pi i}\int_0^{+\infty} \lambda R^+(\lambda^2)v(I-S_2)
	\mathcal{L}_{01}^+(\lambda)S_2v
	\bigl(R_0^+(\lambda^2)-R_0^-(\lambda^2)\bigr)
	\chi(\lambda/\lambda_0)\,d\lambda,\\[2mm]
	\mathcal{W}_{10}^{\mathrm{L}}
	&=\frac{1}{\pi i}\int_0^{+\infty} \lambda R^+(\lambda^2)vS_2
	\mathcal{L}_{10}^+(\lambda)(I-S_2)v
	\bigl(R_0^+(\lambda^2)-R_0^-(\lambda^2)\bigr)
	\chi(\lambda/\lambda_0)\,d\lambda,\\[2mm]
	\mathcal{W}_{11}^{\mathrm{L}}
	&=\frac{1}{\pi i}\int_0^{+\infty} \lambda R^+(\lambda^2)vS_2
	\mathcal{L}_{11}^+(\lambda)S_2v
	\bigl(R_0^+(\lambda^2)-R_0^-(\lambda^2)\bigr)
	\chi(\lambda/\lambda_0)\,d\lambda.
\end{align*}

\subsection{The proof of Proposition~\ref{pro:bound for low energy-1}}\label{sec:contri-of-co}
The contributions from $\mathcal{L}_{00}^+(\lambda)$, $\mathcal{L}_{01}^+(\lambda)$ and $\mathcal{L}_{10}^+(\lambda)$ are governed by the results of Subsection~\ref{sec:bound-of-N0-N2}. Their mapping properties are summarized as follows.

\begin{lemma}\label{lem:contri of L00}
If $H$ has no s-wave resonance, then $\mathcal{W}_{00}^{\mathrm{L}}$ is 
bounded on $L^p(\mathbb{R}^2)$ for $1<p<\infty$; if $H$ has an s-wave 
resonance, then $\mathcal{W}_{00}^{\mathrm{L}}$ is bounded on 
$L^p(\mathbb{R}^2)$ for $1\le p\le\infty$.
\end{lemma}

\begin{proof}
If $H$ has no s-wave resonance, then $\mathcal{L}_{00}^{+}(\lambda)$ 
coincides with $\bigl(M^{\pm}(\lambda)\bigr)^{-1}$ when zero is a regular 
point; if $H$ has an s-wave resonance, then it coincides with 
$\bigl(M^{\pm}(\lambda)\bigr)^{-1}$ when zero is a first-kind threshold 
singularity, upon replacing $S_1$ by $S_1-S_2$. The stated bounds follow 
from Proposition~\ref{pro:bound for low energy}.
\end{proof}

\begin{lemma}
	$\mathcal{W}_{01}^{\mathrm{L}}$ is bounded on $L^p(\mathbb{R}^2)$ for all $1\le p\le\infty$.
\end{lemma}

\begin{proof}
	Since $S_2v=0$ and $\mathcal{L}_{01}^+(\lambda)=\tilde{\mathcal{O}}_2(\log^2\lambda)$ for $\lambda\in(0,\lambda_0)$ by Lemma~\ref{thm-M inverse-even-1}, the conclusion follows from Lemma~\ref{Lemm: bound of N_0}.
\end{proof}

\begin{lemma}\label{lem:contri of L10}
	The operator $\mathcal{W}_{10}^{\mathrm{L}}$ is bounded on $L^p(\mathbb{R}^2)$ for $1\le p<\infty$ in general, and for $1\le p\le\infty$ when $H$ has s-wave but no p-wave resonances.
\end{lemma}

\begin{proof}
In the general case, since $S_2v=0$ and $\mathcal{L}_{10}^+(\lambda)=\tilde{\mathcal{O}}_2(\log^2\lambda)$, Lemma~\ref{Lemm: bound of N_1} gives that $\mathcal{W}_{10}^{\mathrm{L}}$ is bounded on $L^p(\mathbb{R}^2)$ for all $1\le p<\infty$. 
	
In the exceptional case that $H$ has s-wave but no p-wave resonances, we need to consider the contribution of each term in \eqref{eq:special mathcalL10}. Since $Qv=0$, Lemma~\ref{Lemm: bound of N_0} shows that the first term is $L^p$-bounded for $1\le p\le\infty$. By Lemma~\ref{Lemm: bound of N_1}, the second term $\Xi P$ enjoys the same boundedness. Finally, Lemma~\ref{Lemm: bound of N_2} applies to the remainder, giving $L^p$-boundedness for $1\le p\le\infty$. This completes the proof.
\end{proof}

It remains to treat the contribution from $\mathcal{L}_{11}^+(\lambda)$. It is bounded as follows.

\begin{lemma}\label{lem:contri of L11}
The operator $\mathcal{W}_{11}^{\mathrm{L}}$ is $L^p$-bounded for $1<p<2$ 
in general, and for $p=1$ and $2\le p<\infty$ if $H$ has no p-wave resonances. 
Furthermore, if $H$ has no p-wave resonances and every zero-energy eigenfunction $\psi$ 
satisfies \eqref{eq:condi_for_eigen-0},
then $\mathcal{W}_{11}^{\mathrm{L}}$ is also bounded on $L^\infty(\mathbb{R}^2)$.
\end{lemma}

We now assemble the preceding estimates to complete the proof of 
Proposition~\ref{pro:bound for low energy-1}. 

\begin{proof}[Completion of the proof of Proposition~\ref{pro:bound for low energy-1}]
Recall that
\[
\mathcal{W}^{\mathrm{L}}=\sum_{i,j=0}^1\mathcal{W}_{ij}^{\mathrm{L}}.
\]
The $L^2$-boundedness of $\mathcal{W}^{\mathrm{L}}$ follows from the unitarity 
of the wave operators.  When zero is either a second-kind threshold singularity or an 
eigenvalue of $H$, Lemmas~\ref{lem:contri of L00}--\ref{lem:contri of L11} 
collectively yield  
desired $L^p$ bounds of $\mathcal{W}^{\mathrm{L}}$, thereby establishing 
Proposition~\ref{pro:bound for low energy-1}.
\end{proof}

We are left to prove Lemma \ref{lem:contri of L11}.
The proof of this lemma reduces to studying operators of the form
\begin{align}
&\mathcal{N}_3=\frac{1}{\pi i}\int_0^{+\infty} \lambda\chi(\lambda/\lambda_0)\,
R_0^{+}(\lambda^2)v\mathcal{Q}_{m_1}\Gamma_3(\lambda)\mathcal{Q}_{m_2}v
\bigl(R_0^+(\lambda^{2})-R_0^-(\lambda^{2})\bigr)\,d\lambda,
\label{eq-rep-N4}\\[4mm]
&\mathcal{W}_{\mathrm{e}, \mathrm{sing}}^{\mathrm{L}}=\frac{1}{\pi i}\int_0^{+\infty} \lambda^{-1}\chi(\lambda/\lambda_0)\,
R_0^{+}(\lambda^2)vS_3D_3S_3v
\bigl(R_0^+(\lambda^{2})-R_0^-(\lambda^{2})\bigr)\,d\lambda,
\label{eq-rep-N4-2}
\end{align}
where $\Gamma_3(\lambda)\in\mathbb{B}(L^2)$ and $m_1,m_2\in\{1,2\}$, and $\mathcal{Q}_{m}$ is the projection defined in \eqref{eq:def of Qm}, satisfying
\[
\mathcal{Q}_m\bigl(v(x)\,x^{\alpha}\bigr)=0\quad\text{for }|\alpha|<m.
\]

\begin{lemma}\label{lem:bound of N_3-1}
	Let $|V(x)|\lesssim\langle x\rangle^{-6-}$. Suppose that the operator-valued function $\Gamma_3(\lambda)$ in \eqref{eq-rep-N4} satisfies $\Gamma_3(\lambda)=\tilde{\mathcal{O}}_2(\lambda^{-2}\log^{-2}\lambda)$. Then, for $m_1=m_2=1$, the operator $\mathcal{N}_3$ is bounded on $L^1(\mathbb{R}^2)$.
\end{lemma}

\begin{proof}
	To establish the $L^1$ boundedness, it suffices to show that the kernel of $\mathcal{N}_3$ satisfies
	\begin{equation}\label{eq:L1 bound ceri}
		\sup_{y}\int_{\mathbb{R}^2} \left|\mathcal{N}_3(x,y)\right| dx \lesssim 1.
	\end{equation}
	
	Since $|v(x)|\lesssim\langle x\rangle^{-3-}$, we express its kernel via \eqref{eq:S0vR+R-} and \eqref{eq:S0vR-G-0} as
	\[
	R_0^+(\lambda^2)\mathcal{Q}_{1}\Gamma_3(\lambda)\mathcal{Q}_{1}
	\bigl(R_0^+(\lambda^2)-R_0^-(\lambda^2)\bigr)(x,y) 
	= \sum_{\pm}e^{i\lambda(\langle x\rangle \pm \langle y\rangle)}
	\mathcal{K}_{3,0}^{\pm}(\lambda,x,y)
	-\sum_{\pm}e^{\pm i\lambda|y|}
	\mathcal{K}_{3,1}^{\pm}(\lambda,x,y),
	\]
	with (omitting a constant factor)
	\begin{align*}
		\mathcal{K}_{3,0}^{\pm}(\lambda,x,y) 
		&= \big\langle \mathcal{Q}_{m_1}\Gamma_3(\lambda)\mathcal{Q}_{m_2}
		\eta_{m_2}^{\pm}(\lambda,\cdot,y),\, 
		\omega_{m_1}^{-}(\lambda,\cdot,x)\big\rangle,\\[1mm]
		\mathcal{K}_{3,1}^{\pm}(\lambda,x,y) 
		&=\big\langle \mathcal{Q}_{m_1}\Gamma_3(\lambda)\mathcal{Q}_{m_2}
		\eta_{m_2}^{\pm}(\lambda,\cdot,y),\, 
		\psi_{m_1}(\cdot,x)\big\rangle\, \chi(\lambda\langle x\rangle).
	\end{align*}
	Combined with the representation \eqref{eq-rep-N4} of $\mathcal{N}_3$, this leads us to study the oscillatory integrals
	\begin{align*}
		\mathcal{N}_{3,0}^{\pm}(x,y)
		&= \frac{1}{\pi i}\int_0^\infty e^{i\lambda(\langle x\rangle \pm \langle y\rangle)} 
		\mathcal{K}_{3,0}^{\pm}(\lambda,x,y) \, \lambda\chi(\lambda/\lambda_0) \, d\lambda,\\[1mm]
		\mathcal{N}_{3,1}^{\pm}(x,y) 
		&= \frac{1}{\pi i}\int_0^\infty e^{\pm i\lambda\langle y\rangle} 
		\mathcal{K}_{3,1}^{\pm}(\lambda,x,y) \, \lambda\chi(\lambda/\lambda_0) \, d\lambda.
	\end{align*}
	
	We analyze $\mathcal{N}_{3,0}^{\pm}$ separately in the regions $|x|\le 2|y|$ and $|x|>2|y|$.
	
	\smallskip
	\noindent\emph{The region $|x|\le 2|y|$.} 
	Under the condition $\Gamma_3(\lambda)=\tilde{\mathcal{O}}_2(\lambda^{-2}\log^{-2}\lambda)$ and using \eqref{eq:est-eta-w-1}, \eqref{eq:omega-est-3} and the Cauchy--Schwarz inequality, we obtain that when $m_1=m_2=1$, for $\ell=0,1,2$ and $\lambda\in(0,\lambda_0)$,
	\[
	|\partial_\lambda^\ell \mathcal{K}_{3,0}^{\pm}(\lambda,x,y)| 
	\lesssim \lambda^{-1-\ell} (\log\lambda)^{-2} 
	\langle x \rangle^{-1/2}\langle y\rangle^{-1/2}
	\lesssim \lambda^{-1-\ell} (\log\lambda)^{-2} \langle x \rangle^{-1}.
	\]
	Applying Lemma~\ref{lem-oscillatory estimate} yields 
	\[
	|\mathcal{N}_{3,0}^{\pm}(x,y)| 
	\lesssim \langle\langle x\rangle \pm \langle y\rangle\rangle^{-1}
	\langle x \rangle^{-1}\log^{-2}\bigl(2+\big|\langle x\rangle \pm \langle y\rangle\big|\bigr).
	\]
	In spherical coordinates, the above bound shows that
	\begin{align*}
		\sup_{y}\int_{\mathbb{R}^2} \left|\mathcal{N}_{3,0}^\pm (x,y)\right|\,dx
		&\lesssim \sup_{y}\int_{0}^\infty  
		\langle r\pm \langle y\rangle \rangle^{-1}
		\log^{-2}\bigl(2+\big|r \pm  \langle y\rangle\big|\bigr)\,dr \\
		&\lesssim \sup_{y}\int_{0}^\infty 
		\langle r\rangle^{-1}\log^{-2}(2+|r|)\,dr\lesssim 1,
	\end{align*}
	which implies that $\mathcal{N}_{3,0}^{\pm}(x,y)$ satisfies 
	\eqref{eq:L1 bound ceri} in this region.
	
	\smallskip
	\noindent\emph{The region $|x|> 2|y|$.} 
	Using \eqref{eq:est-eta-w-2}, \eqref{eq:omega-est-3} and the Cauchy--Schwarz inequality, we obtain that for $m_1=m_2=1$, $\ell=0,1,2$ and $\lambda\in(0,\lambda_0)$,
	\[
	|\partial_\lambda^\ell \mathcal{K}_{3,0}^{\pm}(\lambda,x,y)| 
	\lesssim \lambda^{-1/2-\ell} (\log\lambda)^{-2} 
	\langle x \rangle^{-1/2}
	\lesssim \lambda^{-1/2-\ell} (\log\lambda)^{-2} \langle x \rangle^{-1/2}.
	\]
	Applying Lemma~\ref{lem-oscillatory estimate} yields 
	\[
	|\mathcal{N}_{3,0}^{\pm}(x,y)| 
	\lesssim \langle\langle x\rangle \pm \langle y\rangle \rangle^{-3/2}
	\langle x \rangle^{-1/2}
	\log^{-2}\bigl(2+\big|\langle x\rangle \pm \langle y\rangle\big|\bigr)
	\lesssim \langle x \rangle^{-2}\log^{-2}\bigl(2+|x|\bigr),
	\]
	since $|x|>2|y|$. This implies that $\mathcal{N}_{3,0}^{\pm}(x,y)$ satisfies 
	\eqref{eq:L1 bound ceri} in this region as well.
	
	\smallskip
	We next consider the kernel $\mathcal{N}_{3,1}^{\pm}(x, y)$. Since 
	\eqref{eq:psi-est-0} holds for $\psi_1(\cdot,x)$, we also have
	\[
	|\mathcal{K}_{3,1}^{\pm}(\lambda,x,y)| 
	\lesssim \lambda^{-1} (\log\lambda)^{-2} \langle x \rangle^{-1}.
	\]
	Noting that the support of $\mathcal{K}_{3,1}^{\pm}(\lambda,x,y)$ satisfies 
	$\lambda\langle x\rangle \le 1$, a direct computation shows
	\[
	\sup_{y}|\mathcal{N}_{3,1}^{\pm}(x,y)| 
	\lesssim \langle x \rangle^{-2}\log^{-2}\bigl(2+|x|\bigr),
	\]
	which also implies that this kernel satisfies \eqref{eq:L1 bound ceri}.
	
	Combining the estimates for $\mathcal{N}_{3,0}^{\pm}$ and 
	$\mathcal{N}_{3,1}^{\pm}$, we complete the proof.
\end{proof}

\begin{lemma}\label{lem:bound of N_3-2-1}
	Suppose $|V(x)|\lesssim\langle x\rangle^{-6-}$ and that 
	$\Gamma_3(\lambda)=\tilde{\mathcal{O}}_2(\lambda^{-2}g_{1}^{+}(\lambda)^{-1})$. 
	Then $\mathcal{N}_3$ is bounded on $L^{p}(\mathbb{R}^2)$ for $1<p<2$ when $m_1=m_2=1$.
\end{lemma}

\begin{proof}
	Note that $g_{1}^{\pm}(\lambda)=\log\lambda+\beta_{\pm}$ with 
	$\beta_{\pm}\in\mathbb{C}\setminus\mathbb{R}$. Following the proof of 
	Lemma~\ref{lem:bound of N_3-1}, we obtain
	\begin{align*}
		|\mathcal{N}_{3,0}^{\pm}(x,y)| &\lesssim 
		\begin{cases}
			\langle\langle x\rangle \pm \langle y\rangle \rangle^{-1}
			\langle x \rangle^{-1}, & |x|\le 2|y|,\\[4pt]
			\langle x \rangle^{-2}, & |x|\ge 2|y|,
		\end{cases}
		\\[4pt]
		|\mathcal{N}_{3,1}^\pm(x,y)| &\lesssim \langle x \rangle^{-2}.
	\end{align*}
	Since 
	\[
	|\partial_{\lambda}^k \mathcal{K}_{3,1}^{\pm}(\lambda,x,y)| 
	\lesssim \lambda^{-1-k} (\log\lambda)^{-2} \langle x \rangle^{-1}, 
	\quad k=0,1,2,
	\]
	oscillatory integral estimates yield another bound 
	$|\mathcal{N}_{3,1}^\pm(x,y)| \lesssim \langle x \rangle^{-1}\langle y\rangle^{-1}$.
	These bounds imply
	\[
	\left\| \mathcal{N}_{3}(x,y) \right\|_{L^{p,\infty}(L^{p',\infty}(\mathbb{R}^2_y),\, \mathbb{R}^2_x)}\lesssim 1,
	\qquad \frac{1}{p}+\frac{1}{p'}=1, \quad 1<p\le 2.
	\]
	Hence $\mathcal{N}_3$ is bounded from $L^{p,\infty}(\mathbb{R}^2)$ to 
	$L^{p,1}(\mathbb{R}^2)$ for each $1<p\le 2$ by H\"older's inequality in 
	Lorentz spaces. Then, in light of the off-diagonal Marcinkiewicz 
	interpolation theorem in \cite{graf}, we complete the proof.
\end{proof}

\begin{lemma}\label{lem:bound of N_3-3}
	Let $|V(x)|\lesssim\langle x\rangle^{-6-}$. Suppose the operator-valued function 
	$\Gamma_3(\lambda)=\tilde{\mathcal{O}}_2(\lambda^{-1+})$. Then, for $m_1=m_2=1$, 
	the operator $\mathcal{N}_3$ is bounded on $L^p(\mathbb{R}^2)$ for all 
	$1 \le p \le \infty$.
\end{lemma}

\begin{proof}
	Following the proof of Lemma~\ref{lem:bound of N_3-1}, it suffices to analyze 
	the operators whose kernels are given by the oscillatory integrals 
	$\mathcal{N}_{3,0}^{\pm}$ and $\mathcal{N}_{3,1}^{\pm}$. Since 
	$\Gamma_3(\lambda)=\tilde{\mathcal{O}}_2(\lambda^{-1+})$, the estimates 
	\eqref{eq:est-eta-w-1} and \eqref{eq:est-eta-w-2} together with the 
	Cauchy--Schwarz inequality yield
	\[
	|\partial_\lambda^\ell \mathcal{K}_{3,0}^{\pm}(\lambda,x,y)| 
	\lesssim \lambda^{1-\ell+} \langle x \rangle^{-1/2} \langle y \rangle^{-1/2},
	\qquad \ell=0,1,2.
	\]
	Applying Lemma~\ref{lem-oscillatory estimate} with $k=2$ and exponent 
	$b=1-$, we obtain
	\[
	|\mathcal{N}_{3,0}^{\pm}(x,y)| 
	\lesssim \langle x \rangle^{-1/2} 
	\langle \langle x\rangle \pm \langle y\rangle \rangle^{-2-} 
	\langle y \rangle^{-1/2},
	\]
	which shows that the kernels $\mathcal{N}_{3,0}^{\pm}(x,y)$ are admissible. Similarly, by \eqref{eq:est-eta-w-1} and \eqref{eq:psi-est-0} together with 
	the Cauchy--Schwarz inequality,
	\[
	|\partial_\lambda^\ell \mathcal{K}_{3,1}^{\pm}(\lambda,x,y)| 
	\lesssim \lambda^{-1/2-\ell+} \langle x \rangle^{-1} \langle y \rangle^{-1/2},
	\qquad \ell=0,1,2,
	\]
	whence
	\[
	|\mathcal{N}_{3,1}^{\pm}(x,y)| 
	\lesssim \langle x \rangle^{-1} 
	\langle \langle x\rangle \pm \langle y\rangle \rangle^{-\frac32-} 
	\langle y \rangle^{-1/2}.
	\]
	These kernels are also admissible. This completes the proof.
\end{proof}

\begin{lemma}\label{lem:bound of N_3-2}
	Suppose $|V(x)|\lesssim\langle x\rangle^{-6-}$ and that 
	$\Gamma_3(\lambda)=\tilde{\mathcal{O}}_2(\lambda^{-2})$. Then the operator 
	$\mathcal{N}_3$ defined in \eqref{eq-rep-N4} is bounded on 
	$L^1(\mathbb{R}^2)$ whenever $m_1+m_2>2$.
\end{lemma}

\begin{proof}
	In this case, following the proof of Lemma~\ref{lem:bound of N_3-1}, we 
	obtain that for $\ell=0,1,2$ and $\lambda\in(0,\lambda_0)$,
	\[
	|\partial_\lambda^\ell \mathcal{K}_{3,0}^{\pm}(\lambda,x,y)| \lesssim
	\begin{cases}
		\lambda^{m_1+m_2-3-\ell} \langle x \rangle^{-1/2}\langle y\rangle^{-1/2}, 
		& |x|\le 2|y|,\\[4pt]
		\lambda^{m_1+m_2-5/2-\ell} \langle x \rangle^{-1/2}, 
		& |x|>2|y|,
	\end{cases}
	\]
	and 
	\[
	|\partial_\lambda^\ell \mathcal{K}_{3,1}^{\pm}(\lambda,x,y)| \lesssim
	\lambda^{m_2-2-\ell} \langle x \rangle^{-m_1}.
	\]
	These estimates yield
	\begin{align*}
		|\mathcal{N}_{3,0}^{\pm}(x,y)| 
		&\lesssim \langle\langle x\rangle \pm \langle y\rangle \rangle^{-(m_1+m_2-1)}
		\langle x \rangle^{-1}+ \langle x \rangle^{-(m_1+m_2)},\\
		|\mathcal{N}_{3,1}^\pm(x,y)| 
		&\lesssim \langle x \rangle^{-(m_1+m_2)},
	\end{align*}
	which immediately implies the $L^1$-boundedness of $\mathcal{N}_3$.
\end{proof}

\begin{lemma}\label{lem:bound of N_4}
	Assume that $|V(x)|\lesssim\langle x\rangle^{-8-}$. Then the operator $\mathcal{W}_{\mathrm{e}, \mathrm{sing}}^{\mathrm{L}}$ is bounded on $L^p(\mathbb{R}^2)$ for every $1\le p<\infty$. If, in addition, each zero-energy eigenfunction of $H$ satisfies \eqref{eq:condi_for_eigen-0}, then $\mathcal{W}_{\mathrm{e}, \mathrm{sing}}^{\mathrm{L}}$ is also bounded on $L^\infty(\mathbb{R}^2)$.
\end{lemma}

\begin{proof}
	By \cite[Lemma 5.6]{Erd-Gre-13},
	\[
		(-\Delta)^{-1}v S_3D_3S_3 v (-\Delta)^{-1}=\mathcal{P}_e,
	\]
	where $\mathcal{P}_e$ is the orthogonal projection onto $\mathscr{N}_e:=\mathcal{P}_eL^2(\mathbb{R}^2)$, the zero-energy eigenspace of $H$; in particular $vS_3D_3S_3v=(-\Delta)\mathcal{P}_e(-\Delta)$. Since every zero-energy eigenfunction $\psi$ of $H$ satisfies $-\Delta\psi+V\psi=0$ (see Remark~\ref{remk:one-to-one}), we have $(-\Delta)\mathcal{P}_e=-V\mathcal{P}_e$ and, by self-adjointness, $\mathcal{P}_e(-\Delta)=-\mathcal{P}_eV$; consequently
	\[
		vS_3D_3S_3v=(-\Delta)\mathcal{P}_e(-\Delta)=V\mathcal{P}_eV ,
	\]
	and the singular operator can be rewritten as
	\begin{align*}
		\mathcal{W}_{\mathrm{e}, \mathrm{sing}}^{\mathrm{L}}
		=\frac{1}{\pi i}\int_0^{+\infty} \lambda^{-1} \chi(\lambda/\lambda_0)\,
		R_0^+(\lambda^{2})V\mathcal{P}_eV
		\bigl(R_0^+(\lambda^{2})-R_0^-(\lambda^{2})\bigr) \, d\lambda .
	\end{align*}

	Let $\{\Psi_{j}\}_{j=1}^k$ be an orthonormal basis of $\mathscr{N}_e$. Recall that $S_3\bigl(v\,x^\alpha\bigr)=0$ for all $|\alpha|<2$; through the one-to-one correspondence of Remark~\ref{remk:one-to-one}, this cancellation takes the form
	\begin{equation}\label{vanishing-con}
		\int_{\mathbb{R}^2}  x^{\alpha}\,(V\Psi_j)(x)\,dx=0, \qquad |\alpha|<2,\quad 1\le j\le k .
	\end{equation}
	In particular,
	\[
		\mathcal{P}_e\bigl(Vx^\alpha\bigr)
		=\sum_{j=1}^k \Psi_j\,\overline{\int_{\mathbb{R}^2}x^\alpha\,(V\Psi_j)(x)\,dx}=0,
		\qquad |\alpha|<2 ,
	\]
	so Lemma~\ref{lem:bound of N_3-2} applies and $\mathcal{W}_{\mathrm{e}, \mathrm{sing}}^{\mathrm{L}}$ is bounded on $L^1(\mathbb{R}^2)$. Moreover, following the proof of Lemma~\ref{lem:bound of N_3-1}, its integral kernel satisfies
	\[
		|\mathcal{W}_{\mathrm{e}, \mathrm{sing}}^{\mathrm{L}}(x,y)|
		\lesssim \langle x \rangle^{-1/2}
		\langle \langle x\rangle \pm \langle y\rangle \rangle^{-2}
		\langle y \rangle^{-1/2}
		+ \langle x \rangle^{-2} \langle y \rangle^{-2},
	\]
	and a standard Schur-type argument then yields the $L^p(\mathbb{R}^2)$ boundedness of $\mathcal{W}_{\mathrm{e}, \mathrm{sing}}^{\mathrm{L}}$ for every $1<p<\infty$.

	It remains to establish the $L^\infty(\mathbb{R}^2)$ boundedness under condition \eqref{eq:condi_for_eigen-0}. Inserting the resolvent identity
	\[
		R_0^+(\lambda^{2})=(-\Delta)^{-1}+\lambda^2 (-\Delta)^{-1}R_0^+(\lambda^{2})
	\]
	into the representation above, we decompose $\mathcal{W}_{\mathrm{e}, \mathrm{sing}}^{\mathrm{L}}=\mathcal{N}_{4,m}+\mathcal{N}_{4,r}$, where
	\begin{align*}
		\mathcal{N}_{4,m}
		&:= \frac{1}{\pi i}\int_0^{+\infty} \lambda^{-1} \chi(\lambda/\lambda_0)\,
		(-\Delta)^{-1}V\mathcal{P}_eV
		\bigl(R_0^+(\lambda^{2})-R_0^-(\lambda^{2})\bigr) \, d\lambda ,\\
		\mathcal{N}_{4,r}
		&:= \frac{1}{\pi i}\int_0^{+\infty} \lambda\, \chi(\lambda/\lambda_0)\,
		(-\Delta)^{-1}R_0^+(\lambda^{2})V\mathcal{P}_eV
		\bigl(R_0^+(\lambda^{2})-R_0^-(\lambda^{2})\bigr) \, d\lambda .
	\end{align*}

	\medskip
	\noindent\emph{Step 1: $\mathcal{N}_{4,r}$ is bounded on $L^\infty(\mathbb{R}^2)$ regardless of condition \eqref{eq:condi_for_eigen-0}.}
	Splitting $(-\Delta)^{-1}=\frac{\chi(\sqrt{-\Delta}/\lambda_0)}{-\Delta}+\frac{\tilde{\chi}(\sqrt{-\Delta}/\lambda_0)}{-\Delta}$ in front of the $\lambda$-integral gives $\mathcal{N}_{4,r}=\mathcal{N}_{4,r,0}+\mathcal{N}_{4,r,1}$, where $\mathcal{N}_{4,r,0}$ and $\mathcal{N}_{4,r,1}$ contain the cutoffs $\chi(\sqrt{-\Delta}/\lambda_0)$ and $\tilde{\chi}(\sqrt{-\Delta}/\lambda_0)$, respectively.

	The symbol $\tilde{\chi}(|\xi|/\lambda_0)\,|\xi|^{-2}$ is smooth and satisfies $\bigl|\partial_\xi^\alpha\bigl(\tilde{\chi}(|\xi|/\lambda_0)|\xi|^{-2}\bigr)\bigr|\lesssim|\xi|^{-2-|\alpha|}$; its inverse Fourier transform is therefore integrable, and $\frac{\tilde{\chi}(\sqrt{-\Delta}/\lambda_0)}{-\Delta}$ is bounded on $L^p(\mathbb{R}^2)$ for all $1\le p\le\infty$. On the other hand, we use
	\[
		V\mathcal{P}_eV=vS_3D_3S_3v,
	\]
	in which the middle operator factors through the projection $S_3\le Q$; in view of the vanishing moments \eqref{vanishing-con}, and since the symbol $f_0(\lambda):=\lambda\chi(\lambda/\lambda_0)$ readily satisfies \eqref{eq:assum for f_0}, the remaining $\lambda$-integral is covered by Lemma~\ref{Lemm: bound of N_0} and is bounded on $L^p(\mathbb{R}^2)$ for all $1\le p\le\infty$. Hence $\mathcal{N}_{4,r,1}$ is bounded on $L^\infty(\mathbb{R}^2)$.

	For $\mathcal{N}_{4,r,0}$ we argue by duality: its kernel can be verified to be locally integrable, thus $\mathcal{N}_{4,r,0}$ is bounded on $L^\infty$ if its duality is bounded on $L^1$.  So it suffices to prove that its adjoint
	\[
		\mathcal{N}_{4,r,0}^*
		=\frac{1}{\pi i}\int_0^{+\infty} \lambda\chi(\lambda/\lambda_0)\,
		\bigl(R_0^+(\lambda^{2})-R_0^-(\lambda^{2})\bigr)V\mathcal{P}_eV
		R_0^{-}(\lambda^{2}) \frac{\chi(\sqrt{-\Delta}/\lambda_0)}{-\Delta}\, d\lambda
	\]
	is bounded on $L^1(\mathbb{R}^2)$; indeed, $\mathcal{N}_{4,r,0}=(\mathcal{N}_{4,r,0}^*)^*$ then acts boundedly on $L^\infty(\mathbb{R}^2)=(L^1(\mathbb{R}^2))^*$. To identify $\mathcal{N}_{4,r,0}^*$ we compute Fourier transforms with the help of the free spectral measure identity
	\begin{equation}\label{eq:free-measure}
		\frac{\lambda}{\pi i}\bigl(R_0^+(\lambda^{2})-R_0^-(\lambda^{2})\bigr)\varphi
		= \bigl(\hat\varphi\, d\sigma_\lambda\bigr)^{\vee}
	\end{equation}
	and Parseval's formula, exactly as in the proof of Lemma~\ref{lem-W0s-bound}. First, expanding $\mathcal{P}_ef=\sum_{j=1}^k\Psi_j\langle f,\Psi_j\rangle$ gives
	\begin{equation*}
		\mathcal{W}_{\mathrm{e}, \mathrm{sing}}^{\mathrm{L}}f
		=\sum_{j=1}^k\frac{1}{\pi i}\int_0^{+\infty} \lambda^{-1} \chi(\lambda/\lambda_0)\,
		R_0^+(\lambda^{2})(V\Psi_j) 
		\bigl\langle\bigl(R_0^+-R_0^-\bigr)(\lambda^{2})f,\,V\Psi_j\bigr\rangle \, d\lambda .
	\end{equation*}
	By \eqref{eq:free-measure} and Parseval's formula,
	\[
		\bigl\langle\bigl(R_0^+(\lambda^{2})-R_0^-(\lambda^{2})\bigr)f,\,V\Psi_j\bigr\rangle
		=\frac{\pi i\,\lambda^{-1}}{(2\pi)^2}\int_{\lambda\mathbb{S}^1}
		\hat f(\eta)\,\overline{\widehat{V\Psi_j}(\eta)}\,d\sigma_{\lambda}(\eta),
	\]
	so taking the Fourier transform in $x$, using $\widehat{R_0^+(\lambda^2)g}\,(\xi)=\frac{\widehat{g}(\xi)}{|\xi|^2-\lambda^2-i0}$ and the coarea formula $d\eta=d\sigma_{\lambda}(\eta)\,d\lambda$ with $\lambda=|\eta|$, we obtain
	\begin{equation}\label{eq:Foure express}
		\mathcal{F}\left(\mathcal{W}_{\mathrm{e}, \mathrm{sing}}^{\mathrm{L}}f\right)(\xi)
		=\sum_{j=1}^k \frac{\widehat{V\Psi_j}(\xi)}{(2\pi)^2}
		\lim_{\varepsilon\to 0^+}\int_{\mathbb{R}^2}
		\frac{\overline{\widehat{V\Psi_j}(\eta)}\,\chi(|\eta|/\lambda_0)\,\hat f(\eta)}
		{|\eta|^2\bigl(|\xi|^2-|\eta|^2-i\varepsilon\bigr)}
		\,d\eta .
	\end{equation}
	Second, \eqref{eq:free-measure} applied to the outer spectral measure of $\mathcal{N}_{4,r,0}^*$ forces $\lambda=|\xi|$ on the Fourier side, and computing the inner pairing as above yields
	\begin{equation}\label{eq:dual-form}
		\mathcal{F}\bigl(\mathcal{N}_{4,r,0}^*f\bigr)(\xi)
		=\sum_{j=1}^k \frac{\chi(|\xi|/\lambda_0)\,\widehat{V\Psi_j}(\xi)}{(2\pi)^2}
		\lim_{\varepsilon\to 0^+}\int_{\mathbb{R}^2}
		\frac{\overline{\widehat{V\Psi_j}(\eta)}\,\chi(|\eta|/\lambda_0)\,\hat f(\eta)}
		{|\eta|^2\bigl(|\eta|^2-|\xi|^2+i\varepsilon\bigr)}
		\,d\eta .
	\end{equation}
	Since $|\eta|^2-|\xi|^2+i\varepsilon=-\bigl(|\xi|^2-|\eta|^2-i\varepsilon\bigr)$, comparing \eqref{eq:dual-form} with \eqref{eq:Foure express} gives
	\[
		\mathcal{F}\bigl(\mathcal{N}_{4,r,0}^*f\bigr)(\xi)
		=-\chi(|\xi|/\lambda_0)\cdot\mathcal{F}\bigl(\mathcal{W}_{\mathrm{e},\mathrm{sing}}^{\mathrm{L}}f\bigr)(\xi),
	\]
	that is, $\mathcal{N}_{4,r,0}^*=-\chi(\sqrt{-\Delta}/\lambda_0)\circ\mathcal{W}_{\mathrm{e},\mathrm{sing}}^{\mathrm{L}}$. The cutoff $\chi(\sqrt{-\Delta}/\lambda_0)$ has an integrable kernel and is therefore bounded on $L^1(\mathbb{R}^2)$, and $\mathcal{W}_{\mathrm{e},\mathrm{sing}}^{\mathrm{L}}$ is bounded on $L^1(\mathbb{R}^2)$ by the first part of the proof. Hence $\mathcal{N}_{4,r,0}^*$ is bounded on $L^{1}(\mathbb{R}^2)$, and $\mathcal{N}_{4,r,0}$ is bounded on $L^\infty(\mathbb{R}^2)$. This completes Step 1.

	\medskip
	\noindent\emph{Step 2: $\mathcal{N}_{4,m}$ under condition \eqref{eq:condi_for_eigen-0}.}
	Since $-\Delta\Psi_j+V\Psi_j=0$ gives $\Psi_j=-(-\Delta)^{-1}(V\Psi_j)$, we have $(-\Delta)^{-1}V\mathcal{P}_e=-\mathcal{P}_e$, hence
	\[
		\mathcal{N}_{4,m}
		=-\frac{1}{\pi i}\int_0^{+\infty} \lambda^{-1} \chi(\lambda/\lambda_0)\,
		\mathcal{P}_eV\bigl(R_0^+(\lambda^{2})-R_0^-(\lambda^{2})\bigr) \, d\lambda .
	\]
	Let $f\in L^\infty(\mathbb{R}^2)\cap L^2(\mathbb{R}^2)$. Applying \eqref{eq:free-measure} with the profile $\varphi(\lambda)=\lambda^{-2}\chi(\lambda/\lambda_0)$ gives the functional-calculus identity
	\[
		\frac{1}{\pi i}\int_0^{+\infty}\lambda^{-1}\chi(\lambda/\lambda_0)\bigl(R_0^+(\lambda^{2})-R_0^-(\lambda^{2})\bigr)\,d\lambda
		=(-\Delta)^{-1}\chi(\sqrt{-\Delta}/\lambda_0).
	\]
	Using it under the pairing with $V\Psi_j$ (which is legitimate in view of \eqref{vanishing-con}), we obtain
	\begin{align*}
		\mathcal{N}_{4,m}f
		&=-\sum_{j=1}^k\Psi_j\,
		\bigl\langle(-\Delta)^{-1}\chi(\sqrt{-\Delta}/\lambda_0)f,\,V\Psi_j\bigr\rangle  \\
		&=\sum_{j=1}^k\Psi_j\,
		\bigl\langle f,\,\chi(\sqrt{-\Delta}/\lambda_0)\Psi_j\bigr\rangle
		=\mathcal{P}_e\,\chi(\sqrt{-\Delta}/\lambda_0)f ,
	\end{align*}
	where we used the self-adjointness of the Fourier multipliers involved and the identity $(-\Delta)^{-1}(V\Psi_j)=-\Psi_j$. Consequently
	\[
		\|\mathcal{N}_{4,m}f\|_{L^\infty}
		\le \biggl(\sum_{j=1}^k\|\Psi_j\|_{L^\infty}\,
		\bigl\|\chi(\sqrt{-\Delta}/\lambda_0)\Psi_j\bigr\|_{L^1}\biggr)\,\|f\|_{L^\infty},
	\]
	and it remains to show that $\Psi_j\in L^\infty(\mathbb{R}^2)$ and $\chi(\sqrt{-\Delta}/\lambda_0)\Psi_j\in L^1(\mathbb{R}^2)$ for each $j$.

	\smallskip
	\noindent\emph{Boundedness of $\Psi_j$.} Since $-\frac{1}{2\pi}\log|x|$ is the fundamental solution of $-\Delta$, the identity $\Psi_j=-(-\Delta)^{-1}(V\Psi_j)$ reads
	\[
		\Psi_j(x)
		=\frac{1}{2\pi}\int_{\mathbb{R}^2} \log|x-y|\, (V\Psi_j)(y)\,dy
		=\frac{1}{2\pi}\int_{\mathbb{R}^2} \log\frac{|x-y|}{\langle x\rangle}\,
		(V\Psi_j)(y)\,dy ,
	\]
	where the second equality uses \eqref{vanishing-con} with $\alpha=0$. Note that $\langle\,\cdot\,\rangle^{7}V\Psi_j\in L^{1}(\mathbb{R}^2)$: indeed, $\Psi_j\in L^2(\mathbb{R}^2)$ and $|V(x)|\lesssim\langle x\rangle^{-8-}$, so $\langle x\rangle^{7}|(V\Psi_j)(x)|\lesssim\langle x\rangle^{-1-}|\Psi_j(x)|$, which is integrable by the Cauchy--Schwarz inequality. Using the first-order moments in \eqref{vanishing-con} as well, one may further subtract the linear part of $y\mapsto\log|x-y|$ at $y=0$ (the remainder is $\mathcal{O}(\langle y\rangle^{2}/\langle x\rangle^{2})$ for $|y|\le |x|/2$); the resulting routine estimates then yield $|\Psi_j(x)|\lesssim\langle x\rangle^{-2}$, and in particular $\Psi_j\in L^\infty(\mathbb{R}^2)$.

	\smallskip
	\noindent\emph{Integrability of $\chi(\sqrt{-\Delta}/\lambda_0)\Psi_j$.} In view of \eqref{vanishing-con}, expanding the Fourier transform of $V\Psi_j$ at $\eta=0$ gives
	\begin{align*}
		\widehat{V\Psi_j}(\eta)
		&=\int_{\mathbb{R}^2} e^{-ix\cdot\eta} (V\Psi_j)(x)\, dx 
		=\frac{1}{2}\int_{\mathbb{R}^2} (-ix\cdot\eta)^2 (V\Psi_j)(x)\,dx
		+\int_{\mathbb{R}^2} \mathcal{R}_{3}(i x\cdot\eta)\, (V\Psi_j)(x)\,dx \\
		&=:\mathbf{P}_j(\eta)+E_j(\eta),
	\end{align*}
	where
	\[
		\mathcal{R}_{3}(t)=e^{-it}-1+it+\frac{t^{2}}{2},\qquad
		|\mathcal{R}_{3}^{(k)}(t)|\lesssim|t|^{3-k}\quad(0\le k\le 3).
	\]
	Write $\mathbf{P}_j(\eta)=-\frac12\sum_{1\le i,\ell\le2}m_{i,\ell}^{(j)}\,\eta_i\eta_l$ with $m_{i,\ell}^{(j)}:=\int_{\mathbb{R}^2}x_ix_\ell\,(V\Psi_j)(x)\,dx$. The condition \eqref{eq:condi_for_eigen-0} shows that 
	\[
		\mathbf{P}_j(\eta)= -\frac{m_{1,1}^{(j)}}{2}\,|\eta|^2 .
	\]
	(Without \eqref{eq:condi_for_eigen-0}, the quotient $\mathbf{P}_j(\eta)/|\eta|^2$ would be a linear combination of the Riesz symbols $\eta_i\eta_l/|\eta|^2$, whose inverse Fourier transforms decay only like $|x|^{-2}$ and fail to be integrable; this is precisely where the condition enters.)
	Since $\Psi_j=-(-\Delta)^{-1}(V\Psi_j)$, we conclude that
	\begin{align*}
		\chi(\sqrt{-\Delta}/\lambda_0)\Psi_j(x)
		&=-\chi(\sqrt{-\Delta}/\lambda_0)(-\Delta)^{-1}(V\Psi_j)(x)\\
		&=\frac{m_{1,1}^{(j)}}{2}\,\mathcal{F}^{-1}\bigl(\chi(|\cdot|/\lambda_0)\bigr)(x)
		-\mathcal{F}^{-1}\Bigl(\frac{E_j(\cdot)\,\chi(|\cdot|/\lambda_0)}{|\cdot|^2}\Bigr)(x).
	\end{align*}
	The first term on the right-hand side belongs to $L^{1}(\mathbb{R}^2)$, since $\chi(|\cdot|/\lambda_0)$ is a Schwartz function.  Note that $\langle\,\cdot\,\rangle^3V\Psi_j\in L^1(\mathbb{R}^2)$.
	Set $a_j(\eta):=E_j(\eta)\,\chi(|\eta|/\lambda_0)\,|\eta|^{-2}$. Then $a_j$ is compactly supported and satisfies $|a_j(\eta)|\lesssim|\eta|$ and $|\partial_\eta^\alpha a_j(\eta)|\lesssim|\eta|^{1-|\alpha|}$ for $|\alpha|\le3$. This implies that $\mathcal{F}^{-1}a_j\in L^{1}(\mathbb{R}^2)$. Thus $\chi(\sqrt{-\Delta}/\lambda_0)\Psi_j\in L^1(\mathbb{R}^2)$ for each $j$, and $\mathcal{N}_{4,m}$ is bounded on $L^\infty(\mathbb{R}^2)$.

	Combining Steps 1 and 2 shows that, under condition \eqref{eq:condi_for_eigen-0}, the operator $\mathcal{W}_{\mathrm{e},\mathrm{sing}}^{\mathrm{L}}=\mathcal{N}_{4,m}+\mathcal{N}_{4,r}$ is bounded on $L^\infty(\mathbb{R}^2)$. This completes the proof of Lemma~\ref{lem:bound of N_4}.
\end{proof}


\section{Negative results: unboundedness of wave operators}\label{sec:unbound}
This section is devoted to the proofs of Propositions~\ref{pro:unbound}
and~\ref{pro:unbound-1}, which establish the negative results in
Theorems~\ref{thm-main result-1} and~\ref{thm-main result-2},
respectively. We prove Proposition~\ref{pro:unbound} in the first
subsection; the proofs of statements~\textup{(i)} and~\textup{(ii)} of
Proposition~\ref{pro:unbound-1} are given in the remaining two
subsections.

\subsection{Proof of Proposition~\ref{pro:unbound}}\label{sec:unb-for-regular}

This subsection is devoted to the proof of Proposition~\ref{pro:unbound}.
We establish the unboundedness by constructing explicit counterexamples:
in the $L^1$ setting, we exhibit a suitably chosen function whose
Fourier transform does not vanish at the origin and whose image under
$\mathcal{W}^{\mathrm{L}}$ fails to be integrable; in the $L^\infty$
setting, we test $\mathcal{W}^{\mathrm{L}}$ against the characteristic
functions of large balls and show that the resulting outputs grow like
$\log\log\mu$ on the annuli $|x|\sim\mu$.

When zero is a regular point of $H$, the low-energy analysis in
Section~\ref{sec:Low-bound} yields the decomposition
\[
\mathcal{W}^{\mathrm{L}}
= \mathcal{W}_{\mathrm{sing}}^{\mathrm{L}}
+ \mathcal{W}_{\mathrm{good}}^{\mathrm{L}},
\]
where $\mathcal{W}_{\mathrm{good}}^{\mathrm{L}}$ is bounded on
$L^p(\mathbb{R}^2)$ for every $1 \le p \le \infty$, and the singular
part is given by
\begin{equation}\label{eq:exp of W-sing}
	\mathcal{W}_{\mathrm{sing}}^{\mathrm{L}}
	= \frac{1}{\pi i} \int_0^{+\infty} \lambda h_{+}(\lambda)^{-1}
	\chi(\lambda/\lambda_0) \, R_0^{+}(\lambda^2) v P v
	\bigl(R_0^+(\lambda^2) - R_0^-(\lambda^2)\bigr) \, d\lambda,
\end{equation}
with
\[
h_{+}(\lambda)=-\frac{\|V\|_{L^1}}{2\pi}\log\lambda+z_0^{+},
\]
where $z_0^{+}\in \mathbb{C}\setminus \mathbb{R}$ is the constant given
by \eqref{eq:def-of-hpm}. In particular,
\begin{equation}\label{eq:dec of inve of h_+}
	\frac{1}{h_{+}(\lambda)}
	=-\frac{2\pi}{\|V\|_{L^1}}\log^{-1}\lambda
	+\mathcal{O}(\log^{-2}\lambda),
	\quad \lambda\in (0,\lambda_0).
\end{equation}
Since $\mathcal{W}_{\mathrm{good}}^{\mathrm{L}}$ (as well as the
high-energy part of $W_{\pm}$) is bounded on every $L^p(\mathbb{R}^2)$,
the wave operators $W_{\pm}$ are unbounded on $L^p(\mathbb{R}^2)$,
$1\le p\le \infty$, if and only if the singular part
$\mathcal{W}_{\mathrm{sing}}^{\mathrm{L}}$ is.
Proposition~\ref{pro:unbound} therefore follows from the next lemma.

\begin{lemma}\label{pro:unb of Wsing}
Assume that $|V(x)|\lesssim \langle x \rangle^{-3-}$. Then the operator
$\mathcal{W}_{\mathrm{sing}}^{\mathrm{L}}$ is unbounded both on
$L^1(\mathbb{R}^2)$ and on $L^\infty(\mathbb{R}^2)$.
\end{lemma}

\begin{proof}
\medskip\noindent\emph{$L^1$-unboundedness.}
We argue by contradiction. Suppose that
$\mathcal{W}_{\mathrm{sing}}^{\mathrm{L}}$ were bounded on
$L^1(\mathbb{R}^2)$. Then $\mathcal{W}_{\mathrm{sing}}^{\mathrm{L}} f\in
L^1(\mathbb{R}^2)$ for every $f\in L^1(\mathbb{R}^2)$, so that the
integral of $\mathcal{W}_{\mathrm{sing}}^{\mathrm{L}} f$ over
$\mathbb{R}^2$ would converge absolutely for every such $f$. We will
exhibit a function $f\in L^1(\mathbb{R}^2)$ for which this integral
diverges.

We choose the test function
\[
f = \mathcal{F}^{-1}\bigl(\chi(2|\cdot|/\lambda_0)\bigr),
\]
where $\chi$ is the cut-off function introduced in \eqref{cut-off-fun}.
Since $\chi(2|\cdot|/\lambda_0)\in C_c^{\infty}(\mathbb{R}^2)$, we have
$f\in \mathcal{S}(\mathbb{R}^2)\subset L^1(\mathbb{R}^2)$, and
$\hat{f}=\chi(2|\cdot|/\lambda_0)$ is smooth with $\hat{f}\equiv 1$
near the origin. Below we use the spectral-measure identity
\begin{equation}\label{eq:free spec-measure}
	\frac{\lambda}{\pi i}\bigl(R_0^+(\lambda^2)-R_0^-(\lambda^2)\bigr)g(z)
	= \mathcal{F}^{-1}\bigl(\hat{g} \, d\sigma_\lambda\bigr)(z),
\end{equation}
where $d\sigma_\lambda$ denotes the arc-length measure on the circle
$\lambda \mathbb{S}^1$. Pairing \eqref{eq:free spec-measure} with $|V|$
and applying Fubini's theorem gives, for every $\lambda>0$,
\begin{equation}\label{eq:pairing-spec}
	\int_{\mathbb{R}^2}|V|(y)\,
	\bigl(R_0^+(\lambda^2)-R_0^-(\lambda^2)\bigr)g(y)\,dy
	=\frac{\pi i}{(2\pi)^2}\int_{\mathbb{S}^1}
	\hat{g}(\lambda\omega)\,
	\overline{\mathcal{F}(|V|)(\lambda\omega)}\,d\omega .
\end{equation}
Moreover, since $vPv=\frac{1}{\|V\|_{L^1}}|V|\otimes|V|$, the
definition \eqref{eq:exp of W-sing} can be rewritten as
\begin{align*}
	\mathcal{W}_{\mathrm{sing}}^{\mathrm{L}} f(x)
	= \frac{1}{\pi i\|V\|_{L^1}} \int_0^{+\infty}
	\lambda h_{+}(\lambda)^{-1}\chi(\lambda/\lambda_0)
	\bigl[R_0^{+}(\lambda^2)|V|\bigr](x)
	\int_{\mathbb{R}^2}|V|(y)\,
	\bigl(R_0^+-R_0^-\bigr)(\lambda^2)f(y)\,dy \, d\lambda .
\end{align*}
To compute the integral of $\mathcal{W}_{\mathrm{sing}}^{\mathrm{L}} f$,
note first that
\[
	\int_{\mathbb{R}^2}\bigl[R_0^{+}(\lambda^2)|V|\bigr](x)\,dx
	=\mathcal{F}\bigl(R_0^{+}(\lambda^2)|V|\bigr)(0)
	=-\frac{\|V\|_{L^1}}{\lambda^2},
	\qquad \lambda>0,
\]
by the identity
$\mathcal{F}(R_0^{+}(\lambda^2)|V|)(\xi)
=\frac{\mathcal{F}(|V|)(\xi)}{|\xi|^2-\lambda^2-i0}$.
(Strictly speaking, one first pairs
$\mathcal{W}_{\mathrm{sing}}^{\mathrm{L}} f$ with the Gaussian
$e^{-\varepsilon|x|^2}$, which legitimizes the application of Fubini's
theorem; letting $\varepsilon\to 0+$ and using the assumed
integrability of $\mathcal{W}_{\mathrm{sing}}^{\mathrm{L}} f$ yields
the same limit.) Combining the above identity with
\eqref{eq:pairing-spec} (applied to $g=f$) and the co-area formula
$d\eta=d\sigma_\lambda\,d\lambda$, we obtain
\begin{equation}\label{eq-rep-pair}
	\begin{aligned}
		\int_{\mathbb{R}^2} \mathcal{W}_{\mathrm{sing}}^{\mathrm{L}} f(x) \, dx
		&= -\frac{1}{(2\pi)^2}
		\int_{0}^{+\infty}
		\frac{\chi(\lambda/\lambda_0)}{\lambda^2 h_{+}(\lambda)}
		\int_{\lambda\mathbb{S}^1}\hat{f}(\eta)
		\overline{\mathcal{F}(|V|)(\eta)} \, d\sigma_\lambda(\eta)
		\, d\lambda \\
		&= -\frac{1}{(2\pi)^2}
		\int_{\mathbb{R}^2}
		\frac{\overline{\mathcal{F}(|V|)(\eta)} \;
			h_{+}(|\eta|)^{-1} \chi(2|\eta|/\lambda_0)}{|\eta|^2} \, d\eta ,
	\end{aligned}
\end{equation}
where in the second equality we used
$\hat{f}=\chi(2|\cdot|/\lambda_0)$ and the fact that
$\chi(\cdot/\lambda_0)\equiv 1$ on the support of
$\chi(2|\cdot|/\lambda_0)$.

The right-hand side of \eqref{eq-rep-pair} diverges obviously by
\eqref{eq:dec of inve of h_+}. Hence cannot, this
contradiction proves the $L^1$-unboundedness of
$\mathcal{W}_{\mathrm{sing}}^{\mathrm{L}}$.

\medskip\noindent\emph{$L^\infty$-unboundedness.}
We next show that $\mathcal{W}_{\mathrm{sing}}^{\mathrm{L}}$ is
unbounded on $L^\infty(\mathbb{R}^2)$. To this end, we test the
operator against the characteristic functions
$\Psi_{\mu}:=\mathbf{1}_{\mathbf{B}(0,\mu)}$ of the balls
$\mathbf{B}(0,\mu)$, and prove that
\begin{equation}\label{eq:log incr of wsing}
	\left|\mathcal{W}_{\mathrm{sing}}^{\mathrm{L}}\Psi_{\mu}(x)\right|
	=\frac{1}{2}\log\log\mu+\mathcal{O}(1)
	\quad \text{for } |x|\in [\mu-1/2, \mu+1/2] \text{ and } \mu\gg 1.
\end{equation}
Since $\left\|\Psi_{\mu}\right\|_{L^\infty}=1$ for every $\mu$, this
yields the unboundedness of $\mathcal{W}_{\mathrm{sing}}^{\mathrm{L}}$
on $L^\infty(\mathbb{R}^2)$.

Set
\[
\mathcal{V} := |V|,\qquad
\mathcal{V}_{*} := \frac{|V|}{\|V\|_{L^1}}.
\]
Rewriting \eqref{eq:exp of W-sing} by means of the kernel
$vPv=\frac{1}{\|V\|_{L^1}}|V|\otimes|V|$ and using the algebraic
identity
\[
\langle\cdot, \mathcal{V}\rangle \mathcal{V}_{*}
=\langle\cdot, \mathcal{V}\rangle\left(\mathcal{V}_{*}
-\frac{4\Psi_{1/2}}{\pi}\right)
+\langle\cdot, \mathcal{V}\rangle \frac{4\Psi_{1/2}}{\pi},
\]
we decompose
$\mathcal{W}_{\mathrm{sing}}^{\mathrm{L}}
=\mathcal{W}_{\mathrm{sing}}^{\mathrm{L},0}
+\mathcal{W}_{\mathrm{sing}}^{\mathrm{L},1}$, where
\begin{align*}
	\mathcal{W}_{\mathrm{sing}}^{\mathrm{L},0}f(x)
	&= \frac{1}{\pi i}  \int_0^{+\infty} \lambda h_{+}(\lambda)^{-1}
	\chi(\lambda/\lambda_0) \left[R_0^{+}(\lambda^2)
	\left(\mathcal{V}_{*}-\frac{4\Psi_{1/2}}{\pi}\right)\right](x)\\
	&\qquad\qquad \times\int_{\mathbb{R}^2}\mathcal{V}(y)
	\bigl(R_0^+(\lambda^2) - R_0^-(\lambda^2)\bigr)f(y) \, dy\, d\lambda,\\[2mm]
	\mathcal{W}_{\mathrm{sing}}^{\mathrm{L},1}f(x)
	&= \frac{4}{\pi^2 i}  \int_0^{+\infty} \lambda h_{+}(\lambda)^{-1}
	\chi(\lambda/\lambda_0) \bigl[R_0^{+}(\lambda^2)
	\Psi_{1/2}\bigr](x)\\
	&\qquad\qquad \times\int_{\mathbb{R}^2}\mathcal{V}(y)
	\bigl(R_0^+(\lambda^2) - R_0^-(\lambda^2)\bigr)f(y) \, dy \, d\lambda .
\end{align*}
The point of this decomposition is that the density in the first piece
has mean zero,
\[
\int_{\mathbb{R}^2} \Bigl(\mathcal{V}_{*}(y)
-\frac{4\Psi_{1/2}(y)}{\pi}\Bigr)\,dy=0,
\]
which provides the extra cancellation needed for boundedness, whereas
the second piece carries the logarithmic growth. Indeed, comparing
$\mathcal{W}_{\mathrm{sing}}^{\mathrm{L},0}$ with \eqref{eq-rep-W1} and
following the proof of Lemma~\ref{Lemm: bound of N_1}, we find that
$\mathcal{W}_{\mathrm{sing}}^{\mathrm{L},0}$ is bounded on
$L^\infty(\mathbb{R}^2)$; in particular,
\begin{equation}\label{eq:est for Wsing0}
	\left\|\mathcal{W}_{\mathrm{sing}}^{\mathrm{L},0}\Psi_\mu\right\|_{L^\infty}
	\lesssim 1
\end{equation}
uniformly in $\mu$. Consequently, to establish
\eqref{eq:log incr of wsing} it suffices to prove the same asymptotics
with $\mathcal{W}_{\mathrm{sing}}^{\mathrm{L}}$ replaced by
$\mathcal{W}_{\mathrm{sing}}^{\mathrm{L},1}$.

For this purpose, we derive a more explicit representation of
$\mathcal{W}_{\mathrm{sing}}^{\mathrm{L},1}\Psi_\mu$. Recall that, with
our normalization
$\hat{g}(\xi)=\int_{\mathbb{R}^2}e^{-ix\cdot\xi}g(x)\,dx$ of the
Fourier transform,
\[
	\mathcal{F}(\Psi_\mu)(\xi)
	=2\pi\mu\,\frac{J_1(\mu|\xi|)}{|\xi|},
\]
where $J_1$ denotes the Bessel function of the first kind (see, e.g.,
\cite[p.~578]{graf}). Combining \eqref{eq:pairing-spec} (applied to
$g=\Psi_\mu$) with this formula, we obtain
\begin{equation}\label{eq:rp  of Wsing-1-1-1}
	\begin{aligned}
		\mathcal{W}_{\mathrm{sing}}^{\mathrm{L},1}\Psi_\mu(x)
		= \frac{2\mu}{\pi^2}  \int_0^{+\infty}  h_{+}(\lambda)^{-1}
		J_1(\mu\lambda)\chi(\lambda/\lambda_0)
		\bigl[R_0^{+}(\lambda^2)\Psi_{1/2}\bigr](x)  \,
		\int_{\mathbb{S}^1}
		\overline{\mathcal{F}(|V|)(\lambda\omega)}\,d\omega \, d\lambda.
	\end{aligned}
\end{equation}
Recalling also the representation
\[
R_0^{+}(\lambda^2)(x,z)=\frac{i}{4}H_0^+(\lambda|x-z|)
\]
of the free resolvent kernel in terms of the Hankel function $H_0^+$
of the first kind, we may rewrite \eqref{eq:rp  of Wsing-1-1-1} as
\begin{equation}\label{eq:rp  of Wsing-1-1}
	\mathcal{W}_{\mathrm{sing}}^{\mathrm{L},1}\Psi_\mu(x)
	=\frac{i\mu}{2\pi^2}   \int_{\mathbb{R}^2}K_\mu(|z|)
	\Psi_{1/2}(x-z)\,dz,
\end{equation}
where the radial kernel $K_\mu(|z|)$ is given by
\begin{equation*}
	K_\mu(|z|)
	=
	\int_0^{+\infty}  h_{+}(\lambda)^{-1} J_1(\mu\lambda)
	H_0^+(\lambda|z|)\chi(\lambda/\lambda_0)\,
	\int_{\mathbb{S}^1}\overline{\mathcal{F}(|V|)(\lambda\omega)}
	\,d\omega \,d\lambda .
\end{equation*}
In this way, the asymptotics of
$\mathcal{W}_{\mathrm{sing}}^{\mathrm{L},1}\Psi_\mu$ are reduced to
those of the one-dimensional oscillatory integral $K_\mu$.

To analyze $K_\mu$, we first expand the angular integral in the
potential. The assumption $|V(x)|\lesssim \langle x \rangle^{-3-}$
implies that $\mathcal{F}(|V|)$ is Lipschitz continuous, and hence
\begin{equation}\label{eq:est of remaind gamma}
\int_{\mathbb{S}^1}\overline{\mathcal{F}(|V|)(\lambda\omega)}\,d\omega
	=2\pi\|V\|_{L^1}+\gamma(\lambda), \quad
	\gamma(\lambda)=\mathcal{O}(\lambda) \quad \text{for } \lambda
	\in (0,\lambda_0).
\end{equation}
Accordingly, we decompose
$K_\mu(|z|)=K_\mu^{0}(|z|)+K_\mu^{1}(|z|)$ with
	\begin{align} \label{eq:def of Kmu}
		K_\mu^{0}(|z|)
		&= 2\pi\|V\|_{L^1}
		\int_0^{+\infty} h_{+}(\lambda)^{-1} J_1(\mu\lambda)
		H_0^+(\lambda|z|)\chi(\lambda/\lambda_0)\,d\lambda,
		\\[2mm] \nonumber
		K_\mu^{1}(|z|)
		&= \int_0^{+\infty} h_{+}(\lambda)^{-1} J_1(\mu\lambda)
		H_0^+(\lambda|z|)\gamma(\lambda)
		\chi(\lambda/\lambda_0)\,d\lambda.
	\end{align}
The remainder term $K_\mu^{1}$ is negligible: in light of
\eqref{eq:est of remaind gamma} and the uniform bounds
\[
H_{0}^{+}(t)=\mathcal{O}(t^{-1/2}),\qquad
J_{1}(t)=\mathcal{O}(t^{-1/2}), \qquad t>0
\]
(for small $t$ one uses $H_{0}^{+}(t)=\mathcal{O}(\log t)$ and
$J_1(t)=\mathcal{O}(t)$), one has
\[
	\bigl|K_\mu^{1}(|z|)\bigr|
	\lesssim \frac{1}{\sqrt{\mu|z|}}  \int_0^{\lambda_0}
	\frac{|\gamma(\lambda)|}{\lambda\,|h_{+}(\lambda)|}\,d\lambda
	\lesssim\frac{1}{\mu}\int_0^{\lambda_0}
	\frac{d\lambda}{|h_{+}(\lambda)|}
	=\mathcal{O}\left(\frac{1}{\mu}\right)
\]
for $|z|-\mu\in [-1,1]$ and $\mu\gg 1$, since
$|h_{+}(\lambda)|^{-1}=\mathcal{O}(|\log\lambda|^{-1})$ is integrable
near the origin. For the main term we appeal to
Lemma~\ref{lem:asy of K tau} below: note that $|x|-\mu\in[-1/2, 1/2]$
and $|x-z|\le 1/2$ imply $|z|-\mu\in[-1,1]$, so inserting
\eqref{eq:est for K0} and the above estimate into
\eqref{eq:rp  of Wsing-1-1} yields, for such $x$ and $\mu\gg 1$,
\begin{equation}
	\begin{aligned}
		\mathcal{W}_{\mathrm{sing}}^{\mathrm{L},1}\Psi_\mu(x)
		&=\frac{i\mu}{2\pi^2}\left(4\pi i\,\frac{\log\log\mu}{\mu}
		+\mathcal{O}\left(\frac{1}{\mu}\right)\right)
		\int_{\mathbb{R}^2}\Psi_{1/2}(x-z)\,dz \\[1mm]
		&= -\frac{2}{\pi}\log\log\mu\int_{\mathbb{R}^2}\Psi_{1/2}(z)\,dz
		+\mathcal{O}(1)
		= -\frac{1}{2}\log\log\mu+\mathcal{O}(1),
	\end{aligned}
\end{equation}
since
$\int_{\mathbb{R}^2}\Psi_{1/2}(z)\,dz=|\mathbf{B}(0,1/2)|=\frac{\pi}{4}$.
In particular,
$\left|\mathcal{W}_{\mathrm{sing}}^{\mathrm{L},1}\Psi_\mu(x)\right|
=\frac{1}{2}\log\log\mu+\mathcal{O}(1)$ on the annulus
$|x|\in[\mu-1/2,\mu+1/2]$, which together with
\eqref{eq:est for Wsing0} implies \eqref{eq:log incr of wsing} and
completes the proof.
\end{proof}

It remains to establish the asymptotic behavior of $K_\mu^{0}(|z|)$ as
$\mu\to\infty$, which was used in the proof above. We begin with a
simple oscillatory integral estimate that provides the main
logarithmic term.

\begin{lemma}\label{lem:est of K-m}
	Uniformly in $\delta\in [-1,1]$, one has
	\[
	\int_{1/\mu}^{\lambda_0}\frac{e^{i\lambda\delta}}
	{\lambda h_{+}(\lambda)}\,d\lambda
	=\frac{2\pi}{\|V\|_{L^1}} \log\log\mu+\mathcal{O}(1),
	\qquad \mu\gg 1,
	\]
	where $\lambda_0\in (0,1/2)$ is the fixed cut-off parameter.
\end{lemma}

\begin{proof}
	Recall that
	$h_{+}(\lambda)=-\frac{\|V\|_{L^1}}{2\pi}\log\lambda+z_0^{+}$
	with $z_0^{+}\in \mathbb{C}\setminus \mathbb{R}$. Using the
	expansion \eqref{eq:dec of inve of h_+}, we split the integral
	accordingly,
	\[
	\int_{1/\mu}^{\lambda_0}\frac{e^{i\lambda\delta}}
	{\lambda h_{+}(\lambda)}\,d\lambda
	=-\frac{2\pi}{\|V\|_{L^1}}\int_{1/\mu}^{\lambda_0}
	\frac{e^{i\lambda\delta}}{\lambda \log\lambda}\,d\lambda
	+ \mathcal{O}\left(\int_{1/\mu}^{\lambda_0}
	\frac{1}{\lambda \log^2\lambda}\,d\lambda\right)
	=:I_1+I_2.
	\]
	For $I_1$, since $e^{i\lambda\delta}=1+\mathcal{O}(\lambda)$
	uniformly for $\delta\in [-1,1]$, and since
	$|\log\lambda|^{-1}$ is bounded on $[1/\mu,\lambda_0]$, we obtain
	\[
	I_1=-\frac{2\pi}{\|V\|_{L^1}}\int_{1/\mu}^{\lambda_0}
	\frac{1}{\lambda \log\lambda}\,d\lambda
	+ \mathcal{O}(1).
	\]
	The antiderivative of $(\lambda\log\lambda)^{-1}$ is
	$\log(-\log\lambda)$ on $(0,\lambda_0)$, so a direct computation
	gives
	\[
	\int_{1/\mu}^{\lambda_0}\frac{1}{\lambda\log\lambda}
	\,d\lambda
	=\log(-\log\lambda_0)-\log(\log\mu)
	=-\log\log\mu+\mathcal{O}(1),
	\]
	and hence
	\[
	I_1=\frac{2\pi}{\|V\|_{L^1}}\log\log\mu+\mathcal{O}(1),
	\qquad \mu\gg 1.
	\]
	For $I_2$, the integrand $(\lambda\log^2\lambda)^{-1}$ has the
	antiderivative $-\log^{-1}\lambda$, bounded on $(0,\lambda_0)$;
	hence $I_2=\mathcal{O}(1)$ for $\mu\gg 1$.
	Combining the estimates for $I_1$ and $I_2$, we obtain the desired
	conclusion.
\end{proof}

We now turn to the asymptotics of $K_\mu^{0}(|z|)$ itself.

\begin{lemma}\label{lem:asy of K tau}
	Let $K_\mu^{0}(|z|)$ be given by \eqref{eq:def of Kmu}. Under the
	assumption that $|V(x)|\lesssim \langle x \rangle^{-3-}$, we have
	\begin{equation}\label{eq:est for K0}
		K_{\mu}^{0}(|z|)=4\pi i\,\frac{\log\log\mu}{\mu}
		+\mathcal{O}\left(\frac{1}{\mu}\right),
		\quad |z|-\mu\in [-1,1], ~\mu\gg 1.
	\end{equation}
\end{lemma}

\begin{proof}
	We split the $\lambda$-integral defining $K_\mu^{0}$ at the
	transition scale $\lambda=\mu^{-1}$, where $\mu\lambda$ passes
	from the small-argument regime of $J_1$ into the oscillatory
	regime:
	\begin{equation}\label{eq:dec of K0}
		\begin{aligned}
			K_\mu^{0}(|z|)
			&=2\pi\|V\|_{L^1} \left(\int_{0}^{\frac{1}{\mu}}
			+ \int_{1/\mu}^{\lambda_0}\right) h_{+}(\lambda)^{-1}
			J_1(\mu\lambda)
			H_0^+(\lambda|z|)\chi(\lambda/\lambda_0)\,d\lambda
			\\[1mm]
			&=:K_{\mu}^{0,r}(|z|)+K_{\mu}^{0,m}(|z|).
		\end{aligned}
	\end{equation}

	\medskip\noindent\emph{Part I: the good piece
	$K_\mu^{0,r}(|z|)$.}
	On this interval $\mu\lambda\le 1$; since $|z|-\mu\in[-1,1]$, this
	together with $\lambda\in (0,\lambda_0)$ implies
	$\lambda|z| \le \frac{3}{2}$ for $\mu$ sufficiently large, so that
	both special functions are in their small-argument regimes.
	Indeed, by the asymptotics of the Hankel function at the origin
	(see \eqref{eq:asy-Hanel-0}), one has
	\[
	H_0^+(\lambda|z|)=\frac{2i}{\pi}\log(\lambda|z|)+\mathcal{O}(1),
	\]
	while, for $\mu\lambda\le 1$, it follows from \cite{Abra} that
	\[
	J_1(\mu\lambda)=\frac{\mu\lambda}{2}+\mathcal{O}(\mu^3\lambda^3).
	\]
	Substituting these expansions into the integral, we obtain
	\begin{equation}\label{eq:dec of Km}
		\begin{aligned}
			K_{\mu}^{0,r}(|z|)
			&=2i\mu\|V\|_{L^1} \int_{0}^{\frac{1}{\mu}}
			\frac{\lambda\big(\log\lambda+\log|z|\big)}{h_{+}(\lambda)}
			\chi(\lambda/\lambda_0)\,d\lambda
			+ \mu\int_{0}^{\frac{1}{\mu}}
			\frac{\lambda\mathcal{O}(1)}{h_{+}(\lambda)}\,d\lambda \\[1mm]
			&\qquad +\int_{0}^{\frac{1}{\mu}}
			\frac{\mathcal{O}(\mu^3\lambda^3)\left[\log(\lambda|z|)
			+\mathcal{O}(1)\right]}
			{h_{+}(\lambda)}\,d\lambda.
		\end{aligned}
	\end{equation}

	We first treat the leading term. Note that
	$\chi(\lambda/\lambda_0)\equiv 1$ on $(0,1/\mu)$ for $\mu$ large.
	Using \eqref{eq:dec of inve of h_+}, we split
	\begin{align*}
		&\int_{0}^{\frac{1}{\mu}}
		\frac{\lambda\big(\log\lambda+\log|z|\big)}{h_{+}(\lambda)}
		\,d\lambda
		=-\frac{2\pi}{\|V\|_{L^1}}\left(\int_{0}^{\frac{1}{\mu}}\lambda\,d\lambda
		+\log|z|\int_{0}^{\frac{1}{\mu}}
		\frac{\lambda}{\log\lambda}\,d\lambda\right)\\
		&\qquad+ \mathcal{O}\left(\int_{0}^{\frac{1}{\mu}}
		\frac{\lambda}{|\log\lambda|}\,d\lambda\right)
		+|\log|z||\,\mathcal{O}\left(\int_{0}^{\frac{1}{\mu}}
		\frac{\lambda}{\log^2\lambda}\,d\lambda\right).
	\end{align*}
	Since
	\[
	\int_{0}^{\frac{1}{\mu}}\frac{\lambda}{\log\lambda}\,d\lambda
	=-\frac{1}{2\mu^2\log\mu}
	+\mathcal{O}\left(\frac{1}{\mu^2\log^2\mu}\right),
	\quad \mu\gg 1,
	\]
	and since $|z|-\mu\in[-1,1]$ with $\mu\gg1$ implies
	$\log|z|=\log\mu+\mathcal{O}(1)$, the first two terms on the
	right-hand side combine into
	\[
	-\frac{2\pi}{\|V\|_{L^1}}\left(\frac{1}{2\mu^2}
	-\frac{\log|z|}{2\mu^2\log\mu}\right)
	+\mathcal{O}\left(\frac{1}{\mu^2\log\mu}\right)
	=\mathcal{O}\left(\frac{1}{\mu^2\log\mu}\right).
	\]
	The remaining two error terms are of the same order: a simple
	computation yields
	\[
	\mathcal{O}\left(\int_{0}^{\frac{1}{\mu}}
	\frac{\lambda}{|\log\lambda|}\,d\lambda\right)
	+|\log|z||\,\mathcal{O}\left(\int_{0}^{\frac{1}{\mu}}
	\frac{\lambda}{\log^2\lambda}\,d\lambda\right)
	=\mathcal{O}\left(\frac{1}{\mu^2\log\mu}\right).
	\]
	These two estimates show that
	\begin{equation}\label{eq:est for first term of km}
		2i\mu\|V\|_{L^1}\int_{0}^{\frac{1}{\mu}}
		\frac{\lambda\big(\log\lambda+\log|z|\big)}{h_{+}(\lambda)}
		\chi(\lambda/\lambda_0)\,d\lambda
		=\mathcal{O}\left(\frac{1}{\mu\log\mu}\right).
	\end{equation}
	Similarly, since
	$|h_{+}(\lambda)|^{-1}=\mathcal{O}(|\log\lambda|^{-1})$, the
	remaining two terms in \eqref{eq:dec of Km} obey
	\[
	\mu\int_{0}^{\frac{1}{\mu}}
	\frac{\lambda\mathcal{O}(1)}{h_{+}(\lambda)}\,d\lambda
	=\mathcal{O}\left(\frac{\mu}{\log\mu}\int_0^{\frac{1}{\mu}}
	\lambda\,d\lambda\right)
	=\mathcal{O}\left(\frac{1}{\mu\log\mu}\right)
	\]
	and, noting that $|\log(\lambda|z|)|\lesssim |\log\lambda|$ on
	$(0,1/\mu)$,
	\[
	\int_{0}^{\frac{1}{\mu}}
	\frac{\mathcal{O}(\mu^3\lambda^3)\left[\log(\lambda|z|)
	+\mathcal{O}(1)\right]}{h_{+}(\lambda)}\,d\lambda
	=\mathcal{O}\left(\mu^3\int_{0}^{\frac{1}{\mu}}
	\lambda^3\,d\lambda\right)
	=\mathcal{O}\left(\frac{1}{\mu}\right).
	\]
	This, together with \eqref{eq:dec of Km} and
	\eqref{eq:est for first term of km}, implies that
	\begin{equation}\label{eq:est of Km}
		K_{\mu}^{0,r}(|z|)
		=\mathcal{O}\left(\frac{1}{\mu\log\mu}\right)
		\quad	\text{for $|z|-\mu\in[-1,1]$ and $\mu\gg 1$}.
	\end{equation}

	\medskip\noindent\emph{Part II: the leading piece
	$K_\mu^{0,m}(|z|)$.}
	Recall from \eqref{eq:dec of K0} that
	\[
	K_\mu^{0,m}(|z|)=2\pi\|V\|_{L^1} \int_{1/\mu}^{\lambda_0}
	h_{+}(\lambda)^{-1} J_1(\mu\lambda)
	H_0^+(\lambda|z|)\chi(\lambda/\lambda_0)\,d\lambda.
	\]
	Here both $\mu\lambda\ge 1$ and $\lambda|z|\gtrsim 1$, so we may
	use the large-argument asymptotic expansions
	\begin{align}
		H_{0}^{+}(t)&=\sqrt{\frac{2}{\pi t}}\,e^{i(t-\pi/4)}
		+\mathcal{O}(t^{-3/2}),
		\label{eq:Hankel-asymp}\\
		J_{1}(t)&=\sqrt{\frac{2}{\pi t}}\cos\Bigl(t-\frac{3\pi}{4}\Bigr)
		+\mathcal{O}(t^{-3/2}).\label{eq:Bessel-asymp}
	\end{align}
	Substituting \eqref{eq:Hankel-asymp}--\eqref{eq:Bessel-asymp} into
	the integral and using
	\[
	e^{-i\pi/4}\cos\Bigl(\lambda\mu-\frac{3\pi}{4}\Bigr)
	=\frac{1}{2}\left(ie^{-i\lambda\mu}-e^{i\lambda\mu}\right),
	\]
	we may decompose $K_{\mu}^{0,m}(|z|)$, for $|z|-\mu\in [-1,1]$, as
	\begin{align*}
		K_{\mu}^{0,m}(|z|)
		={}&
		\frac{2i\|V\|_{L^1}}{\sqrt{\mu|z|}}\int_{1/\mu}^{\lambda_0}
		e^{i\lambda(|z|-\mu)} \frac{\chi(\lambda/\lambda_0)}
		{\lambda h_{+}(\lambda)}\,d\lambda
		-\frac{2\|V\|_{L^1}}{\sqrt{\mu|z|}} \int_{1/\mu}^{\lambda_0}
		e^{i\lambda(|z|+\mu)}\frac{\chi(\lambda/\lambda_0)}
		{\lambda h_{+}(\lambda)}\,d\lambda\\
		&+\mu^{-2}\mathcal{O}\left(\int_{1/\mu}^{\lambda_0}
		\frac{d\lambda}{\lambda^2|h_{+}(\lambda)|}\right),
	\end{align*}
	where the remainder estimate uses
	$(\mu\lambda)^{-1/2}(\lambda|z|)^{-3/2}
	\lesssim\mu^{-2}\lambda^{-2}$ and the analogous bounds for the
	other mixed remainder terms (recall that $\lambda\ge 1/\mu$ and
	$|z|\sim\mu$). The first term carries the slowly varying phase
	$e^{i\lambda(|z|-\mu)}$ (recall that $|z|-\mu\in[-1,1]$) and
	therefore produces the main contribution: since inserting the
	cut-off $\chi(\lambda/\lambda_0)$ into the integral of
	Lemma~\ref{lem:est of K-m} changes it only by $\mathcal{O}(1)$,
	that lemma yields
	\begin{align*}
		\frac{2i\|V\|_{L^1}}{\sqrt{\mu |z|}}\int_{1/\mu}^{\lambda_0}
		e^{i\lambda(|z|-\mu)} \frac{\chi(\lambda/\lambda_0)}
		{\lambda h_{+}(\lambda)}\,d\lambda
		&=\frac{4\pi i}{\sqrt{\mu |z|}}\log\log\mu
		+\mathcal{O}\left(\frac{1}{\sqrt{\mu |z|}}\right)\\
		&=4\pi i\,\frac{\log\log\mu}{\mu}
		+\mathcal{O}\left(\frac{1}{\mu}\right)
	\end{align*}
	for $|z|-\mu\in [-1,1]$ and $\mu\gg 1$, where the last equality
	uses $\sqrt{\mu|z|}=\mu\bigl(1+\mathcal{O}(\mu^{-1})\bigr)$. The
	second term carries the rapidly oscillating phase
	$e^{i\lambda(|z|+\mu)}$ with $|z|+\mu\ge 2\mu-1$, and integration
	by parts gains one power of $\mu^{-1}$: one obtains
	\begin{align*}
		\frac{2\|V\|_{L^1}}{\sqrt{\mu|z|}}\int_{1/\mu}^{\lambda_0}
		e^{i\lambda(|z|+\mu)}\frac{\chi(\lambda/\lambda_0)}
		{\lambda h_{+}(\lambda)}\,d\lambda
		&=\mathcal{O}\left(\frac{1}{\mu\log\mu}\right)
		+\mathcal{O}\left(\frac{1}{\mu^2}\int_{1/\mu}^{\lambda_0}
		\frac{d\lambda}{\lambda^2|h_{+}(\lambda)|}\right)
		=\mathcal{O}\left(\frac{1}{\mu\log\mu}\right),
	\end{align*}
	where we used
	\[
	\int_{1/\mu}^{\lambda_0}
	\frac{1}{\lambda^2 |h_{+}(\lambda)|}\,d\lambda
	\lesssim\frac{1}{\log\mu}\int_{1/\mu}^{\lambda_0}
	\frac{d\lambda}{\lambda^2}
	\lesssim\frac{\mu}{\log\mu}.
	\]
	The same bound shows that the remainder term is also
	$\mathcal{O}(\frac{1}{\mu\log\mu})$. Combining these three
	estimates, we conclude that
	\begin{equation}\label{eq:est for K0m}
		K_{\mu}^{0,m}(|z|)=4\pi i\,\frac{\log\log\mu}{\mu}
		+\mathcal{O}\left(\frac{1}{\mu}\right),
		\quad |z|-\mu\in [-1,1], ~\mu\gg 1.
	\end{equation}

	Finally, the estimates \eqref{eq:est of Km} and
	\eqref{eq:est for K0m} together yield \eqref{eq:est for K0}, which
	completes the proof.
\end{proof}

\subsection{Proof of \textup{(i)} in Proposition~\ref{pro:unbound-1}}

This subsection is devoted to the proof of statement~\textup{(i)} in
Proposition~\ref{pro:unbound-1}, namely the $L^1$-unboundedness of
$\mathcal{W}^{\mathrm{L}}$ in the presence of a second-kind threshold singularity or a zero eigenvalue.
According to whether $H$ has p-wave or s-wave resonances, the proof
splits into the following three subcases:
\begin{itemize}
	\item \emph{Subcase 1:} $H$ has neither p-wave nor s-wave resonances.
    \smallskip
	\item \emph{Subcase 2:} $H$ has both p-wave and s-wave resonances.
      \smallskip
	\item \emph{Subcase 3:} $H$ has p-wave resonances but no s-wave resonances.
\end{itemize}

Recall from Section~\ref{low-bound II} the decomposition
\begin{equation}\label{eq:dec-of-WL-second}
	\mathcal{W}^{\mathrm{L}}=\sum_{i,j=0}^1\mathcal{W}_{ij}^{\mathrm{L}}.
\end{equation}
The results of Subsection~\ref{sec:contri-of-co} show that the components
$\mathcal{W}_{10}^{\mathrm{L}}$ and $\mathcal{W}_{01}^{\mathrm{L}}$ are
both bounded on $L^1(\mathbb{R}^2)$, while
Subsections~\ref{sec:unb-for-regular} and~\ref{sec:Low-bound} show that
$\mathcal{W}_{00}^{\mathrm{L}}$ is bounded on $L^1(\mathbb{R}^2)$ if and
only if $H$ has an s-wave resonance. The unboundedness in all three
subcases is therefore governed by the remaining component
$\mathcal{W}_{11}^{\mathrm{L}}$, to which we now turn.

We first study the $L^1$-unboundedness of $\mathcal{W}_{11}^{\mathrm{L}}$
when $H$ has p-wave resonances. In this case,
Lemma~\ref{lem:one to one lem} ensures that $S_2-S_3$ does not vanish.
By the definition of $\mathcal{W}_{11}^{\mathrm{L}}$ in
Section~\ref{low-bound II}, it is the contribution of the expansion
\begin{equation}\label{eq:exp-of-L11}
	\mathcal{L}_{11}^{+}(\lambda)=
	\dfrac{S_3D_3S_3}{\lambda^2}+\dfrac{N_1}{\lambda^2g_1^{+}(\lambda)}
	+\tilde{\mathcal{O}}_2\!\left(\lambda^{-2}g_1^{+}(\lambda)^{-2}\right),
\end{equation}
where
\begin{equation}\label{eq:def of N1-1}
	\begin{aligned}
		N_1&:=(S_2-S_3)D_2(S_2-S_3)-(S_2-S_3)D_2vG_2vS_3D_3S_3\\
		&\qquad-S_3D_3S_3vG_2vD_2(S_2-S_3)\\
		&\qquad+S_3D_3S_3vG_2v(S_2-S_3)D_2(S_2-S_3)vG_2vS_3D_3S_3.
	\end{aligned}
\end{equation}
Of the three terms in \eqref{eq:exp-of-L11}, the leading term
$\tfrac{S_3D_3S_3}{\lambda^2}$ and the error term
$\tilde{\mathcal{O}}_2\!\left(\lambda^{-2}g_1^{\pm}(\lambda)^{-2}\right)$
contribute $L^1$-bounded operators by virtue of
Lemmas~\ref{lem:bound of N_3-1} and \ref{lem:bound of N_4}, respectively,
so that only the second term $\tfrac{N_1}{\lambda^2g_1^{+}(\lambda)}$ can
be responsible for unboundedness.

Recall from \eqref{eq:def-of-g_1} that
\[
g_1^{+}(\lambda):=\log\lambda+\beta_{+},\qquad \beta_{+}\in\mathbb{C}\setminus\mathbb{R}.
\]
A term-by-term analysis of \eqref{eq:def of N1-1} shows that, apart from
the leading term $\frac{1}{\lambda g_1^{+}(\lambda)}(S_2-S_3)D_2(S_2-S_3)$,
all remaining terms in the expansion \eqref{eq:def of N1-1} are bounded on
$L^1(\mathbb{R}^2)$. The $L^1$-unboundedness of $\mathcal{W}_{11}^{\mathrm{L}}$
is therefore reduced to the following statement.

\begin{lemma}\label{lemm:unb-of-wpsing}
	If $H$ has p-wave resonances, then the operator
	\[
	\mathcal{W}_{\mathrm{p}, \mathrm{sing}}^{\mathrm{L}}
	=\frac{1}{\pi i}\int_0^{+\infty} \lambda^{-1} g_1^+(\lambda)^{-1}
	\chi(\lambda/\lambda_0)\,
	R_0^{+}(\lambda^2)v(S_2-S_3)D_2(S_2-S_3)v
	\bigl(R_0^+(\lambda^{2})-R_0^-(\lambda^{2})\bigr)\, d\lambda
	\]
	is not bounded on $L^1(\mathbb{R}^2)$.
\end{lemma}

\begin{proof}
	The proof proceeds in two steps: we first isolate the singular part of the
	kernel, and then exhibit the $\log\log$ growth of its column norms
	$y\mapsto\|\mathcal{I}_{\mathrm{p},\mathrm{sing}}(\cdot,y)\|_{L^1(\mathbb{R}^2)}$
	along annuli, which forces
	$\|\mathcal{W}_{\mathrm{p},\mathrm{sing}}^{\mathrm{L}}\|_{L^1\to L^1}=\infty$.
	
	\medskip\noindent\emph{Step 1: Reduction to the singular kernel.}
	The kernel of
	$R_0^+(\lambda^{2})v(S_2-S_3)D_2(S_2-S_3)v\bigl(R_0^+(\lambda^{2})-R_0^-(\lambda^{2})\bigr)$
	can be written as
	\begin{equation}\label{eq:exp-of-ker-of-spec}
		\Bigl\langle D_2(S_2-S_3)v\bigl(R_0^{+}(\lambda^2)(\cdot, y)-R_0^{-}(\lambda^2)(\cdot, y)\bigr),\,
		(S_2-S_3)vR_0^{-}(\lambda^2)(\cdot, x)
		\Bigr\rangle.
	\end{equation}
	Since $(S_2-S_3)v$ annihilates the constant function, the argument of
	Lemma~\ref{lem:est-S0-J0} gives
	\begin{equation}\label{eq:exp-for-Q2J-1}
		\begin{aligned}
			(S_2-S_3)v&\bigl(R_0^+(\lambda^2)(\cdot,y)-R_0^-(\lambda^2)(\cdot,y)\bigr)
			=\tfrac{i}{2} \lambda J_0'(\lambda\langle y\rangle)
			(S_2-S_3)v\frac{\langle \cdot,y\rangle}{|y|} \\
			&\qquad\qquad\qquad+e^{i\lambda\langle y\rangle}(S_2-S_3)\eta_2^+(\lambda,\cdot,y)
			+e^{-i\lambda\langle y\rangle}(S_2-S_3)\eta_2^-(\lambda,\cdot,y).
		\end{aligned}
	\end{equation}
	In view of Remark~\ref{rem:rela-of-omegas}, one also has
	\begin{equation}\label{eq:exp-for-Q2H-1}
		\begin{aligned}
			(S_2-S_3)vR_0^-(\lambda^2)(\cdot,x)
			&= e^{-i\lambda\langle x\rangle}(S_2-S_3)\omega_2^{-}(\lambda,\cdot,x)
			- (S_2-S_3)\psi_2(\cdot,x)\chi(\lambda\langle x\rangle)\\
			&\quad -\frac{i}{4}\lambda (H_0^-)'(\lambda\langle x\rangle)\,
			(S_2-S_3)v \frac{\langle \cdot,x\rangle}{|x|}.
		\end{aligned}
	\end{equation}
	Substituting \eqref{eq:exp-for-Q2J-1} and \eqref{eq:exp-for-Q2H-1} into
	\eqref{eq:exp-of-ker-of-spec} and using the improved estimates for
	$\eta_2^\pm$, $\omega_2^-$ and $\psi_2$, the same arguments as in
	Subsection~\ref{sec:contri-of-co} show that all terms except
	\begin{equation}\label{eq:main-term}
		-\dfrac{1}{8} \lambda^2 (H_0^+)'(\lambda\langle x\rangle)  J_0'(\lambda\langle y\rangle)
		\Bigl\langle D_2(S_2-S_3)v \frac{\langle \cdot,y\rangle}{|y|},\,
		(S_2-S_3)v \frac{\langle \cdot,x \rangle}{|x|} \Bigr\rangle
	\end{equation}
	contribute operators that are bounded on $L^1(\mathbb{R}^2)$.
	
	For the reader's convenience, we illustrate this on a typical example: the
	cross-term of
	$e^{i\lambda\langle y\rangle}(S_2-S_3)\eta_2^+(\lambda,\cdot,y)$ in
	\eqref{eq:exp-for-Q2J-1} and
	$-\frac{i}{4}\lambda (H_0^-)'(\lambda\langle x\rangle)\,
	(S_2-S_3)v \frac{\langle \cdot,x\rangle}{|x|}$ in
	\eqref{eq:exp-for-Q2H-1}. The corresponding operator has integral kernel
	(up to a constant)
	\begin{align*}
		\mathcal{I}_{\mathrm{good}}(x,y)
		&=\frac{1}{\pi i}\int_0^{+\infty} \lambda^{-1}
		g_1^{+}(\lambda)^{-1} \chi(\lambda/\lambda_0)\,
		e^{i\lambda\langle y\rangle}\lambda \, (H_0^+)'(\lambda\langle x\rangle)\\
		&\qquad\qquad\times\Bigl\langle D_2(S_2-S_3) \eta_2^+(\lambda,\cdot,y),\,
		(S_2-S_3)v\frac{\langle \cdot,x\rangle}{|x|}\Bigr\rangle\, d\lambda\\
		&=:\frac{1}{\pi i}\int_0^{+\infty} \lambda^{-1}
		g_1^{+}(\lambda)^{-1} \chi(\lambda/\lambda_0)\,
		\mathcal{E}_g(\lambda, x, y)\, d\lambda .
	\end{align*}
	Recall from Subsection~\ref{Sec:some Technical lemmas} that
	$\chi(z)(H_0^-)'(z)=\mathcal{O}(z^{-1})$ and
	$\tilde{\chi}(z)(H_0^-)'(z)=e^{-iz}\mathcal{O}_{\infty}(z^{-1/2})$.
	Together with the estimates \eqref{eq:est-eta-w-1},
	\eqref{eq:est-eta-w-2} and the Cauchy--Schwarz inequality,
	this allows us to decompose $\mathcal{E}_g(\lambda, x, y)$ as
	\[
	e^{i\lambda\langle y\rangle}\chi(\lambda \langle x\rangle)
	\mathcal{K}_{g,1}(\lambda, x, y)
	+e^{-i\lambda(\langle x\rangle-\langle y\rangle)}
	\mathcal{K}_{g,2}(\lambda, x, y),
	\]
	where $\mathcal{K}_{g,1}$ and $\mathcal{K}_{g,2}$ satisfy
	\[
	|\mathcal{K}_{g,1}|\lesssim \lambda^{3/2} \langle x\rangle^{-1},
	\qquad
	\bigl|\partial_{\lambda}^{k}\mathcal{K}_{g,2}\bigr|
	\lesssim \lambda^{2} \langle x\rangle^{-1/2}\langle y\rangle^{-1/2},\quad k=0,1,2.
	\]
	As in the proofs of Lemmas~\ref{lem:bound of N_3-2-1} and
	\ref{lem:bound of N_3-3}, this yields
	\[
	\bigl|\mathcal{I}_{\mathrm{good}}(x,y)\bigr|
	\lesssim \langle x\rangle^{-5/2}
	+\langle x\rangle^{-1/2}\langle y\rangle^{-1/2}
	\langle\langle x\rangle-\langle y\rangle\rangle^{-2},
	\]
	so that
	\[
	\sup_{y\in\mathbb{R}^2}\int_{\mathbb{R}^2}
	\bigl|\mathcal{I}_{\mathrm{good}}(x,y)\bigr|\,dx<\infty,
	\]
	and the corresponding operator is bounded on $L^1(\mathbb{R}^2)$ by
	Schur's lemma. All other cross-terms are treated in exactly the same way.
	
	\medskip\noindent\emph{Step 2: $\log\log$ growth of the singular kernel.}
	It remains to show that the contribution of \eqref{eq:main-term} is
	unbounded on $L^1(\mathbb{R}^2)$. Set
	\begin{equation*}
		T(x,y):=\Bigl\langle D_2(S_2-S_3)v \frac{\langle \cdot,y\rangle}{|y|},\,
		(S_2-S_3)v \frac{\langle \cdot,x \rangle}{|x|} \Bigr\rangle ,
	\end{equation*}
	and define the scalar kernel, with $\mu=\langle y\rangle$,
	$\nu=\langle x\rangle$,
	\begin{equation}\label{eq:def-ksing}
		k(\nu,\mu):=\frac{i}{8\pi}\int_0^{+\infty}\lambda\, g_1^+(\lambda)^{-1}
		\chi(\lambda/\lambda_0)\,(H_0^+)'(\lambda\nu)\,J_0'(\lambda\mu)\, d\lambda ,
	\end{equation}
	so that the kernel contributed by \eqref{eq:main-term} factors as
	\begin{equation*}
		\mathcal{I}_{\mathrm{p},\, \mathrm{sing}}(x,y)
		=k\bigl( \langle x\rangle, \langle y\rangle\bigr)\,T(x,y).
	\end{equation*}
We denote by $\mathcal{I}_{\mathrm{p},\, \mathrm{sing}}$ the operator with kernel
$\mathcal{I}_{\mathrm{p},\, \mathrm{sing}}(x,y)$.
By \cite[Chapter 9]{Abra}, we have
\[
J_0'(z)=-J_1(z),\qquad (H_0^{\pm})'(z)=-H_1^{\pm}(z),
\]
so that
	\begin{equation*}
		k(\nu,\mu)=\frac{i}{8\pi}\int_0^{+\infty}\lambda\, g_1^+(\lambda)^{-1}
		\chi(\lambda/\lambda_0)\, H_1^+(\lambda\nu)\,J_1(\lambda\mu)\, d\lambda .
	\end{equation*}
	The pointwise lower bound for $k$ along the ``shoulder'' $\nu=\mu+\delta$
	is the content of Lemma~\ref{lem:k-shoulder-lower} below: there exist a
	constant $c>0$, $\mu_0\ge 2$ and $\delta_1\ge 2$, depending only on
	$\lambda_0$, $\chi$ and $\beta_+$, such that
	\begin{equation}\label{eq:k-asymp-lem-0}
		|k(\nu,\mu)|
		\ge \frac{c}{\mu\,\delta\,\log\delta},
		\qquad\text{whenever }\ \nu=\mu+\delta,\ \ \mu\ge\mu_0,\ \
		\delta\in[\delta_1,\sqrt{\mu}\,].
	\end{equation}
	
	Before exploiting \eqref{eq:k-asymp-lem-0}, we record the endpoint Schur
	identity for $\mathcal{I}_{\mathrm{p},\, \mathrm{sing}}$. One checks directly that the
	kernel is bounded, in fact
	$\sup_{x,\,y}|\mathcal{I}_{\mathrm{p},\, \mathrm{sing}}(x,y)|\lesssim 1$:
	for $k$ this follows from \eqref{eq:bessel-bds} by splitting the integral
	at $\lambda=\nu^{-1}$ and $\lambda=\mu^{-1}$, and for $T$ from the
	finiteness of the pairings defining it, which depend only on the
	directions $x/|x|$ and $y/|y|$. In particular the kernel is locally
	integrable, and the endpoint Schur identity gives
	\begin{equation}\label{eq:schur-endpoint}
		\|\mathcal{I}_{\mathrm{p},\, \mathrm{sing}}\|_{L^1(\mathbb{R}^2)\to L^1(\mathbb{R}^2)}
		=\operatorname{ess\,sup}_{y\in\mathbb{R}^2}\int_{\mathbb{R}^2}
		|\mathcal{I}_{\mathrm{p},\, \mathrm{sing}}(x,y)|\,dx .
	\end{equation}
	
	We next produce a region on which $|T(x,y)|$ is bounded from below; this is
	where the p-wave resonance enters. Since $H$ has p-wave resonances,
	$(S_2-S_3)D_2(S_2-S_3)$ is nonzero by
	Lemma~\ref{lem:one to one lem}, and $S_2-S_3$ is a finite-rank projection
	by Riesz--Schauder theory. Recall from \eqref{eq:property of S-3} that
	\[
	S_3L^2=\left\{\varphi\in S_2L^2\;\Big|\;
	\int_{\mathbb{R}^2}x_i v(x)\varphi(x)\,dx=0\text{ for }i=1,2\right\}.
	\]
	Since $S_2v=0$, we have $\dim\bigl((S_2-S_3)L^2\bigr)\le 2$.
	We treat the two-dimensional case, the one-dimensional case being simpler.
	Choose an orthonormal basis
	$\{\phi_1, \phi_2\}$ of $(S_2-S_3)L^2$ that diagonalizes
	$(S_2-S_3)D_2(S_2-S_3)$:
	\begin{equation}\label{eq:dia-of-SD2S}
		(S_2-S_3)D_2(S_2-S_3)
		=\sum_{j=1}^r \alpha_j \langle\cdot, \phi_j\rangle\phi_j,
		\qquad \alpha_j \in \mathbb{R}\setminus \{0\},
	\end{equation}
	with $r=2$. Let
	$$a_{ij}=\int_{\mathbb{R}^2} x_i v(x)\phi_j(x)\,dx,\qquad i,j=1,2.$$
	Since $\dim\bigl((S_2-S_3)L^2\bigr)=2$, the vectors
	$(S_2-S_3)(x_1v)$ and $(S_2-S_3)(x_2v)$ form a basis of $(S_2-S_3)L^2$;
	equivalently, the matrix $A=(a_{ij})_{i,j=1,2}$ is invertible. We now
	define a new basis $\{\varphi_1, \varphi_2\}$ by
	\[
	\varphi_1:=a_{22}\,\phi_1-a_{21}\,\phi_2,\qquad
	\varphi_2:=a_{12}\,\phi_1-a_{11}\,\phi_2.
	\]
	A direct computation shows that this basis diagonalizes the two moment
	functionals associated with $x_1v$ and $x_2v$:
	\begin{equation}\label{eq:relation-be-basis}
		\begin{aligned}
			\Bigl|\int_{\mathbb{R}^2} x_1 v(x)\varphi_1(x)\,dx\Bigr|
			=|\det A|>0, \quad
			\Bigl|\int_{\mathbb{R}^2} x_1 v(x)\varphi_2(x)\,dx\Bigr|=0, \\[2mm]
			\Bigl|\int_{\mathbb{R}^2} x_2 v(x)\varphi_2(x)\,dx\Bigr|
			=|\det A|>0, \quad
			\Bigl|\int_{\mathbb{R}^2} x_2 v(x)\varphi_1(x)\,dx\Bigr|=0.
		\end{aligned}
	\end{equation}
	Consequently, we may write
	\[
	(S_2-S_3)D_2(S_2-S_3)
	=\sum_{i,j=1}^2 \beta_{i,j} \langle\cdot, \varphi_i\rangle\varphi_j,
	\qquad \beta_{i,j} \in \mathbb{C},
	\]
	where at least one coefficient $\beta_{i,j}$ is nonzero; relabelling the
	coordinate axes and the basis if necessary, we may assume
	$\beta_{1,1}\neq 0$.
	
	Now write $\hat z:=z/|z|$ for $z\neq 0$. Since
	$\frac{\langle\cdot,z\rangle}{|z|}=(\cdot)\cdot\hat z$, expanding $T(x,y)$
	in the basis $\{\varphi_1,\varphi_2\}$ gives
	\[
	T(x,y)=\sum_{i,j=1}^2\beta_{i,j}\,
	\Bigl\langle v\,(\cdot)\cdot\hat y,\,\varphi_i\Bigr\rangle
	\Bigl\langle\varphi_j,\,v\,(\cdot)\cdot\hat x\Bigr\rangle .
	\]
	By \eqref{eq:relation-be-basis}, at $\hat x=\hat y=(1,0)$ every summand
	vanishes except the $\beta_{1,1}$-term, which equals
	$\beta_{1,1}\,|\det A|^2\neq 0$. Both factors are continuous in the
	angular variables, so there exist $\varepsilon>0$ and $c_0>0$ such that
	\begin{equation}\label{eq:T-lower}
		|T(x,y)|\ge c_0
		\qquad\text{whenever }\ |\hat x-(1,0)|\le\varepsilon,\ \
		|\hat y-(1,0)|\le\varepsilon .
	\end{equation}
	
	Finally, we integrate over the region where both \eqref{eq:k-asymp-lem-0}
	and \eqref{eq:T-lower} apply. For $\mu\ge\mu_0+1$ define the annular
	sector
	\[
	\mathbf{Ann}_{\mu}
	:=\Bigl\{y\in\mathbb{R}^2:\ \bigl|\langle y\rangle-\mu\bigr|\le\tfrac12,\ \
	\bigl|\hat y-(1,0)\bigr|\le\varepsilon\Bigr\},
	\]
	and the cone
	\[
	\Gamma_\mu:=\Bigl\{x\in\mathbb{R}^2:\
	\langle x\rangle-\mu\in
	\bigl[\delta_1,\sqrt{\mu}\,\bigr],\ \
	\bigl|\hat x-(1,0)\bigr|\le\varepsilon\Bigr\}.
	\]
	If $y\in\mathbf{Ann}_{\mu}$ and $x\in\Gamma_{\mu}$ for a large $\mu$, then
	$\langle y\rangle\sim\mu$ and, writing $\nu=\langle x\rangle$ and
	$\delta=\nu-\langle y\rangle$, the bound \eqref{eq:k-asymp-lem-0} gives
	\[
	\bigl|k\bigl(\nu, \langle y\rangle\bigr)\bigr|
	\ge\frac{c}{\langle y\rangle\,\delta\,\log\delta}
	\ge\frac{c}{2\,\mu\,\delta\,\log\delta}.
	\]
	Let $\varepsilon$ be a fixed parameter such that \eqref{eq:T-lower} holds.
	Using polar coordinates $x=r\omega$, in which
	$dx=r\,dr\,d\omega=\nu\,d\nu\,d\omega$, and
	noting that the sector $\{|\omega-(1,0)|\le\varepsilon\}$ has angular
	measure $\sim 2\varepsilon$, we obtain, for every
	$y\in\mathbf{Ann}_{\mu}$,
	\begin{align*}
		\int_{\mathbb{R}^2}|\mathcal{I}_{\mathrm{p},\,\mathrm{sing}}(x,y)|\,dx
		&\ge c_0\int_{\Gamma_{\mu}}
		\bigl|k\bigl(\langle x\rangle, \langle y\rangle\bigr)\bigr|\,dx\\
		&\ge c_0\cdot 2\varepsilon\int_{\delta_{1}}^{\sqrt{\mu}}
		\frac{c\,\nu\,d\delta}{\mu\,\delta\,\log\delta}
		\Big|_{\nu=\langle y\rangle+\delta}\\
		&\ge 2\varepsilon\,c\,c_0\int_{\delta_{1}}^{\sqrt{\mu}}
		\frac{d\delta}{\delta\,\log\delta}
		=2\varepsilon\,c\,c_0\,
		\Bigl(\log\log\sqrt{\mu}-\log\log\delta_1\Bigr)\\
		&\ge c'\,\log\log\mu,
	\end{align*}
	where we used $\langle x\rangle\les\mu$ and
	$\langle y\rangle\sim \mu$. Since $\mathbf{Ann}_{\mu}$ has positive
	measure and $\mu\ge\mu_0+1$ is arbitrary, \eqref{eq:schur-endpoint} yields
	\[
	\|\mathcal{I}_{\mathrm{p},\,\mathrm{sing}}\|_{L^1\to L^1}
	=\operatorname{ess\,sup}_{y}\int_{\mathbb{R}^2}
	|\mathcal{I}_{\mathrm{p},\,\mathrm{sing}}(x,y)|\,dx
	\ge c'\,\log\log\mu\longrightarrow\infty
	\qquad(\mu\to\infty).
	\]
	Hence $\mathcal{I}_{\mathrm{p},\,\mathrm{sing}}$ does not give rise to an
	$L^1$-bounded operator; since the remaining terms are bounded by Step~1,
	$\mathcal{W}_{\mathrm{p},\mathrm{sing}}^{\mathrm{L}}$ is unbounded on
	$L^1(\mathbb{R}^2)$.
\end{proof}

\begin{lemma}\label{lem:k-shoulder-lower}
	Let $k(\nu,\mu)$ be the scalar kernel defined in \eqref{eq:def-ksing}, i.e.
	\[
	k(\nu,\mu)=\frac{i}{8\pi}\int_0^{+\infty}\lambda\, g_1^+(\lambda)^{-1}
	\chi(\lambda/\lambda_0)\,H_1^+(\lambda\nu)\,J_1(\lambda\mu)\, d\lambda .
	\]
	Then there exist a constant $c>0$, $\mu_0\ge 2$ and $\delta_1\ge 2$,
	depending only on $\lambda_0$, $\chi$ and $\beta_+$, such that, for
	$\nu=\mu+\delta$ with $\mu\ge\mu_0$ and $\delta\in[\delta_1,\sqrt{\mu}]$,
	\begin{equation}\label{eq:k-asymp-lem}
		|k(\nu,\mu)|
		\ge \frac{c}{\mu\,\delta\,\log\delta}.
	\end{equation}
\end{lemma}

\begin{proof}
	Since $g_1^+(\lambda)=\log\lambda+\beta_+$ with
	$\beta_+\notin\mathbb{R}$, one has
	\begin{equation}\label{eq:g-comp}
		|g_1^+(\lambda)|\sim |\log\lambda|,
		\quad
		g_1^+(\lambda)^{-2}=(\log\lambda)^{-2}+\mathcal{O}\bigl(|\log\lambda|^{-3}\bigr),
		\qquad \lambda\in(0,\lambda_0].
	\end{equation}
	Recall the standard asymptotics (cf.\ \cite{Abra}): 
	$H_1^{+}(z)=\frac{2}{\pi i z}+\mathcal{O}(1)$ and
	$J_1(w)=\frac{w}{2}+\mathcal{O}(w^{3})$ for $0<z,w\lesssim 1$, while, as 
	$z,w\to+\infty$,
	\begin{equation*}
		H_1^{+}(z)=\Bigl(\frac{2}{\pi z}\Bigr)^{1/2}e^{i(z-3\pi/4)}+\mathcal{O}(z^{-3/2}),
		\qquad
		J_1(w)=\Bigl(\frac{2}{\pi w}\Bigr)^{1/2}
		\cos\Bigl(w-\frac{3\pi}{4}\Bigr)+\mathcal{O}(w^{-3/2}) .
	\end{equation*}
	In particular,
	\begin{equation}\label{eq:bessel-bds}
		|J_1(u)|\lesssim \min\{u,u^{-1/2}\},
		\qquad
		|H_1^{+}(v)|\lesssim \max\{v^{-1},v^{-1/2}\},
		\qquad u,v>0;
	\end{equation}
	and, writing $\cos\theta=\frac12\bigl(e^{i\theta}+e^{-i\theta}\bigr)$ and
	noting that $e^{-3\pi i/2}=i$,
	\begin{equation}\label{eq:HJ-prod}
		H_1^{+}(z)J_1(w)
		=\frac{e^{i(z-w)}+i\,e^{i(z+w)}}{\pi\sqrt{zw}}
		+\mathcal{O}\bigl(z^{-3/2}w^{-1/2}+z^{-1/2}w^{-3/2}\bigr),
		\qquad z,w\ge \tfrac12 .
	\end{equation}
	Throughout the proof we write $\delta:=\nu-\mu$; since the conclusion only
	involves the range $\delta\in[\delta_1,\sqrt{\mu}]$, one has $\delta>0$
	everywhere below.
	
	\medskip\noindent\emph{Step 1: asymptotics of $\mathcal{L}_\chi$.}
	We claim that there exists $\delta_1\ge 2$ such that, for all $\delta\ge\delta_1$,
	\begin{equation}\label{eq:cchi-asy-lem}
		\mathcal{L}_\chi(\delta)=:\int_0^{\lambda_0} g_1^+(\lambda)^{-1}
		\chi(\lambda/\lambda_0)\,e^{i\lambda\delta}\, d\lambda=-\frac{i}{\delta\log\delta}
		+\mathcal{O}\Bigl(\frac{1}{\delta\,(\log\delta)^{2}}\Bigr).
	\end{equation}
	Indeed, set
	\[
	I(\delta):=\int_0^{\lambda_0}
	\frac{\chi(\lambda/\lambda_0)\,e^{i\lambda\delta}}
	{\lambda\,g_1^+(\lambda)^{2}}\,d\lambda .
	\]
	Since $g_1^+(\lambda)^{-1}\chi(\lambda/\lambda_0)\to0$ as $\lambda\to 0^+$
	and $\chi(1)=0$, and since
	$\int_0\frac{d\lambda}{\lambda|\log\lambda|^{2}}<\infty$, an integration by
	parts, using
	$\partial_\lambda g_1^+(\lambda)^{-1}=-\bigl(\lambda g_1^+(\lambda)^2\bigr)^{-1}$, gives
	\[
	\mathcal{L}_\chi(\delta)
	=-\frac{i}{\delta}\,I(\delta)
	+\frac{i}{\delta\lambda_0}\int_0^{\lambda_0}
	\chi'(\lambda/\lambda_0)\,g_1^+(\lambda)^{-1}e^{i\lambda\delta}\,d\lambda .
	\]
	Since $\chi'(\lambda/\lambda_0)$ is smooth and compactly supported in
	$(0,\lambda_0)$, the second integral is $\mathcal{O}(\delta^{-N})$ for
	every $N$, by repeated integration by parts. Hence \eqref{eq:cchi-asy-lem}
	follows once we show that
	\begin{equation}\label{eq:model-log}
		I(\delta)=\frac{1}{\log\delta}+\mathcal{O}\bigl((\log\delta)^{-2}\bigr),
		\qquad \delta\to+\infty .
	\end{equation}
	To this end, assume $\delta\ge 2/\lambda_0$, so that
	$\chi(\lambda/\lambda_0)\equiv 1$ on $(0,1/\delta]$, and split the integral
	at $\lambda=1/\delta$. On $(0,1/\delta]$ one has
	$e^{i\lambda\delta}=1+\mathcal{O}(\lambda\delta)$, and \eqref{eq:g-comp}
	applies; hence
	\begin{align*}
		\int_0^{1/\delta}\frac{e^{i\lambda\delta}}{\lambda\,g_1^+(\lambda)^{2}}\,d\lambda
		&=\int_0^{1/\delta}\frac{d\lambda}{\lambda(\log\lambda)^{2}}
		+\mathcal{O}\Bigl(\int_0^{1/\delta}\frac{d\lambda}{\lambda|\log\lambda|^{3}}
		+\delta\int_0^{1/\delta}\frac{d\lambda}{(\log\lambda)^{2}}\Bigr)\\
		&=\frac{1}{\log\delta}+\mathcal{O}\bigl((\log\delta)^{-2}\bigr).
	\end{align*}
	For the tail, one integration by parts gives
	\begin{align*}
		\int_{1/\delta}^{\lambda_0}
		\frac{\chi(\lambda/\lambda_0)\,e^{i\lambda\delta}}
		{\lambda\,g_1^+(\lambda)^{2}}\,d\lambda
		=\Bigl[\frac{e^{i\lambda\delta}}{i\delta}\,
		\frac{\chi(\lambda/\lambda_0)}{\lambda\,g_1^+(\lambda)^{2}}
		\Bigr]_{1/\delta}^{\lambda_0}
		-\frac{1}{i\delta}\int_{1/\delta}^{\lambda_0}e^{i\lambda\delta}\,
		\partial_\lambda\Bigl[
		\frac{\chi(\lambda/\lambda_0)}{\lambda\,g_1^+(\lambda)^{2}}\Bigr]d\lambda .
	\end{align*}
	The boundary terms are $\mathcal{O}\bigl((\log\delta)^{-2}\bigr)$ (the one
	at $\lambda_0$ in fact vanishes, since $\chi(1)=0$); in the remaining
	integral, the $\chi'$-piece contributes $\mathcal{O}(\delta^{-1})$, and the
	other terms are estimated by
	\[
	\frac{1}{\delta}\int_{1/\delta}^{\lambda_0}
	\frac{d\lambda}{\lambda^{2}|\log\lambda|^{m}}
	=\frac{1}{\delta}\int_{1/\lambda_0}^{\delta}\frac{dt}{(\log t)^{m}}
	\lesssim (\log\delta)^{-m},
	\qquad m=2,3,
	\]
	via the substitution $t=1/\lambda$, the last integral being dominated by the
	top endpoint. The tail is therefore
	$\mathcal{O}\bigl((\log\delta)^{-2}\bigr)$, and \eqref{eq:model-log} follows.
	
	\medskip\noindent\emph{Step 2: decomposition of $k$.}
	We claim that, for all $\mu\ge 2$ and $\nu\in[\mu,2\mu]$,
	\begin{equation}\label{eq:shoulder-lem}
		k(\nu,\mu)=\frac{i\,\mathcal{L}_\chi(\delta)}{8\pi^{2}\,\sqrt{\mu\nu}}
		+\mathcal{J}(\mu,\nu),
		\qquad
		|\mathcal{J}(\mu,\nu)|\le C\,\frac{\log\log\mu}{\mu^{2}} .
	\end{equation}
	To see this, split the integral defining $k$ at $\lambda_1:=1/(2\mu)$ and
	write $k=\frac{i}{8\pi}(k_{<}+k_{>})$ accordingly.
	
	On $(0,\lambda_1]$ one has $\lambda\mu\le \tfrac12$ and
	$\lambda\nu\le 1$; by \eqref{eq:bessel-bds},
	$|J_1(\lambda\mu)|\lesssim \lambda\mu$ and
	$|H_1^{+}(\lambda\nu)|\lesssim (\lambda\nu)^{-1}$. Using
	\eqref{eq:g-comp}, $\nu\in[\mu,2\mu]$ and
	$\int_0^{\varepsilon}\lambda|\log\lambda|^{-1}d\lambda
	\lesssim \varepsilon^{2}/|\log\varepsilon|$, we get
	\[
	|k_{<}|\lesssim \int_0^{\lambda_1}\lambda\,|g_1^+(\lambda)|^{-1}
	(\lambda\nu)^{-1}(\lambda\mu)\,d\lambda
	\lesssim \frac{\mu}{\nu}\cdot\frac{\lambda_1^{2}}{|\log\lambda_1|}
	\lesssim \frac{1}{\mu^{2}\log\mu}.
	\]
	
	On $[\lambda_1,\lambda_0]$ both $\lambda\nu\ge\lambda\mu\ge \tfrac12$, so
	\eqref{eq:HJ-prod} applies with $z=\lambda\nu$, $w=\lambda\mu$:
	\begin{align*}
		k_{>}&=\frac{1}{\pi\sqrt{\mu\nu}}\int_{\lambda_1}^{\lambda_0}
		g_1^+(\lambda)^{-1}\chi(\lambda/\lambda_0)
		\bigl(e^{i\lambda\delta}+i\,e^{i\lambda(\mu+\nu)}\bigr)d\lambda \\
		&\qquad+\int_{\lambda_1}^{\lambda_0}\lambda\, g_1^+(\lambda)^{-1}
		\chi(\lambda/\lambda_0)\,G(\lambda)\,d\lambda ,
	\end{align*}
	where
	$|G(\lambda)|\lesssim (\lambda\nu)^{-3/2}(\lambda\mu)^{-1/2}
	+(\lambda\nu)^{-1/2}(\lambda\mu)^{-3/2}$. We estimate the resulting three
	pieces. First,
	\[
	\int_{\lambda_1}^{\lambda_0}g_1^+(\lambda)^{-1}\chi(\lambda/\lambda_0)
	e^{i\lambda\delta}d\lambda
	=\mathcal{L}_\chi(\delta)
	+\mathcal{O}\Bigl(\int_0^{\lambda_1}\frac{d\lambda}{|\log\lambda|}\Bigr)
	=\mathcal{L}_\chi(\delta)+\mathcal{O}\Bigl(\frac{1}{\mu\log\mu}\Bigr).
	\]
	Second, since $\mu+\nu\ge 2\mu$, \eqref{eq:cchi-asy-lem} applied with
	$\delta$ replaced by $\mu+\nu$ (for $\mu\ge\delta_1/2$) gives
	\[
	\Bigl|\int_{\lambda_1}^{\lambda_0}g_1^+(\lambda)^{-1}\chi(\lambda/\lambda_0)
	e^{i\lambda(\mu+\nu)}d\lambda\Bigr|
	\le |\mathcal{L}_\chi(\mu+\nu)|+\mathcal{O}\Bigl(\frac{1}{\mu\log\mu}\Bigr)
	\lesssim \frac{1}{\mu\log\mu}.
	\]
	Both oscillatory modes carry the prefactor
	$\frac{1}{\pi\sqrt{\mu\nu}}\sim \mu^{-1}$ in $k_{>}$, so their contribution
	beyond the main term $\mathcal{L}_\chi(\delta)/\bigl(\pi\sqrt{\mu\nu}\bigr)$
	is $\mathcal{O}\bigl(\mu^{-2}/\log\mu\bigr)$. Finally, since
	$\int_{\lambda_1}^{\lambda_0}\frac{d\lambda}{\lambda|\log\lambda|}
	=\log\log(2\mu)+\mathcal{O}(1)$ and $\nu\sim \mu$,
	\[
	\Bigl|\int_{\lambda_1}^{\lambda_0}\lambda\, g_1^+(\lambda)^{-1}
	\chi(\lambda/\lambda_0)\,G(\lambda)\,d\lambda\Bigr|
	\lesssim \bigl(\nu^{-3/2}\mu^{-1/2}+\nu^{-1/2}\mu^{-3/2}\bigr)\log\log\mu
	\lesssim \frac{\log\log\mu}{\mu^{2}} .
	\]
	Collecting the three estimates and multiplying by $\frac{i}{8\pi}$ gives
	\eqref{eq:shoulder-lem}.
	
	\medskip\noindent\emph{Step 3: conclusion.}
	Let $\nu=\mu+\delta$ with $\delta\in[\delta_1,\sqrt{\mu}]$, and let
	$\mu\ge\max\{2,\delta_1/2\}$; then $\nu\le 2\mu$, so Step~2 applies. By
	\eqref{eq:cchi-asy-lem}, the main term in \eqref{eq:shoulder-lem} satisfies
	\[
	\frac{|\mathcal{L}_\chi(\delta)|}{8\pi^{2}\sqrt{\mu\nu}}
	=\frac{1}{8\pi^{2}\sqrt{\mu\nu}\,\delta\,\log\delta}
	\Bigl(1+\mathcal{O}\bigl((\log\delta)^{-1}\bigr)\Bigr),
	\]
	while, using $\delta\log\delta\le \sqrt{\mu}\cdot\tfrac12\log\mu$,
	\[
	|\mathcal{J}(\mu,\nu)|\cdot\bigl(\mu\,\delta\,\log\delta\bigr)
	\lesssim \frac{\delta\,\log\delta\,\log\log\mu}{\mu}
	\lesssim \frac{\log\mu\,\log\log\mu}{\sqrt{\mu}}
	=o(1)\qquad(\mu\to\infty),
	\]
	so the remainder is relatively negligible. Enlarging $\delta_1$ if
	necessary so that the $\mathcal{O}\bigl((\log\delta)^{-1}\bigr)$ term is at
	most $\tfrac12$, and then choosing $\mu_0$ so large that
	$|\mathcal{J}(\mu,\nu)|
	\le \frac{1}{16\pi^{2}\sqrt{\mu\nu}\,\delta\,\log\delta}$
	for all $\mu\ge\mu_0$, we obtain from \eqref{eq:shoulder-lem} and
	$\sqrt{\mu\nu}=\mu\sqrt{1+\delta/\mu}\le 2\mu$ that
	\[
	|k(\nu,\mu)|
	\ge \frac{1}{16\pi^{2}\sqrt{\mu\nu}\,\delta\,\log\delta}
	\ge \frac{1}{32\pi^{2}\,\mu\,\delta\,\log\delta},
	\qquad \mu\ge\mu_0,\ \ \delta\in[\delta_1,\sqrt{\mu}].
	\]
	This is \eqref{eq:k-asymp-lem} with $c=1/(32\pi^{2})$.
\end{proof}

We can now complete the proof of statement~\textup{(i)} in
Proposition~\ref{pro:unbound-1}.

\begin{proof}[Completion of the proof of \textup{(i)} in Proposition~\ref{pro:unbound-1}]
By \eqref{eq:dec-of-WL-second}, together with the
$L^1$-boundedness of $\mathcal{W}_{10}^{\mathrm{L}}$ and
$\mathcal{W}_{01}^{\mathrm{L}}$ established in
Subsection~\ref{sec:contri-of-co}, the operator $\mathcal{W}^{\mathrm{L}}$
fails to be bounded on $L^1(\mathbb{R}^2)$ if and only if the same holds
for the sum $\mathcal{W}_{00}^{\mathrm{L}}+\mathcal{W}_{11}^{\mathrm{L}}$. 

Combining the reduction in
Subsection~\ref{sec:unb-for-regular} with the discussion at the beginning
of this subsection, we further deduce that $\mathcal{W}^{\mathrm{L}}$ is
$L^1$-unbounded if and only if
\begin{itemize}
	\item $\mathcal{W}_{\mathrm{sing}}^{\mathrm{L}}$ is $L^1$-unbounded, when
	$H$ has neither p-wave nor s-wave resonances;
	\item $\mathcal{W}_{\mathrm{p},\,\mathrm{sing}}^{\mathrm{L}}$ is
	$L^1$-unbounded, when $H$ has both p-wave and s-wave resonances;
	\item $\mathcal{W}_{\mathrm{sing}}^{\mathrm{L}}+
	\mathcal{W}_{\mathrm{p},\,\mathrm{sing}}^{\mathrm{L}}$ is $L^1$-unbounded,
	when $H$ has p-wave resonances but no s-wave resonances.
\end{itemize}
The first two statements follow from Lemmas~\ref{pro:unb of Wsing} and
\ref{lemm:unb-of-wpsing}, respectively. It remains to prove the third
one. Note that the unboundedness of the two summands does not by itself
imply the unboundedness of the sum, since cancellations between the two
singularities could a priori occur; to rule this out, we test the sum
against a sequence of functions whose Fourier transforms vanish near the
origin.

Let $f_k$ be the sequence of functions defined by
\[
\hat{f_k}(|\xi|) = \chi\Big(\frac{2|\xi|}{\lambda_0}\Big)\bigl(1-\chi(k|\xi|)\bigr).
\]
Each $f_k$ is a Schwartz function, and a simple computation shows that
$\sup_k \|f_k\|_{L^1} \lesssim 1$. Since the divergence in
\eqref{eq-rep-pair} originates from the logarithmic singularity at the
origin, applying the identity \eqref{eq-rep-pair} to $f_k$ yields
\[
\lim_{k\to\infty} \langle \mathcal{W}_{\mathrm{sing}}^{\mathrm{L}} f_k,  1\rangle = \infty.
\]
It therefore suffices to show that the p-wave part annihilates every
$f_k$ in the same pairing, i.e. that
$\langle \mathcal{W}_{\mathrm{p},\,\mathrm{sing}}^{\mathrm{L}} f_k,
1\rangle = 0$ for each $k$.
Indeed, using the diagonalization \eqref{eq:dia-of-SD2S} and arguing as
for \eqref{eq-rep-pair}, we obtain, for each pair $(f_k,1)$,
\begin{align*}
	\langle \mathcal{W}_{\mathrm{p},\,\mathrm{sing}}^{\mathrm{L}} f_k, \, 1 \rangle
	&= \sum_{j} \alpha_{j} \frac{\mathcal{F}(v\phi_j)(0)}{(2\pi)^2}
	\lim_{\varepsilon \to 0+} \int_{\mathbb{R}^2}
	\frac{\overline{\mathcal{F}^{-1}(v\phi_j)(\eta)} \;
		\chi(2|\eta|)\bigl(1-\chi(k|\eta|)\bigr)}
	{g_{1}^{+}(|\eta|)|\eta|^2 (|\eta|^2 + i\varepsilon)} \, d\eta \\
	&= 0,
\end{align*}
where we used that $\mathcal{F}(v\phi_j)(0) = \int_{\mathbb{R}^2}
v\phi_j\,dx = 0$ for each $j$: indeed,
$\phi_j\in (S_2-S_3)L^2\subset S_2L^2$ and $S_2v=0$, so that
$\int_{\mathbb{R}^2}v\phi_j\,dx=\langle S_2v,\phi_j\rangle=0$.
Consequently,
$$\langle (\mathcal{W}_{\mathrm{sing}}^{\mathrm{L}}+
\mathcal{W}_{\mathrm{p},\,\mathrm{sing}}^{\mathrm{L}}) f_k,\,1\rangle
\to\infty$$ as $k\to\infty$, while
$\sup_k\|f_k\|_{L^1}\lesssim 1$. This shows that
$\mathcal{W}_{\mathrm{sing}}^{\mathrm{L}}+
\mathcal{W}_{\mathrm{p},\,\mathrm{sing}}^{\mathrm{L}}$ is unbounded on
$L^1(\mathbb{R}^2)$ and completes the proof.
\end{proof}

\subsection{Proof of \textup{(ii)} in Proposition~\ref{pro:unbound-1}}

This subsection is devoted to the proof of statement~(ii) of 
Proposition~\ref{pro:unbound-1}. Since Yajima~\cite{Yajima-2022} has shown 
that $W_{-}$ fails to be bounded on $L^p(\mathbb{R}^2)$ for any 
$2<p\le \infty$ when $H$ has a p-wave resonance, it follows from 
Proposition~\ref{pro:bound for high} that $\mathcal{W}^{\mathrm{L}}$ is not 
bounded on $L^\infty(\mathbb{R}^2)$ in this case. Therefore, it suffices to 
establish the $L^\infty$-unboundedness of $\mathcal{W}^{\mathrm{L}}$ in the 
absence of p-wave resonances.

Since we also assume that zero is a second-kind threshold singularity or an 
eigenvalue of $H$ in Proposition~\ref{pro:unbound-1}, the operator $H$ must 
possess zero-energy eigenfunctions. This reduces the proof of the 
$L^\infty$-unboundedness of $\mathcal{W}^{\mathrm{L}}$ to the following three 
subcases under the assumption that $H$ has no p-wave resonances:
\begin{itemize}
	\item \emph{Subcase 4:} $H$ has no s-wave resonance and 
	\eqref{eq:condi_for_eigen-0} holds for every zero-energy eigenfunction of $H$.
      \smallskip
	\item \emph{Subcase 5:} $H$ has an s-wave resonance and 
	\eqref{eq:condi_for_eigen-0} does not hold for some zero-energy eigenfunction 
	of $H$.
	\item \emph{Subcase 6:} $H$ has no s-wave resonance and 
	\eqref{eq:condi_for_eigen-0} does not hold for some zero-energy eigenfunction 
	of $H$.
\end{itemize}

If $H$ has no p-wave resonance, then $\mathcal{W}_{10}^{\mathrm{L}}$ and 
$\mathcal{W}_{01}^{\mathrm{L}}$ are both bounded on $L^\infty(\mathbb{R}^2)$. 
Moreover, in this case,
\[
\mathcal{L}_{11}^{+}(\lambda)=
\dfrac{S_3D_3S_3}{\lambda^2}+\tilde{\mathcal{O}}_2(\lambda^{-1+}),
\]
and the contribution of the error term 
$\tilde{\mathcal{O}}_2(\lambda^{-1+})$ is also bounded on 
$L^\infty(\mathbb{R}^2)$ by the proof of Lemma~\ref{lem:contri of L11}.

Collecting these results, it remains to verify that
\begin{itemize}
	\item $\mathcal{W}_{\mathrm{sing}}^{\mathrm{L}}$ is unbounded on 
	$L^\infty(\mathbb{R}^2)$ in Subcase~4;
	\item $\mathcal{W}_{\mathrm{e},\,\mathrm{sing}}^{\mathrm{L}}$ is unbounded 
	on $L^\infty(\mathbb{R}^2)$ in Subcase~5;
	\item $\mathcal{W}_{\mathrm{sing}}^{\mathrm{L}}+
	\mathcal{W}_{\mathrm{e},\,\mathrm{sing}}^{\mathrm{L}}$ is unbounded on 
	$L^\infty(\mathbb{R}^2)$ in Subcase~6.
\end{itemize}

We first establish the unboundedness of 
$\mathcal{W}_{\mathrm{e},\,\mathrm{sing}}^{\mathrm{L}}$.

\begin{lemma}\label{lem:unbound-of-N_4}
	If \eqref{eq:condi_for_eigen-0} does not hold for some zero-energy 
	eigenfunction of $H$, then $\mathcal{W}_{\mathrm{e}, \mathrm{sing}}^{\mathrm{L}}$ 
	is unbounded on $L^\infty(\mathbb{R}^2)$.
\end{lemma}

\begin{proof}
	Recall from the proof of Lemma~\ref{lem:bound of N_4} that we have the 
	decomposition
	\[
	\mathcal{W}_{\mathrm{e}, \mathrm{sing}}^{\mathrm{L}}f(x)
	=\frac{1}{2}\sum_{j=1}^k \Psi_j(x) 
	\int_{\mathbb{R}^2} 
	\overline{\mathcal{F}^{-1}\left(\frac{\chi(|\cdot|/\lambda_0)\mathbf{P}_j(\cdot)}
		{|\cdot|^2}\right)(w)}\,
	f(w)\,dw + \mathcal{W}_{\mathrm{e}, \mathrm{good}}^{\mathrm{L}} f(x),
	\]
	for any $f\in L^2\cap L^\infty$, where the operator $\mathcal{W}_{\mathrm{e}, \mathrm{good}}^{\mathrm{L}}$ is bounded on 
	$L^\infty(\mathbb{R}^2)$ and $\mathbf{P}_j$ is the homogeneous polynomial
	\[
	\mathbf{P}_j(\eta)=\int_{\mathbb{R}^2} (x\cdot \eta)^2 (V\Psi_j)(x)\,dx.
	\]
	Here, $\{\Psi_{j}\}_{j=1}^k$ is an orthonormal basis of the zero-energy 
	eigenspace of $H$ with $\Psi_{j}\in L^{\infty}(\mathbb{R}^2)$. By the 
	orthonormality of $\{\Psi_j\}$, the $L^\infty$-boundedness of 
	$\mathcal{W}_{\mathrm{e}, \mathrm{sing}}^{\mathrm{L}}$ is equivalent to 
	showing that each map
	\begin{equation}\label{eq:mapping-to-unb}
		f\mapsto \int_{\mathbb{R}^2} 
		\overline{\mathcal{F}^{-1}\left(\frac{\chi(|\cdot|/\lambda_0)\mathbf{P}_j(\cdot)}
			{|\cdot|^2}\right)(w)}\,
		f(w)\,dw
	\end{equation}
	defines a bounded linear functional on $L^\infty(\mathbb{R}^2)$.
	
	To prove the lemma, we argue by contradiction. Without loss of generality, 
	assume that $\Psi_1$ does not satisfy \eqref{eq:condi_for_eigen-0}. Note that 
	the homogeneous polynomials $x_1x_2$ and $x_1^2-x_2^2$ form a basis for the 
	space of homogeneous harmonic polynomials of degree $2$ in two dimensions. 
	Since $\Psi_1$ does not satisfy \eqref{eq:condi_for_eigen-0}, there exist a 
	non-zero homogeneous harmonic polynomial $P_1$ of degree $2$ and a constant 
	$c_0$ such that
	\[
	\mathbf{P}_1(\eta)=P_1(\eta)+c_0|\eta|^2.
	\]
	It is known that the term $c_0|\eta|^2$ contributes a bounded linear functional 
	on $L^\infty(\mathbb{R}^2)$. Since $P_1$ is a non-zero homogeneous harmonic 
	polynomial of degree $2$, one has
	\[
	\int_{\mathbb{S}^1} P_1(\omega)\, d\omega=0.
	\]
	We now decompose
	\[
	\mathcal{F}^{-1}\left(\frac{\chi(|\cdot|/\lambda_0)\, P_1(\cdot)}
	{|\cdot|^2}\right)
	=\mathcal{F}^{-1}\left(\frac{P_1(\cdot)}
	{|\cdot|^2}\right)
	-\mathcal{F}^{-1}\left(\frac{\tilde{\chi}(|\cdot|/\lambda_0)\, P_1(\cdot)}
	{|\cdot|^2}\right),
	\]
	and observe that $\mathcal{F}^{-1}\left(\tilde{\chi}(|\cdot|/\lambda_0)P_1(\cdot/|\cdot|)\right)$ is a rapidly decaying function away from the origin.
	By \cite[Proposition 2.2.7]{graf}, there exist a constant $b$ and a 
	$C^\infty$ function $\Omega$ on $\mathbb{S}^1$ such that, 
	in the sense of distributions,
	\[
	\mathcal{F}^{-1}\left(\frac{P_1(\cdot)}
	{|\cdot|^2}\right)
	=b\delta +W_\Omega,
	\]
	where $W_\Omega$ is the distribution defined by
	\[
	\langle W_\Omega, \varphi\rangle
	=\lim_{\varepsilon\to 0^+}\int_{|x|\ge \varepsilon} 
	\frac{\Omega(x/|x|)}{|x|^2} \varphi(x)\,dx,
	\qquad \varphi \in C_0^\infty(\mathbb{R}^2),
	\]
	with $\int_{\mathbb{S}^1} \Omega(\omega)\, d\omega=0$.
	
	Now, let $f_{\kappa}(x)=\overline{\Omega(x/|x|)}
	\bigl(\chi(|x|)-\chi(|x|/\kappa)\bigr)$, which is a $C_0^\infty$ function. 
	Then, for some non-zero constant $c_1$,
	\[
	\left|\int_{\mathbb{R}^2} 
	\overline{\mathcal{F}^{-1}\left(\frac{\chi(|\cdot|/\lambda_0)P_1(\cdot)}
		{|\cdot|^2}\right)(w)}\,
	f_\kappa(w)\,dw\right|
	\ge c_1 |\log \kappa|
	\]
	as $\kappa\to\infty$. This implies that the map in \eqref{eq:mapping-to-unb} 
	cannot be a bounded linear functional on $L^\infty(\mathbb{R}^2)$, thereby 
	completing the proof.
\end{proof}

We now complete the proof of statement~(ii) of 
Proposition~\ref{pro:unbound-1}.

\begin{proof}[Completion of the proof of statement~\textup{(ii)} of 
	Proposition~\ref{pro:unbound-1}]
	From the reduction at the beginning of this subsection, it remains to show 
	that $\mathcal{W}_{\mathrm{sing}}^{\mathrm{L}}$ is unbounded on 
	$L^\infty(\mathbb{R}^2)$ in Subcase~4; this follows from 
	Lemma~\ref{pro:unb of Wsing}. In Subcase~5, the operator $\mathcal{W}_{\mathrm{e}, \mathrm{sing}}^{\mathrm{L}}$ 
	is unbounded on $L^\infty(\mathbb{R}^2)$ by 
	Lemma~\ref{lem:unbound-of-N_4}. Finally, by Remark~\ref{eq:rem-to-Linfty}, 
	the proof of Lemma~\ref{lem-W0s-bound} shows that 
	$\chi(|\cdot|/r)\mathcal{W}_{\mathrm{sing}}^{\mathrm{L}}$ is bounded on 
	$L^\infty(\mathbb{R}^2)$  for each fixed $r$. However, by the proof of 
	Lemma~\ref{lem:unbound-of-N_4}, the operator 
	$\chi(|\cdot|/r)\mathcal{W}_{\mathrm{e},\,\mathrm{sing}}^{\mathrm{L}}$ is 
	unbounded on $L^\infty(\mathbb{R}^2)$ by Subcase~6 for some large $r$. Consequently, the sum 
	$\mathcal{W}_{\mathrm{sing}}^{\mathrm{L}}+
	\mathcal{W}_{\mathrm{e},\,\mathrm{sing}}^{\mathrm{L}}$ is unbounded on 
	$L^\infty(\mathbb{R}^2)$ in Subcase~6, and the proof is complete.
\end{proof}

\appendix

\section{Proof of Lemma \ref{thm-M inverse-even-1}}\label{sec:proof of exp lemma}

This section is devoted to the proof of Lemma~\ref{thm-M inverse-even-1}. The main tool is the Feshbach inversion formula, which we recall as follows.

\begin{lemma}[{\cite[Lemma 2.3]{JN01}}]\label{lemma-Feshbach formula}
	Let $\mathbb{A}$ be an operator matrix on the orthogonal direct sum of Hilbert spaces $\mathcal{H}_1\oplus\mathcal{H}_2$,
	\[
	\mathbb{A}=
	\begin{pmatrix}
		a_{11} & a_{12}\\[2pt]
		a_{21} & a_{22}
	\end{pmatrix},
	\qquad a_{ij}\colon \mathcal{H}_j\rightarrow\mathcal{H}_i,
	\]
	where $a_{11},a_{22}$ are closed and $a_{12},a_{21}$ are bounded. Assume that $a_{11}$ has a bounded inverse. Then $\mathbb{A}$ has a bounded inverse if and only if the operator $d:= a_{22}-a_{21}a_{11}^{-1}a_{12}$ has a bounded inverse, in which case
	\begin{equation*}
		(\mathbb{A})^{-1}=
		\begin{pmatrix}
			a_{11}^{-1}a_{12}d^{-1}a_{21}a_{11}^{-1}+a_{11}^{-1} & -a_{11}^{-1}a_{12}d^{-1}\\[4pt]
			-d^{-1}a_{21}a_{11}^{-1} & d^{-1}
		\end{pmatrix}.
	\end{equation*}
\end{lemma}

As a first step, we reduce the invertibility of $M^{\pm}(\lambda)$ to that of a $2\times2$ operator matrix adapted to the decomposition $L^2=(I-S_2)L^2\oplus S_2L^2$. Recall that, by Lemma~\ref{lem:one to one lem}, $S_2 \neq 0$ whenever zero is a second-kind threshold singularity or an eigenvalue of $H$; in this case, for each $\lambda>0$ the map
\[
\sigma_{\lambda}\colon L^2 \to (I-S_2)L^2 \oplus S_2L^2, 
\qquad \sigma_{\lambda}\varphi := \bigl((I-S_2)\varphi, \; \lambda^{-1}S_2\varphi\bigr),
\]
is an isomorphism; denoting its adjoint by $\sigma_{\lambda}^*$, whenever $\bigl(\sigma_{\lambda}M^{\pm}(\lambda)\sigma_{\lambda}^*\bigr)^{-1}$ exists we have
\begin{equation}\label{eq:iso to M inverse}
	(M^{\pm}(\lambda))^{-1} = \sigma_{\lambda}^*\bigl(\sigma_{\lambda}M^{\pm}(\lambda)\sigma_{\lambda}^*\bigr)^{-1}\sigma_{\lambda}.
\end{equation}
The proof of Lemma~\ref{thm-M inverse-even-1} is thus reduced to the invertibility and asymptotic behavior of $\sigma_{\lambda}M^{\pm}(\lambda)\sigma_{\lambda}^*$ as $\lambda\to 0$. We proceed in three steps: we first expand $M^{\pm}(\lambda)$ near zero energy, then invert the diagonal part of the resulting operator matrix, and finally recover the full inverse by a Neumann series argument.

\begin{lemma}
	Assume that $|V(x)| \lesssim \langle x\rangle^{-8-}$. Then the following expansion holds in $\mathbb{B}(L^2)$:
	\begin{equation}\label{eq:exp for M at zero}
		M^{\pm}(\lambda) = g_0^{\pm}(\lambda)P + T + g_1^\pm(\lambda)\lambda^2 v G_1 v + \lambda^2 v G_2 v + \tilde{\mathcal{O}}_2(\lambda^{3+}), 
		\qquad \lambda\in (0,1/2).
	\end{equation}	
\end{lemma}

\begin{proof}
	By the asymptotic properties of the Hankel functions summarized in \eqref{eq:asy-Hanel-0}--\eqref{eq:est for Theta-0}, the kernel of the free resolvent expands as
	\[
	R_0^{\pm}(\lambda^2)(x,y) = \pm\frac{i}{4} H_0^{\pm}(z) 
	= -\frac{1}{2\pi}\log\!\left(\frac{z}{2}\right) \pm \frac{i}{4} - \frac{\gamma}{2\pi} 
	+ \frac{1}{8\pi}\bigl(z^2\log z + \beta_{\pm} z^2\bigr) + r^{\pm}(z),
	\qquad z = \lambda|x-y|,
	\]
	where the remainder satisfies, for every $k\in \mathbb{N}_0$, the bound $\bigl|d^k r^{\pm}(z)/dz^k\bigr| \lesssim |z|^{3-k+}$ if $|z| \lesssim 1$ and $\bigl|d^k r^{\pm}(z)/dz^k\bigr| \lesssim |z|^{\max\{3-k+,\,-1/2\}}$ if $|z| \gtrsim 1$. Substituting this expansion into the definition of $M^{\pm}(\lambda)$ and estimating the resulting Hilbert--Schmidt norms, we obtain \eqref{eq:exp for M at zero}.
\end{proof}

For $\lambda>0$, we introduce the notation
\begin{equation*}
	\begin{aligned}
		m_{00}^\pm(\lambda)&=(I-S_2)M^{\pm}(\lambda)(I-S_2), 
		& m_{01}^{\pm}(\lambda)&= \lambda^{-1}(I-S_2)M^{\pm}(\lambda)S_2,\\ 
		m_{10}^{\pm}(\lambda)&= \lambda^{-1}S_2M^{\pm}(\lambda)(I-S_2), 
		& m_{11}^\pm(\lambda)&=\lambda^{-2}S_2M^{\pm}(\lambda)S_2.
	\end{aligned}
\end{equation*}
By the definition of $S_2$ in \eqref{eq-def-for-S2}, $(g_0^{\pm}(\lambda)P+T)S_2 = S_2(g_0^{\pm}(\lambda)P+T)=0$; hence \eqref{eq:exp for M at zero} yields the expansions
\begin{equation}\label{eq:asy of m_ij}
	\begin{aligned}
		m_{00}^\pm(\lambda)&= (I-S_2)\bigl(g_0^{\pm}(\lambda)P+T\bigr)(I-S_2) 
		+ \lambda^2(I-S_2)\bigl(g_1^\pm(\lambda)vG_1v + vG_2v\bigr)(I-S_2) 
		+ \tilde{\mathcal{O}}_2(\lambda^{3+}), \\
		m_{11}^\pm(\lambda)&= S_2\bigl(g_1^\pm(\lambda)vG_1v + vG_2v\bigr)S_2 
		+ \tilde{\mathcal{O}}_2(\lambda^{1+}), \\
		m_{01}^{\pm}(\lambda) &= \lambda\,(I-S_2)\bigl( g_1^\pm(\lambda)vG_1v 
		+ vG_2v \bigr)S_2 + \tilde{\mathcal{O}}_2(\lambda^{2+}), \\
		m_{10}^{\pm}(\lambda) &= \lambda \, S_2\bigl(g_1^\pm(\lambda)vG_1v 
		+ vG_2v \bigr)(I-S_2) + \tilde{\mathcal{O}}_2(\lambda^{2+}).
	\end{aligned}
\end{equation}
Accordingly, we decompose
\begin{equation}\label{eq:exp for AMA^*}
	\sigma_{\lambda}M^{\pm}(\lambda)\sigma_{\lambda}^*
	=\begin{pmatrix} 
		m_{00}^\pm(\lambda)  & m_{01}^\pm(\lambda) \\[2pt] 
		m_{10}^\pm(\lambda) & m_{11}^\pm(\lambda)
	\end{pmatrix}
	=:\Lambda^{\pm}(\lambda) + L^{\pm}(\lambda),
\end{equation}
where $\Lambda^{\pm}(\lambda)$ and $L^{\pm}(\lambda)$ denote the diagonal and off-diagonal parts, respectively. By \eqref{eq:asy of m_ij}, $L^{\pm}(\lambda)$ is a higher-order perturbation of $\Lambda^{\pm}(\lambda)$; we therefore first invert $\Lambda^{\pm}(\lambda)$ and then recover the full inverse by perturbation theory.

\begin{lemma}\label{lem:inve of Lambda}
Assume that $|V(x)|\lesssim \langle x\rangle^{-8-}$ and that zero is either a second-kind threshold singularity or an eigenvalue of $H$. Then there exists a sufficiently small $\lambda_1>0$ such that, for all $\lambda\in (0,\lambda_1)$, the operator $\Lambda^{\pm}(\lambda)$ is invertible on $(I-S_2)L^2\oplus S_2L^2$ and
\begin{equation*}
\Lambda^{\pm}(\lambda)^{-1}=
\begin{pmatrix} 
m_{00}^\pm(\lambda)^{-1}  & 0 \\[2mm] 
0 & m_{11}^\pm(\lambda)^{-1}
\end{pmatrix},
\end{equation*}
	where
	\begin{equation}\label{eq:inverse of m00} 
		m_{00}^\pm(\lambda)^{-1} = 
		\begin{cases}
			h_{\pm}(\lambda)^{-1}N + QD_{0}Q +\tilde{\mathcal{O}}_2(\lambda^{3/2}), 
			& \text{if no s-wave is present}, \\[3mm]
			\begin{aligned}[t]
				&-\bigl(h_{\pm}(\lambda)+N\bigr)(S_{1}-S_2)D_{1}(S_{1}-S_2) 
				- (S_{1}-S_2)D_{1}(S_{1}-S_2)N \\
				&\qquad + QD_{0}Q +\tilde{\mathcal{O}}_2(\lambda^{3/2}),
			\end{aligned}
			& \text{if an s-wave is present},
		\end{cases}
	\end{equation}
	and
	\begin{equation}\label{eq:inverse of m11}
		m_{11}^\pm(\lambda)^{-1}=
		\begin{cases}
			S_3 D_3 S_3+\tilde{\mathcal{O}}_2(\lambda^{1+}), 
			& \text{if no p-wave is present},\\[3mm]
			S_3 D_3 S_3 + g_1^\pm(\lambda)^{-1}N_1 
			+ \tilde{\mathcal{O}}_2(g_1^\pm(\lambda)^{-2}), 
			& \text{if a p-wave is present}.
		\end{cases}
	\end{equation}
	Here the operators $N$ and $N_1$ are defined in \eqref{eq:def of K} and \eqref{eq:def of K_1}, respectively.
\end{lemma}

\begin{proof}
	Let $D_1 = \bigl((S_1-S_2)TPT(S_1-S_2)\bigr)^{-1} = \bigl(S_1TPTS_1\bigr)^{-1}$ and let $N$ be given by \eqref{eq:def of K}. Applying Lemma~\ref{thm-M inverse-even} with $S_1$ replaced by $S_1-S_2$, we find that $m_{00}^\pm(\lambda)$ is invertible on $(I-S_2)L^2$ for all sufficiently small $\lambda$, with inverse given by \eqref{eq-M expansion-even} or \eqref{eq-M expansion-even-1-1}, accordingly. This proves \eqref{eq:inverse of m00}.
	
	We next turn to $m_{11}^\pm(\lambda)$ on $S_2L^2$, and first treat the case in which p-wave resonances and zero-energy eigenfunctions are simultaneously present, i.e. $S_2 \neq S_3$ and $S_3 \neq 0$; the two degenerate cases are settled at the end. Introduce the isomorphism
	\[
	\sigma_{1,1}\colon S_2L^2 \to (S_2-S_3)L^2 \oplus S_3L^2, 
	\qquad \sigma_{1,1}\varphi := \bigl((S_2 - S_3)\varphi, \, S_3\varphi\bigr). 
	\]
	By \eqref{eq:asy of m_ij}, the invertibility of $m_{11}^\pm(\lambda)$ is governed by its leading term $S_2(g_1^\pm(\lambda)vG_1v + vG_2v)S_2$, which is invertible on $S_2L^2$ if and only if the block matrix
	\begin{align*} 
		\mathfrak{M}^\pm(\lambda)
		:=&\sigma_{1,1}\bigl(S_2(g_1^\pm(\lambda)vG_1v + vG_2v)S_2\bigr)\sigma_{1,1}^* \\
		=&
		\begin{pmatrix} 
			(S_2-S_3)\bigl(g_1^\pm(\lambda)vG_1v + vG_2v\bigr)(S_2-S_3) 
			& (S_2-S_3)vG_2vS_3 \\[2mm] 
			S_3vG_2v(S_2-S_3) & S_3vG_2vS_3
		\end{pmatrix}
	\end{align*}
	is invertible on $(S_2-S_3)L^2 \oplus S_3L^2$; here the second equality follows from $vG_1vS_3 = S_3vG_1v=0$, a consequence of the definition of $S_3$. Write $\mathfrak{M}^{\pm}(\lambda)=(a_{ij})_{i,j=1}^2$. Since $a_{22}=S_3vG_2vS_3$ is invertible on $S_3L^2$ (see \cite[Lemma 5.4]{Erd-Gre-13} or \cite{JN01}), Lemma~\ref{lemma-Feshbach formula}, applied with the roles of the two diagonal blocks interchanged, shows that the invertibility of $\mathfrak{M}^{\pm}(\lambda)$ is equivalent to the invertibility, on $(S_2-S_3)L^2$, of the Schur complement
	\begin{align*}
		d &= a_{11} - a_{12}a_{22}^{-1}a_{21} \\
		&=(S_2-S_3)\bigl(g_1^\pm(\lambda)vG_1v + vG_2v\bigr)(S_2-S_3) 
		- (S_2-S_3)vG_2vS_3D_3 S_3vG_2v(S_2-S_3),
	\end{align*}
	where $D_3 := (S_3vG_2vS_3)^{-1}$, in which case
	\begin{equation*}
		\mathfrak{M}^\pm(\lambda)^{-1}=
		\begin{pmatrix}
			d^{-1} & -d^{-1} a_{12} a_{22}^{-1} \\[2pt]
			-a_{22}^{-1} a_{21} d^{-1} & a_{22}^{-1} a_{21} d^{-1} a_{12} a_{22}^{-1} + a_{22}^{-1}
		\end{pmatrix}. 
	\end{equation*}
	It therefore remains to invert $d$. Since $vG_1v$ is self-adjoint and $S_3L^2 = \ker\bigl(S_2vG_1vS_2\bigr)$, the operator $(S_2-S_3)vG_1v(S_2-S_3)$ is invertible on $(S_2-S_3)L^2$, so the leading term of $d$ is $g_1^\pm(\lambda)(S_2-S_3)vG_1v(S_2-S_3)$; as $g_1^\pm(\lambda) = \log\lambda + \beta_\pm$ with $\beta_\pm\in \mathbb{C}\setminus \mathbb{R}$, a Neumann series argument yields
	\[
	d^{-1} = g_1^\pm(\lambda)^{-1} D_2 + \tilde{\mathcal{O}}_2(g_1^\pm(\lambda)^{-2})
	\]
	for all sufficiently small $\lambda$. Inserting this expansion into the formula above and conjugating the result back to $S_2L^2$ by $\sigma_{1,1}$, a direct computation gives
	\begin{equation*}
		\bigl(S_2(g_1^\pm(\lambda)vG_1v + vG_2v)S_2\bigr)^{-1}
		= S_3 D_3 S_3 + g_1^\pm(\lambda)^{-1}N_1 
		+ \tilde{\mathcal{O}}_2(g_1^\pm(\lambda)^{-2}),
	\end{equation*}
	where $N_1$ is given by \eqref{eq:def of K_1}. In view of \eqref{eq:asy of m_ij}, another Neumann series argument, after shrinking $\lambda_1$ if necessary, then yields the second branch of \eqref{eq:inverse of m11}.
	
	It remains to consider the two degenerate cases. If $H$ has p-wave resonances but no zero-energy eigenfunctions, then $S_2\neq0$ and $S_3 = 0$, so $S_2vG_1vS_2$ is invertible on $S_2L^2$ since $S_3L^2=\ker\bigl(S_2vG_1vS_2\bigr)$ by \eqref{eq-def-for-S3}; the same argument gives $m_{11}^{\pm}(\lambda)^{-1} = g_1^\pm(\lambda)^{-1}S_2D_2S_2 + \tilde{\mathcal{O}}_2(g_1^\pm(\lambda)^{-2})$, which again matches the second branch of \eqref{eq:inverse of m11}, as $S_3 D_3 S_3 = 0$ and $N_1= S_2 D_2 S_2$ in this case by \eqref{eq:def of K_1}. If $H$ has zero-energy eigenfunctions but no p-wave resonances, then $S_2=S_3\neq0$, and \eqref{eq:asy of m_ij} reduces to $m_{11}^{\pm}(\lambda)=S_3vG_2vS_3 + \tilde{\mathcal{O}}_2(\lambda^{1+})$; since $S_3vG_2vS_3$ is invertible on $S_3L^2$, we obtain $m_{11}^{\pm}(\lambda)^{-1} = S_3D_3S_3 +\tilde{\mathcal{O}}_2(\lambda^{1+})$, i.e. the first branch of \eqref{eq:inverse of m11}. This establishes \eqref{eq:inverse of m11} in all cases and completes the proof of the lemma.
\end{proof}

\begin{proof}[Completion of the proof of Lemma~\ref{thm-M inverse-even-1}]
	By \eqref{eq:asy of m_ij} and Lemma~\ref{lem:inve of Lambda}, $\Lambda^\pm(\lambda)^{-1}L^{\pm}(\lambda)=\tilde{\mathcal{O}}_2(\lambda^{1-})$; hence, in view of \eqref{eq:exp for AMA^*}, the Neumann series converges, and there exists $\lambda_0 > 0$ such that, for all $\lambda\in (0, \lambda_0)$, the operator $\sigma_{\lambda}M^{\pm}(\lambda)\sigma_{\lambda}^*$ is invertible on $(I - S_2)L^2 \oplus S_2L^2$ with
	\begin{equation}\label{eq:exp inverse of AmA*}
		\begin{aligned}
			\bigl(\sigma_{\lambda}M^{\pm}(\lambda)\sigma_{\lambda}^*\bigr)^{-1}
			&= \Lambda^\pm(\lambda)^{-1} 
			-\Lambda^\pm(\lambda)^{-1} L^{\pm}(\lambda) \Lambda^\pm(\lambda)^{-1}
			+\tilde{\mathcal{O}}_2(\lambda^{2-}) \\
			&= 
			\begin{pmatrix} 
				m_{00}^\pm(\lambda)^{-1}  & -\ell_{01}^\pm(\lambda) \\[2mm] 
				-\ell_{10}^\pm(\lambda) & m_{11}^\pm(\lambda)^{-1}
			\end{pmatrix}
			+ \tilde{\mathcal{O}}_2(\lambda^{2-}),
		\end{aligned}
	\end{equation}
	where
	\[
	\ell_{01}^\pm(\lambda) 
	= m_{00}^{\pm}(\lambda)^{-1} m_{01}^{\pm}(\lambda)\, 
	m_{11}^{\pm}(\lambda)^{-1}, \qquad
	\ell_{10}^\pm(\lambda) 
	= m_{11}^{\pm}(\lambda)^{-1} m_{10}^{\pm}(\lambda)\, 
	m_{00}^{\pm}(\lambda)^{-1}.    
	\]
	Accordingly, we set
	\begin{equation}\label{eq:def of mathcal L}
		\begin{cases}
			\mathcal{L}_{00}^{\pm}(\lambda)
			= m_{00}^{\pm}(\lambda)^{-1}+\tilde{\mathcal{O}}_2(\lambda^{2-}),  \quad
			\mathcal{L}_{01}^{\pm}(\lambda)
			= \dfrac{1}{\lambda}\left(-\ell_{01}^\pm(\lambda)
			+\tilde{\mathcal{O}}_2(\lambda^{2-})\right), 
			\\[4mm]
			\mathcal{L}_{10}^{\pm}(\lambda)
			= \dfrac{1}{\lambda}\left(-\ell_{10}^\pm(\lambda)
			+\tilde{\mathcal{O}}_2(\lambda^{2-})\right), \quad
			\mathcal{L}_{11}^{\pm}(\lambda)
			= \dfrac{1}{\lambda^2}\left(m_{11}^\pm(\lambda)^{-1}
			+\tilde{\mathcal{O}}_2(\lambda^{2-})\right).	
		\end{cases}
	\end{equation}
	The identity \eqref{eq:exp inverse of AmA*}, together with \eqref{eq:iso to M inverse} and Lemma~\ref{lem:inve of Lambda}, yields the desired expansions \eqref{eq-M expansion-even-2} and \eqref{eq:mathcalL00}; moreover, \eqref{eq:inverse of m11} implies \eqref{eq:mathcalL11}. It remains to prove \eqref{eq:mathcalL01 10} and \eqref{eq:special mathcalL10}.
	
	\par\medskip
	\noindent\textbf{Proof of \eqref{eq:mathcalL01 10}.} From \eqref{eq:asy of m_ij} and Lemma~\ref{lem:inve of Lambda} we read off
	\[
	m_{01}^{\pm}(\lambda),\, m_{10}^{\pm}(\lambda)
	=\tilde{\mathcal{O}}_2\bigl(\lambda g_1^{\pm}(\lambda)\bigr), 
	\qquad
	m_{00}^{\pm}(\lambda)^{-1}=\tilde{\mathcal{O}}_2(\log\lambda),
	\qquad
	m_{11}^{\pm}(\lambda)^{-1}
	=S_3D_3S_3+\tilde{\mathcal{O}}_2\bigl(g_1^{\pm}(\lambda)^{-1}\bigr). 	
	\]
	Since $g_1^{\pm}(\lambda)\sim \log\lambda$ for $\lambda\in (0,\lambda_0)$, combining these estimates gives
	\begin{equation*}
		\ell_{01}^\pm(\lambda),\, \ell_{10}^\pm(\lambda)
		= \tilde{\mathcal{O}}_2\bigl(\lambda\log^{2}\lambda\bigr),
	\end{equation*}
	which proves \eqref{eq:mathcalL01 10}.
	
	\par\medskip
	\noindent\textbf{Proof of \eqref{eq:special mathcalL10}.} Assume now that $S_1-S_2\neq0$ and $S_2=S_3$; in this case the rough bound above no longer identifies the leading term of $\mathcal{L}_{10}^{\pm}(\lambda)$, and a more precise expansion of $\ell_{10}^\pm(\lambda)$ is required. Substituting \eqref{eq:asy of m_ij} and the relevant branches of \eqref{eq:inverse of m00}--\eqref{eq:inverse of m11} (s-wave present, no p-wave) into the definition of $\ell_{10}^\pm(\lambda)$, and using $S_1vG_1vS_3=S_3vG_1vS_1=0$ from \eqref{eq-def-for-QS}--\eqref{eq:property of S-3}, we obtain
	\begin{equation*}
		\begin{aligned}
			\ell_{10}^\pm(\lambda)
			&= m_{11}^{\pm}(\lambda)^{-1}m_{10}^{\pm}(\lambda)
			\bigl[-h_{\pm}(\lambda)-N\bigr](S_1-S_2)D_{1}(S_{1}-S_2)\\
			&\quad +m_{11}^{\pm}(\lambda)^{-1}m_{10}^{\pm}(\lambda) QD_{0}Q
			+\lambda S_3D_3S_3vG_2v(S_{1}-S_2)D_{1}(S_{1}-S_2)N
			+\tilde{\mathcal{O}}_2(\lambda^{2}).
		\end{aligned}
	\end{equation*}
	Since $m_{11}^{\pm}(\lambda)^{-1}m_{10}^{\pm}(\lambda)=\tilde{\mathcal{O}}_2(\lambda\log\lambda)$ and $h_{\pm}(\lambda)\sim\log\lambda$, the first two terms on the right-hand side are $\tilde{\mathcal{O}}_2(\lambda\log^2\lambda)(S_1-S_2)$ and $\tilde{\mathcal{O}}_2(\lambda\log\lambda)Q$, respectively. For the last term, the definition of $N$ in \eqref{eq:def of K} and the fact that $(S_{1}-S_2)P=0$ give
	\begin{align*}
	 \lambda S_3D_3S_3vG_2v(S_{1}-S_2)D_{1}(S_{1}-S_2)N
	=& \lambda\,\Xi P 
	+ \lambda S_3D_3S_3vG_2v(S_{1}-S_2)D_{1}(S_{1}-S_2)D_0QTPTQD_0Q \\
	= &\lambda\,\Xi P+\tilde{\mathcal{O}}_2(\lambda)Q,   
	\end{align*}
	with $\Xi :=-S_3D_3S_3vG_2v(S_{1}-S_2)D_{1}(S_{1}-S_2)D_0QT$. Since $S_1, S_2 \leq Q$, the factor $(S_1-S_2)$ on the right of the first term may be replaced by $Q$; collecting the above estimates, we arrive at
	\[
	\ell_{10}^\pm(\lambda)
	=\tilde{\mathcal{O}}_2(\lambda\log^2\lambda)Q+\lambda\,\Xi P
	+\tilde{\mathcal{O}}_2(\lambda^{2}).
	\]
	In view of \eqref{eq:def of mathcal L}, this yields \eqref{eq:special mathcalL10} and completes the proof of Lemma~\ref{thm-M inverse-even-1}.
\end{proof}

\section{Proofs of Lemmas in Section~\ref{sec:Pre}}\label{sec:pro for lemm in sec2}
This appendix contains the proofs of Lemmas~\ref{lem-oscillatory estimate}--\ref{lem-Fou-multiplier} stated in Section~\ref{sec:Pre}.

\begin{proof}[Proof of Lemma \ref{lem-oscillatory estimate}]
	For $|x|\le 1$, the estimate \eqref{eq.oscillatory-est} follows from the trivial bound $|I(x)|\lesssim 1$, since the right-hand side of \eqref{eq.oscillatory-est} is bounded from below by a positive constant in this regime.
	For $|x|>1$, we split the integral in \eqref{eq.oscillatory-est} at the scale $\lambda\sim\langle x\rangle^{-1}$ and write $I(x)=I_{1}(x)+I_{2}(x)$, where
	\begin{align*}
		I_{1}(x)=\int_0^\infty e^{i\lambda x}m(\lambda)\chi(\lambda/\lambda_0)\chi(\lambda\langle x\rangle)\,d\lambda, \quad
		I_{2}(x)=\int_0^\infty e^{i\lambda x}m(\lambda)\chi(\lambda/\lambda_0)\tilde{\chi}(\lambda\langle x\rangle)\,d\lambda.
	\end{align*}
	On the support of the integrand of $I_1$ one has $\lambda\lesssim\langle x\rangle^{-1}$; the bound $|m(\lambda)|\lesssim|\log\lambda|^{-a}\lambda^{b}$ and the change of variables $\mu=\lambda\langle x\rangle$ therefore give
	\begin{equation*}
		\left| I_{1}(x)\right| \lesssim \int_0^{\min\{\langle x\rangle^{-1},\lambda_0\}} \left| \log\lambda\right|^{-a} \lambda^{b}\,d\lambda \lesssim \left( 1+\log \langle x\rangle \right)^{-a} \langle x\rangle^{-(b+1)}.
	\end{equation*}
	For $I_2$, integrating by parts $k=\lfloor b\rfloor+2$ times, using the derivative bounds $|m^{(j)}(\lambda)|\lesssim|\log\lambda|^{-a}\lambda^{b-j}$ and noting that all boundary terms vanish, we obtain
	\begin{equation*}
		\left| I_{2}(x)\right| \lesssim \langle x\rangle^{-k} \int_{\frac 12\langle x\rangle^{-1}}^{\lambda_0} \left| \log\lambda\right|^{-a} \lambda^{b-k}\,d\lambda \lesssim \left(1+ \log\langle x\rangle \right)^{-a} \langle x\rangle^{-(b+1)},
	\end{equation*}
	where the last inequality uses $b-k<-1$. Combining these estimates yields \eqref{eq.oscillatory-est}.
\end{proof}

\begin{proof}[Proof of Lemma \ref{lem-oscillatory estimate-h}]
	We again split the integral in \eqref{eq.oscillatory-est-h} at the scale $\lambda\sim\langle x\rangle^{-1}$:
	\begin{align*}
		I_{1}(x)=\int_0^\infty e^{i\lambda x}m(\lambda)\tilde{\chi}(\lambda/\lambda_0)\chi(\lambda\langle x\rangle)\,d\lambda, \quad
		I_{2}(x)=\int_0^\infty e^{i\lambda x}m(\lambda)\tilde{\chi}(\lambda/\lambda_0)\tilde{\chi}(\lambda\langle x\rangle)\,d\lambda.
	\end{align*}
	The two terms are then estimated exactly as in the proof of Lemma~\ref{lem-oscillatory estimate}: $I_1$ by the support property and a change of variables, and $I_2$ by repeated integration by parts. We omit the details.
\end{proof}

\begin{proof}[Proof of Lemma \ref{eq-fourier-mulitiplier-1}]
	Let $K$ denote the convolution kernel,
	\[
	K(x):=\bigl(m(|\xi|)\chi(|\xi|/\lambda_0)\bigr)^{\vee}(x), \qquad x\in\mathbb{R}^n.
	\]
	We claim that
	\begin{equation}\label{eq-fou-L1-0}
		K \in L^1(\mathbb{R}^n).
	\end{equation}
	The $L^p$-boundedness for $1 \leq p \leq \infty$ then follows from Young's inequality. For the $\mathcal{H}^1$-boundedness, we use the characterization
	\[
	\|f\|_{\mathcal{H}^1} \sim \|f\|_{L^1} + \sum_{k=1}^n \|\mathbf{R}_k f\|_{L^1} \quad (n \geq 2), \qquad \|f\|_{\mathcal{H}^1} \sim \|f\|_{L^1} + \|\mathbf{H} f\|_{L^1} \quad (n=1),
	\]
	where $\mathbf{R}_k$ denotes the $k$-th Riesz transform and $\mathbf{H}$ the Hilbert transform (see \cite[Theorem~2.4.6]{graf2}). Since $m(\sqrt{-\Delta})\chi(\sqrt{-\Delta}/\lambda_0)$ commutes with the Riesz transforms (with the Hilbert transform when $n=1$), its $\mathcal{H}^1$-boundedness is immediate from \eqref{eq-fou-L1-0}.
	
	It remains to prove \eqref{eq-fou-L1-0}. By the Hankel inversion formula for radial functions,
	\[
	K(x) = \frac{1}{(2\pi)^{\frac{n}{2}}|x|^{\frac{n}{2}-1}} \int_0^\infty m(\lambda)\chi(\lambda/\lambda_0)\lambda^{\frac{n}{2}}J_{\frac{n}{2}-1}(\lambda|x|)\,d\lambda, \qquad x\in \mathbb{R}^n,
	\]
	where $J_{n/2-1}$ is the Bessel function of the first kind. Inserting its asymptotic representation
	\begin{equation}\label{eq:dec of Bessel J}
		J_{\frac{n}{2}-1}(z) = e^{iz}G_{\frac{n}{2}-1}^+(z)+e^{-iz}G_{\frac{n}{2}-1}^-(z), \quad z\in(0,\infty),
	\end{equation}
	where $G_{\frac{n}{2}-1}^\pm(z)=\mathcal{O}_{\infty}(z^{-1/2})$, and applying Lemma~\ref{lem-oscillatory estimate} to each of the resulting two integrals, we obtain
	\[
	\bigl|K(x) \bigr| \lesssim \bigl(1 + \log\langle x\rangle\bigr)^{-a} |x|^{-\frac{n-1}{2}} \langle x\rangle^{-\frac{n+1}{2}},
	\]
	which implies \eqref{eq-fou-L1-0}.
\end{proof}

\begin{proof}[Proof of Lemma \ref{lem-Fou-multiplier}]
	Let
	\[
	h_a(x)=\mathcal{F}^{-1}\!\left((\log|\cdot|)^{-a}\chi\!\left(\tfrac{|\cdot|}{\lambda_0}\right)\right)(x).
	\]
	The key step is to establish the pointwise decay
	\begin{equation}\label{eq:bound of inv}
		h_a(x)
		= \mathcal{O}\!\left(\frac{1}{|x|\log^{a+1}|x|}\right), \quad |x|\gg 1.
	\end{equation}
	
	We begin with the case $n\ge 2$. By the Hankel inversion formula,
	\begin{equation}\label{eq:inver by Bessel}
		h_a(x)=\frac{1}{(2\pi)^{\frac{n}{2}}|x|^{\frac{n}{2}-1}} \int_0^\infty \frac{\chi(\lambda/\lambda_0)}{\log^a \lambda}\,\lambda^{\frac{n}{2}}J_{\frac{n}{2}-1}(\lambda|x|)\,d\lambda.
	\end{equation}
	The Bessel recurrence relation
	$\frac{d}{dt}\bigl[t^{\nu}J_{\nu}(t)\bigr]=t^{\nu}J_{\nu-1}(t)$
	with $\nu=\frac n2$ gives the identity
	\[
	\lambda^{\frac{n}{2}}J_{\frac{n}{2}-1}(\lambda|x|)
	= \frac{1}{|x|}\,\frac{d}{d\lambda}\Bigl[\lambda^{\frac{n}{2}}J_{\frac{n}{2}}(\lambda|x|)\Bigr],
	\]
	with which we integrate by parts in \eqref{eq:inver by Bessel} and obtain
	\begin{align*}
		h_a(x)
		&=\frac{1}{(2\pi)^{\frac{n}{2}}|x|^{\frac{n}{2}}} \int_0^\infty \frac{\chi(\lambda/\lambda_0)}{\log^a \lambda}\,
		\frac{d}{d\lambda}\Bigl[\lambda^{\frac{n}{2}}J_{\frac{n}{2}}(\lambda|x|)\Bigr]\,d\lambda \\
		&=\frac{1}{(2\pi)^{\frac{n}{2}}|x|^{\frac{n}{2}}}
		\int_0^\infty \left(\frac{a\,\chi(\lambda/\lambda_0)}{\lambda\log^{a+1}\lambda}
		-\frac{\chi'(\lambda/\lambda_0)}{\lambda_0\log^a\lambda}\right)
		\lambda^{\frac{n}{2}}J_{\frac{n}{2}}(\lambda|x|)\,d\lambda \\
		&=:I_1^n(x)-I_2^n(x),
	\end{align*}
	where the boundary terms vanish: at $\lambda=0$ because $\lambda^{\frac{n}{2}}J_{\frac{n}{2}}(\lambda|x|)=\mathcal{O}(\lambda^{n})$, and at the upper end because $\chi(\lambda/\lambda_0)$ vanishes for $\lambda\ge\lambda_0$.
	
	The integrand of $I_2^n(x)$ is smooth and supported in $[\lambda_0/2,\lambda_0]$, away from the origin; repeated integration by parts therefore gives
	\[
	I_2^n(x)=\mathcal{O}(|x|^{-N}) \quad\text{for any $N$ as $|x|\to\infty$.}
	\]
	For $I_1^n(x)$, we insert the asymptotics \eqref{eq:dec of Bessel J} of $J_{\frac n2}$ and apply Lemma~\ref{lem-oscillatory estimate} (with $a$ replaced by $a+1$ and $b=\frac{n-1}{2}$) to deduce
	\[
	I_1^n(x)=\mathcal{O}\!\left(\frac{1}{|x|^{n}\log^{a+1}|x|}\right) \quad\text{for $|x|\gg 1$.}
	\]
	Combining the estimates for $I_1^n$ and $I_2^n$ yields \eqref{eq:bound of inv} for all $n\ge 2$.
	
	\medskip
	It remains to treat the one-dimensional case. We first reduce to an even cutoff: let $\psi\in C_c^\infty(\R)$ be even with $\psi(\lambda)=1$ for $|\lambda|\le\frac14$, and decompose
	\[
	\chi\!\left({\lambda}/{\lambda_0}\right)
	=\chi\!\left({\lambda}/{\lambda_0}\right)\left(1-\psi\!\left({10\lambda}/{\lambda_0}\right)\right)
	+ \psi\!\left({10\lambda}/{\lambda_0}\right),
	\]
	which is valid because $\chi(\lambda/\lambda_0)=1$ on the support of $\psi(10\lambda/\lambda_0)$. The first summand is smooth and compactly supported away from the origin, where the factor $(\log|\xi|)^{-a}$ is smooth; its inverse Fourier transform is therefore a Schwartz function. It thus suffices to treat the second summand, that is, to assume that $\chi$ is even. In that case
	\[
	h_a(x)=\frac{1}{\pi} \int_0^\infty \frac{\chi(\lambda/\lambda_0)}{\log^a \lambda}\cos(\lambda x)\,d\lambda.
	\]
	Integrating by parts once gives
	\[
	h_a(x)=\frac{1}{\pi x} \int_0^\infty \left(\frac{a\,\chi(\lambda/\lambda_0)}{\lambda\log^{a+1}\lambda}
	-\frac{\chi'(\lambda/\lambda_0)}{\lambda_0\log^{a}\lambda}\right)\sin(\lambda x)\,d\lambda
	=:I_{1}^1(x)-I_{2}^1(x).
	\]
	As before, $I_{2}^1(x)=\mathcal{O}(|x|^{-N})$ for any $N$ as $|x|\to\infty$. For $I_{1}^1(x)$, we split the integral at $\lambda=|x|^{-1}$:
	\[
	I_{1}^1(x)=\frac{a}{\pi x}\int_0^{\frac{1}{|x|}} \frac{\chi(\lambda/\lambda_0)}{\lambda\log^{a+1}\lambda}\sin(\lambda x)\,d\lambda
	+ \frac{a}{\pi x}\int_{\frac{1}{|x|}}^{\infty} \frac{\chi(\lambda/\lambda_0)}{\lambda\log^{a+1}\lambda}\sin(\lambda x)\,d\lambda.
	\]
	In the first piece, using $|\sin(\lambda x)|\lesssim \lambda|x|$ and the monotonicity of $\log^{-(a+1)}\!(\lambda)$ near the origin, we obtain
	\[
	\frac{a}{\pi x}\int_0^{\frac{1}{|x|}} \frac{\chi(\lambda/\lambda_0)}{\lambda\log^{a+1}(\lambda)}\sin(\lambda x)\,d\lambda
	=\mathcal{O}\!\left(\frac{1}{|x|\log^{a+1}(|x|)}\right).
	\]
	For the second piece, one more integration by parts yields the same bound:
	\[
	\frac{a}{\pi x}\int_{\frac{1}{|x|}}^{\infty} \frac{\chi(\lambda/\lambda_0)}{\lambda\log^{a+1}(\lambda)}\sin(\lambda x)\,d\lambda
	=\mathcal{O}\!\left(\frac{1}{|x|\log^{a+1}|x|}\right).
	\]
	Thus \eqref{eq:bound of inv} also holds for $n=1$.
	
	\medskip
	Finally, $h_a\in L^1(\mathbb{R}^n)$: the elementary bound $|h_a(x)|\lesssim 1$ controls the region $|x|\lesssim1$, while \eqref{eq:bound of inv} controls $|x|\gg1$. The $L^p$-boundedness of $(\log\sqrt{-\Delta})^{-a}\chi(\sqrt{-\Delta}/\lambda_0)$ for $1\le p\le\infty$ now follows from Young's inequality, and the endpoint bounds on $\mathcal{H}^1$ and $\BMO$ follow from the same commutation argument as in the proof of Lemma~\ref{eq-fourier-mulitiplier-1}. This completes the proof.
\end{proof}

{\bf Acknowledgements:}    This Project was supported by the National Key R\&D program of China: No. 2021YFA1002500.
    H. Cheng was supported by CPSF under grant 2025M784444. C. Miao was supported by NSFC under grants 12371095 and 12531005.
   X. Yao was supported by NSFC (Grants No.~12531005  and 12671120).


\begin{thebibliography}{BAKS00}
	
\bibitem{Abra}
M.~Abramowitz and I.~A.~Stegun, \emph{Handbook of Mathematical Functions with Formulas, Graphs and Mathematical Tables}, U.S. Government Printing Office, Washington, DC, 1965.

\bibitem{Agmon}
S.~Agmon, \emph{Spectral properties of Schr\"odinger operators and scattering theory}, Ann. Scuola Norm. Sup. Pisa Cl. Sci. (4) \textbf{2} (1975), no.~2, 151--218.

\bibitem{CSWY25}
H.~Cheng, A.~Soffer, Z.~Wu and X.~Yao, \emph{The $L^p$-boundedness of wave operators for higher order Schr\"odinger operator in low odd dimensions}, arXiv:2505.07009, 2025.

\bibitem{AF06}
P.~D'Ancona and L.~Fanelli, \emph{$L^p$-boundedness of the wave operator for the one dimensional Schr\"odinger operator}, Comm. Math. Phys. \textbf{268} (2006), no.~2, 415--438.

\bibitem{Beceanu-Schlag-20}
M.~Beceanu and W.~Schlag, \emph{Structure formulas for wave operators}, Amer. J. Math. \textbf{142} (2020), no.~3, 751--807.

\bibitem{Chen-Pusat}
G.~Chen and F.~Pusateri, \emph{On the 1d cubic NLS with a non-generic potential}, Comm. Math. Phys. \textbf{405} (2024), no.~2, Paper No.~35, 59 pp.

\bibitem{Donninger-Krie}
R.~Donninger and J.~Krieger, \emph{A vector field method on the distorted Fourier side and decay for wave equations with potentials}, Mem. Amer. Math. Soc. \textbf{241} (2016), no.~1142.

\bibitem{Deift-T}
P.~Deift and E.~Trubowitz, \emph{Inverse scattering on the line}, Comm. Pure Appl. Math. \textbf{32} (1979), 121--251.

\bibitem{Erdogan-Goldberg-Green-JFA-2018}
M.~B.~Erdo\u{g}an, M.~Goldberg and W.~R.~Green, \emph{On the $L^p$ boundedness of wave operators for two-dimensional Schr\"odinger operators with threshold obstructions}, J. Funct. Anal. \textbf{274} (2018), 2139--2161.

\bibitem{EG22}
M.~B.~Erdo\u{g}an and W.~R.~Green, \emph{The $L^p$-continuity of wave operators for higher order Schr\"odinger operators}, Adv. Math. \textbf{404} (2022), Paper No.~108450, 41 pp.

\bibitem{EG23}
M.~B.~Erdo\u{g}an and W.~R.~Green, \emph{The $L^p$-continuity of wave operators for higher order Schr\"odinger operators}, J. Differential Equations \textbf{355} (2023), 144--161.

\bibitem{Erd-Gre-13}
M.~B.~Erdo\u{g}an and W.~R.~Green, \emph{Dispersive estimates for Schr\"odinger operators in dimension two with obstructions at zero energy}, Trans. Amer. Math. Soc. \textbf{365} (2013), 6403--6440.

\bibitem{Erd-Gre-13-cmp}
M.~B.~Erdo\u{g}an and W.~R.~Green, \emph{A weighted dispersive estimate for Schr\"odinger operators in dimension two}, Comm. Math. Phys. \textbf{319} (2013), no.~3, 791--811.

\bibitem{Erd-Gre-arxiv}
M.~B.~Erdo\u{g}an, W.~R.~Green and K.~Lamaster, \emph{$L^p$ boundedness of wave operators for higher order Schr\"odinger operators with threshold eigenvalues}, Trans. Amer. Math. Soc., doi:10.1090/tran/9898.

\bibitem{Finco-Yajima-2006-even}
D.~Finco and K.~Yajima, \emph{The $L^p$ boundedness of wave operators for Schr\"odinger operators with threshold singularities II. The even dimensional case}, J. Math. Sci. Univ. Tokyo \textbf{13} (2006), no.~3, 277--346.

\bibitem{Gal-Yaj-00}
A.~Galtbayar and K.~Yajima, \emph{The $L^p$-continuity of wave operators for one dimensional Schr\"odinger operators}, J. Math. Sci. Univ. Tokyo \textbf{7} (2000), no.~2, 221--240.

\bibitem{GY24}
A.~Galtbayar and K.~Yajima, \emph{The $L^p$-boundedness of wave operators for fourth order Schr\"odinger operators on $\mathbf{R}^4$}, J. Spectr. Theory \textbf{14} (2024), no.~1, 271--354.

\bibitem{GY26}
A.~Galtbayar and K.~Yajima, \emph{The $L^p$-boundedness of wave operators for fourth-order Schr\"odinger operators on $\mathbb{R}^2$}, I: Regular case, Rev. Math. Phys. (2026), Paper No.~2650006, 59 pp.

\bibitem{Ger-Pu-Rou}
P.~Germain, F.~Pusateri and F.~Rousset, \emph{The nonlinear Schr\"odinger equation with a potential}, Ann. Inst. H. Poincar\'e Anal. Non Lin\'eaire \textbf{35} (2018), no.~6, 1477--1530.

\bibitem{Goldberg-Green-2016}
M.~Goldberg and W.~Green, \emph{The $L^p$ boundedness of wave operators for Schr\"odinger operators with threshold singularities}, Adv. Math. \textbf{303} (2016), 360--389.

\bibitem{Goldberg-Green-2017}
M.~Goldberg and W.~Green, \emph{On the $L^p$ boundedness of wave operators for four-dimensional Schr\"odinger operators with a threshold eigenvalue}, Ann. Henri Poincar\'e \textbf{18} (2017), 1269--1288.

\bibitem{GG21}
M.~Goldberg and W.~Green, \emph{$L^p$ boundedness of wave operators for fourth order Schr\"odinger operators}, Trans. Amer. Math. Soc. \textbf{374} (2021), 4075--4092.

\bibitem{GV06}
M.~Goldberg and M.~Visan, \emph{A counterexample to dispersive estimates for Schr\"odinger operators in higher dimensions}, Comm. Math. Phys. \textbf{266} (2006), 211--238.

\bibitem{GR02}
I.~S.~Gradshteyn and I.~M.~Ryzhik, \emph{Table of Integrals, Series, and Products}, 6th ed., Academic Press, San Diego, 2002.

\bibitem{graf}
L.~Grafakos, \emph{Classical Fourier Analysis}, 3rd ed., Graduate Texts in Mathematics 249, Springer, New York, 2014.

\bibitem{graf2}
L.~Grafakos, \emph{Modern Fourier Analysis}, 3rd ed., Graduate Texts in Mathematics 250, Springer, New York, 2014.

\bibitem{Schlag-2D-dis}
W.~Schlag, \emph{Dispersive estimates for Schr\"odinger operators in dimension two}, Comm. Math. Phys. \textbf{257} (2005), 87--117.

\bibitem{Huang-Yao-2026}
S.~Huang and X.~Yao, \emph{Counterexamples to the $L^1$ and $L^\infty$ boundedness of the one-dimensional wave operators}, preprint, arXiv:2606.17898.

\bibitem{JN01}
A.~Jensen and G.~Nenciu, \emph{A unified approach to resolvent expansions at thresholds}, Rev. Math. Phys. \textbf{13} (2001), no.~6, 717--754.

\bibitem{Jensen-Yajima-02-CMP}
A.~Jensen and K.~Yajima, \emph{A remark on $L^p$-boundedness of wave operators for two-dimensional Schr\"odinger operators}, Comm. Math. Phys. \textbf{225} (2002), no.~3, 633--637.

\bibitem{Jensen-Yajima-08-PLMS}
A.~Jensen and K.~Yajima, \emph{On $L^p$ boundedness of wave operators for 4-dimensional Schr\"odinger operators with threshold singularities}, Proc. Lond. Math. Soc. \textbf{96} (2008), no.~1, 136--162.

\bibitem{Kuroda}
S.~T.~Kuroda, \emph{Scattering theory for differential operators. I. Operator theory}, J. Math. Soc. Japan \textbf{25} (1973), no.~1, 75--104.

\bibitem{MWY23}
H.~Mizutani, Z.~Wan and X.~Yao, \emph{$L^p$ boundedness of wave operators for fourth-order Schr\"odinger operators on the line}, Adv. Math. \textbf{451} (2024), Paper No.~109806.

\bibitem{MWY23-1}
H.~Mizutani, Z.~Wan and X.~Yao, \emph{$L^p$ boundedness of wave operators for fourth-order Schr\"odinger operators with zero resonance on $\mathbb{R}^3$}, J. Funct. Anal. \textbf{289} (2025), no.~8, Paper No.~111013, 68 pp.

\bibitem{MWY23-2}
H.~Mizutani, Z.~Wan and X.~Yao, \emph{Counterexamples and weak $(1,1)$ estimates of wave operators for fourth-order Schr\"odinger operators in dimension three}, J. Spectr. Theory \textbf{14} (2024), no.~4, 1409--1450.

\bibitem{Murata-82}
M.~Murata, \emph{Asymptotic expansions in time for solutions of Schr\"odinger-type equations}, J. Funct. Anal. \textbf{49} (1982), no.~1, 10--56.

\bibitem{Reed-Simon-III}
M.~Reed and B.~Simon, \emph{Methods of Modern Mathematical Physics III: Scattering Theory}, Academic Press, New York, 1979.

\bibitem{Pusa-Soffer}
F.~Pusateri and A.~Soffer, \emph{Bilinear estimates in the presence of a large potential and a critical NLS in 3D}, Mem. Amer. Math. Soc. \textbf{299} (2024), no.~1498, 107 pp.

\bibitem{Weder-99-CMP}
R.~Weder, \emph{The $W^{k,p}$-continuity of the Schr\"odinger wave operator on the line}, Comm. Math. Phys. \textbf{208} (1999), no.~2, 507--520.

\bibitem{Yajima-JMSJ-95}
K.~Yajima, \emph{The $W^{k,p}$-continuity of wave operators for Schr\"odinger operators}, J. Math. Soc. Japan \textbf{47} (1995), no.~3, 551--581.

\bibitem{Yajima-1995-even}
K.~Yajima, \emph{The $W^{k,p}$-continuity of wave operators for Schr\"odinger operators. III. Even-dimensional cases $m\ge4$}, J. Math. Sci. Univ. Tokyo \textbf{2} (1995), no.~2, 311--346.

\bibitem{Yajima-99-CMP}
K.~Yajima, \emph{$L^p$-boundedness of wave operators for two-dimensional Schr\"odinger operators}, Comm. Math. Phys. \textbf{208} (1999), no.~3, 125--152.

\bibitem{Yajima-JMSJ-2006}
K.~Yajima, \emph{The $L^p$ boundedness of wave operators for Schr\"odinger operators with threshold singularities I. The odd dimensional case}, J. Math. Sci. Univ. Tokyo \textbf{13} (2006), no.~1, 43--93.

\bibitem{Yajima-2016}
K.~Yajima, \emph{Remark on the $L^p$-boundedness of wave operators for Schr\"odinger operators with threshold singularities}, Doc. Math. \textbf{21} (2016), 391--443.

\bibitem{Yajima-2016-3d}
K.~Yajima, \emph{$L^1$ and $L^\infty$-boundedness of wave operators for three dimensional Schr\"odinger operators with threshold singularities}, Tokyo J. Math. \textbf{41} (2018), no.~2, 385--406.

\bibitem{Yajima-2022}
K.~Yajima, \emph{The $L^p$-boundedness of wave operators for two dimensional Schr\"odinger operators with threshold singularities}, J. Math. Soc. Japan \textbf{74} (2022), no.~4, 1169--1217.

\bibitem{Yajima-2022-4}
K.~Yajima, \emph{The $L^p$-boundedness of wave operators for four dimensional Schr\"odinger operators with threshold resonances}, in \emph{The Physics and Mathematics of Elliott Lieb---The 90th Anniversary. Vol.~II}, 517--563, EMS Press, Berlin, 2022.
\end{thebibliography}
\end{document}